\documentclass[reqno,12pt]{amsart}

\usepackage{amsmath}
\usepackage{amscd}
\usepackage{amssymb}
\usepackage{enumerate}
\usepackage{amsfonts}
\usepackage{graphicx}
\usepackage[all]{xy}
\usepackage{mathrsfs}
\usepackage{hyperref}

\usepackage{bbm}

\newcommand{\e}{\mathbbm{1}}

\newcommand{\limto}{{\displaystyle\lim_{\longrightarrow}}}
\newcommand{\rightlim}{\mathop{\limto}}

\newcommand{\leftlim}{\mathop{\displaystyle\lim_{\longleftarrow}}}
\newcommand{\limfromn}{\leftlim\limits_{\raise3pt\hbox{$n$}}}
\newcommand{\limton}{\rightlim\limits_{\raise3pt\hbox{$n$}}}

\newcommand{\rightlimit}[1]{\mathop{\lim\limits_{\longrightarrow}}\limits%
                   _{\raise3pt\hbox{$\scriptstyle #1$}}}
\newcommand{\leftlimit}[1]{\mathop{\lim\limits_{\longleftarrow}}\limits%
                   _{\raise3pt\hbox{$\scriptstyle #1$}}}

\numberwithin{equation}{section}

\newcommand{\rar}[1]{\stackrel{#1}{\longrightarrow}}
\newcommand{\lar}[1]{\stackrel{#1}{\longleftarrow}}
\newcommand{\xrar}[1]{\xrightarrow{#1}}

\newcommand{\iso}{\buildrel{\sim}\over{\longrightarrow}}

\newcommand{\into}{\hookrightarrow}

\newcommand{\al}{\alpha}
\newcommand{\be}{\beta}
\newcommand{\ga}{\gamma}
\newcommand{\Ga}{\Gamma}
\newcommand{\de}{\delta}
\newcommand{\De}{\Delta}
\newcommand{\la}{\lambda}

\newcommand{\eps}{\epsilon}
\newcommand{\sg}{\sigma}

\newcommand{\te}{\theta}

\newcommand{\Om}{\Omega}

\newcommand{\vp}{\varphi}
\newcommand{\vt}{\vartheta}

\newcommand{\bA}{{\mathbb A}}

\newcommand{\bC}{{\mathbb C}}
\newcommand{\bD}{{\mathbb D}}

\newcommand{\bF}{{\mathbb F}}
\newcommand{\bG}{{\mathbb G}}

\newcommand{\bK}{{\mathbb K}}
\newcommand{\bL}{{\mathbb L}}

\newcommand{\bN}{{\mathbb N}}

\newcommand{\bQ}{{\mathbb Q}}
\newcommand{\bR}{{\mathbb R}}

\newcommand{\bZ}{{\mathbb Z}}

\newcommand{\cA}{{\mathcal A}}

\newcommand{\cC}{{\mathcal C}}
\newcommand{\cD}{{\mathcal D}}
\newcommand{\cE}{{\mathcal E}}
\newcommand{\cF}{{\mathcal F}}
\newcommand{\cG}{{\mathcal G}}

\newcommand{\cK}{{\mathcal K}}
\newcommand{\cL}{{\mathcal L}}
\newcommand{\cM}{{\mathcal M}}
\newcommand{\cN}{{\mathcal N}}
\newcommand{\cO}{{\mathcal O}}

\newcommand{\cS}{{\mathcal S}}

\newcommand{\sC}{{\mathscr C}}
\newcommand{\sD}{{\mathscr D}}
\newcommand{\sE}{{\mathscr E}}
\newcommand{\sF}{{\mathscr F}}

\newcommand{\sL}{{\mathscr L}}
\newcommand{\sM}{{\mathscr M}}

\newcommand{\sP}{{\mathscr P}}

\newcommand{\sS}{{\mathscr S}}

\newcommand{\sX}{{\mathscr X}}
\newcommand{\sY}{{\mathscr Y}}
\newcommand{\sZ}{{\mathscr Z}}

\newcommand{\fA}{{\mathfrak A}}

\newcommand{\fG}{{\mathfrak G}}

\newcommand{\fq}{{\mathfrak q}}

\newcommand{\abs}[1]{\lvert #1\rvert}

\newcommand{\In}{\operatorname{in}}
\newcommand{\out}{\operatorname{out}}
\newcommand{\Ar}{\operatorname{Ar}}
\newcommand{\Isom}{\operatorname{Isom}}
\newcommand{\Ker}{\operatorname{Ker}}
\newcommand{\Cob}{{\text {\bf Cob}}}
\newcommand{\Cobin}{\Cob_{\In}}
\newcommand{\Cobout}{\Cob_{\out}}
\newcommand{\sCob}{{\text {\bf sCob}}}
\newcommand{\sCobin}{\sCob_{\In}}
\newcommand{\sCobout}{\sCob_{\out}}
\newcommand{\Cat}{{\text {\bf Cat}}}

\newcommand{\End}{\operatorname{End}}
\newcommand{\Hom}{\operatorname{Hom}}
\newcommand{\Ext}{\operatorname{Ext}}
\newcommand{\Aut}{\operatorname{Aut}}
\newcommand{\Spec}{\operatorname{Spec}}

\newcommand{\id}{\operatorname{id}}
\newcommand{\Id}{\operatorname{Id}}
\newcommand{\pr}{\mathrm{pr}}

\newcommand{\Ad}{\operatorname{Ad}}
\newcommand{\ev}{\operatorname{ev}}

\newcommand{\Supp}{\operatorname{supp}}

\newcommand{\Ind}{\operatorname{Ind}}
\newcommand{\ind}{\operatorname{ind}}
\newcommand{\Res}{\operatorname{Res}}

\newcommand{\Av}{\operatorname{Av}}
\newcommand{\av}{\operatorname{av}}

\newcommand{\red}{\operatorname{red}}
\newcommand{\Mor}{\sM\!or}

\newcommand{\HOM}{\underline{\operatorname{Hom}}}

\newcommand{\codim}{\operatorname{codim}}

\newcommand{\tens}{\otimes}

\newcommand{\st}{\,\big\vert\,}

\newcommand{\sbr}{\smallbreak}
\newcommand{\mbr}{\medbreak}
\newcommand{\bbr}{\bigbreak}

\newcommand{\can}{\operatorname{can}}

\newtheorem{thm}{Theorem}[section]
\newtheorem{cor}[thm]{Corollary}
\newtheorem{lem}[thm]{Lemma}

\newtheorem{prop}[thm]{Proposition}

\newtheorem{quest}[thm]{Question}

\theoremstyle{remark}
\newtheorem{rem}[thm]{Remark}
\newtheorem{rems}[thm]{Remarks}
\newtheorem{example}[thm]{Example}

\newtheorem{defin}[thm]{Definition}
\newtheorem{predef}[thm]{Pre-definition}
\newtheorem{defins}[thm]{Definitions}
\newtheorem{convention}[thm]{Convention}

\newcommand{\cst}{\bC^\times}

\newcommand{\Fun}{\operatorname{Fun}}

\newcommand{\ql}{\overline{\bQ}_\ell}

\newcommand{\Perv}{\operatorname{Perv}}

\newcommand{\ig}{\ind_{G'}^G}
\newcommand{\Ig}{\Ind_{G'}^G}
\newcommand{\Rg}{\Res_{G'}^G}
\newcommand{\iga}{\ind_{\Ga'}^{\Ga}}

\newcommand{\avg}{\av_{G/G'}}
\newcommand{\Avg}{\Av_{G/G'}}

\newcommand{\igo}{\ind_{G_1}^G}

\newcommand{\pu}{\operatorname{\mathfrak{pu}}}
\newcommand{\puc}{\pu^\circ}
\newcommand{\cpu}{\operatorname{\mathfrak{cpu}}}
\newcommand{\cpuc}{\cpu^\circ}

\newcommand{\qzp}{\bQ_p/\bZ_p}

\newcommand{\Nt}{\widetilde{N}}

\newcommand{\GV}{Grothendieck-Verdier\ }

\title[Character sheaves on unipotent groups]{Character sheaves on unipotent groups in positive characteristic: foundations}

\author{Mitya Boyarchenko \and Vladimir Drinfeld}

\thanks{Both authors were supported by the NSF grant
DMS-0701106. M.B. was also supported by the NSF Postdoctoral Research Fellowship DMS-0703679 and by the NSF grant DMS-1001769. V.D. was also supported by the NSF grant DMS-1001660. \\
\textit{Address}: Department of
Mathematics, University of Chicago, Chicago, IL 60637. \\
\textit{E-mail}: {\tt mitya@math.uchicago.edu} (M.B.), \ {\tt
drinfeld@math.uchicago.edu} (V.D.)}

\begin{document}

\begin{abstract}
In this article we formulate and prove the main theorems of the theory of character sheaves on unipotent groups over an algebraically closed field of characteristic $p>0$. In particular, we show that every admissible pair for such a group $G$ gives rise to an $\bL$-packet of character sheaves on $G$, and that, conversely, every $\bL$-packet of character sheaves on $G$ arises from a (non-unique) admissible pair.

In the appendices we discuss two abstract category theory patterns related to the study of character sheaves. The first appendix sketches a theory of duality for monoidal categories, which generalizes the notion of a rigid monoidal category and is close in spirit to the Grothendieck-Verdier duality theory. In the second one we use a topological field theory approach to define the canonical braided monoidal structure and twist on the equivariant derived category of constructible sheaves on an algebraic group; moreover, we show that this category carries an action of the surface operad. The third appendix proves that the ``naive'' definition of the equivariant $\ell$-adic derived category with respect to a unipotent algebraic group is equivalent to the ``correct'' one.
\end{abstract}

\dedicatory{Dedicated to the memory of our friend Leonid Vaksman}

\maketitle

\setcounter{tocdepth}{1}

\tableofcontents



\section*{Introduction}

In a series of works beginning in the 1980s, George Lusztig
developed a theory of character sheaves for reductive algebraic
groups and explored its relation to the character theory of finite
groups of Lie type. In 2003, he conjectured that there should also
exist an interesting theory of character sheaves for unipotent
groups in positive characteristic, and calculated the first example
of a nontrivial $\bL$-packet in this setting \cite{lusztig}. A general
definition of an $\bL$-packet of character sheaves on a unipotent
group over an algebraically closed field $k$ of characteristic $p>0$
was given in \cite{intro}. We also formulated there a list of
conjectures related to this notion and discussed the orbit method
for unipotent groups of nilpotence class $<p$.

\mbr

%

A brief introduction to the theory of character sheaves on unipotent
groups was given in a joint talk by the authors, the slides for
which are available online \cite{talk}. The present article contains
the proofs of the results announced in Parts I and II of that talk
(Part III was devoted to the relationship between character sheaves
and characters of unipotent groups over finite fields, which is discussed in \cite{characters-char-sheaves}).

\mbr

Let us summarize the main features of the theory to which our work
is devoted. Precise statements of the main results, as well as all
the background definitions, are contained in \S\ref{s:results},
which can be viewed as an ``extended introduction'' to the paper.

\mbr

Let $G$ be a unipotent algebraic group over $k$, and fix a prime
$\ell\neq p=\operatorname{char}k$. Let $\sD_G(G)$ denote the
$G$-equivariant derived category of constructible $\ql$-complexes on
$G$, where $G$ acts on itself by conjugation. It is a braided
monoidal category with respect to the functor of convolution with
compact supports (Def.~\ref{d:convolution}), which we denote by
$(M,N)\mapsto M*N$. Let us say that an object $e\in\sD_G(G)$ is a
\emph{closed idempotent} if there exists an arrow $\e\to e$ that
becomes an isomorphism after convolving with $e$, where $\e$ is the
unit object in $\sD_G(G)$ (the delta-sheaf at the identity $1\in
G$). The notion of a \emph{minimal closed idempotent} is defined in
the obvious way. If $e\in\sD_G(G)$ is a minimal closed idempotent,
the \emph{$\bL$-packet of character sheaves corresponding to $e$} is
defined as the collection of objects $M\in\sD_G(G)$ such that
$e*M\cong M$, and such that the underlying complex of $M$ (obtained
by discarding the $G$-equivariant structure) is an irreducible
perverse sheaf on $G$.

\mbr

Some of the fundamental properties of character sheaves, proved in
this article, are as follows. Every $\bL$-packet of character sheaves
on $G$ is finite. If $M$ and $N$ are two character sheaves on $G$,
then $\Ext^i_{\sD_G(G)}(M,N)=0$ for all $i>0$. Moreover, if $M$ and
$N$ lie in the same $\bL$-packet defined by a minimal closed
idempotent $e\in\sD_G(G)$, then $M*N$ is perverse up to
cohomological shift by an integer $n_e$ determined only by $e$ (if
$M$ and $N$ lie in different $\bL$-packets, then $M*N=0$).

\mbr

One of the ingredients in the proof is an explicit construction of
minimal closed idempotents in $\sD_G(G)$, based on the notion of an
\emph{admissible pair} for a unipotent group
(\S\ref{ss:Serre-duality-admissible-pairs}) and on the construction
of the \emph{induction functor with compact supports}
$\sD_{G'}(G')\rar{}\sD_G(G)$ for a closed subgroup $G'\subset G$
(see \S\ref{ss:induction-functors}).

\mbr

These tools were used previously in \cite{characters} to develop a
geometric approach to the study of characters of unipotent groups
over finite fields. In particular, it was proved there that every
admissible pair for $G$ gives rise to a minimal \emph{weak}
idempotent in $\sD_G(G)$, where a weak idempotent is defined as an
object $e\in\sD_G(G)$ that satisfies $e*e\cong e$. In the present
article we complete the picture by showing that the classes of
minimal closed idempotents and minimal weak idempotents in
$\sD_G(G)$ coincide (this result is parallel to the classical
theorem that the coadjoint orbits of a unipotent group are closed).
In addition, we prove that \emph{every} minimal (weak or closed)
idempotent in $\sD_G(G)$ arises from an admissible pair for $G$.
%
%

\mbr

One of the motivations behind our work was to provide a foundation for the theory of character sheaves and characters on unipotent groups over \emph{finite} fields, whose existence was conjectured by Lusztig. Suppose that $k$ is an algebraic closure of a finite subfield $\bF_q\subset k$ and that $G=G_0\tens_{\bF_q}k$, where $G_0$ is a unipotent group over $\bF_q$. If $M$ is a Frobenius-invariant character sheaf on $G$, the corresponding trace function $G_0(\bF_q)\rar{}\ql$ is invariant under conjugation. The relationship between these functions and the irreducible characters of the finite group $G_0(\bF_q)$ is studied in \cite{characters-char-sheaves}. In the case when $G$ is \emph{easy} in the terminology of \emph{op.~cit.} (every point of $G(k)$ is contained in the neutral connected component of its centralizer in $G$), the functions on $G_0(\bF_q)$ coming from Frobenius-invariant character sheaves on $G$ coincide with the irreducible characters up to scaling. For the general case we refer to Theorem 2.17 in \emph{op.~cit.} We remark that the present article relies on the results of \cite{characters} in a few places, while \cite{characters-char-sheaves} depends significantly both on \cite{characters} and on the present article.

\subsection*{Acknowledgments} We are indebted to George Lusztig, who originally suggested in 2003 that there should exist a theory of character sheaves on unipotent groups in positive characteristic and computed the first interesting examples in this theory. We thank A.~Beilinson, K.~Costello, J.~Lurie, and U.~Tillmann for valuable
advice. We also thank the referees for pointing out several misprints and omissions in an earlier version of our article.



\section{Main definitions and results}\label{s:results}

Most of this section is devoted to recalling several definitions and
constructions that were discussed at length in \cite{intro} and/or
\cite{characters}. In
\S\S\ref{ss:basic-definitions}--\ref{ss:equivariant-derived} we
recall some facts about derived categories of constructible
$\ell$-adic complexes, along with their equivariant versions.
In \S\ref{ss:themonoidal}-\ref{ss:properties-structures} we introduce
and discuss the monoidal categories $\sD(G)$ and $\sD_G(G)$
associated with any unipotent group $G$. The definitions of
character sheaves on unipotent groups in positive characteristic and
their functional dimension are given in
\S\S\ref{ss:char-sheaves-L-packets}--\ref{ss:functional-dimension}.
They are followed by a digression in \S\ref{ss:algebraic-reminders},
where we recall several well known results of character theory for
finite groups that serve as a motivation behind our approach to the
analysis of character sheaves. The main results of our work, along
with various preliminaries, appear in
\S\S\ref{ss:Serre-duality-admissible-pairs}--\ref{ss:mackey-theory}.
Finally, in \S\ref{ss:organization} we explain the organization of
the remaining sections of the article.

\subsection{Basic definitions and notation}\label{ss:basic-definitions}
Throughout this article we work over an algebraically closed field
$k$ of characteristic $p>0$. We also fix, once and for all, a prime
$\ell\neq p$ and an algebraic closure $\ql$ of the field $\bQ_\ell$
of $\ell$-adic numbers.

\mbr

By an \emph{algebraic group} over $k$ we will mean a smooth group
scheme (equivalently, a reduced group scheme of finite type) over
$k$. A \emph{unipotent algebraic group} (or ``unipotent group,'' for
brevity) over $k$ is an algebraic group over $k$ that is isomorphic
to a closed subgroup of the group $UL_n(k)$ of unipotent
upper-triangular matrices of size $n$ over $k$, for some $n\in\bN$.

\begin{rem}\label{r:why-algebraically-closed}
Many of the definitions and results of our work can be formulated
for unipotent groups over an arbitrary field $k$ of positive
characteristic (sometimes it is necessary to assume that $k$ is
perfect), and most of the auxiliary facts (see
\S\S\ref{s:serre-FD}--\ref{s:reduction}) remain valid in this more
general setting. On the other hand, their proofs can be trivially
reduced to the case where $k$ is algebraically closed. In addition,
certain important properties of character sheaves require $k$ to be
algebraically closed. For these reasons, we find it more convenient
to assume that $k$ is algebraically closed from the very beginning.
\end{rem}

If $X$ is an arbitrary scheme of finite type over $k$, one knows how
to define the bounded derived category $D^b_c(X,\ql)$ of
constructible complexes of $\ql$-sheaves on $X$ (see, e.g.,
\cite[\S\S1.1.2--1.1.3]{deligne-weil-2}). We will denote it simply
by $\sD(X)$, since $\ell$ is fixed.
It is a triangulated $\ql$-linear category. Furthermore, it is
equipped with a self-dual \emph{perverse} $t$-structure
$({}^p\sD^{\leq 0}(X),{}^p\sD^{\geq 0}(X))$ \cite{bbd}, whose heart,
$\Perv(X)={}^p\sD^{\leq 0}(X)\cap {}^p\sD^{\geq 0}(X)$, is called
the category of \emph{perverse sheaves} on $X$. It is an abelian
category in which every object has finite length (\emph{op.~cit.},
Thm.~4.3.1(i)).

\subsection{Formalism of the six functors}  \label{ss:six-functors}
In what follows, we will frequently employ Grothendieck's ``formalism
of the six functors'' for the categories $\sD(X)$ (as well as their
equivariant versions, defined in \S\ref{ss:equivariant-derived}
below). In particular, for any\footnote{One does not have to assume that $f$ is separated
(see, e.g., \cite{Las-Ols06}).} morphism $f:X\rar{}Y$ of $k$-schemes of finite type
one has the pullback functor $f^*:\sD(Y)\rar{}\sD(X)$,
the pushforward functor $f_*:\sD(X)\rar{}\sD(Y)$,
the functor $f_!:\sD(X)\rar{}\sD(Y)$ (pushforward with compact supports),
and the functor $f^!:\sD(Y)\rar{}\sD(X)$. \emph{We always omit the letters
``$L$'' and ``$R$'' from our notation for the six functors; thus,
$f_!$ will stand for $Rf_!$ and $\tens$ will stand for
$\overset{L}{\tens}_{\ql}$, etc.}

\begin{rem}
In \cite{sga4.5} the functor $f_!$ is defined only if $f$ is
separated. This is enough for our purposes.
\end{rem}

\subsection{Equivariant derived
categories}\label{ss:equivariant-derived} We remain in the setup of
\S\ref{ss:basic-definitions}. Let $G$ be an algebraic group over
$k$, let $X$ be a scheme of finite type over $k$, and suppose that
we are given a left action of $G$ on $X$. In general, to get
the correct definition of the ``equivariant derived category''
$\sD_G(X)$, one must either adopt the approach of Bernstein and
Lunts \cite{ber-lunts} (when $G$ is affine), or use the definition
of $\ell$-adic derived categories for Artin stacks due to Laszlo and
Olsson \cite{Las-Ols06} and define $\sD_G(X)=\sD(G\bigl\backslash
X)$, where $G\bigl\backslash X$ is the quotient stack of $X$ by $G$.

\mbr

\emph{From now on we assume that $G$ is unipotent.} In this case, the naive definition of $\sD_G(X)$ given below is equivalent to the correct one by Proposition \ref{p:equivalence-definitions-equivariant-category}.

\mbr

Let us write $\al:G\times X\to X$ for the action morphism and
$\pi:G\times X\to X$ for the projection. Let $\mu:G\times G\to G$ be
the product in $G$. Let $\pi_{23}:G\times G\times X\to G\times X$ be
the projection along the first factor $G$. The category $\sD_G(X)$
is defined as follows.

\begin{defin}\label{d:equiv-derived}
An {\em object} of the category $\sD_G(X)$ is a pair $(M,\phi)$,
where $M\in\sD(X)$ and $\phi:\al^*M\rar{\simeq}\pi^*M$ is an
isomorphism in $\sD(G\times X)$ such that
\begin{equation}\label{e:compatibility}
\pi_{23}^*(\phi)\circ(\id_G\times\al)^*(\phi) =
(\mu\times\id_X)^*(\phi),
\end{equation}
i.e., the composition of the natural isomorphisms
\[
(\id_G\times\al)^*\al^* M \cong (\mu\times\id_X)^*\al^* M \xrar{\
(\mu\times\id_X)^*(\phi)\ } (\mu\times\id_X)^*\pi^* M \cong
\pi_{23}^*\pi^* M
\]
equals the composition
\[
(\id_G\times\al)^*\al^* M \xrar{\ (\id_G\times\al)^*(\phi)\ }
(\id_G\times\al)^*\pi^* M \cong \pi_{23}^*\al^* M \xrar{\
\pi_{23}^*(\phi)\ } \pi_{23}^*\pi^* M.
\]
A {\em morphism} $(M,\phi)\rar{}(N,\psi)$ in $\sD_G(X)$ is a
morphism $\nu:M\rar{}N$ in $\sD(X)$ satisfying
$\phi\circ\al^*(\nu)=\pi^*(\nu)\circ\psi$. The {\em composition} of
morphisms in $\sD_G(X)$ is defined to be equal to their composition
in $\sD(X)$.
\end{defin}

\begin{rem}
With the aid of the proper base change theorem and the smooth base
change theorem, one can easily ``upgrade'' the functors $f^*$, $f_*$
and $f_!$ mentioned in \S\ref{ss:six-functors} to functors between
the equivariant derived categories $\sD_G(X)$ and $\sD_G(Y)$, in the
case where a unipotent group $G$ acts on $X$ and $Y$, and the
morphism $f$ commutes with the $G$-action
(cf.~\cite[\S4.4]{characters}).
\end{rem}

\begin{rem}
By Definition~\ref{d:equiv-derived}, we have a faithful forgetful functor
$\sD_G(X)\to\sD(X)$. If $G$ is a \emph{connected} unipotent group
over $k$, one can show that the forgetful functor is fully faithful.
In other words, being $G$-equivariant becomes a
\emph{property} of an $\ell$-adic complex on $X$ in this case. In particular, if $(M,\phi)\in\sD_G(X)$, then $\phi$ is determined uniquely by $M$.
\end{rem}

\subsection{The monoidal categories $\sD (G)$ and $\sD_G(G)$} \label{ss:themonoidal}
\begin{defin}
If $G$ is a unipotent algebraic group over $k$, the
\emph{equivariant derived category of $G$} is defined as the
equivariant derived category $\sD_G(G)$ with respect to the
conjugation action of $G$ on itself.
\end{defin}

\begin{defin}\label{d:convolution}
Let $G$ be a unipotent algebraic group over $k$. If $M$ and $N$ are
objects of $\sD(G)$ (respectively, $\sD_G(G)$), the
\emph{convolution with compact supports} of $M$ and $N$ is the
object of $\sD(G)$ (respectively, $\sD_G(G)$) defined by
$M*N=\mu_!\bigl((p_1^* M)\tens(p_2^* N)\bigr)$, where $\mu:G\times
G\rar{}G$ is the multiplication morphism and $p_1,p_2:G\times
G\rar{}G$ are the first and second projections.
\end{defin}

\begin{convention}
From now on, the words ``with compact supports'' will be dropped,
and the bifunctor $*$ will simply be referred to as the
\emph{convolution} of complexes on $G$. The other convolution bifunctor, obtained by replacing
$\mu_!$ with $\mu_*$ in Definition~\ref{d:convolution}, will not be used in the present work\footnote{In fact, convolution without compact supports, call it $*_*$ for now, can be expressed in terms of convolution with compact supports. Namely, if $M,N$ are objects of $\sD(G)$ or $\sD_G(G)$ one has a canonical isomorphism $M*_*N\rar{\simeq}\bD_G^-(\bD_G^-N*\bD_G^-M)$, where $\bD_G^-$ is the functor introduced in Definition \ref{d:duality}. By Remark \ref{r:duality-D(G)-inner-hom}, the functor $\bD_G^-$ also has intrinsic meaning in terms of the monoidal structure given by convolution with compact supports.}.
\end{convention}

It is easy to construct associativity constraints for the bifunctors
$*$ on $\sD(G)$ and on $\sD_G(G)$, making each category monoidal,
with unit object $\e=\e_G$ being the delta-sheaf at the identity element $1\in G(k)$.

\subsection{Properties and additional structures on $\sD_G(G)$ and $\sD(G)$}
\label{ss:properties-structures}
Here we give a brief outline (which is enough to read the most part of the article)
and refer to Appendices~\ref{s:dualityformalism} and \ref{s:CFT}
for the precise definitions and constructions.


\subsubsection{A duality property weaker than rigidity}   \label{sss:r-intro}
If $G$ is finite then the monoidal categories $\sD_G(G)$ and $\sD(G)$ are rigid.
In general, they are not rigid but have a weaker property related to duality.
Namely, $\sD_G(G)$ and $\sD(G)$ are  \emph{r-categories} in the sense of Definition~\ref{def:r-category}, see Example~\ref{example:dual-gen4} and Lemma~\ref{l:duality-on-D(G)}.


\subsubsection{$\sD_G(G)$ is a braided category} \label{sss:br}
The braiding on $\sD_G(G)$ is constructed in Definition \ref{d:braiding-equivariant-derived}.
Although this braiding enters the formulation of some of our results (notably,
Theorem \ref{t:properties-character-sheaves}(d)), in many situations it suffices to know its existence or even the following corollary of its existence.

\begin{rem}\label{r:braiding-equivariant-derived}
The functors $(M,N)\mapsto M*N$ and $(M,N)\mapsto N*M$ on $\sD_G(G)$ are isomorphic.
\end{rem}



\subsubsection{Pivotal structure on $\sD (G)$ and $\sD_G(G)$}
As explained in \S\ref{sss:pivotal_G},  each of the monoidal categories $\sD (G)$ and $\sD_G(G)$ has
a canonical \emph{pivotal structure} in the sense of Definition~\ref{d:pivotal1}.

\subsubsection{Ribbon structure on $\sD_G(G)$}   \label{sss:Rib}
According to Definition~\ref{d:ribbon}, a \emph{ribbon structure} on a braided
r-category\footnote{As explained in \S\ref{sss:r-intro}, the usual framework of rigid braided categories is too restrictive for us.} $\cM$ is an automorphism $\te$ of the functor $\Id_{\cM}$ satisfying
certain conditions\footnote{One can also consider a ribbon structure as a pivotal structure satisfying a certain condition, see Corollary~\ref{c:ribbon}.}; in particular, $\te$ has to be a \emph{twist} in the sense of Definition~\ref{d:tw}.
In \S\ref{ss:ribbon_G} we define a \emph{canonical ribbon structure on $\sD_G(G)$.}
For each $M\in\sD_G(G)$ and $g\in G$ the action of the canonical twist $\te_M :M\rar{\simeq}M$
on the stalk $M_g$ equals the action of $g\in Z(g)$ on $M_g$. Here $Z(g)\subset G$
is the centralizer of $g$ and the action of $Z(g)$ on $M_g$
comes\footnote{By Definition~\ref{d:equiv-derived}, for each $\ga ,g\in G$ one has
$\phi_{\ga ,g}:M_{\ga g\ga^{-1}}\iso M_g\,$. The (left) action of $Z(g)$ on $M_g$ is defined by
$\ga\mapsto\phi_{\ga ,g}^{-1}\,$, $\ga\in Z(g)$.}
from the equivariant structure on $M$.

\begin{example}
If $G$ is finite then $\sD_G(G)$ is the derived category of 
$\cA$, where $\cA$ is the category of $G$-equivariant constructible sheaves on $G$, also
known as the category of modules over the \emph{quantum double} of the group algebra of $G$.
$\cA$ is a standard example of a ribbon category  (and in fact, a modular one), see  \cite[\S3.2]{BK}.
\end{example}

\subsubsection{Action of the surface operad}
The category $\sD_G(G)$ is equipped with a canonical action of the surface operad, see
\S\ref{sss:surface_operad} of  Appendix~\ref{s:CFT} and Remark~\ref{r:surface_operad}(ii).
The action of the genus 0 part of the surface operad amounts to the braided monoidal structure and twist mentioned in \S\ref{ss:themonoidal}, \S\ref{sss:br}, and \S\ref{sss:Rib}.

\subsection{Character sheaves and
$\bL$-packets}\label{ss:char-sheaves-L-packets} The notion of a
character sheaf on a unipotent algebraic group $G$ is defined in
terms of certain ``idempotents'' in the category $\sD_G(G)$. A more
exhaustive study of idempotents in monoidal categories (which does
not depend on any of the other results of the present article)
appears in \S\ref{s:idempotents}. Here we will briefly summarize
some of the definitions given in that section, specialized to the
monoidal category $(\sD_G(G),*,\e)$.

\begin{defins}[Weak and closed idempotents; minimal idempotents; Hecke
subcategories]\label{d:idempotents-Hecke} In the setup of
\S\ref{ss:basic-definitions}, let $G$ be a unipotent algebraic group
over $k$.
\begin{enumerate}[(1)]
\item An object $e\in\sD_G(G)$ is said to be a \emph{weak
idempotent} if $e*e\cong e$. It is said to be a
\emph{closed\footnote{The origin of the adjective ``closed'' is
explained in \S\ref{ss:idempotents-definitions}.} idempotent} if
there exists a morphism $\pi:\e\rar{}e$ that becomes an isomorphism
after convolving with $e$. Such a $\pi$ is called an
\emph{idempotent arrow.}
 \sbr
\item If $e\in\sD_G(G)$ is a weak idempotent, the \emph{Hecke
subcategory} associated to $e$ is the full subcategory
$e\sD_G(G)\subset\sD_G(G)$ consisting of objects $M\in\sD_G(G)$ such
that $e*M\cong M$ (equivalently, such that $M\cong e*N$ for some
$N\in\sD_G(G)$).
 \sbr
\item An object $e\in\sD_G(G)$ is said to be a \emph{minimal weak
idempotent} (respectively, a \emph{minimal closed idempotent}) if
$e$ is a weak (respectively, closed) idempotent, $e\neq 0$, and for
every weak (respectively, closed) idempotent $e'$ in $\sD_G(G)$, we
have either $e*e'=0$, or $e*e'\cong e$.
\end{enumerate}
\end{defins}

\begin{rems}   \label{r:Hecke_cat}
\begin{enumerate}[(i)]
\item If $e\in\sD_G(G)$ is a closed idempotent, the
Hecke subcategory $e\sD_G(G)$ is a monoidal category with unit
object $e$ (see Lemma \ref{l:Hecke-monoidal}).
 \sbr
\item  If $e\in\sD_G(G)$ is a weak idempotent,
the category $e\sD_G(G)$ is not necessarily monoidal already for
$G=\bG_a$ (see Remark~\ref{r:caution}, including its last sentence), although by Remark
\ref{r:braiding-equivariant-derived}, the subcategory
$e\sD_G(G)\subset\sD_G(G)$ is closed under convolution. There are
also other reasons why the notion of weak idempotent is not really
good. Nevertheless, the interplay between weak and closed
idempotents, studied in \S\ref{s:idempotents}, turns out to be
useful for proving the main results of our work.
 \sbr
\item Let $e,e'\in\sD_G(G)$ be minimal closed (respectively, minimal weak) idempotents.
If $e*e'\ne 0$, then $e\cong e*e'\cong e'*e\cong e'$ (we have used
Remark~\ref{r:braiding-equivariant-derived}). So if $e\ncong e'$, then
$e*e'=0$, and therefore $e\sD_G(G)\cap e'\sD_G(G)=0$.
\end{enumerate}
\end{rems}


\begin{defins}[Character sheaves and
$\bL$-packets]\label{d:char-sheaves-L-packets} With the assumptions of
\S\ref{ss:basic-definitions}, let $G$ be a unipotent algebraic group
over $k$.
\begin{enumerate}[(1)]
\item Let $e\in\sD_G(G)$ be a minimal closed idempotent. We write
$\sM_e^{perv}$ for the full subcategory of the Hecke subcategory
$e\sD_G(G)$ consisting of those objects for which the underlying
$\ell$-adic complex is a perverse sheaf on $G$. It is clear that
$\sM_e^{perv}$ is an additive $\ql$-linear subcategory of
$\sD_G(G)$. The \emph{Lusztig packet of character sheaves} on $G$
defined by $e$ is the set of (isomorphism classes of) indecomposable
objects of the category $\sM_e^{perv}$.
 \sbr
\item An object of $\sD_G(G)$ is a \emph{character sheaf} if it lies
in the Lusztig packet of character sheaves defined by some minimal
closed idempotent in $\sD_G(G)$.
\end{enumerate}
\end{defins}

From now on, for brevity, we write ``$\bL$-packet'' in place of ``Lusztig packet.''

\begin{rem}
If minimal closed idempotents $e,e'\in\sD_G(G)$ are not isomorphic then
the corresponding $\bL$-packets are disjoint by Remark~\ref{r:Hecke_cat} (iii).
\end{rem}

The first main result of this article is the following

\begin{thm}\label{t:properties-character-sheaves}
Let $G$ be a unipotent algebraic group over $k$, and let
$e\in\sD_G(G)$ be a minimal closed idempotent.
\begin{enumerate}[$($a$)$]
\item Then $\sM_e^{perv}$
is a semisimple abelian category with finitely many simple objects.
In particular, $\bL$-packets of character sheaves on $G$ are finite.
 \sbr
\item There exists a $($necessarily unique$)$ integer $n_e$ such
that $e[-n_e]\in\sM_e^{perv}$. One has $0\le n_e\le \dim G$.
The subcategory $\sM_e:=\sM_e^{perv}[n_e]$
of the monoidal category $e\sD_G(G)$ is monoidal.
 \sbr
\item The monoidal categories $e\sD_G(G)$ and $\sM_e$ are rigid; moreover, $\sM_e$ is a fusion
category\footnote{A {\em fusion category\,} over $\ql$ is a rigid $\ql$-linear monoidal category $\cC$ such that the unit object of $\cC$ is indecomposable and as a $\ql$-linear category,  $\cC$
is equivalent to a direct sum of finitely many copies of the category of finite-dimensional vector spaces.} in the sense of \cite{ENO}.
 \sbr
\item The restrictions to $\sM_e$ of the braiding constructed in Definition \ref{d:braiding-equivariant-derived} and of the twist constructed in Definition \ref{d:twist-equivariant-derived} define a ribbon structure on $\sM_e$. This ribbon structure is modular.
 \sbr
\item The perverse $t$-structure on $\sD(G)$ induces a $t$-structure
on $e\sD_G(G)$, and the canonical functor
$D^b(\sM_e^{perv})\rar{}e\sD_G(G)$ is an equivalence.
\end{enumerate}
\end{thm}

\begin{rems}\label{r:properties-char-sheaves}
\begin{enumerate}[(i)]
\item By Remark \ref{r:Hecke_cat}(i), the last part of (b) is
equivalent to $\sM_e$ being closed under convolution (the fact that
$\sM_e$ contains the unit object $e$ of the Hecke subcategory
$e\sD_G(G)$ is clear from the definition of $n_e$).
 \sbr
\item It is not hard to show that a minimal closed idempotent $e$ is indecomposable
(cf.~Corollary \ref{c:minimal-closed-idempotent-indecomposable}). So
in the situation of Theorem~\ref{t:properties-character-sheaves}(b),
the object $e[-n_e]$ is a character sheaf. In particular,
$\bL$-packets are nonempty.
\end{enumerate}
\end{rems}

The proof of Theorem~\ref{t:properties-character-sheaves} is given
in \S\ref{ss:proof-t:properties-character-sheaves}. It relies on a
certain construction of all minimal closed idempotents in $\sD_G(G)$
(and the corresponding $\bL$-packets), which is the keystone of our
approach to the theory of character sheaves on unipotent groups.
This construction, based on the notion of an admissible pair
(\S\ref{ss:Serre-duality-admissible-pairs}) and on the induction
functors for equivariant derived categories
(\S\ref{ss:induction-functors}), is given in
\S\ref{ss:construction-L-packets}.

\mbr

The duality functor on the rigid monoidal category $e\sD_G(G)$ can
be calculated explicitly as follows.

\begin{defin}\label{d:duality}
For a unipotent algebraic group $G$ over $k$, we define the contravariant functor
$\bD^-_G:\sD_G(G)\rar{}\sD_G(G)$ by
$\bD^-_G=\bD_G\circ\iota^*=\iota^*\circ\bD_G$, where
$\iota:G\rar{}G$ is given by $g\mapsto g^{-1}$ and
$\bD_G:\sD_G(G)\rar{}\sD_G(G)$ is the Verdier duality functor. By abuse of notation, we also denote by $\bD_G^-$ and $\bD_G$ the corresponding functors $\sD(G)\rar{}\sD(G)$ in the non-equivariant setting.
\end{defin}

\begin{rem}\label{r:duality-D(G)-inner-hom}
The functor $\bD_G^-$ can be interpreted as inner Hom to the unit object in the monoidal category $\sD_G(G)$ or $\sD(G)$: see Lemma \ref{l:duality-on-D(G)}.
\end{rem}

\begin{prop}\label{p:duality}
In the situation of Theorem \ref{t:properties-character-sheaves},
\begin{enumerate}[$($a$)$]
\item There is a canonical isomorphism
\begin{equation}\label{e:dual-e}
\bD_G^- e \rar{\simeq} e[-2n_e](-n_e)\tens L_e,
\end{equation}
where $L_e$ is a certain line\footnote{That is, a $1$-dimensional vector space.}  over $\ql$.
 \sbr
\item Given an idempotent arrow $\e\rar{\pi}e$, the duality functor on $e\sD_G(G)$ canonically
identifies with\footnote{It is not hard to show that $e\sD_G(G)$ is
stable under $\bD_G^-$; see Lemma \ref{l:GV-idempotents}.}
$M\longmapsto (\bD_G^- M)[2n_e](n_e)\tens L_e^{-1}$, where
$L_e^{-1}=\Hom_{\ql}(L_e,\ql)$.
\end{enumerate}
\end{prop}

The proposition is proved in \S\ref{ss:proof-p:duality}.

\begin{rems}\label{r:duality}
\begin{enumerate}[(i)]
\item By Theorem \ref{t:properties-character-sheaves}(a) and Remark
\ref{r:properties-char-sheaves}(ii), the object $e\in\sM_e$ is
simple, so $\Hom_{\sD_G(G)}(e,e)$ is $1$-dimensional. Thus assertion
(a) of Proposition \ref{p:duality} is equivalent to the statement
that $\bD_G^-e\cong e[-2n_e](-n_e)$.
 \sbr
\item We do not know whether the line $L_e$ can be canonically trivialized.
In Proposition~\ref{p:duality-canonical} and \S\ref{ss:proof-p:duality-canonical}
below we describe a trivialization of $L_e$ depending on an additional choice.
\end{enumerate}
\end{rems}

\subsection{The notion of functional
dimension}\label{ss:functional-dimension} Theorem
\ref{t:properties-character-sheaves}(b) allows 
to introduce
\begin{defin}
Let $G$ be a unipotent algebraic group over $k$, and let
$e\in\sD_G(G)$ be a minimal closed idempotent. The \emph{functional
dimension} of $e$ is the number $d_e=(\dim G-n_e)/2$. We also call
$d_e$ \emph{the functional dimension of every character sheaf in the
$\bL$-packet defined by $e$.}
\end{defin}

By Theorem~\ref{t:properties-character-sheaves}(b), $0\le d_e\le
(\dim G)/2$.

\begin{rem}
This notion is analogous to the classical notion of functional
dimension in the representation theory of real Lie groups. First, we
show in Theorem \ref{t:geometric-mackey}(b) that functional
dimension behaves as expected under induction functors. Second, if
$G$ is connected and its nilpotence class is less than
$p=\operatorname{char}k$, then according to \cite[Thm.~5.10]{intro},
isomorphism classes of minimal closed idempotents $e\in\sD_G(G)$
bijectively correspond to ``coadjoint orbits" in the sense of
\cite[\S4.1]{intro} and
\[
d_e=\frac{1}{2}\cdot\dim\Om_e, \quad n_e=\codim\Om_e \, ,
\]
where $\Om_e$ is the orbit corresponding to $e$.
%
%
\end{rem}

\begin{rem}
The functional dimension may fail to be an integer in our setup. For
example, let $G$ be a {\em fake Heisenberg group,} i.e., a
noncommutative central extension of $\bG_a$  by $\bG_a$ (such
extensions do exist in characteristic $p>0$:
cf.~\cite[\S3.7]{intro}). Then there exist minimal closed
idempotents $e\in\sD_G(G)$ with $d_e=1/2$ (\emph{loc.~cit.}).
\end{rem}

The next theorem is proved in
\S\ref{ss:proof-t:functional-dimension}.

\begin{thm}\label{t:functional-dimension}
Let $G$ and $e$ be as in
Theorem~\ref{t:properties-character-sheaves}. The Gauss sum\footnote{See \cite[\S6.2]{dgno}, formula (122), for the definition of the Gauss sum $\tau^+(\cM_e)$ of $\cM_e$.} of the modular category $\cM_e$ equals $\eps p^n$, where $n\geq 0$ is an integer, $\eps=1$ when $d_e\in\bZ$ and $\eps=-1$ when $d_e\not\in\bZ$. In particular, if the $\bL$-packet of
$e$ has only one element, then $d_e\in\bZ$.
\end{thm}

\subsection{Elementary reminders}\label{ss:algebraic-reminders} We
now make a short digression to recall several constructions and
results from character theory for finite groups that are, for the
most part, very standard and well known. In the remainder of the
article we will see that all of them admit suitable geometric
analogues in the world of equivariant $\ell$-adic complexes on
unipotent groups, which are both interesting in their own right, and
play an essential role in the proofs of our main results.

\sbr

\begin{enumerate}[1.]
\item Let $\Ga$ be a finite group, let $\Fun(\Ga)$ denote the
algebra of functions $\Ga\rar{}\bC$ under pointwise addition and
convolution, and let $\Fun(\Ga)^\Ga\subset\Fun(\Ga)$ denote the
subalgebra of conjugation-invariant functions\footnote{It coincides
with the center of $\Fun(\Ga)$. Its elements are often called
``class functions'' on $\Ga$.}. Then there is a bijection between
the set of complex irreducible characters of $\Ga$ and the set of
minimal (or ``indecomposable'') idempotents in $\Fun(\Ga)^\Ga$,
given by $\chi\longmapsto\frac{\chi(1)}{\abs{\Ga}}\cdot\chi$.
 \sbr
This fact is one of the reasons why the definition of character
sheaves involves minimal idempotents. Another motivation, coming
from the orbit method for unipotent groups of ``small'' nilpotence
class, is explained in detail in \cite{intro}.
 \sbr
\item An important role in the character theory of finite groups is
played by the operation of induction of class functions. If
$\Ga'\subset\Ga$ is a subgroup and $f\in\Fun(\Ga')^{\Ga'}$, the
induced function $\iga f\in\Fun(\Ga)^\Ga$ can be obtained in two
steps. First, we extend $f$ by zero outside of $\Ga'$ to obtain a
$\Ga'$-invariant function $\overline{f}:\Ga\rar{}\bC$. Next, for
every coset $\ga\Ga'\subset\Ga$, we form the corresponding
conjugate, $\overline{f}^\ga$, of $\overline{f}$ (it depends only on
the coset $\ga\Ga'$ and not on the particular element $\ga$), and we
define $\iga f$ as the sum of all these conjugates, indexed by the
elements of $\Ga/\Ga'$.
 \sbr
\item If $\chi$ is a complex irreducible character of $\Ga'$, there
is a result, called \emph{Mackey's irreducibility criterion}, which
gives necessary and sufficient conditions for the induced character
$\iga\chi$ to be irreducible as well. It is not hard to show that
this result can be reformulated as follows: with the notation of \#2
above, $\iga\chi$ is irreducible if and only if
$\overline{\chi}*\de_x*\overline{\chi}=0$ for every
$x\in\Ga\setminus\Ga'$, where $\de_x$ denotes the delta-function at
$x$.
 \sbr
\item Some complex irreducible representations of a
finite group $\Ga$ can be obtained by means of the following
construction. Consider a pair $(H,\chi)$ consisting of a subgroup
$H\subset\Ga$ and a homomorphism $\chi:H\rar{}\cst$. Let $\Ga'$ be
the stabilizer of the pair $(H,\chi)$ with respect to the
conjugation action of $\Ga$. We say that the pair $(H,\chi)$ is
\emph{admissible} if the following three conditions are satisfied:
 \sbr
\begin{enumerate}[(1)]
\item $\Ga'/H$ is commutative;
 \sbr
\item the map $B_\chi:(\Ga'/H)\times(\Ga'/H)\rar{}\cst$
induced by
$(\ga_1,\ga_2)\longmapsto\chi(\ga_1\ga_2\ga_1^{-1}\ga_2^{-1})$
(which, in view of (1), is well defined and biadditive) is a perfect
pairing, i.e., induces an isomorphism
$\Ga'/H\rar{\simeq}\Hom(\Ga'/H,\cst)$; and
 \sbr
\item for every $g\in\Ga$, $g\not\in\Ga'$, we have
$\chi\big\lvert_{H\cap H^g}\neq\chi^g\bigl\lvert_{H\cap H^g}$, where
$H^g=g^{-1}Hg$ and $\chi^g:H^g\rar{}\cst$ is obtained by transport
of structure: $\chi^g(h)=\chi(ghg^{-1})$.
\end{enumerate}
 \sbr
Now, properties (1) and (2) above imply that the group $\Ga'$ has a
unique complex irreducible representation $\pi_\chi$ such that $H$ acts
 by $\chi$. Property (3) then implies that the
induced representation $\Ind_{\Ga'}^\Ga\pi_\chi$ is irreducible, in
view of Mackey's irreducibility criterion (see \#3).
 \sbr
\item If $\Ga$ is a finite \emph{nilpotent} group then
\emph{every} complex irreducible representation $\rho$ of $\Ga$ arises from
some admissible pair $(H,\chi)$ by means of the construction explained in \#4;
moreover, the ``reduction process" described in one of the appendices to \cite{intro}
allows to construct $(H,\chi)$ {\em canonically\,} (up to conjugation) for a given $\rho$.
Note that the existence, for a given $\rho$, of a non-canonical admissible pair $(H,\chi)$
(with $\Gamma'=H$) immediately follows from the classical theorem that every
irreducible representation of a finite nilpotent group is induced from
a 1-dimensional representation of a subgroup.


\end{enumerate}

\subsection{Perfect schemes and groups}\label{ss:perfect} One of the
main technical tools used in the present article is the
\emph{geometric} notion of an admissible pair for a \emph{unipotent}
group over $k$, introduced in
\S\ref{ss:Serre-duality-admissible-pairs}. The definition of this
notion uses Serre duality for unipotent groups. However, as
explained in \cite{intro}, the Serre dual of a (connected) unipotent
group over $k$ can only be defined canonically (i.e., by means of a
universal property) as a group object in the category of
\emph{perfect} schemes over $k$. For this reason, it will be
technically more convenient for us to place ourselves in the
framework of perfect schemes and perfect group schemes from the very
start.

\mbr

Let us recall that a scheme $S$ in characteristic $p$ (i.e., such
that $p$ annihilates the structure sheaf $\cO_S$ of $S$) is said to
be \emph{perfect} if the morphism $\cO_S\rar{}\cO_S$, given by
$f\longmapsto f^p$ on the local sections of $\cO_S$, is an
isomorphism of sheaves. In particular, a commutative ring $A$ of
characteristic $p$ is perfect \cite{greenberg} if and only if $\Spec
A$ is a perfect scheme. We will denote by $\mathfrak{Sch}_k$ the
category of all $k$-schemes, and by $\mathfrak{Perf}_k$ the full
subcategory consisting of perfect schemes. The inclusion functor
$\mathfrak{Perf}_k\into\mathfrak{Sch}_k$ admits a right adjoint,
which we call the \emph{perfectization functor} and denote by
$X\longmapsto X_{per\!f}$. (This follows from the results of
Greenberg \cite{greenberg}. In our setup, $k$ is algebraically
closed, but, in fact, it suffices to assume that $k$ is perfect.)

\mbr

It is not hard to see that a group object in the category
$\mathfrak{Perf}_k$ is automatically a group scheme over $k$.
Conversely, if $G$ is any group scheme over $k$, then $G_{per\!f}$
can be canonically equipped with a $k$-group scheme structure as
well. For more details, we refer the reader to
\cite[\S\S{}A.3--A.4]{characters}.

\begin{defins}\label{d:perfect-quasi-algebraic}
A \emph{perfect quasi-algebraic scheme} over $k$ is an object of
$\mathfrak{Perf}_k$ that is isomorphic to the perfectization of a
scheme of finite type over $k$. A \emph{perfect quasi-algebraic
group} over $k$ is a group object of $\mathfrak{Perf}_k$ that is
isomorphic to the perfectization of an algebraic group over $k$. For
brevity, by a \emph{perfect unipotent group} over $k$ we will mean a
perfect quasi-algebraic group over $k$ that is isomorphic to the
perfectization of a unipotent algebraic group over $k$.
\end{defins}

\begin{rem}
If $X$ is an arbitrary scheme over $k$, then, by adjunction, we
obtain a canonical morphism $X_{per\!f}\rar{}X$. It is known
\cite{greenberg} to be a homeomorphism of the underlying topological
spaces. Furthermore, if $f:U\rar{}X$ is an \'etale morphism, the
induced morphism $U_{per\!f}\rar{}U\times_X X_{per\!f}$ is an
isomorphism, and the functor $U\longmapsto U_{per\!f}$ induces an
equivalence between the \'etale topos of $X$ and that of
$X_{per\!f}$.

\mbr

It follows that every construction that can be formulated in terms
of the \'etale topos of a scheme is ``insensitive'' to replacing a
scheme over $k$ with its perfectization. In particular, if $X$ is a
perfect quasi-algebraic scheme over $k$, the derived category
$\sD(X)=D^b_c(X,\ql)$ of $X$ and the abelian category $\Perv(X)$ of
perverse sheaves on $X$ can be defined in the same way as for
schemes of finite type over $k$. We also have the formalism of the
six functors (cf.~\S\ref{ss:six-functors}) for the derived
categories of perfect quasi-algebraic schemes over $k$. Further, if
$G$ is a perfect unipotent group acting on a quasi-algebraic scheme
$X$ over $k$, the equivariant derived category $\sD_G(X)$ can be
defined as in \S\ref{ss:equivariant-derived}, and character sheaves
on $G$ can be defined as in \S\ref{ss:char-sheaves-L-packets}.
Moreover, every result about character sheaves on $G$ (or, more
generally, about the category $\sD_G(G)$) that can be proved for
perfect unipotent groups is also automatically valid for ordinary
unipotent algebraic groups over $k$.
\end{rem}

\begin{convention}
From now on, unless explicitly stated otherwise, all schemes under
consideration will be assumed to be perfect schemes over $k$. This
convention will allow us to simplify the formulation of all of our
results that depend on Serre duality.
\end{convention}

\subsection{Serre duality and admissible
pairs}\label{ss:Serre-duality-admissible-pairs} We remain in the
setup of \S\ref{ss:basic-definitions}. From the viewpoint of the
sheaves-to-functions correspondence, the next definition is a
geometric analogue of the notion of a $1$-dimensional representation
of a finite group.
\begin{defin}\label{d:multiplicative}
A \emph{multiplicative local system} on a perfect connected
quasi-algebraic group $H$ over $k$ is a rank $\ql$-local system
$\cL$ of rank $1$ equipped with an isomorphism
$\mu^*(\cL)\rar{\simeq}\cL\boxtimes\cL$, where $\mu:H\times
H\rar{}H$ denotes the multiplication morphism.
\end{defin}
By abuse of notation, we usually denote a multiplicative local
system by a single letter such as $\cL$. This is harmless in view of
Remark \ref{r:family-mult-loc-sys}(iii).

\mbr

For our purposes, the \emph{Serre dual} of $H$ should be thought of
as the moduli space of multiplicative $\ql$-local systems on $H$ (a
better approach is mentioned in Remark
\ref{r:serre-duality-canonical}). The notion of Serre duality for
unipotent groups (not to be confused with Serre duality in the
cohomology theory for coherent sheaves) goes back to the article
\cite{serre}, and is discussed in more detail in \S\ref{s:serre-FD}
below. For the time being, it will suffice to know that the Serre
dual, $H^*$, of an arbitrary perfect \emph{connected} unipotent
group $H$ exists as a (possibly disconnected) perfect commutative
unipotent group over $k$ (see Proposition \ref{p:serre-dual}).

\mbr

Furthermore, the construction of the biadditive pairing $B_\chi$
that enters condition (2) in the definition of an admissible pair
for a finite group (see \#4 in \S\ref{ss:algebraic-reminders})
admits a geometrization, which plays a role in the definition of an
admissible pair for a perfect unipotent group. This geometric
construction is explained in more detail in
\S\ref{ss:auxiliary-construction}. For the time being, we will use
this construction as a ``black box.''

\begin{defin}\label{d:normalizer}
Let $G$ be a perfect unipotent group over $k$, and let $(H,\cL)$ be
a pair consisting of a connected subgroup $H\subset G$ and a
multiplicative local system $\cL$ on $H$. The \emph{normalizer}
$N_G(H,\cL)$ of $(H,\cL)$ in $G$ is defined as the stabilizer of the
isomorphism class $[\cL]\in H^*(k)$ in the normalizer\footnote{Since
$H^*$ is defined by a universal property, the conjugation action of
$N_G(H)$ on $H$ induces an action of $N_G(H)$ on $H^*$. Note also
that $[\cL]$ is a point of $H^*$ over $k$ by the definition of
$H^*$.} $N_G(H)\subset G$ of $H$ in $G$.
\end{defin}

\begin{defin}\label{d:admissible-pair}
Let $G$ be a perfect unipotent group over $k$. An \emph{admissible
pair} for $G$ is a pair $(H,\cL)$ consisting of a connected subgroup
$H\subset G$ and a multiplicative local system $\cL$ on $H$ such
that the following three conditions are satisfied.
 \sbr
\begin{enumerate}[(1)]
\item Let $G'$ be the normalizer of $(H,\cL)$ in $G$ (see Definition
\ref{d:normalizer}), and let $G^{\prime\circ}$ denote its neutral
connected component. Then $G^{\prime\circ}/H$ is commutative.
 \sbr
\item The $k$-group morphism
$\vp_\cL : G^{\prime\circ}/H \rar{}
\bigl(G^{\prime\circ}/H\bigr)^*$ defined by applying the
construction of \S\ref{ss:auxiliary-construction} below to
$U=Z=G^{\prime\circ}$, $N=H$, and $\cN=\cL$ is an isogeny.
 \sbr
\item For every $g\in G(k)$ such that $g\not\in G'(k)$, we have
\[
\cL\bigl\lvert_{(H\cap H^g)^\circ} \not\cong
\cL^g\bigl\lvert_{(H\cap H^g)^\circ},
\]
where $H^g=g^{-1}Hg$ and $\cL^g$ is the multiplicative local system
on $H^g$ obtained from $\cL$ by transport of structure (via the map
$h\mapsto g^{-1}hg$).
\end{enumerate}
\end{defin}

\begin{rem}
Since
$\dim\bigl(G^{\prime\circ}/H\bigr)=\dim\bigl(G^{\prime\circ}/H\bigr)^*$
(see \S\ref{ss:serre-duality-properties}), condition (2) in the last
definition is equivalent to finiteness of $\Ker\vp_\cL$.
\end{rem}

\subsection{Heisenberg minimal idempotents}
\label{ss:Heisenberg-minimal-idempotents} The notion of an
admissible pair allows us to construct a certain special class of
minimal idempotents in equivariant derived categories of unipotent
groups. Namely, let $G$ be a perfect unipotent group over $k$, let
$(H,\cL)$ be an admissible pair for $G$, and let $G'$ be its
normalizer in $G$, as defined above. Write $\bK_H$ for the dualizing
complex of $H$, which in our setup is isomorphic to $\ql[2\dim H]$,
where $\ql$ is the constant $\ell$-adic local system of rank $1$ on
$H$ (because $k$ is algebraically closed). Finally, put
$e_\cL=\cL\tens\bK_H$, and let $e'_\cL$ denote the object of
$\sD_{G'}(G')$ obtained from $e_\cL$ by extension by zero. (Since
$(H,\cL)$ is invariant under the conjugation action of $G'$, both
$\cL$ and $\bK_H$ have canonical $G'$-equivariant structures.)

\begin{lem}\label{l:heisenberg-minimal-idempotent}\label{lemma6}
The object $e'_\cL$ is a closed idempotent in $\sD_{G'}(G')$, which,
moreover, is minimal as a weak idempotent in $\sD_{G'}(G')$.
\end{lem}
The lemma results from \cite[Prop.~8.1(a)--(b)]{characters}. In what
follows, we refer to $e'_\cL$ as the \emph{Heisenberg minimal
idempotent} on $G'$ defined by the admissible pair $(H,\cL)$. As we
will see shortly, one can obtain a minimal closed idempotent in
$\sD_G(G)$ from $e'_\cL$ via ``induction with compact supports''
(cf.~Theorem \ref{t:construction-L-packets}).

\mbr

The following result covers the base case in the inductive proof of
Theorems \ref{t:properties-character-sheaves} and \ref{t:functional-dimension} that will be given in
this article. 

\begin{thm}[S.~Datta and T.~Deshpande]\label{t:Heisenberg-L-packets}
If $G'$ is a perfect unipotent group over $k$ and
$e'_{\cL}\in\sD_{G'}(G')$ is a Heisenberg minimal idempotent, then
Theorems \ref{t:properties-character-sheaves} and
\ref{t:functional-dimension} hold for $(G',e'_{\cL})$ in place of
$(G,e)$.
\end{thm}

The proof given in
\S\ref{ss:proof-t:Heisenberg-L-packets} consists of references to the articles
by S.~Datta \cite{swarnendu} and T.~Deshpande \cite{tanmay}.


\subsection{Averaging functors and induction functors}\label{ss:induction-functors}
We remain in the setup of \S\ref{ss:basic-definitions}. Let $G$ be a
perfect unipotent group acting on a perfect quasi-algebraic scheme
$X$ over $k$, and let $G'\subset G$ be a closed subgroup. There is
an obvious forgetful functor
\begin{equation}\label{e:forgetful-equivariant}
F:\sD_G(X) \rar{} \sD_{G'}(X).
\end{equation}
We will need its right and left adjoints. Let us recall their
construction. It is based on the following remark.

\begin{rem}   \label{r:factorization}
The functor \eqref{e:forgetful-equivariant} can be factored as
\begin{equation}   \label{e:factorization}
\xymatrix{
   \sD_G(X) \ar[rr]^{\pr_2^*\ \ \ \ \ \ } & &  \sD_G((G/G')\times X)
 \ar[rr]^{\ \ \ \ \ i^*\circ\Phi}_{\ \ \ \ \ \sim} & &  \sD_{G'}(X)
  }
\end{equation}
Here $\Phi:\sD_G((G/G')\times X)\rar{}\sD_{G'}((G/G')\times X)$ is
the forgetful functor, $G$ acts on both $G/G'$ and $X$, $\pr_2
:(G/G')\times X\to X$ is the projection, and $i:X\hookrightarrow
(G/G')\times X$ takes $x\in X$ to $(\bar 1,x)$, where $\bar 1\in
G/G'$ is the image of $1\in G$.
\end{rem}


\begin{lem}\label{l:averaging}
The forgetful functor \eqref{e:forgetful-equivariant} has a right
adjoint and a left adjoint. The right adjoint is the functor
\begin{equation}   \label{e:averaging}
\Av_{G/G'} : \sD_{G'}(X)\rar{}\sD_G(X); \quad
\Av_{G/G'}:=(\pr_2)_*(i^*\circ\Phi)^{-1}
\end{equation}
$($we are using the notation of Remark~\ref{r:factorization}$)$. The
left adjoint is the functor \[M\longmapsto \av_{G/G'}(M)[2d](d),\]
where $d:=\dim (G/G')$ and $\av_{G/G'} : \sD_{G'}(X)\rar{}\sD_G(X)$
is defined by
\begin{equation}   \label{e:!averaging}
\av_{G/G'}:=(\pr_2)_!(i^*\circ\Phi)^{-1}.
\end{equation}
\end{lem}

\begin{proof}
Use the factorization \eqref{e:factorization} and the isomorphism
$\pr_2^! (M)\iso \pr_2^*(M)[2d](d)$, $d:=\dim (G/G')$.
\end{proof}

\begin{defin}\label{d:averaging}
The functor 
\eqref{e:averaging} will be called the \emph{averaging functor} (for the pair $G'\subset G$).
The functor \eqref{e:!averaging} will be called \emph{averaging with
compact supports}. If $G'=\{1\}$, we will write $\av_G$ and $\Av_G$
in place of $\avg$ and $\Av_{G/G'}$.
\end{defin}

Note that the canonical morphism $(\pr_{2})_!\rar{}(\pr_{2})_*$ induces a canonical  morphism
\begin{equation}  \label{e:!to*}
\av_{G/G'}\rar{}\Av_{G/G'}.
\end{equation}
\begin{rem}
Applying the forgetful functor $F$ to \eqref{e:!to*} and using the
adjunction of Lemma \ref{l:averaging} yields the following two
morphisms:
\begin{equation}\label{e:average-morphisms}
F\circ\avg\rar{}F\circ\Avg\rar{}\Id_{\sD_{G'}(X)}.
\end{equation}
These morphisms will be described more explicitly in Lemma
\ref{l:lem2}.
\end{rem}

Now let $G$ be a perfect unipotent group over $k$, and let $G'\subset G$
be a closed subgroup. Let $G$ (respectively, $G'$) act on itself by
conjugation. Further, let $\iota:G'\into G$ denote the inclusion
morphism. Then $\iota$ induces the functor
\[
\iota_!=\iota_* : \sD_{G'}(G')\rar{}\sD_{G'}(G).
\]

\begin{defin}\label{d:induction}
The functors of \emph{induction} and \emph{induction with compact
supports} are defined by
\[
\Ig=\Av_{G/G'}\circ\iota_* : \sD_{G'}(G')\rar{}\sD_G(G)
\]
and
\[
\ig=\av_{G/G'}\circ\iota_* : \sD_{G'}(G')\rar{}\sD_G(G),
\]
respectively.
\end{defin}

The morphism \eqref{e:!to*} induces a canonical morphism
$\ig\rar{}\Ig$. Lemma~\ref{l:averaging} implies the following statement.

\begin{cor}\label{c:induction-adjunctions}
\begin{enumerate}
\item[$($i$)$] Define the functor $\operatorname{Res}^G_{G'}:\sD_G(G)\rar{}\sD_{G'}(G')$ to be the
composition $\sD_G(G)\rar{}\sD_{G'}(G)\rar{\iota^*}\sD_{G'}(G')$,
where the first functor is the forgetful one. Then $\Ig$ is right
adjoint to $\operatorname{Res}^G_{G'}$.
 \sbr
\item[$($ii$)$]  Define the functor $\operatorname{res}^G_{G'}:\sD_G(G)\rar{}\sD_{G'}(G')$ to be the
composition \\ $\sD_G(G)\rar{}\sD_{G'}(G)\rar{\iota^!}\sD_{G'}(G')$.
Then $\ig$ is left adjoint to the functor
\[
M\longmapsto \operatorname{res}^G_{G'}M[2d](d), \quad d:=\dim
(G/G').
\]
\end{enumerate}
\end{cor}

The functor $\operatorname{Res}^G_{G'}$ defined in the corollary will be called the {\em restriction
functor.}

\begin{rem}
The functors $\avg$, $\Av_{G/G'}$, $\ig$ and $\Ig$ are triangulated.
The canonical morphisms $\avg\rar{}\Av_{G/G'}$ and $\ig\rar{}\Ig$
are morphisms of triangulated functors.
\end{rem}

\subsection{Construction of
$\bL$-packets}\label{ss:construction-L-packets} The next result
provides an explicit construction of all $\bL$-packets of character
sheaves on a unipotent group over $k$.

%

\begin{thm}\label{t:construction-L-packets}
Let $G$ be a perfect unipotent group over $k$.
\begin{enumerate}[$($a$)$]
\item Let $(H,\cL)$ be an admissible pair for $G$
$($Def.~\ref{d:admissible-pair}$)$, let $G'$ be its normalizer in
$G$ $($Def.~\ref{d:normalizer}$)$, and let $e'_{\cL}\in\sD_{G'}(G')$
be the corresponding Heisenberg minimal idempotent
$($\S\ref{ss:Heisenberg-minimal-idempotents}$)$. Then $f=\ig
e'_{\cL}$ is a minimal closed idempotent in $\sD_G(G)$.
 \sbr
\item With the notation of Theorem
\ref{t:properties-character-sheaves}, $n_{e'_\cL}=\dim H$ and
$n_f=\dim H-\dim(G/G')$.
 \sbr
\item Every minimal closed idempotent $f$ in $\sD_G(G)$ arises from some
admissible pair $(H,\cL)$ by means of the construction described in
part $($a$)$.
\end{enumerate}
\end{thm}

For the proof, see \S\ref{ss:proof-t:construction-L-packets}.


\begin{cor}
Let $(H,\cL)$ be an admissible pair for $G$ and let  $G'$ be its normalizer in $G$.
Then $\dim H\ge \dim (G/G')$.
\end{cor}

\begin{proof}
Combine Theorem~\ref{t:construction-L-packets}(b) with the
inequality $n_f\ge 0$, which is a part of
Theorem~\ref{t:properties-character-sheaves}(b).
\end{proof}

\begin{rem}   \label{r:non-unique}
In Theorem~\ref{t:construction-L-packets}(c) the pair $(H,\cL)$ is
not unique, in general. However, we do have a weaker uniqueness
statement:
\end{rem}

\begin{prop}\label{p:uniqueness}
Let $G$, $H$ and $\cL$ be as in Theorem
\ref{t:construction-L-packets}$($a$)$. Then there exists a connected
normal subgroup $A\subset G$ such that
\begin{enumerate}[$($a$)$]
\item $A\subset H$,
 \sbr
\item $\cL\bigl\lvert_A$ is $G^\circ$-invariant, and
 \sbr
\item $A$ contains all connected normal subgroups of $G$ satisfying
$($a$)$ and $($b$)$.
\end{enumerate}
Moreover, the $G$-orbit of $\bigl(A,\cL\bigl\lvert_A\bigr)$ depends
only on the idempotent $f\in\sD_G(G)$ constructed in Theorem
\ref{t:construction-L-packets}$($a$)$.
\end{prop}

The proposition is proved in \S\ref{ss:proof-p:uniqueness}.

\begin{rem}
If $G$ is connected, then $\bigl(A,\cL\bigl\lvert_A\bigr)$ is
$G$-invariant by property (b), so in this case the last statement of
the proposition amounts to the assertion that the pair
$\bigl(A,\cL\bigl\lvert_A\bigr)$ is uniquely determined by $f$.
\end{rem}

\subsection{The Verdier dual of a minimal idempotent}
The next result is a more precise version of Proposition
\ref{p:duality}. In order to state it we temporarily return to the
framework of ordinary (as opposed to perfect) unipotent groups.

\begin{prop}\label{p:duality-canonical}
Let $G$ be a unipotent algebraic group over $k$, let $(H,\cL)$ be an
admissible pair for $G$, write $G'=N_G(H,\cL)$, and let $f=\ig
e'_{\cL}\in\sD_G(G)$ be the corresponding minimal closed idempotent.
There is a natural isomorphism
\begin{equation}\label{e:duality-canonical}
\bD_G^- f \rar{\simeq} f[-2n_f](-n_f).
\end{equation}
\end{prop}

This result is proved in \S\ref{ss:proof-p:duality-canonical}.

\begin{rem}
If we work in the framework of perfect unipotent groups, the
assertion above fails already when $H=G$. The reason is that the
dualizing complex of a smooth perfect quasi-algebraic variety of
dimension $d$ over $k$ \emph{cannot} be canonically identified with
$\ql[2d](d)$.
\end{rem}

Proposition \ref{p:duality-canonical} states that with the
notation of Proposition \ref{p:duality}, 
a choice of an admissible pair $(H,\cL)$ for $G$ 
giving rise to the minimal closed
idempotent $e$ 
defines an isomorphism $\ql\rar{\simeq}L_e$.

\begin{quest}
Does this isomorphism depend on the choice of $(H,\cL)$ ?
\end{quest}

\subsection{Existence of minimal idempotents}
\begin{thm}\label{t:idempotents}
Let $G$ be a perfect unipotent group over $k$. Then
\begin{enumerate}[$($a$)$]
\item Every minimal closed idempotent in $\sD_G(G)$ is minimal
as a weak idempotent.
 \sbr
\item Every minimal weak idempotent in $\sD_G(G)$ is closed.
 \sbr
\item For every nonzero object $N\in\sD(G)$, there exists a minimal
closed idempotent $f\in\sD_G(G)$ with $f*N\neq 0$.
\end{enumerate}
\end{thm}

For the proof, see \S\ref{ss:proof-t:idempotents}.

\begin{rem}
Parts (a) and (b) of the last theorem can be reformulated together
as follows: the class of minimal weak idempotents in $\sD_G(G)$
coincides with the class of minimal closed idempotents in
$\sD_G(G)$.
\end{rem}

\subsection{Some Mackey theory}\label{ss:mackey-theory}
We remain in the setup of \S\ref{ss:basic-definitions}. Let $G$ be a
perfect unipotent group over $k$, and let $G'\subset G$ be a closed
subgroup. The following notion is an obvious geometrization of the
Mackey irreducibility criterion for induced characters of finite
groups, as stated in \#3 of \S\ref{ss:algebraic-reminders}.
\begin{defin}\label{d:geom-mackey-condition}
A weak idempotent $e\in\sD_{G'}(G')$ is said to satisfy the
\emph{geometric Mackey condition} with respect to $G$ if for every
$x\in G(k)\setminus G'(k)$, we have
$\overline{e}*\de_x*\overline{e}=0$, where $\de_x\in\sD(G)$ denotes
the delta-sheaf at the point $x$, and $\overline{e}$ denotes the
object of $\sD(G)$ obtained by extending $e$ to $G$ by zero outside
of $G'$.
\end{defin}

The last main result of our work, proved in
\S\ref{ss:proof-t:geometric-mackey}, is

\begin{thm}\label{t:geometric-mackey}
Let $G$ be a perfect unipotent group over $k$, let $G'\subset G$ be
a closed subgroup, let $e\in\sD_{G'}(G')$ be a minimal closed
idempotent satisfying the geometric Mackey condition with respect to
$G$, and let $f=\ig e$. Then
 \sbr
\begin{enumerate}[$($a$)$]
\item $f$ is a minimal closed idempotent in $\sD_G(G)$;
 \sbr
\item we have $n_f=n_e-\dim(G/G')$, or, equivalently,
$d_f=d_e+\dim(G/G')$;
 \sbr
\item the functor $\ig$ restricts to a monoidal equivalence
$e\sD_{G'}(G')\rar{\sim}f\sD_G(G)$, which is compatible with the braidings constructed in Definition \ref{d:braiding-equivariant-derived} and the twists constructed in Definition \ref{d:twist-equivariant-derived};
 \sbr
\item for every $M\in e\sD_{G'}(G')$, the canonical arrow $\can_M:\ig
M\rar{}\Ig M$ is an isomorphism;
 \sbr
\item if $M\in e\sD_{G'}(G')$ is perverse, then so is $(\ig
M)[\dim(G/G')]$.
\end{enumerate}
\end{thm}

\begin{rem}
The functor $\ig:\sD_{G'}(G')\rar{}\sD_G(G)$ is usually not
monoidal\footnote{Similarly, if $\Ga'\subset\Ga$ are finite groups,
the induction map from class functions on $\Ga'$ to class functions
on $\Ga$ does not preserve convolution.}. Nevertheless, in
\S\ref{ss:weak-semigroupal} we construct a {\em weak semigroupal
structure} on $\ig$ in the sense of
Definition~\ref{d:weak-semigroupal}. The first part of assertion (c) of
Theorem~\ref{t:geometric-mackey} means that $\ig$ restricts to an
equivalence $F:e\sD_{G'}(G')\rar{\sim}f\sD_G(G)$ and that the
induced weak semigroupal structure on $F$ is actually strong (see
Definition \ref{d:weak-semigroupal} and Remark
\ref{r:semigroupal-equivalence}). The other two statements of assertion (c) follow from the more general Lemma~\ref{l:compatibility-ind-braiding-twists}.
\end{rem}


\begin{rem}
In the situation of Theorem~\ref{t:geometric-mackey} there is a canonical
bijection between the set of idempotent arrows $\e\rar{}e$ and the
set of idempotent arrows $\e\rar{}f$ (see Remark~\ref{r:canonical_bij} for more details).
\end{rem}

\subsection{The main facts used in the proofs} \label{ss:key_ingredients}
The proofs of the main results of this article rely on
the theory of Fourier-Deligne transform,  the equality
\begin{equation}  \label{Ext=0}
\Ext^2(\bG_a,\bQ_p/\bZ_p)=0
\end{equation}
proved by L.~Breen \cite{breen-exts},
and the fact that the orbits of a unipotent group acting on an affine variety are closed.
The latter is used in the proof of Proposition~\ref{p:avg-closed-idemp}, and the equality \eqref{Ext=0} was used in \cite{characters} to prove the key Proposition~\ref{p:extension-locsys}.

\subsection{Organization of the text}\label{ss:organization}
The rest of the article is organized as follows. In
\S\ref{s:idempotents} we define and study the notion of a closed
idempotent in an abstract monoidal category. We establish several
general results that are used in the proofs of the main theorems of
the paper.

\mbr

In \S\ref{s:serre-FD} we recall the notion of Serre
duality \cite{serre,begueri} for perfect connected unipotent groups
in characteristic $p>0$ and that of Fourier-Deligne
transform. We have to work in a setting slightly more general than the standard one
(namely, to consider multiplicative local systems on non-necessarily commutative
connected unipotent groups). The main results on the Fourier-Deligne transform in this setting are
Proposition \ref{p:fourier-deligne-*-isomorphism} and
Theorem \ref{t:FD-closed-idempotent}.

%
%

\mbr

In \S\ref{s:reduction} we establish Proposition
\ref{p:normal-admissible-pairs} and Theorem
\ref{t:reduction-complexes}, which are used in \S\ref{s:proofs} to
prove the main theorems of the article by induction. The proofs of
the results of \S\ref{s:reduction} are somewhat similar to the proof
of the classical theorem  that any irreducible representation of a
finite nilpotent group is induced from a 1-dimensional
representation of a subgroup.

\mbr

We devote \S\ref{s:mackey} to a general study of triples $(G,G',e)$
satisfying the geometric Mackey condition from Definition
\ref{d:geom-mackey-condition}, while \S\ref{s:averaging} establishes
several properties of the functors $\avg$ and $\Avg$.

\mbr

The main results of the article, stated 
in \S\S\ref{ss:char-sheaves-L-packets}--\ref{ss:mackey-theory}, are
proved in \S\ref{s:proofs}. They are deduced without difficulty from
the crucial Proposition \ref{p:crucial}, proved
in \S\ref{ss:proof-p:crucial}.

\mbr

In Appendices~\ref{s:dualityformalism} and \ref{s:CFT}  we discuss certain abstract categorical formalisms
and apply them to the category $\sD_G(G)$ from \S\ref{ss:themonoidal}.

\mbr

Appendix~\ref{s:dualityformalism} is devoted to
a duality formalism in the spirit of Grothendieck and Verdier. We introduce the notions of a Grothendieck-Verdier category and of an r-category. The latter generalizes the notion of a rigid monoidal category. We define the notions of a pivotal structure and a ribbon structure on a Grothendieck-Verdier category. (A more detailed exposition of the theory of \GV categories appears in \cite{GV}.) Assuming that $G^{\circ}$
is unipotent, we prove in Appendix~\ref{s:dualityformalism} that the monoidal category $\sD_G(G)$  is an r-category and describe a natural ribbon structure on it (without the unipotence assumption this is done in Appendix~\ref{s:CFT}). We also prove that every Hecke subcategory of a pivotal (respectively, ribbon) Grothendieck-Verdier category is also a pivotal (respectively, ribbon) \GV category.

\mbr

Appendix~\ref{s:CFT} is devoted to a formalism of ``pre-topological field theories'' (the notion of a pre-TFT is a ``lax'' version of the notion of a TFT). Slightly modifying \cite[\S6]{BFN08}, we associate
a $2$-dimensional pre-TFT to any algebraic stack $\sX$ of finite type over a field $k$.
In the case where $\sX$ is the classifying stack of an algebraic group $G$ this pre-TFT allows us to solve
two problems. First, we define a canonical ribbon structure on the r-category $\sD_G(G)$, where $G$ is an \emph{arbitrary} algebraic group over $k$ and $\sD_G(G)$ is understood as the bounded derived category of constructible $\ell$-adic complexes on the stack quotient of $G$ by the adjoint action of $G$
\cite{Las-Ols06}. Second, we prove a key Lemma~\ref{l:compatibility-ind-braiding-twists}
(on the compatibility of the functor $\ig :\sD_{G'}(G')\rar{}\sD_{G}(G)$ with braidings and twists).

\mbr

In Appendix \ref{a:equivalence-equivariant-derived} we show that if $G$ is an algebraic group over $k$ acting on a scheme $X$ of finite type over $k$, then the category $\sD_G(X)$ constructed in Definition \ref{d:equiv-derived} is naturally equivalent to the category $D^b_c(G\backslash X,\ql)$, where $G\backslash X$ is the stack quotient of $X$ by $G$, provided that the neutral connected component $G^\circ$ of $G$ is \emph{unipotent}.

%
%
%



\section{Idempotents in monoidal categories}\label{s:idempotents}

\subsection{Monoidal categories}\label{ss:monoidal-notation}
There are several equivalent approaches to defining the notion of a
monoidal category. The following approach will be the most
convenient for us. If $\cM$ is a category, a \emph{semigroupal
structure} on $\cM$ is a pair consisting of a functor
$\tens:\cM\times\cM\rar{}\cM$ and a functorial collection, $\al$, of
isomorphisms
\[
\al_{X,Y,Z}:(X\tens Y)\tens Z\rar{\simeq}X\tens(Y\tens Z)
\]
for all $X,Y,Z\in\cM$, satisfying the ``pentagon axiom''
\cite[p.~162]{m}. One calls $\al$ the \emph{associativity
constraint} for $\tens$. A \emph{semigroupal category} is a triple
$(\cM,\tens,\al)$ consisting of a category $\cM$ and a semigroupal
structure $(\tens,\al)$ on $\cM$. Often, by abuse of notation, one
denotes a semigroupal category by a single letter such as $\cM$.

\begin{defins}   \label{defs:unital}
\begin{enumerate}[(1)]
\item An object $E$ of a semigroupal category $\cM$ is said to be
\emph{unital} if the functors $X\mapsto E\tens X$ and $X\mapsto
X\tens E$ are isomorphic to the identity functor on $\cM$.
 \sbr
\item A semigroupal category is said to be \emph{monoidal} if it has
a unital object.
 \sbr
\item A \emph{unit object} in a monoidal category $\cM$ is a pair
$(E,u)$ consisting of a unital object $E$ and an
isomorphism\footnote{By definition, there exists such an isomorphism
for every unital object $E$.} $u:E\tens E\rar{\simeq}E$.
\end{enumerate}
\end{defins}

\begin{lem}[\cite{categories-sheaves}, Rem.~4.2.7]\label{l:unit-uniqueness}
If $\cM$ is a monoidal category and $(E,u)$, $(E',u')$ are two unit
objects of $\cM$, there exists a unique isomorphism
$f:E\rar{\simeq}E'$ such that $u'\circ(f\tens f)=f\circ u$.
\end{lem}


The lemma allows us to speak of `the' unit object in a monoidal
category (rather than `a' unit object). The unit object of a
monoidal category will usually be denoted by $\e$.

\begin{lem}[\cite{categories-sheaves}, Lem.~4.2.6]\label{l:mon-cat-unit-constraints}
If $\cM$ is a monoidal category and $\e=(\e,u)$ is a unit object in
$\cM$, there exist unique functorial collections of
isomorphisms
\[
\la=\bigl\{\la_X:\e\tens X\rar{\simeq}X \bigr\}_{X\in\cM}
\qquad\text{and}\qquad \rho =\bigl\{\rho_X:X\tens\e\rar{\simeq}X
\bigr\}_{X\in\cM}
\]
such that $\la_\e=u=\rho_\e$ and the triangle axiom \cite[p.~163]{m}
holds.
\end{lem}

This lemma implies that the definition of a monoidal category given
above is equivalent to the one from \cite{m}.

\subsection{Semigroupal and monoidal
functors}\label{ss:semigroupal-monoidal-functors}
\begin{defin}\label{d:weak-semigroupal}
Let $\cM_1$ and $\cM_2$ be semigroupal categories.
\begin{enumerate}[(a)]
\item
A \emph{weak semigroupal structure} on a functor $F:\cM_1\to\cM_2$
between the underlying categories is a functorial collection,
$\phi$, of morphisms
\[
\phi(X,Y) : F(X)\tens F(Y) \rar{} F(X\tens Y) \qquad
\forall\,X,Y\in\cM_1
\]
that are compatible with the associativity constraints on $\cM_1$
and $\cM_2$.
 \sbr
\item If each $\phi(X,Y)$ is an isomorphism, then
$\phi$ is called a {\em semigroupal structure} (or a {\em strong
semigroupal structure}).
 \sbr
\item A {\em semigroupal} (respectively, {\em weakly semigroupal})
functor $\cM_1\to\cM_2$ is a functor $\cM_1\rar{}\cM_2$ equipped
with a semigroupal (respectively, weakly semigroupal) structure.
 \sbr
\item If $\cM_1$ and $\cM_2$ are monoidal categories, a semigroupal
functor $\cM_1\to\cM_2$ is said to be \emph{monoidal} if it takes
unital objects to unital objects.
\end{enumerate}
\end{defin}

\begin{rem}\label{r:semigroupal-equivalence}
Suppose that $\cM_1$ and $\cM_2$ are monoidal categories and
$F:\cM_1\to\cM_2$ is a semigroupal functor that is an equivalence of
the underlying categories. Then $F$ automatically takes unital
objects to unital objects. In other words, a semigroupal equivalence
between monoidal categories is the same as a monoidal equivalence.
\end{rem}

\begin{rem}
The notion of a weakly semigroupal functor is closely related to the
notion of {\em pseudo-tensor functor} between {\em pseudo-tensor
categories} from \cite{chiral} (see \S1.1.1 and \S1.1.5 of
\emph{op.~cit.}). Namely, by \cite[\S1.1.3]{chiral}, a semigroupal
category can be considered as a special kind of pseudo-tensor
category and by \cite[\S1.1.6(ii)]{chiral}, a pseudo-tensor functor
between semigroupal categories is the same as a weakly semigroupal
functor.
\end{rem}

\subsection{Algebras} If $(\cM,\tens,\al)$ is a semigroupal category
and $A\in\cM$, one defines the notion of an associative morphism
$\mu:A\tens A\rar{}A$ in the usual way. A pair $(A,\mu)$ consisting
of an object $A\in\cM$ and an associative morphism $\mu:A\tens
A\rar{}A$ is called an \emph{algebra} in $\cM$.
A weakly semigroupal functor takes algebras to algebras.

If $\cM$ is monoidal, one also has the notion of a \emph{unital algebra} (called
a ``monoid'' in \cite{m} and a ``ring'' in \cite{categories-sheaves}) in $\cM$.
In this case $\e_{\cM}$ has a canonical structure of unital algebra
and any object of $\cM$ has a canonical structure of $\e_{\cM}$-bimodule.

\subsection{Main definitions}\label{ss:idempotents-definitions}
\begin{defins}\label{defs:idempotents}
Let $\cM$ be a monoidal category.
\begin{enumerate}[(a)]
\item An object $e\in\cM$ is said to be a \emph{weak idempotent} if
$e\tens e\cong e$.
 \sbr
\item An \emph{idempotent algebra} in $\cM$ is an algebra $(e,\mu)$
such that $\mu:e\tens e\rar{}e$ is an isomorphism.
 \sbr
\item
A morphism $\e\rar{\pi}e$ in $\cM$ is said to be an \emph{idempotent arrow}
if both 
morphisms
\begin{equation}\label{e:morphisms-idempotent-arrow}
\e\tens e\xrar{\ \ \pi\tens\id_e\ \ }e\tens e \qquad\text{and}\qquad
e\tens\e\xrar{\ \ \id_e\tens\pi\ \ }e\tens e
\end{equation}
are isomorphisms.
 \sbr
\item An object $e\in\cM$ is said to be a \emph{closed idempotent} if
there exists an idempotent arrow $\e\rar{}e$.
\end{enumerate}
\end{defins}

\begin{rems}\label{rems:idempotents}
\begin{enumerate}[(i)]
\item In the situation of (b), (c) or (d), the object $e\in\cM$ is a weak
idempotent.
 \sbr
\item Lemma \ref{l:areequal} below asserts that if \emph{both}
arrows \eqref{e:morphisms-idempotent-arrow} are isomorphisms, then
they are equal as isomorphisms $e\iso e\tens e$. On the other hand, it is \textbf{not} enough to
require \emph{one} of them to be an isomorphism: see Remark
\ref{r:one-isom-not-enough}.
 \sbr
\item The term ``closed idempotent'' may be new, but the notion
itself is not. For instance, a special class of closed idempotents,
known as \emph{idempotent monads}, was known for several decades
(see \S\ref{ss:idempotent-monads} for more details). Moreover,
closed idempotents in the general sense appear in
\cite[Exercise~4.2]{categories-sheaves}.
 \sbr
\item The origin of the adjective ``closed'' is explained in
\S\ref{ss:examples-idempotents}.
 \sbr
\item A unit of an associative algebra is unique if it exists. In
particular, this remark applies to idempotent algebras.
 \sbr
\item If $(e,\mu)$ is an idempotent \emph{unital} algebra in $\cM$,
then the unit $\pi:\e\rar{} e$ is an idempotent arrow in $\cM$.
Indeed, the compositions $\e\tens e\xrar{\ \pi\tens\id_e\ }e\tens
e\rar{\mu}e$ and $e\tens\e\xrar{\ \id_e\tens\pi\ }e\tens
e\rar{\mu}e$ are equal to $\la_e$ and $\rho_e$, respectively, and
hence are isomorphisms. But $\mu$ is an isomorphism, hence so are
the arrows \eqref{e:morphisms-idempotent-arrow}.
 \sbr
\item We will see later (Proposition \ref{p:closed-idemp-idemp-alg})
that the notion of 
idempotent arrow is in fact \emph{equivalent\,}
to the notion of an idempotent unital algebra.
 \sbr
\item The notion of a weak idempotent is not very well behaved (see, for instance, Remark~\ref{r:caution}).
The notion of a closed idempotent is much better behaved. For
example, idempotent arrows have no automorphisms and isomorphism
classes of idempotent arrows bijectively correspond to isomorphism
classes of closed idempotents (see Corollary
\ref{c:uniqueness-idempotent-arrows}).
\end{enumerate}
\end{rems}

\begin{lem} \label{l:areequal}
If $\pi: \e\rar{}e$ is an idempotent arrow in a monoidal category
$\cM$ then the isomorphisms $e\iso e\tens e$ corresponding to isomorphisms \eqref{e:morphisms-idempotent-arrow} are equal to each other.
\end{lem}

\begin{proof}
By \cite[Lem.~4.1.2]{categories-sheaves}, the lemma holds if $\cM$ is the category
of endofunctors of some category. To prove the lemma for any $\cM$, note that
the left action of $\cM$ on itself defines a monoidal functor $F:\cM\rar{}\End(\cM)$,
where $\End(\cM)$ is the monoidal category of functors $\cM\rar{}\cM$.
Clearly $F(\pi )$ is an idempotent arrow in $\End(\cM)$, so by \emph{loc.~cit.},
the isomorphisms \eqref{e:morphisms-idempotent-arrow} have equal images under $F$.
Since $F$ is faithful we are done.
\end{proof}

\subsection{An example}\label{ss:examples-idempotents}
In this subsection we give an example that explains the origin of
the adjective ``closed'' in Definition \ref{defs:idempotents}(c). At
the same time we show that in general, not every weak idempotent in
a monoidal category is closed. Another important class of examples
appears in \S\ref{ss:idempotent-monads} below.

\mbr

Let $X$ be a scheme, and let $i:Y\rar{}X$ be a morphism of finite
type such that the diagonal morphism $Y\rar{}Y\times_X Y$ is an
isomorphism\footnote{Such an $i$ is not necessarily a locally closed
embedding. For instance (if $X$ is Noetherian), one can take
$Y=Z\coprod(X\setminus Z)$, where $Z\subset X$ is any closed
subscheme.}. If $A$ is any unital commutative ring, then sheaves of
$A$-modules on the \'etale site $X_{\text{\'et}}$ form a monoidal
category $\cM$ with respect to the usual tensor product. Consider
the sheaf $i_! A_Y\in\cM$, where $A_Y$ is the constant sheaf on
$Y_{\text{\'et}}$ with stalk $A$.
Each stalk of $i_! A_Y$ equals $A$ or $0$, so it has a canonical
ring structure. It is easy to show that these structures come from a
unique structure of idempotent algebra on $i_!A_Y$. In particular,
$i_!A_Y$ is a weak idempotent. It is easy to show that $i_!A_Y$ is a
closed idempotent in $\cM$ if and only if $i:Y\rar{}X$ is a closed
embedding. Furthermore, if $A$ is a field, then every closed
idempotent in $\cM$ is isomorphic to $i_!A_Y$ for some closed
$Y\subset X$.

\subsection{Closed idempotents via adjoint functors}\label{ss:idempotent-monads}
If $\cC$ is a (small)
category, the category $\End(\cC)$ of functors $\cC\rar{}\cC$ is a
strictly associative and strictly unital monoidal category with
respect to composition of functors. Closed idempotents in this
category were studied, under the name of \emph{idempotent monads},
in a number of earlier works, among which we mention
\cite{adams,dfh1,dfh2}. In this subsection we summarize the basic
facts relating idempotent monads to the notions of adjoint functors
and reflective subcategories, mostly following
\cite[\S4.1]{categories-sheaves} (the term ``projector'' is used in
\emph{loc.~cit.} in place of ``idempotent monad'').

\mbr

Let us first recall the following
\begin{defin}
A subcategory $\cD$ of a category $\cC$ is said to be
\emph{reflective} if the inclusion functor $\cD\into\cC$ admits a
left adjoint.
\end{defin}

\begin{prop}[\cite{categories-sheaves}, Prop.~4.1.3--4.1.4]\label{p:idempotent-monads}
\begin{enumerate}[$($a$)$]
\item Let $\cC$ be a category, let $\cD\subset\cC$ be a full
reflective subcategory, and let $L:\cC\rar{}\cD$ be a left adjoint
of the inclusion functor $I:\cD\into\cC$. Further, let us denote the
adjunction morphisms by $\eta:L\circ I\rar{}\Id_{\cD}$ and
$\eps:\Id_{\cC}\rar{}I\circ L$. Then $\eta$ is an isomorphism, and
$\eps$ is an idempotent arrow in the category of endofunctors of
$\cC$.
 \mbr
\item Conversely, let $P:\cC\rar{}\cC$ be a functor, and let
$\eps:\Id_{\cC}\rar{}P$ denote an idempotent arrow in the category
of endofunctors of $\cC$. Given an object $X\in\cC$, the following
three statements are equivalent:
 \sbr
\begin{enumerate}[$($i$)$]
\item the arrow $\eps_X:X\rar{}P(X)$ is an isomorphism;
 \sbr
\item for any $Y\in\cC$, the map $\Hom(P(Y),X)\rar{}\Hom(Y,X)$
 \sbr
given by $f\mapsto f\circ\eps_Y$ is bijective;
\item the map in $($ii$)$ is surjective for $Y=X$.
 \sbr
\end{enumerate}
If $\cD$ is the full subcategory of $\cC$ consisting of objects
$X\in\cC$ satisfying the equivalent conditions $($i$)$--$($iii$)$,
then $P(X)\in\cD$ for any $X\in\cC$. Furthermore, $P$ induces a
functor $\cC\rar{}\cD$ which is left adjoint to the inclusion
functor $\cD\into\cC$. A fortiori, $\cD$ is a reflective subcategory
of $\cC$.
\end{enumerate}
\end{prop}

\begin{rem}\label{r:one-isom-not-enough}
The following example, taken from
\cite[Exercise~4.1]{categories-sheaves}, shows that in Definition
\ref{defs:idempotents}(c) it is not enough to require one of the two
arrows in \eqref{e:morphisms-idempotent-arrow} to be an isomorphism.

\mbr

Let $\cC$ be the category with one object, $X$, such that
$\End_{\cC}(X)=\{\id_X,p\}$ and the arrow $p:X\rar{}X$ satisfies
$p^2=p$. Further, let $P:\cC\rar{}\cC$ be the functor defined by
$P(\id_X)=\id_X=P(p)$, and let $\eps:\Id_{\cC}\rar{}P$ be the
morphism of functors given by $\eps_X=p$. Then
$P\circ\eps:P\rar{}P^2$ is an isomorphism, but $\eps\circ
P:P\rar{}P^2$ is not.
\end{rem}

\subsection{Hecke subcategories}\label{ss:hecke-subcategories}
Let $\cM$ be a monoidal category and $e\in\cM$ a weak idempotent.
The essential image of the functor $\cM\to\cM$ defined by $X\mapsto e\tens X$
(resp. $X\mapsto X\tens e$, $X\mapsto e\tens X\tens e$) will be denoted by
$e\cM$ (resp. $\cM e$,  $e\cM e$). The full subcategories $e\cM$, $\cM e$, $e\cM e$
of the category $\cM$ can also be described as follows:
\begin{equation}  \label{eM}
e\cM=\bigl\{X\in\cM\st e\tens X\cong X \bigr\}, \qquad \cM e =
\bigl\{X\in\cM\st X\tens e\cong X \bigr\},
\end{equation}
\begin{equation}   \label{eMe}
e\cM e = \bigl\{X\in\cM\st e\tens X\tens e\cong X \bigr\}.
\end{equation}

\begin{lem}\label{l:hecke-intersection}
$e\cM e=e\cM\cap\cM e$.
\end{lem}

\begin{proof}
By definition, $e\cM e\subset e\cM\cap\cM e$. By \eqref{eM}-\eqref{eMe},
$e\cM e\supset e\cM\cap\cM e$.
\end{proof}

If the idempotent $e$ is closed then $e\cM$ and $\cM e$ have the following description.

\begin{lem}\label{onemoredescription}
Let $\pi:\e\rar{}e$ be an idempotent arrow. An object $X\in\cM$
belongs to $e\cM$ $($resp. $\cM e${}$)$ if and only if the morphism
$\pi\tens\id_X:\e\tens X\rar{}e\tens X$ $($resp.
$\id_X\tens\pi:X\tens\e\rar{}X\tens e${}$)$ is an isomorphism.
\hfill\qedsymbol
\end{lem}

\begin{cor}\label{c:hecke-closed-under-retracts}
If $e\in\cM$ is a closed idempotent, then the subcategories $e\cM$,
$\cM e$ and $e\cM e$ of $\cM$ are closed under retracts.
\end{cor}

\begin{proof}
It suffices to consider $e\cM$. If $\pi:\e\to e$ is an idempotent
arrow, $X\in e\cM$, and $Y$ is a retract of $X$ then the morphism
$\pi\tens\id_Y:\e\tens Y\rar{}e\tens Y$ is a retract of
$\pi\tens\id_X:\e\tens X\rar{}e\tens X$. But a retract of an
isomorphism is an isomorphism.
\end{proof}

\begin{defin}
We call $e\cM e\subset\cM$ the \emph{Hecke subcategory} of $\cM$
defined by $e$.
\end{defin}

The full subcategory $e\cM e\subset\cM$ is closed under $\tens$. So it is
a semigroupal category.

\begin{lem}\label{l:Hecke-monoidal}
If $e\in\cM$ is a closed
idempotent then the semigroupal category $e\cM e$ is monoidal and $e$ is
a unital object of $e\cM e$.
\end{lem}

\begin{proof}
Lemma~\ref{onemoredescription} shows that
the functor $X\mapsto e\tens X$ is isomorphic
to the identity functor on the subcategory $e\cM\subset\cM$ and
\emph{a fortiori}, on the subcategory $e\cM e\subset\cM$.
Similarly, the functor $X\mapsto X\tens e$ is isomorphic to the
identity functor on the subcategory $\cM e\subset\cM$, and hence
also on $e\cM e$.
\end{proof}

\begin{rem}\label{r:caution}
If $e\in\cM$ is a {\em weak\,} idempotent then the semigroupal
category $e\cM e$ may fail to be monoidal.\footnote{If $e$ is an
idempotent algebra there is a remedy, see Remark~\ref{r:remedy}.}
For instance, this happens if $\cM$ is the category $\sD
(\bA^1)=D^b_c(\bA^1,\ql)$ equipped with the usual tensor product and
$e=j_! \ql\oplus i_!\ql$, where $j: \bA^1\setminus \{ 0
\}\hookrightarrow\bA^1$ is the open embedding and $i: \{ 0
\}\hookrightarrow\bA^1$ is the corresponding closed
embedding\footnote{Let $e_1=i_!\ql$ and $e_2=j_!\ql$. Then $e_1\tens
e_1\cong e_1$, $e_2\tens e_2\cong e_2$ and $e_1\tens e_2=0$. Thus
$e_1,e_2\in e\cM e$. For $n=1,2$ there exist nonzero morphisms $e_1\to e_2[n]$ in $\cM$. They are annihilated by the functors $X\mapsto e_1\tens X$
and $X\mapsto e_2\tens X$, and hence also by the functor $X\mapsto
e\tens X$. In particular, the latter functor is not an
autoequivalence of $e\cM$.}. Note that if $G=\bG_a$ then the
monoidal category $(\sD (\bA^1), \tens )$ is equivalent to $(\sD_G
(G),*)$ via the Fourier-Deligne transform.
\end{rem}

\begin{lem}\label{l:transitivity}
Let $\e\rar{\pi}e$ be an idempotent arrow in a monoidal category
$\cM$, and let $e\rar{\varpi}f$ be an idempotent arrow in the Hecke
subcategory $e\cM e$. Then $\varpi\circ\pi$ is an idempotent arrow
in $\cM$.
\end{lem}

\begin{proof}
Since $f\in e\cM e\subset e\cM$ 
the arrow $\pi\tens\id_f:\e\tens f\rar{}e\tens f$ is an isomorphism by Lemma~\ref{onemoredescription}.
The arrow $\varpi\tens\id_f:e\tens f\rar{}f\tens f$ is an isomorphism by assumption, so
$(\varpi\circ\pi)\tens\id_f:\e\tens f\rar{}f\tens f$ is an isomorphism. Similarly,
$\id_f\tens(\varpi\circ\pi):f\tens\e\rar{}f\tens f$ is an isomorphism.
\end{proof}

\subsection{Adjunction properties}\label{ss:interplay-weak-closed-idempotents}
The following result (Proposition
\ref{p:interplay-weak-closed-idempotents}) will play an important
role in the proof of Proposition \ref{p:when-ind-closed}.

\begin{defin}\label{d:weakly-symmetric}
A semigroupal category $\cM$ is said to be \emph{weakly symmetric}
if the functors $\cM\times\cM\rar{}\cM$ given by $(X,Y)\longmapsto
X\tens Y$ and $(X,Y)\longmapsto Y\tens X$ are isomorphic.
(The name ``weakly symmetric'' may be non-standard.)
\end{defin}

For example, a braided monoidal category is weakly symmetric. In
general, if $\cM$ is weakly symmetric and $e\in\cM$ is a weak
idempotent, it is clear that all three subcategories $e\cM$, $\cM e$
and $e\cM e$ of $\cM$ coincide. Normally, we will write $e\cM$ for
the Hecke subcategory of $\cM$ defined by $e$ in this situation.

\begin{prop}\label{p:interplay-weak-closed-idempotents}
\begin{enumerate}[$($a$)$]
\item If $\e\rar{\pi}e$ is an idempotent arrow in a monoidal category $\cM$, then the functor
$L:\cM\rar{}e\cM$ given by $X\longmapsto e\tens X$ is left adjoint
to the inclusion functor $I:e\cM\into\cM$. More precisely, the
family of morphisms
\begin{equation}\label{e:adjunction-morphism}
\{\pi\tens\id_X:X\rar{}e\tens X\}_{X\in\cM}
\end{equation}
defines an adjunction morphism $\Id_{\cM}\rar{}I\circ L$.
 \sbr
\item Here is a partial converse. Let $\cM$ be a weakly
symmetric monoidal category, and let $e\in\cM$ be a weak idempotent.
Suppose that the functor $\cM\rar{}e\cM$ given by $X\longmapsto
e\tens X$ is left adjoint to the inclusion functor $e\cM\into\cM$.
Then $e$ is a closed idempotent in $\cM$.
\end{enumerate}
\end{prop}

\begin{proof}
(a) Consider the functor
\begin{equation}\label{e:(*)}
P=I\circ L:\cM\rar{}\cM, \qquad P(X)=e\tens X,
\end{equation}
and note that \eqref{e:adjunction-morphism} defines an idempotent
arrow $\eps:\Id_\cM\to P$ in the category of endofunctors of $\cM$.
It remains to apply Proposition \ref{p:idempotent-monads}(b) and
Lemma~\ref{onemoredescription}.

\mbr

\noindent (b) By Proposition \ref{p:idempotent-monads}(a), there
exists an idempotent arrow $\eps:\Id_\cM\rar{}P$, where $P$ is
defined by \eqref{e:(*)}.
Define $\pi :\e\to e$ to be the composition of $\eps_\e:\e\to e\tens\e$
and the canonical isomorphism $\rho_e :e\tens\e\rar{\simeq} e$.
Since $\eps$ is an idempotent arrow, the morphism
$\id_e\tens\eps_X:e\tens X\to e\tens(e\tens X)$ is an isomorphism
for every $X\in\cM$. Setting $X=\e$, we see that
$\id_e\tens\pi:e\tens\e\to e\tens e$ is an isomorphism. Since $\cM$
is weakly symmetric, $\pi\tens\id_e:\e\tens e\to e\tens e$ is also
an isomorphism.
\end{proof}

\subsection{Minimal idempotents and Jacobson monoidal categories}
\label{ss:minimal-idempotents}
\begin{defin}\label{d:minimal-idempotents}
Let $\cM$ be a weakly symmetric monoidal category that has a zero
object\footnote{In all the applications we have in mind, $\cM$ will
in fact be additive.}, $0$. A \emph{minimal closed} (respectively,
\emph{weak}) \emph{idempotent} in $\cM$ is a closed (respectively,
weak) idempotent $e\in\cM$ such that $e\neq 0$, and such that for
every closed (respectively, weak) idempotent $e'\in\cM$, we have
either $e\tens e'=0$, or $e\tens e'\cong e$.
\end{defin}

\begin{example}\label{example1}
Let $X$, $A$, $\cM$ be as in \S\ref{ss:examples-idempotents}.
Suppose that $A$ is a field. Then for every closed point $x\in X$,
the closed idempotent $i_!A_{\{x\}}\in\cM$ is minimal among weak
idempotents. Moreover, all minimal closed idempotents in $\cM$ have
this form.
\end{example}

\begin{rem}
A minimal closed idempotent might fail to be minimal as a weak
idempotent. On the other hand, if a minimal weak idempotent happens
to be a closed idempotent, then it is clearly minimal as a closed
idempotent as well.
\end{rem}

\begin{rem}
We will show in \S\ref{ss:idempotents-partial-order} that minimal
closed idempotents can also be defined in terms of a certain partial
order on the set of isomorphism classes of \emph{all} nonzero closed
idempotents in $\cM$ (cf.~Remark~\ref{r:classical} (iii)).
Moreover, the construction of this partial order does not require the
assumption that $\cM$ is weakly symmetric.
\end{rem}

The following observation is occasionally useful.

\begin{lem}\label{l:characterize-minimal-idempotents}
Let $\cM$ be a weakly symmetric monoidal category with a zero
object, $0$. A closed idempotent $e\in\cM$ is minimal among closed
idempotents if and only if $e\neq 0$ and the Hecke subcategory
$e\cM$ has no closed idempotents $($up to isomorphism$)$ apart from
$0$ and its unit object, $e$. Similarly, a weak idempotent $e\in\cM$
is minimal among weak idempotents if and only if $e\neq 0$ and every
nonzero weak idempotent in the Hecke subcategory $e\cM\subset\cM$ is
isomorphic to $e$.
\end{lem}

\begin{proof}
The ``only if" statements are clear. Let us prove the ``if"
statements. Let $\cM$ be a weakly symmetric monoidal category. If
$\pi:\e\to e$ and $\pi':\e\to e'$ are idempotent arrows in $\cM$,
the fact that $\cM$ is weakly symmetric implies that
$\id_e\tens\pi:e\rar{}e\tens e'$ is an idempotent arrow in $e\cM$.
This proves the first assertion. For the second assertion, observe
that if $e$ and $e'$ are weak idempotents in a weakly symmetric
semigroupal category $\cM$, then $e\tens e'$ is also a weak
idempotent.
\end{proof}

\begin{defin}\label{d:jacobson-monoidal-categories}
A weakly symmetric monoidal category $\cM$ with a zero object is
\emph{Jacobson} if for every nonzero $N\in\cM$ there exists a closed
idempotent $e\in\cM$ such that $e\tens N\neq 0$ and $e$ is minimal
among weak idempotents in $\cM$.
\end{defin}

\begin{example}\label{example2}
In the situation of Example \ref{example1}, $\cM$ is Jacobson if and
only if $X$ is a Jacobson scheme in the sense of
\cite[\S10]{ega4-3}.
\end{example}

\begin{prop}\label{p:jacobson1}
Let $\cM$ be a Jacobson monoidal category.
 \sbr
\begin{enumerate}[$($a$)$]
\item Every minimal closed idempotent in $\cM$ is minimal as a weak
idempotent.
 \sbr
\item Every minimal weak idempotent in $\cM$ is closed.
 \sbr
\item If $e\in\cM$ is a weak idempotent such that the semigroupal
category $e\cM$ is monoidal, then $e\cM$ is Jacobson. In particular,
if $e$ is a closed idempotent in $\cM$, then $e\cM$ is Jacobson.
\end{enumerate}
\end{prop}
\begin{proof}
We first prove (a) and (b) simultaneously. Let $f\in\cM$ be
\emph{either} a minimal closed idempotent \emph{or} a minimal weak
idempotent, and let $e\in\cM$ be a closed idempotent that is minimal
as a weak idempotent and satisfies $e\tens f\neq 0$. In either case,
using the minimality of $e$ and $f$, we find $e\cong e\tens f\cong f$.

%

\mbr

Let us prove (c).
For any nonzero object $N\in e\cM$, we can find
a closed idempotent $f\in\cM$ that is minimal as a weak idempotent and
satisfies $f\tens N\neq 0$. Since 
$N\in e\cM$ and $f\tens N\ne 0$ we have 
$f\tens e\neq 0$, so the minimality of $f$ implies that $f\tens e\simeq f$,
i.e., $f\in e\cM$. Clearly $f$ is  a minimal weak idempotent in $e\cM$.
Moreover, $f$ is a closed idempotent in $e\cM$
(if $\pi:\e\rar{}f$ is an idempotent arrow in
$\cM$ and $e\cM$ is monoidal, then $\id_e\tens\pi:e\rar{}e\tens f$
is an idempotent arrow in $e\cM$). 
\end{proof}

\begin{prop}\label{p:jacobson2}
Let $\cM$ be a weakly symmetric monoidal category with a zero
object. Suppose that for every nonzero $N\in\cM$ there exists a
closed idempotent $\widetilde{e}\in\cM$ such that
$\widetilde{e}\tens N\neq 0$ and $\widetilde{e}\cM$ is Jacobson.
Then $\cM$ is Jacobson.
\end{prop}

\begin{proof}
Let $N\in\cM$ be nonzero, and choose a closed idempotent
$\widetilde{e}\in\cM$ such that $\widetilde{e}\tens N\neq 0$ and
$\widetilde{e}\cM$ is Jacobson. Then $\widetilde{e}\tens N$ is a
nonzero object of $\widetilde{e}\cM$, so there exists a closed
idempotent $e\in\widetilde{e}\cM$ such that
$e\tens\widetilde{e}\tens N\neq 0$ and $e$ is minimal as a weak
idempotent in $\widetilde{e}\cM$. Clearly, $e\tens N\neq 0$. By
Lemma \ref{l:transitivity}, $e$ is a closed idempotent in $\cM$.
Moreover, $e\cM=e(\widetilde{e}\cM)$, so by Lemma
\ref{l:characterize-minimal-idempotents}, $e$ is minimal as a weak
idempotent in $\cM$. This proves that $\cM$ is Jacobson.
\end{proof}

\subsection{Modules over an idempotent algebra}
Statement (ii) of the following lemma allows to interpret the adjunction from
Proposition~\ref{p:interplay-weak-closed-idempotents} as a ``free-forget" adjunction.

\begin{lem}   \label{l:idemp2-alg-closed-idemp}
Let $(e,\mu_e)$ be a unital idempotent algebra in a monoidal
category $\cM$.
\begin{enumerate}[$($i$)$]
\item A left $e$-module $($resp. right $e$-module, $e$-bimodule$)$
$X$ is unital if and only if the action morphism $e\tens X\rar{} X$
$($resp. $X\tens e\rar{} X$, $e\tens X\tens e\rar{} X${}$)$ is an
isomorphism.
 \sbr
\item Let $e\!-\!\operatorname{mod}_{\cM}$ denote the category of
unital left $e$-modules. Then the forgetful functor
$F:e\!-\!\operatorname{mod}_{\cM}\rar{}\cM$ is fully faithful and
its essential image equals $e\cM$. Similarly, the category of unital
right $e$-modules $($resp. $e$-bimodules$)$ identifies with $\cM e$
$($resp. $e\cM e${}$)$.
\end{enumerate}
\end{lem}

\begin{proof}
We will consider only left modules.

\mbr

(i) The action morphism $\al :e\tens X\to X$ is a $e$-module
morphism, so if $\al$ is an isomorphism then the $e$-module $X$ is
isomorphic to the free $e$-module $e\tens X$, which is clearly
unital. Now suppose that $X$ is unital, i.e., the composition
\begin{equation}    \label{unitality}
X\rar{\simeq}\e\tens X \xrar{\ \ \pi\tens\id_X\ \ } e\tens X  \xrar{\ \ \al\ \ } X
\end{equation}
equals $\id_X$. Then $X$ is a retract of $e\tens X\in e\cM$, so $X\in e\cM$ by
Corollary~\ref{c:hecke-closed-under-retracts}.  Therefore $\pi\tens\id_X$ is an isomorphism by
Lemma~\ref{onemoredescription}. But the composition \eqref{unitality} equals $\id_X$,
so $\alpha$ is an isomorphism.

\mbr

(ii) By \cite[p.~174]{m}, $F$ has a left adjoint, namely, the ``free
module" functor $\Phi :\cM\to e\!-\!\operatorname{mod}_{\cM} , \quad
\Phi (X):=e\tens X$. By (i), the adjunction $\Phi F\rar{} \id_{\cM}$
is an isomorphism, so $F$ is fully faithful. It remains to show that
the adjunction $X\rar{} F(\Phi (X))$,  $X\in\cM$, is an isomorphism
if and only if $X\in e\cM$. This is a reformulation of
Lemma~\ref{onemoredescription}.
\end{proof}

Lemma~\ref{l:idemp2-alg-closed-idemp} (i) allows to give the following definition.

\begin{defin}
Let $e$ be a not necessarily unital idempotent algebra. A left
$e$-module (resp. right $e$-module, $e$-bimodule) $X$ is said to be
{\em unital\,} if the action morphism $e\tens X\rar{} X$ (resp.
$X\tens e\rar{} X$, $e\tens X\tens e\rar{} X$) is an isomorphism.
\end{defin}

\begin{rem}    \label{r:remedy}
If $e$ is not unital then the category of unital left $e$-modules
(resp. unital right $e$-modules, unital $e$-bimodules) seems to be
more reasonable than the category $e\cM$ (resp. $\cM e$ and $e\cM
e$). For instance, the category of unital $e$-bimodules is
automatically monoidal while $e\cM e$ may fail to be monoidal (see
Remark~\ref{r:caution}).
\end{rem}

\subsection{Idempotent unital algebras and closed idempotents}\label{ss:idempotent-algebras}
\begin{lem}\label{l:unit}
Let $E$ be a unital object in a semigroupal category $\cM$. Choose
an isomorphism $u:E\tens E\rar{\simeq} E$. Then
\begin{enumerate}[$($a$)$]
\item $u$ defines a structure of associative algebra on $E$;
 \sbr
\item for every pair of morphisms $f,g:E\rar{}E$, we have
$u\circ(f\tens g)=(f\circ g)\circ u$.
\end{enumerate}
\end{lem}

\begin{proof}
The pair $(E,u)$ is a unit object in $\cM$, see
Definition~\ref{defs:unital} (3). It remains to apply
Lemma~\ref{l:mon-cat-unit-constraints}.
\end{proof}

The next result is converse to Remark~\ref{rems:idempotents} (vi).

\begin{prop}\label{p:closed-idemp-idemp-alg}
If $\pi:\e\rar{}e$ is an idempotent arrow in a monoidal category
$\cM$, there exists a unique associative morphism $\mu:e\tens
e\rar{}e$ making $e$ an idempotent algebra with unit $\pi$.
\end{prop}

\begin{proof}
The condition that $\pi:\e\rar{}e$ is a left unit means that the
composition $\e\tens e\xrar{\ \pi\tens\id_e\ }e\tens e\rar{\mu}e$
equals $\la_e :\e\tens e\rar{\simeq} e$. Since $\pi\tens\id_e
:\e\tens e\xrar{}e$ is an isomorphism there is a unique $\mu :e\tens
e\rar{} e$ with this property, and this $\mu$ is an isomorphism. By
Lemma~\ref{l:Hecke-monoidal}, $e$ is a unital object of $e\cM e$, so
by Lemma~\ref{l:unit}, $\mu :e\tens e\rar{\simeq} e$ is
automatically associative. By construction, $\pi:\e\rar{}e$ is a
left unit for $\mu$. By Lemma~\ref{l:areequal}, $\pi$ is also a
right unit for $\mu$.
\end{proof}

\begin{cor}  \label{c:closed-idemp-arrows}
Let $e\in\cM$ be a closed idempotent. Then
\begin{enumerate}[$($a$)$]
\item every isomorphism $\mu:e\tens e\iso e$ makes $e$ into an idempotent algebra with unit;
 \sbr
\item associating to $\mu$ the corresponding unit $\pi:\e\rar{}e$ one gets a one-to-one correspondence
between isomorphisms $e\tens e\iso e$ and idempotent arrows $\e\rar{}e$.
\end{enumerate}
\end{cor}

\begin{proof}
By Proposition~\ref{p:closed-idemp-idemp-alg}, there exists an isomorphism $\mu:e\tens e\iso e$
making $e$ into an idempotent algebra with unit. By Lemma~\ref{l:Hecke-monoidal}, $e$ is a
unital object of $e\cM e$. So if $\mu':e\tens e\iso e$ is {\em any\,} isomorphism then by
Lemma~\ref{l:unit-uniqueness}, $(e, \mu')\simeq (e, \mu)$ and therefore
$(e, \mu')$ is also an idempotent algebra with unit. Thus we have proved (a). Statement (b) follows
from (a) by Remark~\ref{rems:idempotents}(vi) and Proposition~\ref{p:closed-idemp-idemp-alg}.
\end{proof}

\subsection{A partial order on closed idempotents  }\label{ss:idempotents-partial-order}
In \S\ref{sss:general} we define a partial order on the set of
isomorphism classes of closed idempotents in any monoidal category
$\cM$. In \S\ref{sss:weakly symmetric} we show that if $\cM$ is
weakly symmetric then any two elements of this partially ordered set
have an infimum (which equals their product).

\subsubsection{General case}  \label{sss:general}
If $\pi:\e\rar{}e$ and $\pi':\e\rar{}e'$ are arrows in a monoidal
category $\cM$, we write $\Hom(\pi,\pi')$ for the set of morphisms
$f:e\rar{}e'$ such that $f\circ\pi=\pi'$.

\begin{lem}\label{l:at-most-one-morphism}
Let $\pi$ and $\pi'$ be idempotent arrows in a monoidal category
$\cM$.
\begin{enumerate}[$($a$)$]
\item The set $\Hom(\pi,\pi')$ has at most one element.
 \sbr
\item If both $\Hom(\pi,\pi')$ and $\Hom(\pi',\pi)$ are nonempty,
then the unique $f\in\Hom(\pi,\pi')$ and $g\in\Hom(\pi',\pi)$ are
isomorphisms, inverse to each other.
\end{enumerate}
\end{lem}

\begin{proof}
We apply the method used to prove Lemma \ref{l:areequal}. The
functor $F:\cM\rar{}\End(\cM)$ defined by the left action of $\cM$
on itself enjoys the following properties:
\begin{enumerate}[(i)]
\item $F$ takes idempotent arrows to idempotent arrows;
 \sbr
\item $F$ is faithful;
 \sbr
\item $F$ reflects isomorphisms: if $f$ is a morphism in $\cM$
and $F(f)$ is an isomorphism in $\End(\cM)$, then $f$ is also an
isomorphism.
\end{enumerate}
In view of these facts, it suffices to prove the lemma when
$\cM=\End(\cC)$ for some category $\cC$. In this case the
lemma follows from Proposition~\ref{p:idempotent-monads}.
\end{proof}

\begin{lem}   \label{l:Hom(e,X)}
Let $\pi :\e\rar{}e$ be an idempotent arrow in a monoidal category
$\cM$. If $X\in e\cM$ then the map $\Hom(e,X) \rar{}\Hom (\e ,X)$ defined by
$f\mapsto f\circ\pi$ is a bijection.
\end{lem}

\begin{proof}
Use the implication (i)$\Rightarrow$(ii) from
Proposition~\ref{p:idempotent-monads} (b) in the following
situation: $\cC=\cM$,  $P(Y):=e\tens Y$, and $\eps:\Id_{\cM}\rar{}P$
is the idempotent arrow corresponding to $\pi :\e\rar{}e$.
\end{proof}

\begin{cor}\label{c:uniqueness-idempotent-arrows}
If $\pi,\pi':\e\rar{}e$ are two idempotent arrows with the same
codomain in a monoidal category $\cM$, there exists a unique
automorphism $f:e\rar{\simeq}e$ such that $f\circ\pi=\pi'$. \hfill\qedsymbol
\end{cor}


\begin{rem}\label{r:isom-classes-closed-idempotents}
The corollary implies that the map $(\e\rar{\pi}e)\mapsto e$ induces
a bijection between the set of isomorphism classes of idempotent
arrows in $\cM$ and the set of isomorphism classes of closed
idempotents in $\cM$.
\end{rem}

\begin{defin}\label{d:partial-order}
Given idempotent arrows $\pi$ and $\pi'$ in a monoidal category
$\cM$, we write $\pi\geq\pi'$ if $\Hom(\pi,\pi')\neq\varnothing$. By
Lemma \ref{l:at-most-one-morphism}, this preorder induces a partial
order on the set of isomorphism classes of idempotent arrows in
$\cM$, or, equivalently (by Remark
\ref{r:isom-classes-closed-idempotents}), on the set of isomorphism
classes of closed idempotents in $\cM$.
\end{defin}

Here are some equivalent descriptions  of the partial order in
Definition~\ref{d:partial-order}.

\begin{prop}    \label{p:orderdescriptions}
Let $e$ and $e'$ be closed idempotents in a monoidal category $\cM$.
The following conditions are equivalent:
\begin{enumerate}[$($i$)$]
\item $e\geq e'$ in the sense of Definition~\ref{d:partial-order};
 \sbr
\item $e'\in e\cM e$;
 \sbr
\item $e'\in e\cM$;
 \sbr
\item $e'\in \cM e$.
\end{enumerate}
\end{prop}

\begin{proof}
By symmetry, it suffices to check that (i)$\iff$(iii).
Lemma~\ref{l:Hom(e,X)} shows that (iii)$\Rightarrow$(i). Let us
prove that (i)$\Rightarrow$(iii). Choose idempotent arrows
$\pi:\e\rar{} e$ and $\pi':\e\rar{} e'$. Suppose there is a morphism
$f:e\rar{}e'$ with $f\circ\pi=\pi'$. Since
$\pi'\tens\id_{e'}:\e\tens e'\rar{} e'\tens e'$ is an isomorphism
and $\pi'=f\circ\pi$, we see that $e'$ is a retract of $e\tens e'$.
Since $e\tens e'\in e\cM$ we get $e'\in e\cM$ by
Corollary~\ref{c:hecke-closed-under-retracts}.
%
\end{proof}

\subsubsection{Weakly symmetric case}   \label{sss:weakly symmetric}
\begin{lem}   \label{l:productclosed}
Let $\cM$ be a weakly symmetric monoidal category.
\begin{enumerate}[$($i$)$]
\item If $\pi :\e\rar{} e$ and $\pi' :\e\rar{} e'$  are idempotent arrows in $\cM$, then so is  $\pi\tens\pi' :\e\rar{} e\tens e'$.
 \sbr
\item If $e,e'\in\cM$ are closed idempotents, so is  $e\tens e'$.
\end{enumerate}
\end{lem}

\begin{proof}
It suffices to check (i). Since $\id_e\tens\pi :e\tens\e\to e\tens e$ and
$\id_{e'}\tens\pi' :e'\tens\e\to e'\tens e'$
are isomorphisms so is
$\id_e\tens\pi\tens\id_{e'}\tens\pi'  :e\tens\e\tens e'\tens \e\to e\tens e\tens e'\tens e'$.
So by weak symmetry, $\id_{e\tens e'}\tens\pi\tens\pi' :e\tens e'\tens\e\tens\e\to e\tens e'\tens e\tens e'$
is an isomorphism, i.e., $\pi\tens\pi'$ is an idempotent arrow.
\end{proof}

\begin{rems}   \label{r:classical}
\begin{enumerate}[(i)]
\item The set of idempotents in a commutative monoid $A$ can be equipped with the following partial order:
\[
e_1\le e_2 \quad \mbox{if} \quad e_1e_2=e_1.
\]
Any two idempotents $e_1, e_2\in A$ have an infimum; namely, $\inf (e_1, e_2)=e_1e_2$.
The unit of $A$ is the biggest idempotent.
 \sbr
\item Now let $\cM$ be a weakly symmetric monoidal category and $W(\cM )$ (resp. $C(\cM )$)
the set of isomorphism classes of weak (resp. closed) idempotents in
$\cM$. Applying the previous remark to the monoid of isomorphism
classes of objects of $\cM$ one gets a partial order on $W(\cM )$.
The induced partial order on $C(\cM )\subset W(\cM )$ equals that
from Definition~\ref{d:partial-order} (to see this, use the
equivalence (i)$\iff$(iv) from
Proposition~\ref{p:orderdescriptions}).
 \sbr
\item If $\cM$ has a zero object then minimal weak (resp. closed) idempotents in $\cM$ in the sense
of Definition \ref{d:minimal-idempotents} are the same as minimal
elements in the set of (isomorphism classes of) nonzero weak (resp. closed)
idempotents in $\cM$ (in the closed case this follows from Lemma~\ref{l:productclosed}).
\end{enumerate}
\end{rems}


\begin{cor}   \label{c:infimum}
Let $\cM$ be a weakly symmetric monoidal category and $C(\cM )$ the
set of isomorphism classes of closed idempotents in $\cM$ equipped
with the partial order from Definition~\ref{d:partial-order}. Then
any two elements of $C(\cM )$ have an infimum $($which equals their
product$)$.
\end{cor}

\begin{proof}
Use Remarks~\ref{r:classical} (i-ii) and Lemma~\ref{l:productclosed}.
\end{proof}

\subsection{Retracts and direct sums of closed idempotents}\label{ss:retracts-idempotents}

\begin{lem} \label{l:retract-idempotents}
Let $e\rar{i}\e\rar{p}e$ be morphisms in a monoidal category $\cM$ such that $pi=\id_e$.
Then $p:\e\rar{}e$ is an idempotent arrow.
\end{lem}

\begin{proof}
We have to show that $p\tens\id_e:\e\tens e\to e\tens e$ and
$\id_e\tens p:   e\tens\e\to e\tens e$ are isomorphisms. Let us show this for $p\tens\id_e$.
Since $pi=\id_e$ it suffices to show that $ip\tens\id_e:\e\tens e\rar{} \e\tens e$ equals
$\id_{\e\tens e}$. Representing $\e\tens e$ as a retract of $\e\tens\e$ we see that this is
equivalent to showing that $ip\tens ip:\e\tens\e\to\e\tens\e$ equals $\id_{\e}\tens ip$.
For any $f\in\End(\e )$ one has $f\tens\id_{\e}=\id_{\e}\tens f$, so
$ip\tens ip=(ip\tens\id_{\e})\cdot (\id_{\e}\tens ip)=\id_{\e}\tens ipip=\id_{\e}\tens ip$.
\end{proof}

%

\begin{cor}\label{c:retracts-closed-idempotents}
If $e$ is a closed idempotent in a monoidal category $\cM$ and
$e'\in\cM$ is a retract of $e$, then $e'$ is a closed idempotent in
$\cM$, and $e'\leq e$ with respect to the partial order introduced
in Definition \ref{d:partial-order}.
\end{cor}

\begin{proof}
By Corollary~\ref{c:hecke-closed-under-retracts}, $e'\in e\cM e$.
By Lemma~\ref{l:Hecke-monoidal}, $e\cM e$ is a monoidal category with unital object $e$.
Applying Lemma~\ref{l:retract-idempotents} to $e\cM e$,  we get an idempotent arrow $e\to e'$
in $e\cM e$. So by Lemma~\ref{l:transitivity}, $e'$ is a closed idempotent in $\cM$.
Clearly $e'\leq e$.
\end{proof}


Corollary \ref{c:retracts-closed-idempotents} immediately implies the following statement.

\begin{cor}\label{c:minimal-closed-idempotent-indecomposable}
In an additive monoidal category every minimal closed idempotent $e$ is indecomposable.
\end{cor}

Here ``minimal'' means ``minimal among all nonzero closed
idempotents in $\cM$ with respect to the partial order introduced in
Definition~\ref{d:partial-order}". We note that by
Remark~\ref{r:classical} (iii), if $\cM$ is weakly symmetric, this
is the same as minimality in the sense of
Definition~\ref{d:minimal-idempotents}.

\begin{prop}  \label{p:dir_sum_decomps}
Let $\cM$ be an additive monoidal category. Let $e_1,\ldots , e_n\in\cM$ and
$e=e_1\oplus\ldots\oplus e_n$. Then the following conditions are equivalent:
\begin{enumerate}[$($i$)$]
\item $e$ is a closed idempotent;
 \sbr
\item each $e_i$ is a closed idempotent and $e_i\tens e_j=0$ for $i\ne j$.
\end{enumerate}
If these conditions hold, then
\begin{equation}  \label{e:dir_sum_decomps}
e\cM =\bigoplus_i e_i\cM , \quad \cM e=\bigoplus_i \cM e_i, \quad e\cM e=\bigoplus_{i,j} e_i \cM e_j
\end{equation}
\end{prop}

\begin{proof}
It is clear that (ii)$\Rightarrow$(i). Let us show that
(i)$\Rightarrow$(ii). By
Corollary~\ref{c:retracts-closed-idempotents}, each $e_i$ is a
closed idempotent in $\cM$. Moreover, the proof of
Corollary~\ref{c:retracts-closed-idempotents} shows that the
projection $p_i:e\to e_i$ is an idempotent arrow in the monoidal
category $e\cM e$. So the morphism $\id_{e_i}\tens p_i: e_i\tens
e\rar{} e_i\tens e_i$ is an isomorphism, which implies that
$e_i\tens e_j=0$ for $i\ne j$.

\mbr

Let us prove the first decomposition in \eqref{e:dir_sum_decomps}
(the other two are proved similarly). It is clear that
$e_i\cM=ee_i\cM\subset e\cM$ and that each object of $e\cM$ is a
direct sum of objects of $e_1\cM,\ldots , e_n\cM$. It remains to
show that if $X\in e_i\cM$, $Y\in e_j\cM$, and $i\ne j$ then $\Hom
(Y,X)=0$. By
Proposition~\ref{p:interplay-weak-closed-idempotents}(a), $\Hom
(Y,X)=\Hom (e_i\tens Y,X)$ and $e_i\tens Y=0$ because $e_i\tens
e_j=0$.
\end{proof}



\section{Serre duality and Fourier-Deligne transform}\label{s:serre-FD}

\subsection{Definition of the Serre
dual}\label{ss:serre-dual-definition} We keep the assumptions of
\S\ref{ss:basic-definitions} (although, in fact, it would suffice to
require the field $k$ to be perfect throughout this section). In
particular, $\operatorname{char}k=p>0$, and $\ell$ is a fixed prime
different from $p$.

\mbr

The notion of a multiplicative $\ql$-local system on a connected
perfect quasi-algebraic group over $k$ was introduced in Definition
\ref{d:multiplicative}. In order to formulate the definition of the
Serre dual of a connected perfect unipotent group, we need a
relative version:

\begin{defin}\label{d:family-mls}
Let $H$ be a connected perfect quasi-algebraic group over $k$, and
let $S$ be a perfect quasi-algebraic scheme over $k$. A \emph{family
of multiplicative $\ql$-local systems on $H$ parameterized by $S$}
(or an \emph{$S$-family of multiplicative local systems on $H$}) is
a $\ql$-local system $\sL$ on $H\times S$ of rank $1$ equipped with
an isomorphism
\begin{equation}  \label{e:monoidal_structure}
\vp:\mu_{12}^*(\sL)\iso p_{13}^*(\sL)\tens p_{23}^*(\sL),
\end{equation}
where
\[\mu_{12}=\mu\times\id_S :H\times H\times S \rar{} H\times S \]
(with $\mu$ being the multiplication morphism for $H$) and
\[
p_{13},p_{23}:H\times H\times S \rar{} H\times S
\]
are the projections along the second and first factors,
respectively. The isomorphism $\vp$ is called a \emph{multiplicative
structure} on the local system $\sL$.
\end{defin}

\begin{rems}\label{r:family-mult-loc-sys}
\begin{enumerate}[(i)]
\item Since $H$ is connected, an $S$-family of
multiplicative local systems on $H$ has no nontrivial automorphisms.
 \mbr
\item For $s\in S(k)$ and $h_1,h_2,h_3\in H(k)$ one has the diagram
\[
\xymatrix{
 \sL_{h_1h_2h_3,s} \ar[d]_\simeq \ar[rr]^{\simeq\ \ \ \ \ } & & \sL_{h_1h_2,s}\tens\sL_{h_3,s} \ar[d]^\simeq \\
 \sL_{h_1,s}\tens\sL_{h_2h_3,s} \ar[rr]^{\simeq\ \ \ \ \ } & & \sL_{h_1,s}\tens\sL_{h_2,s}\tens\sL_{h_3,s}
   }
\]
in which all arrows come from \eqref{e:monoidal_structure}. It
clearly commutes if $h_1=h_2=h_3=1$. Since $H$ is connected this
implies commutativity for all $h_1,h_2,h_3\in H(k)$.
 \mbr
\item Suppose that $\sL$ is a rank 1 $\ql$-local system on $H\times S$ such that
\[\mu_{12}^*(\sL)\cong p_{13}^*(\sL)\tens p_{23}^*(\sL).\] Then the
pullback of $\sL$ to $\{ 1\}\times S\subset H\times S$ is trivial.
Moreover, since $H$ is connected, specifying an isomorphism
\eqref{e:monoidal_structure} is equivalent to specifying a
trivialization of this pullback. In particular, $\sL$ has a
multiplicative structure, and any two multiplicative structures on
$\sL$ are canonically isomorphic.
\end{enumerate}
\end{rems}
In view of the last remark, we will always denote a family of
multiplicative local systems by a single letter such as $\sL$ or
$\sE$.

%
%
%
%
%

\begin{prop}\label{p:serre-dual}
Let $H$ be a $($possibly noncommutative$)$ connected perfect
unipotent group over $k$. There exists a $($possibly disconnected$)$
perfect commutative unipotent group $H^*$ over $k$ and an
$H^*$-family $\sE$ of multiplicative $\ql$-local systems on $H$ with
the following universal property.

\mbr

If $S$ is a perfect quasi-algebraic scheme over $k$, the map
$f\longmapsto (\id_H\times f)^*\sE$ is an isomorphism between the
group of $k$-morphisms $f:S\rar{}H^*$ and the group of isomorphism
classes of $S$-families of multiplicative $\ql$-local systems on
$H$.
\end{prop}

In the case where $H$ is commutative, the idea of this construction
goes back to Serre's article \cite{serre}, and the result itself is
proved in \cite{begueri}. For the proof in general, we refer the
reader to \S{}A.12 of the appendix in \cite{characters}.

\begin{rem}\label{r:uniqueness-morphism-to-dual}
If $\sL$ is a rank $1$ $\ql$-local system on $H\times S$ such that
$\mu_{12}^*(\sL)\cong p_{13}^*(\sL)\tens p_{23}^*(\sL)$, then by
Remark \ref{r:family-mult-loc-sys}(iii), we obtain a uniquely
defined morphism $S\rar{}H^*$ (it is independent of the choice of a
multiplicative structure on $\sL$).
\end{rem}

\begin{defin}\label{d:serre-dual}
The pair $(H^*,\sE)$ satisfying the conclusion of Proposition
\ref{p:serre-dual}
is called a \emph{Serre dual} of $H$. Of course, it is determined
uniquely up to unique isomorphism. As usual, by abuse of
terminology, we will often refer to $H^*$ itself as the Serre dual
of $H$, in which case $\sE$ will be called ``the universal local
system on $H\times H^*$''.
\end{defin}

\begin{rem}\label{r:serre-duality-canonical}
Following Serre \cite{serre}, the authors of  \cite{begueri,saibi}
define $H^*$ to be the moduli space of central extensions of $H$ by
$\qzp$ rather than the moduli space of multiplicative $\ql$-local
systems on $H$. In fact, a choice of a group homomorphism
$\psi:(\bQ_p,+)\rar{}\ql^\times$ with kernel $\bZ_p$ allows to
identify the two spaces. Namely, if $\widetilde{H}$ is a central
extension of the group scheme $H$ by $\qzp$ then $\widetilde{H}$ is
a $\qzp$-torsor over $H$; moreover, the $\ql$-local system
$\widetilde{H}_\psi$ on $H\times S$ corresponding to
$\psi:\qzp\rar{}\ql^\times$ and this torsor is multiplicative. Thus
one gets a map between the two moduli spaces, and it is easy to see
that this is an isomorphism.
%
%
%
%
\end{rem}

\subsection{Properties of Serre
duality}\label{ss:serre-duality-properties} Serre duality for
perfect connected unipotent groups has several properties that are
similar to the properties of the functor $\Ga\longmapsto\Ga^*$,
where $\Ga$ is an arbitrary finite group and $\Ga^*=\Hom(\Ga,\cst)$.
The first one and the third one follow easily from the definition of
the Serre dual.

\mbr

\begin{enumerate}[(1)]
\item Serre duality, $H\longmapsto H^*$, is a contravariant functor from the category
$\puc_k$ of perfect connected unipotent groups over $k$ to the
category $\cpu_k$ of commutative perfect unipotent groups over $k$.
 \mbr
\item If $H\in\puc_k$, the restriction homomorphism
$H^*\rar{}[H,H]^*$ has finite image. This follows from \cite[Theorem
1.2]{masoud}.
 \mbr
\item If $G\rar{}H\rar{}K\rar{}1$ is an exact sequence of perfect
connected unipotent groups over $k$, the induced sequence
$0\rar{}K^*\rar{}H^*\rar{}G^*$ of commutative perfect unipotent
groups is also exact.
 \mbr
\item If $G\rar{f}H$ is an isogeny between perfect connected
unipotent groups over $k$, the induced homomorphism $H^*\rar{}G^*$
has finite kernel. Indeed, multiplicative $\ql$-local systems $\cL$
on $H$ such that $f^*\cL$ is trivial are parameterized (up to
isomorphism) by the finite group $\Hom\bigl((\Ker
f)(k),\ql^\times\bigr)$.
 \mbr
\item The restriction of the functor $H\longmapsto H^*$ to the
category $\cpuc_k=\puc_k\cap\cpu_k$ of perfect connected commutative
unipotent groups over $k$ is an exact\footnote{Note that $\cpu_k$ is
an abelian category, and $\cpuc_k$ is an exact subcategory of
$\cpu_k$.} equivalence between $\cpuc_k$ and its opposite category.
Furthermore, the square of this restriction is naturally isomorphic
to the identity functor on $\cpuc_k$ (cf.~\cite{serre,begueri}).
 \mbr
\item If $H\in\puc_k$ and $H^{ab}=H/[H,H]$ is the abelianization of
$H$, then the induced morphism $(H^{ab})^*\rar{}H^*$ is a
monomorphism, and identifies $(H^{ab})^*$ with the neutral connected
component of $H^*$. To see this, combine (2), (3) and (5).
 \mbr
\item If $H\in\cpuc_k$, then $\dim H^*=\dim H$. Hence, in view of
(6), we deduce that if $H\in\puc_k$, then $\dim H^*=\dim H^{ab}$.
\end{enumerate}

\subsection{An auxiliary
construction}\label{ss:auxiliary-construction} In this subsection we
recall a construction, which plays an important role in the
definition of an admissible pair
(\S\ref{ss:Serre-duality-admissible-pairs}) and in the proofs
appearing in \S\ref{s:reduction}. On the other hand, it is not
used in the rest of \S\ref{s:serre-FD}.

\mbr

Let $U$ be a (possibly disconnected and/or noncommutative) perfect
unipotent group, and let $N\subset U$ be a normal connected
subgroup. Since the Serre dual, $N^*$, of $N$ is defined by a
universal property, it is clear that $U$ acts on $N^*$ by $k$-group
scheme automorphisms. Now let $\cN$ be a multiplicative $\ql$-local
system on $N$ such that the corresponding element of $N^*(k)$ is
$U$-invariant. Further, let $Z\subset U$ be a connected subgroup
such that\footnote{Without loss of generality, one can take $Z$ to
be the preimage in $U$ of the neutral connected component of the
center of $U/N$.} $N\subset Z$ and $[U,Z]\subset N$. According to
the next lemma, these data define a
$k$-group scheme morphism
\[
\vp_{\cN} : U/N \rar{} (Z/N)^*.
\]

\begin{lem}
$($i$)$ Let $c:U\times Z\rar{}N$ denote the commutator morphism,
$c(u,z)=uzu^{-1}z^{-1}$. Then $\sL:=c^*\cN$ has a structure of a
$U$-family of multiplicative $\ql$-local systems on $Z$.
 \sbr
$($ii$)$ The morphism of $k$-schemes
$\widetilde{\vp}_{\cN}:U\rar{}Z^*$ corresponding to\footnote{The
morphism is well defined by Remark
\ref{r:uniqueness-morphism-to-dual}.} $\sL$ factors as
\[
U \twoheadrightarrow U/N \rar{\ \ \vp_{\cN}\ \ } (Z/N)^*
\hookrightarrow Z^*.
\]
Moreover, $\widetilde{\vp}_{\cN}$ and $\vp_{\cN}$ are homomorphisms of group schemes over $k$.
\end{lem}

For the proof, see \cite[\S{}A.13]{characters}.

\subsection{The Fourier-Deligne
transform}\label{ss:fourier-deligne-transform} Let us now define a
version of the Fourier-Deligne transform that will be used in our
article. We consider a more general setting than
\cite{deligne,katz-laumon,saibi} by working with a possibly
noncommutative connected unipotent group. In the commutative case
our definition amounts to the \emph{inverse} of the usual
Fourier-Deligne transform (up to 
a cohomological shift).

\mbr

Let $H$ be a perfect connected unipotent group over $k$, and let $S$
be an arbitrary perfect quasi-algebraic scheme over $k$. Let
$(H^*,\sE)$ be the Serre dual of $H$, and let $\sE_S$ denote the
pullback of $\sE$ to $H\times H^*\times S$. Further, let
$\pr:H\times H^*\rar{}H$ and $\pr':H\times H^*\rar{}H^*$ denote the
first and second projection morphisms. As usual, we write $\bK_H$ for the
dualizing complex of $H$.

\begin{defin}\label{d:FD-inverse}
Define triangulated functors $\cF'_S\,,\cF'_{S,*}  : \sD(H^*\times S) \rar{} \sD(H\times S)$
by
\[
\cF'_S ( M ) =\bK_H\tens (\pr\times\id_S)_! \bigl( \sE_S\tens
(\pr'\times\id_S)^*M \bigr),
\]
\[
\cF'_{S,*} (M ) :=\bK_H\tens (\pr\times\id_S)_* \bigl( \sE_S\tens (\pr'\times\id_S)^*M
\bigr).
\]
We write
$\cF'=\cF'_{\Spec k}$ and $\cF'_*=\cF'_{\Spec k,*}$, so we have
\begin{equation} \label{e:fourier-deligne}
\cF':\sD(H^*)\rar{}\sD(H), \quad \cF'(M):=\bK_H\tens\pr_!(\sE\tens\pr'^*M),
\end{equation}
\begin{equation} \label{e:fourier*deligne}
\cF'_*:\sD(H^*)\rar{}\sD(H),\quad \cF'_*(M):=\bK_H\tens\pr_*(\sE\tens\pr'^*M).
\end{equation}
\end{defin}

One has a canonical morphism of functors $\cF'_S\rar{}\cF'_{S,*}$ (in particular, $\cF'\rar{}\cF'_{*}$).

\begin{rem}\label{r:FD-commutative-case}
We consider $\sD(H^*\times S)$ as a monoidal category with respect
to tensor product, and $\sD(H\times S)$ as a monoidal category with
respect to convolution with compact supports (see
\eqref{e:convol-S}). It is well known
\cite{deligne,katz-laumon,saibi} that if $H$ is commutative, then
$\cF'_S$ is a monoidal equivalence, and $\cF'_S\rar{}\cF'_{S,*}$ is
an isomorphism.
\end{rem}

\begin{rem}\label{r:FD-general-case}
One can show that in general, $\sF'_S$ can be decomposed as
\begin{equation}\label{e:ast}
\sD(H^*\times S)\rar{\sim}\sD(H^{\text{ab,st}}\times
S)\overset{\pi^!}{\into}\sD(H\times S),
\end{equation}
where $H^{\text{ab,st}}$ is the ``stacky abelianization'' of $H$
introduced in \cite[\S4.1]{masoud} with
$\pi:H\rar{}H^{\text{ab,st}}$ being the canonical projection, the
first arrow in \eqref{e:ast} is a monoidal equivalence
(Rem.~\ref{r:semigroupal-equivalence}), and $\pi^!$ is a fully
faithful semigroupal functor (Def.~\ref{d:weak-semigroupal}(c)).
\end{rem}

In the case $S=\Spec k$ we will construct a version of the decomposition
\eqref{e:ast} that does not use $H^{\text{ab,st}}$ explicitly, and we will prove
that the morphism $\cF'_S\rar{}\cF'_{S,*}$ is an isomorphism without assuming $H$
to be commutative, see Proposition~\ref{p:fourier-deligne-*-isomorphism} and
Theorem~\ref{t:FD-closed-idempotent}. The same arguments allow to treat the
case of any quasi-algebraic $k$-scheme $S$, but we will not need this degree of
generality.

\subsection{Semigroupal structure on $\cF'$}\label{ss:semigroupal-FD}
In \S\S\ref{3.6.1}-\ref{3.6.3} we will construct a semigroupal
structure (Def.~\ref{d:weak-semigroupal}(b)) on a quite general
class of functors. In \S\S\ref{3.6.4}-\ref{3.6.5} we will show that
the functor $\cF'$ defined by \eqref{e:fourier-deligne} belongs to
this class.


\subsubsection{}\label{3.6.1}
Let $S$ be a perfect quasi-algebraic scheme over $k$, and let $H$ be
any\footnote{Not necessarily connected or unipotent.} perfect
quasi-algebraic group over $k$. We equip the category $\sD(H\times
S)$ with the monoidal structure given by convolution with compact
support:
\begin{equation}\label{e:convol-S}
M*N=\mu_{12!}(p_{13}^*M\tens p_{23}^*N),
\end{equation}
where $\mu_{12},p_{13},p_{23}:H\times H\times S\rar{}H\times S$ are
as in Definition \ref{d:family-mls}.

\subsubsection{}
Every object $\cC\in\sD(H\times S)$ defines a functor
\begin{equation}\label{eq:0}
F_\cC:\sD(S)\rar{}\sD(H), \qquad
F_\cC(M):=\pr_{H!}(\cC\tens\pr_S^*M),
\end{equation}
where $\pr_S:H\times S\rar{}S$ and $\pr_H:H\times S\rar{}H$ are the
two projections. Similarly, each $\cK\in\sD(H\times S\times H\times
S)=\sD(H\times H\times S\times S)$ defines a functor
\[
F_\cK : \sD(S\times S)\rar{}\sD(H\times H).
\]

\mbr

One has a canonical functorial morphism
\begin{equation}\label{eq:1}
F_{\cC_1}(M_1)*F_{\cC_2}(M_2) \rar{} F_{\cC_1*\cC_2}(M_1\tens M_2)
\end{equation}
for all $M_t\in\sD(S)$ and $\cC_t\in\sD(H\times S)$ (where $t=1,2$),
namely, the composition
\begin{equation}\label{eq:2}
\begin{split}
& F_{\cC_1}(M_1)*F_{\cC_2}(M_2) \rar{\simeq} \mu_!
F_{\cC_1\boxtimes\cC_2}(M_1\boxtimes M_2) \rar{} \\
& \rar{} \mu_!F_{\De_!\De^*(\cC_1\boxtimes\cC_2)}(M_1\boxtimes M_2)
\rar{\simeq} F_{\cC_1*\cC_2}(M_1\tens M_2),
\end{split}
\end{equation}
where $\De:H\times H\times S\rar{}H\times S\times H\times S$ is
defined by $\De(h_1,h_2,s)=(h_1,s,h_2,s)$ and the second morphism in
\eqref{eq:2} comes from the adjunction
\[
\cC_1\boxtimes\cC_2\rar{}\De_*\De^*(\cC_1\boxtimes\cC_2) =
\De_!\De^*(\cC_1\boxtimes\cC_2).
\]

\subsubsection{}\label{3.6.3}
Now suppose we have a pair $(\cC,m)$, where $\cC\in\sD(H\times S)$
and $m:\cC*\cC\rar{}\cC$ is a morphism. Then one has a canonical
functorial morphism
\begin{equation}\label{eq:3}
F_\cC(M_1)*F_\cC(M_2)\rar{}F_\cC(M_1\tens M_2), \qquad
M_1,M_2\in\sD(S),
\end{equation}
namely, the composition
\[
F_\cC(M_1)*F_\cC(M_2)\rar{}F_{\cC*\cC}(M_1\tens
M_2)\rar{}F_\cC(M_1\tens M_2),
\]
where the first arrow is the morphism \eqref{eq:1} and the second
one is induced by $m$.

\begin{lem}\label{lem:1}
Suppose that $(\cC,m)$ is an associative algebra in $\sD(H\times
S)$. Then \eqref{eq:3} makes $F_\cC$ into a weakly semigroupal
functor $($Def.~\ref{d:weak-semigroupal}$($c$)${}$)$. It is strongly
semigroupal if and only if the algebra $(\cC,m)$ is idempotent and
\begin{equation}\label{eq:4}
\cC_s*\cC_{s'}=0 \qquad \text{for any } s,s'\in S(k) \text{ such
that } s'\neq s,
\end{equation}
where $\cC_s\in\sD(H)$ is the restriction of $\cC$ to
$H\times\{s\}\subset H\times S$.  \hfill\qedsymbol
\end{lem}

\begin{rem}\label{r:fourier-equivariant}
If $(\cC,m)$ is an associative algebra in $\sD_H(H\times S)$ (where
$H$ acts by conjugation on the first factor and trivially on the
second factor), then $F_\cC$ is a weakly semigroupal functor
$\sD(S)\rar{}\sD_H(H)$.
\end{rem}

\subsubsection{}\label{3.6.4}
From now on we assume that $H$ is connected and unipotent. Let $\sL$
be an $S$-family of multiplicative local systems on $H$. Set
\begin{equation}\label{eq:5}
\cC:=\sL\tens\pr_H^*\bK_H=\sL\tens\pr_S^!(\ql)_S,
\end{equation}
where $\bK_H$ is the dualizing complex of $H$. Let $1:S\rar{}H\times
S$ denote the unit section; then $1^!\cC=(\ql)_S$, so we get a
canonical morphism
\begin{equation}\label{eq:6}
1_!(\ql)_S\rar{}\cC.
\end{equation}
Note that $1_!(\ql)_S$ is the unit object of $\sD(H\times S)$.

\begin{lem}\label{lem:2}
\begin{enumerate}[$($i$)$]
\item The morphism \eqref{eq:6} is an idempotent arrow.
 \sbr
\item Condition \eqref{eq:4} holds if and only if the map
$S(k)\rar{}H^*(k)$ corresponding to $\sL$ is injective.
\end{enumerate}
\end{lem}

\begin{proof}
In the case $S=\Spec k$ statement (i) is proved in
\cite[Prop.~8.1(a)]{characters}, and the general case is similar.
The ``only if'' part of (ii) is clear, and the ``if'' part follows
from \cite[Lem.~9.4]{characters}.
\end{proof}

\begin{cor}\label{cor:3}
$\cC$ has a unique structure of an idempotent algebra for which
\eqref{eq:6} is a unit.
\end{cor}
\begin{proof}
Use Lemma \ref{lem:2}(i) and Proposition
\ref{p:closed-idemp-idemp-alg}.
\end{proof}

\subsubsection{}\label{3.6.5} Finally, let $S=H^*$, and let $\sL=\sE\in\sD(H\times
H^*)$ be the universal family of multiplicative local systems on
$H$. Define $\cC$ by \eqref{eq:5}; then $F_\cC=\cF'$ (to see this,
compare \eqref{eq:0} with \eqref{e:fourier-deligne}). Combining
Lemma \ref{lem:1}, Corollary \ref{cor:3}, and Lemma \ref{lem:2}(ii)
we get a semigroupal structure (Def.~\ref{d:weak-semigroupal}(b)) on
$\cF'=F_\cC$.

\subsection{Construction of closed idempotents in $\sD(H)$}
In this subsection we will use the functor $\cF':\sD(H^*)\rar{}\sD(H)$ to construct
certain closed idempotents in the monoidal category $\sD(H)$, where
$H$ is a perfect connected unipotent group over $k$. We begin by
stating
\begin{prop}\label{p:fourier-deligne-*-isomorphism}
The canonical morphism $\cF'\rar{}\cF'_*$ is an isomorphism.
\end{prop}
This proposition is proved in \S\ref{ss:proofs} below.

\mbr

Set $e_0:=\cF'\bigl((\ql)_{H^*}\bigr)$, where $(\ql)_{H^*}$ is the constant sheaf
on $H^*$ with stalk $\ql$. Since $(\ql)_{H^*}$ is the unit object of $\sD(H^*)$
and $\cF'$ is a semigroupal functor (see \S\ref{ss:semigroupal-FD}) it is clear
that $e_0$ is a weak idempotent in $\sD(H)$. The next theorem says that
$e_0$ is a closed idempotent.

\mbr

To construct an idempotent arrow $\e\rar{} e_0$, rewrite
formula~\eqref{e:fourier*deligne} as
$\cF'_*(M)=\pr_*(\sE\tens\pr'^!M)$. 
This shows that $\cF'_*$ is right adjoint to the functor
\begin{equation}\label{e:fourier-deligne-unnormalized}
\cF:\sD(H)\rar{}\sD(H^*), \qquad N\longmapsto \pr'_!(\sE^{-1}\tens
\pr^* N).
\end{equation}
We have a canonical isomorphism $\cF(\e)\rar{\simeq}(\ql)_{H^*}$,
which by adjunction yields a morphism
$\pi_0:\e\rar{}\cF'_*\bigl((\ql)_{H^*}\bigr)\rar{\simeq}e_0$ (see
Proposition \ref{p:fourier-deligne-*-isomorphism}).

\begin{thm}\label{t:FD-closed-idempotent}
\begin{enumerate}[$($a$)$]
\item The morphism $\pi_0:\e \rar{}e_0$ is an idempotent arrow in $\sD(H)$.
 \sbr
\item The functor $\cF'$ induces a monoidal equivalence
\[
\sD(H^*) \rar{\sim} e_0\sD(H)\subset\sD(H).
\]
 \sbr
\item Let $i:[H,H]\into H$ denote the inclusion, and let $K\subset [H,H]^*$ denote the image of the restriction morphism
$i^*:H^*\rar{}[H,H]^*$. Then
\[
e_0 \cong \bigoplus_{[\cN]\in K(k)} i_!e_{\cN},
\]
the sum being finite by property
\S\ref{ss:serre-duality-properties}$($2$)$. Here $e_\cN$ denotes the
closed idempotent in $\sD([H,H])$ defined by the multiplicative
local system $\cN$
$($cf.~\S\ref{ss:Heisenberg-minimal-idempotents}$)$.
\end{enumerate}
\end{thm}

The theorem is proved in \S\ref{ss:proofs} below.

\begin{rem}
By Remark \ref{r:fourier-equivariant}, the image of the functor
$\cF'$ lies in the subcategory $\sD_H(H)\subset\sD(H)$ of complexes
equivariant under conjugation. In particular, $e_0\in\sD_H(H)$, so
$e_0\sD(H)=\sD(H)e_0$ is the Hecke subcategory of $\sD(H)$ defined
by $e_0$.
\end{rem}

\begin{cor}\label{c:FD-closed-idempotent}
Let $i:F\into H^*$ be a closed immersion and
$\pi_F:(\ql)_{H^*}\rar{}i_!(\ql)_F$ the corresponding idempotent
arrow in $\sD(H^*)$. Put $e_F=\cF'\bigl(i_!(\ql)_F\bigr)$. Then
\[
\cF'(\pi_F)\circ\pi_0 : \e\rar{}e_F
\]
is an idempotent arrow in $\sD(H)$.
\end{cor}
\begin{proof}
Use Theorem \ref{t:FD-closed-idempotent}(b) and Lemma
\ref{l:transitivity}.
\end{proof}

\subsection{Proof of Proposition
\ref{p:fourier-deligne-*-isomorphism} and Theorem
\ref{t:FD-closed-idempotent}}\label{ss:proofs} Let us fix a
multiplicative local system $\cN$ on $H$ and write
$\cN'=i^*\cN\overset{\text{def}}{=}\cN\bigl\lvert_{[H,H]}$. Let
$H^*_{\cN}$ denote the fiber of the restriction morphism
$i^*:H^*\rar{}[H,H]^*$ over $[\cN']\in[H,H]^*(k)$. Properties (2),
(5) and (6) in \S\ref{ss:serre-duality-properties} imply that
$H^*_{\cN}$ is a connected component of $H^*$.

\mbr

Further, let us write $e'_{\cN}:=i_!e_{\cN'}$; it is a closed
idempotent in $\sD(H)$.

\mbr

We will now state an auxiliary lemma and deduce Proposition
\ref{p:fourier-deligne-*-isomorphism} and Theorem
\ref{t:FD-closed-idempotent} from it. The lemma will be proved at
the end of this subsection.

\begin{lem}\label{l:fourier-deligne-auxiliary}
\begin{enumerate}[$($a$)$]
\item The restriction of $\cF':\sD(H^*)\rar{}\sD(H)$ to the full
subcategory $\sD(H^*_{\cN})$ is an equivalence onto the full
subcategory $e'_{\cN}\sD(H)\subset\sD(H)$.
 \sbr
\item If $M\in\sD(H^*_{\cN})$, the canonical map
$\cF'(M)\rar{}\cF'_*(M)$ is an isomorphism.
\end{enumerate}
\end{lem}

For the time being, let us assume that this lemma holds. Choose
multiplicative local systems $\cN_1,\dotsc,\cN_n$ on $H$ such that
$[\cN'_1],\dotsc,[\cN'_n]$ is a list of all elements of $K^*(k)$,
without repetitions. Then the $H^*_{\cN_j}$ are all the connected
components of $H^*$, whence
\begin{equation}  \label{orthog1}
\sD(H^*)=\bigoplus_j \sD(H^*_{\cN_j}).
\end{equation}
Thus Lemma \ref{l:fourier-deligne-auxiliary}(b)
implies Proposition \ref{p:fourier-deligne-*-isomorphism}.

\mbr

One has $e'_{\cN_j}*e'_{\cN_r}=0$ for $j\ne r$ because the
$e_{\cN'_j}$ are pairwise non-isomorphic minimal closed idempotents
in $\sD([H,H])$. So by Proposition~\ref{p:dir_sum_decomps},
$e:=\bigoplus_{j=1}^n e'_{\cN_j}$ is a closed idempotent in $\sD(H)$
and
\begin{equation}  \label{orthog2}
e\sD(H)=\bigoplus_j e'_{\cN_j}\sD(H).
\end{equation}
In \S\ref{ss:semigroupal-FD} we constructed a semigroupal structure
on the functor $\cF':\sD(H^*)\rar{}\sD(H)$. So by \eqref{orthog1}, \eqref{orthog2},
and Lemma~\ref{l:fourier-deligne-auxiliary}(a), $\cF'$ induces a
semigroupal equivalence $\sD(H^*)\rar{\sim}e\sD(H)$.
By Remark~\ref{r:semigroupal-equivalence}, this is the same as monoidal equivalence.
By Lemma~\ref{l:Hecke-monoidal}, $e$ is a unital object of $e\sD(H)$ and $e_0$
was defined to be the image of the unit object of $\sD(H^*)$, so $e_0\cong e$.
This proves parts (c) and (b) of Theorem \ref{t:FD-closed-idempotent}.


\mbr

Finally, to prove Theorem \ref{t:FD-closed-idempotent}(a), note that
by Proposition \ref{p:fourier-deligne-*-isomorphism}, the functor
\eqref{e:fourier-deligne-unnormalized} is left adjoint to $\cF'$,
and the morphism $\pi_0:\e\rar{}e_0$ comes by adjunction from the
natural isomorphism $\cF(\e)\rar{\simeq}(\ql)_{H^*}$. We already saw
that $e_0$ is a closed idempotent in $\sD(H)$, so it remains to
apply Theorem \ref{t:FD-closed-idempotent}(b) together with
Proposition \ref{p:interplay-weak-closed-idempotents}(a) and the
uniqueness of adjunctions.

\begin{proof}[Proof of Lemma \ref{l:fourier-deligne-auxiliary}] As a
first step, we reduce the lemma to the case where $\cN$ is trivial.
If $\cL$ is any multiplicative local system on $H$, we have an
automorphism $\la_\cL:H^*\rar{\simeq}H^*$ given by
$[\cN]\longmapsto[\cL\tens\cN]$, as well as a monoidal
autoequivalence $\sg_\cL:\sD(H)\rar{}\sD(H)$ given by $M\longmapsto
\cL\tens M$. There are natural isomorphisms
\[
\cF'\rar{\simeq}\sg_\cL\circ\cF'\circ\la_{\cL}^*
\qquad\text{and}\qquad
\cF'_*\rar{\simeq}\sg_\cL\circ\cF'_*\circ\la_{\cL}^*
\]
that are compatible with the canonical morphism $\cF'\rar{}\cF'_*$.
Furthermore, if $\cN$ is another multiplicative local system on $H$,
then $\la_\cL$ restricts to an isomorphism
$H^*_{\cN}\rar{\simeq}H^*_{\cL\tens\cN}$, while
$\sg_\cL(e'_{\cN})\cong e'_{\cL\tens\cN}$. Hence Lemma
\ref{l:fourier-deligne-auxiliary} holds for $\cN$ if and only if it
holds for $\cL\tens\cN$. In particular, from now on we may and do
assume that $\cN$ is trivial. To simplify notation, we will write
$H^*_1=H^*_{\cN}$ and $e'_1=e'_\cN$ in this case.

\mbr Let $\pi:H\rar{}H^{ab}\overset{\text{def}}{=}H/[H,H]$ denote
the natural projection. Note that
$H^*_1\overset{\text{def}}{=}(H^*)^{\circ}$, so by property
\S\ref{ss:serre-duality-properties}(6), the homomorphism
$\pi^*:(H^{ab})^*\rar{}H^*$ induces an isomorphism
$(H^{ab})^*\rar{\simeq}H^*_1$ and therefore an equivalence
$\sD((H^{ab})^*)\rar{\simeq}\sD (H^*_1)$. By
Remark~\ref{r:FD-commutative-case} (applied to $H^{ab}$ in place of
$H$), the inverse Fourier transform
$\null^{ab}\cF':\sD((H^{ab})^*)\rar{}\sD(H^{ab})$ is an equivalence
and the canonical morphism $\null^{ab}\cF'\rar{}\null^{ab}\cF'_*$ is
an isomorphism. We have commutative diagrams
\[
\xymatrix{
  \sD((H^{ab})^*) \ar[d]_{\sim} \ar[rr]^{{}^{ab}\cF'}_{\sim} & & \sD(H^{ab}) \ar@{^{(}->}[d]^{\pi^*} \\
  \sD(H_1^*) \ar[rr]^{\cF'} & & \sD(H)
   } \qquad\text{and}\qquad
\xymatrix{
  \sD((H^{ab})^*) \ar[d]_{\sim} \ar[rr]^{{}^{ab}\cF'_*}_{\sim} & & \sD(H^{ab}) \ar@{^{(}->}[d]^{\pi^*} \\
  \sD(H_1^*) \ar[rr]^{\cF'_*} & & \sD(H)
   }
\]
(the right one is commutative by smooth base change). Moreover, the
canonical morphism $\cF'\rar{}\cF'_*$ agrees via these diagrams with
the canonical morphism $\null^{ab}\cF'\rar{}\null^{ab}\cF'_*$, which
is known to be an isomorphism. Now statement (b) of the lemma is
clear and statement (a) follows from the fact that the functor
$\pi^*:\sD(H^{ab})\rar{}\sD(H)$ induces an equivalence $\sD
(H^{ab})\rar{\simeq}e_1'\sD (H)\subset\sD(H)$.
\end{proof}



\section{Reduction process for constructible complexes}\label{s:reduction}

As before, we use the notation of \S\ref{ss:basic-definitions}. In
particular, $k$ denotes an algebraically closed field of
characteristic $p>0$ and $\ell$ is a prime with $\ell\neq p$.

\mbr

In this section we establish Proposition
\ref{p:normal-admissible-pairs} and
Theorem~\ref{t:reduction-complexes}, which will allow us to prove
the main results of this article by induction. In particular, we
show that if $G$ is a perfect unipotent group over $k$ and
$M\in\sD(G)$ is nonzero, then there exists an admissible pair
$(H,\cL)$ for $G$ such that $(i^{H\subset G}_!\cL)*M\neq 0$, where
$i^{H\subset G}:H\into G$ stands for the inclusion morphism.

\mbr

The arguments used in the proofs of Proposition
\ref{p:normal-admissible-pairs} and
Theorem~\ref{t:reduction-complexes} are analogous to the reduction
process for representations of finite nilpotent groups described in
one of the appendices to \cite{intro}, as well as to the ``geometric
reduction process'' for representations of finite groups of the form
$G(\bF_q)$, where $G$ is a unipotent group over $\bF_q$, introduced
in \cite[\S7]{characters}.


\subsection{Formulation of the results}\label{ss:reduction-formulation}
If $G$ is a (quasi-)algebraic group over $k$, we write $\sP(G)$ for
the set of pairs $(H,\cL)$ where $H\subset G$ is a connected
subgroup and $\cL$ is a multiplicative $\ql$-local system on $H$. We
let $\sP_{norm}(G)\subset\sP(G)$ denote the subset of pairs
$(H,\cL)\in\sP(G)$ such that $H$ is normal in $G$.

\begin{defin}\label{d:pairs-order}
We define a partial order on $\sP(G)$ by
$(H_1,\cL_1)\leq(H_2,\cL_2)$ if $H_1\subset H_2$ and
$\cL_2\bigl\lvert_{H_1}\cong\cL_1$.
\end{defin}

\begin{defin}\label{d:compatible}
A pair $(H,\cL)\in\sP(G)$ is said to be \emph{compatible with} a
given object $M\in\sD(G)$ if $(i^{H\subset G}_!\cL)*M\neq 0$.
Here $i^{H\subset G}: H\into G$ is the inclusion morphism.
\end{defin}

\begin{rem}\label{r:compatible}
If $M\ne 0$, then there exists a pair $(H,\cL)\in\sP_{norm}(G)$
compatible with $M$ (for instance, one can take $H=\{ 1\})$. Clearly
among all such pairs $(H,\cL)$ there is a maximal one.
\end{rem}

\begin{prop}\label{p:normal-admissible-pairs}
Let $G$ be a perfect unipotent group over $k$ and $M\in\sD(G)$,
$M\ne 0$. Suppose that $(H,\cL)\in\sP_{norm}(G)$ is maximal among
all pairs in $\sP_{norm}(G)$ that are compatible with $M$. If $\cL$
is invariant under the conjugation action of $G$, then the pair
$(H,\cL)$ is admissible for $G$.
\end{prop}

The proposition is proved in
\S\ref{ss:proof-p:normal-admissible-pairs}.

\begin{thm}\label{t:reduction-complexes}
Let $G$ be a perfect unipotent group over $k$, let $M\in\sD(G)$, and
let $(A,\cN)\in\sP_{norm}(G)$ be compatible with $M$. Then there
exists an admissible pair $(H,\cL)\in\sP(G)$ that is compatible with
$M$ and satisfies $(A,\cN)\leq(H,\cL)$
$($Def.~\ref{d:pairs-order}$)$.
\end{thm}

The proof of Theorem~\ref{t:reduction-complexes} is contained in
\S\ref{ss:proof-t:reduction-complexes} below.

\subsection{An auxiliary extension result}\label{ss:auxiliary-extension}
Our next goal is to prove Corollary \ref{c:extension-tricky}, which
will be used in the proofs of Proposition
\ref{p:normal-admissible-pairs} and Theorem
\ref{t:reduction-complexes}. First we recall the following key
result from \cite{characters}.

\begin{prop}   \label{p:extension-locsys}
Let $G$ be a connected perfect unipotent group over $k$, let
$A\subset G$ be a connected subgroup such that $[G,G]\subset A$, and
let $\cN$ be a multiplicative local system on $A$. Then the following
properties of $\cN$ are equivalent:
 \sbr
\begin{enumerate}[$($i$)$]
\item $\cN$ can be extended to a multiplicative local system on $G$;
 \sbr
\item the pullback of $\cN$ by the commutator morphism $G\times G\rar{}A$ is trivial.
\end{enumerate}
\end{prop}

This is \cite[Prop.~7.7]{characters}. Its proof given in \cite{characters} is based on the equality
$\Ext^2(\bG_a,\bQ_p/\bZ_p)=0$, where $\Ext^2$ is computed in the category of
fppf sheaves on the category of $k$-schemes.

\begin{rem}
According to \cite[Example 3.2]{masoud}, condition (ii) from
Proposition~\ref{p:extension-locsys} is \emph{weaker} than the
triviality of $\cN\bigl\lvert_{[G,G]}$;  in fact, (ii) is equivalent
to the triviality of the pullback of $\cN$ to the ``true commutator"
of $G$ defined in \cite[\S4.1]{masoud}.
\end{rem}

\begin{prop}\label{p:extension-special}
Let $G$, $A$, and $\cN$ be as in
Proposition~\ref{p:extension-locsys}, and let $M\in\sD(G)$. Suppose
that $(A,\cN)$ is compatible with $M$ and that $\cN$ has the
equivalent properties $($i$)$--$($ii$)$ of
Proposition~\ref{p:extension-locsys}.
Then there exists a multiplicative local system $\cL$ on $G$ such that
$\cL\bigl\lvert_A\cong\cN$ and $\cL*M\neq 0$.
\end{prop}

\begin{proof}
We begin with the following observation. Applying the projection
formula to the multiplication morphism $G\times G\rar{}G$ we see that
if $\cE$ is any multiplicative local system on $G$ then for all
$M,N\in\sD(G)$, we have
\begin{equation} \label{e:tensorfunctor}
(\cE\tens M)*(\cE\tens N) \cong \cE\tens(M*N).
\end{equation}

\mbr

Our $\cN$ extends to a
multiplicative local system $\widetilde{\cN}$ on $G$. Formula~\eqref{e:tensorfunctor}
allows us to replace $M$ with
$M\tens\widetilde{\cN}^{-1}$ and to assume that $\cN$ is trivial. Then
compatibility of $M$ with $(A,\cN)$ means that $\pi_!M\ne 0$,
where $\pi:G\rar{}G/A$ is the projection.
But $G/A$ is a connected commutative perfect unipotent group over
$k$, so the functor $\cF':\sD((G/A)^*)\rar{}\sD(G/A)$ is an
equivalence (see \S\ref{ss:fourier-deligne-transform} and
Rem.~\ref{r:FD-commutative-case}).
Thus $\cF'^{-1}(\pi_!M )\ne 0$. This means that $\cL'*\pi_!M\ne 0$
for some multiplicative local system $\cL'$ on $G/A$. Setting
$\cL=\pi^*\cL'$, we obtain $\cL*M=\pi^* (\cL'*\pi_!M)\neq 0$.
\end{proof}

\begin{cor}\label{c:extension-tricky}
Let $G$ be a perfect unipotent group over $k$, let $M\in\sD(G)$, and
let $(A,\cN)\in\sP(G)$ be compatible with $M$. Let $A_1$ be a
connected subgroup of $G$ such that $[A_1,A_1]\subset A\subset A_1$.
If the pullback of $\cN$ by the commutator morphism $A_1\times
A_1\rar{}A$ is trivial then there exists a multiplicative local
system $\cN_1$ on $A_1$ such that $\cN_1\bigl\lvert_A\cong\cN$ and
$(A_1,\cN_1)$ is compatible with $M$.
\end{cor}

\begin{proof}
Set $\overline{M}:=(i^{A\subset G}_!\cN)*M$. By definition,
compatibility of $(A,\cN)$ with $M$ means that $\overline{M}\neq 0$,
so there exists a coset $A_1g$, $g\in G(k)$, such that
$\overline{M}\bigl\lvert_{A_1 g}\neq 0$. Without loss of generality,
we may suppose that $g=1$ (otherwise replace $M$ with
$M*\de_{g^{-1}}$, where $\de_{g^{-1}}$ is the delta-sheaf at
$g^{-1}$). Then we can replace $G$ with $A_1$ and $M$ with
$M\bigl\lvert_{A_1}$,
which reduces us to Proposition \ref{p:extension-special}.
\end{proof}

\subsection{Proof of Proposition
\ref{p:normal-admissible-pairs}}\label{ss:proof-p:normal-admissible-pairs}
To prove that $(H,\cL)$ is an admissible pair, we have to check
three conditions (see Definition~\ref{d:admissible-pair}). Condition
(3) is vacuous because the normalizer of $(H,\cL)$ in $G$ equals
$G$. To check conditions (1) and (2), we use

\begin{lem}  \label{l:isogeny}
Let $Z\subset G^\circ$ be the preimage of the neutral connected
component of the center of $G^\circ/H$. Let
$\vp_{\cL}:G^\circ/H\rar{}(Z/H)^*$ be the group morphism defined in
\S\ref{ss:auxiliary-construction}. Then the morphism
$Z/H\rar{}(Z/H)^*$ induced by $\vp_{\cL}$ is an isogeny.
\end{lem}

\begin{proof}
Let $K\subset G^\circ$ be the preimage of the neutral connected
component of the kernel of $\vp_{\cL}\bigl\lvert_{(Z/H)}$. We have
to show that $K=H$.

\mbr

Clearly $K$ is connected and $[K,K]\subset H\subset K$. Since $\cL$
is $G$-invariant $K$ is normal in $G$. So by
Corollary~\ref{c:extension-tricky}, there exists a multiplicative
local system $\cL'$ on $K$ such that $\cL'\bigl\lvert_H\cong\cL$ and
$(K,\cL')$ is compatible with $M$. Therefore the maximality
assumption on $(H,\cL)$ implies that $K=H$.
\end{proof}

Set $\widetilde G:=G^\circ/H$. Let $\widetilde Z$ be the neutral
connected component of the center of $\widetilde G$.
Lemma~\ref{l:isogeny} says that the homomorphism $\widetilde
Z\rar{}\widetilde Z^*$ induced by $\vp_{\cL}:\widetilde
G\rar{}\widetilde Z^*$ is an isogeny. On the other hand, conditions
(1)--(2) of Definition~\ref{d:admissible-pair} amount to the
requirement that $\widetilde G$ is commutative and $\vp_{\cL}$ is an
isogeny. So to check these conditions, it remains to prove the
following lemma.

\begin{lem}  \label{l:remaining}
Let $f:\widetilde G\rar{}\fA$ be a homomorphism of connected
algebraic groups with $\widetilde G$ nilpotent and $\fA$
commutative. Let $\widetilde Z$ be the neutral connected component
of the center of $\widetilde G$. Suppose that $f
\bigl\lvert_{\widetilde Z}:\widetilde Z\rar{}\fA$ is an isogeny.
Then $\widetilde G$ is commutative.
\end{lem}

\begin{proof}
Assume the contrary. Then $\dim\widetilde G>\dim\widetilde
Z=\dim\fA$, so $\dim\Ker f>0$. Let $K$ be the neutral connected
component of $\Ker f$. The last nontrivial term of the sequence
$K,[\widetilde G,K],[\widetilde G,[\widetilde G,K]],\dotsc$ is a
nontrivial connected subgroup of $\widetilde Z\cap K$. So
$\dim(\widetilde Z\cap K)>0$, contrary to the assumption that $f
\bigl\lvert_{\widetilde Z}:\widetilde Z\rar{}\fA$ is an isogeny.
\end{proof}

\subsection{Proof of Theorem
\ref{t:reduction-complexes}}\label{ss:proof-t:reduction-complexes}
Without loss of generality, we may assume that $(A,\cN)$ is maximal
among all pairs in $\sP_{norm}(G)$ that are compatible with $M$. If
$\cN$ is $G$-invariant, then, by Proposition
\ref{p:normal-admissible-pairs}, we can take $(H,\cL)=(A,\cN)$. So
let $\cN$ be not $G$-invariant, and define $G_1$ as the normalizer
of $\cN$ in $G$. Using the functor $X\longmapsto X*\de_{g^{-1}}$ for
a suitable $g\in G(k)$, as in the proof of Corollary
\ref{c:extension-tricky}, we may assume that $(A,\cN)$ is compatible
with $M\bigl\lvert_{G_1}$. Since $G_1\subsetneq G$, we may assume by
induction that there exists an admissible pair $(H,\cL)$ for $G_1$
satisfying the conclusion of Theorem \ref{t:reduction-complexes} for
the quadruple $\bigl(G_1,M\bigl\lvert_{G_1},A,\cN\bigr)$ in place of
$\bigl(G,M,A,\cN)$.

\mbr

We claim that $(H,\cL)$ also satisfies the conclusion of Theorem
\ref{t:reduction-complexes} for the quadruple $\bigl(G,M,A,\cN)$. We
have $(A,\cN)\leq(H,\cL)$ by construction. Since $(H,\cL)$ is
compatible with $M\bigl\lvert_{G_1}$, it is also compatible with
$M$. It remains to prove the following lemma.

\begin{lem}   \label{l:admissibility}
$(H,\cL)$ is admissible for $G$ and $N_G(H,\cL)=N_{G_1}(H,\cL)$.
\end{lem}
\begin{proof}
Let $G'$ be the normalizer of $(H,\cL)$ in $G$. Since $(A,\cN)\leq
(H,\cL)$, we have $G'\subset G_1$, proving the second statement. To
show that $(H,\cL)$ is admissible for $G$ we have to check
conditions (1)--(3) from Definition~\ref{d:admissible-pair}. Since
$G'\subset G_1$, condition (3) holds automatically for every $g\in
G(k)$ such that $g\not\in G_1(k)$. Conditions (1)--(2) and condition
(3) for $g\in G_1(k)$ hold because $(H,\cL)$ is admissible for
$G_1$.
\end{proof}




\section{Idempotents satisfying the Mackey condition}\label{s:mackey}

Throughout this section, $k$ denotes an algebraically closed field
of characteristic $p>0$, and $\ell$ denotes a prime different from
$p$. We also fix a perfect unipotent group $G$ over $k$ and a closed
subgroup $G'\subset G$. Our goal is to prove two auxiliary results
(Propositions \ref{p:Lemma} and \ref{p:when-ind-closed}) on weak and
closed idempotents in $\sD_{G'}(G')$ satisfying the geometric Mackey
condition with respect to $G$ (Definition
\ref{d:geom-mackey-condition}). They will be used in the proofs of
the main results of the article, given in \S\ref{s:proofs}. We also
review some results on induction functors that were obtained in
\cite{characters} (see \S\ref{ss:characters-review}).

\subsection{Setup}\label{ss:idempotents-setup}
Let $e\in\sD_{G'}(G')$ be a weak idempotent satisfying the geometric
Mackey condition with respect to $G$ (see Definition
\ref{d:geom-mackey-condition}). As usual, we will write
$\overline{e}$ for the object of $\sD(G)$ obtained by extending $e$
to all of $G$ by zero outside of $G'$. It follows from Lemma
\ref{l:convolution-support} below that for every $N\in\sD_G(G)$, the
convolution $\overline{e}*N$ is supported on $G'$. Thus
$\overline{e}*N$ can be viewed as the extension of
$e*\bigl(N\bigl\lvert_{G'}\bigr)$ by zero. \emph{In view of this
fact, and in order to save space, we introduce the notation $e*N$
for $e*\bigl(N\bigl\lvert_{G'}\bigr)$, and we view $N\longmapsto
e*N$ as a functor $\sD_G(G)\rar{}e\sD_{G'}(G')$.}

\subsection{Induction functors} Let $G'\subset G$ be as above.
The induction functors
\[
\ig:\sD_{G'}(G')\rar{}\sD_G(G) \qquad\text{and}\qquad
\Ig:\sD_{G'}(G')\rar{}\sD_G(G)
\]
were defined in \S\ref{ss:induction-functors}. The functor $\Ig$ is
right adjoint to the restriction functor
$\sD_G(G)\rar{}\sD_{G'}(G')$, and for every $M\in\sD_{G'}(G')$, we
let
\[
\eta_M : \bigl( \Ig M \bigr) \bigl\lvert_{G'} \rar{} M
\]
denote the adjunction morphism. Furthermore, by construction (see
\S\ref{ss:induction-functors}), for every $M\in\sD_{G'}(G')$, we
have a canonical arrow
\[
\can_M : \ig M \rar{} \Ig M.
\]

\subsection{Summary of some results of
\cite{characters}}\label{ss:characters-review} Let us now summarize
some of the results that were obtained in \cite[\S5]{characters} in
the setup of \S\ref{ss:idempotents-setup}.

\begin{lem}\label{lemma1}
The object $f=\ig e$ is a weak idempotent in $\sD_G(G)$. If $e$ is a
minimal weak idempotent, then so is $f$.
\end{lem}
This follows from \cite[Prop.~5.11~and~Cor.~5.13]{characters}.

\begin{lem}\label{lemma2}
\begin{enumerate}[$($a$)$]
\item
We have $\ig M\in f\sD_G(G)$ for every $M\in e\sD_{G'}(G')$.
 \sbr
\item The functor
\begin{equation}\label{e:induction-Hecke}
\ig\Bigl\lvert_{e\sD_{G'}(G')} \,:\, e\sD_{G'}(G') \rar{} f\sD_G(G)
\end{equation}
is a bijection at the level of isomorphism classes of objects.
 \sbr
\item If $e$ is a closed idempotent in $\sD_{G'}(G')$, then the functor
\eqref{e:induction-Hecke} is faithful.
 \sbr
\item If both $e$ and $f$ are closed idempotents $($in
$\sD_{G'}(G')$ and $\sD_G(G)$, respectively$)$, then the functor
\eqref{e:induction-Hecke} is an equivalence, 
whose quasi-inverse is isomorphic to the functor $N\longmapsto e*N$.
\end{enumerate}
\end{lem}
This follows from \cite[Thm.~5.12]{characters}. In
Remark~\ref{r:induction-inverse-explicit} below we will see that
in the situation of part (d) of the lemma, a choice of an idempotent
arrow $\e_{G'}\rar{}e$ provides a {\em canonical\,} isomorphism
between the quasi-inverse of \eqref{e:induction-Hecke} and the
functor $N\longmapsto e*N$.


\begin{lem}[See Prop.~5.14 in \cite{characters}]\label{lemma3}
For every $N\in\sD_G(G)$, there exists an isomorphism
$f*N\rar{\simeq}\ig(e*N)$, functorial with respect to $N$.
\end{lem}

\begin{lem}\label{lemma4}
For every $M\in e\sD_{G'}(G')$, the composition
\begin{equation}\label{e:compos}
e*\ig M \xrar{\ \ \id_e*\can_M\ \ } e*\Ig M
\xrar{\ \ \id_e*\eta_M\ \ } e*M
\end{equation}
is an isomorphism in $\sD_{G'}(G')$. In particular,
\begin{equation}  \label{e:ef=e}
e*f\cong e.
\end{equation}
\end{lem}
This result, which strengthens \cite[Prop.~5.15]{characters}, is
proved in \S\ref{ss:proof-lemma4}.

\begin{rem}\label{r:induction-inverse-explicit}
Suppose that the idempotent $e$ is closed. Choose an idempotent arrow
$\pi:\e_{G'}\rar{}e$, then for
every $M\in e\sD_{G'}(G')$ the map $\pi*\id_M:M\rar{}e*M$ is an
isomorphism. So by Lemma \ref{lemma4}, the composition of \eqref{e:compos}
and the isomorphism $(\pi*\id_M)^{-1}:e*M\iso M$ gives a functorial isomorphism
\[
e*\Rg\ig M\iso M, \quad M\in e\sD_{G'}(G').
\]

\end{rem}

\begin{lem}[See Lemma 5.16 in \cite{characters}]\label{l:convolution-support}
Let $\overline{e}\in\sD(G)$ denote the extension of $e$ by zero to
$G$. If $N\in\sD_G(G)$, then
$(\overline{e}*N)\bigl\lvert_{G\setminus G'}=0$.
\end{lem}

\subsection{Main results} The next result will be used in the
proof of Proposition \ref{p:crucial}.

\begin{prop}\label{p:Lemma}
Let $G$, $G'$, $e$, $f$ be as in
\S\S\ref{ss:idempotents-setup}--\ref{ss:characters-review}. Let
$(A,\cN)\in\sP_{norm}(G)$ $($see \S\ref{ss:reduction-formulation}$)$
be such that $\cN$ is $G$-invariant, and let
$e_1:=\cN\tens\bK_A\in\sD_G(A)\subset\sD_G(G)$ be the corresponding
idempotent (cf. \S\ref{ss:Heisenberg-minimal-idempotents}).
If $f*e_1\cong f$ and $f\neq 0$, then $G'\supset A$ and
$e*e_1\cong e$.
\end{prop}
\begin{proof}
By Lemma \ref{l:convolution-support}, $e*e_1$ is supported on $G'$,
and by Lemma \ref{lemma3}, $\ig(e*e_1)\cong f$. Now by Lemma
\ref{lemma2}(b), $e*e_1\cong e$. This clearly implies that
$G'\supset A$ unless $e= 0$. But $f\ne 0$, so $e\ne 0$.
\end{proof}
\begin{rem}\label{r:Lemma}
If $e$ is a minimal weak idempotent in $\sD_{G'}(G')$, then $f$ is
also minimal by Lemma \ref{lemma1}, so in this case, it suffices to
assume that $f*e_1\neq 0$.
\end{rem}

The rest of the section is devoted to a proof of
\begin{prop}\label{p:when-ind-closed} Let $G$ be a perfect unipotent group
over $k$, let $G'\subset G$ be a closed subgroup, and let
$e\in\sD_{G'}(G')$ be a closed idempotent satisfying the geometric
Mackey condition $($Def.~\ref{d:geom-mackey-condition}$)$. Then the
following statements are equivalent.
\begin{enumerate}
\item[$($i$)$] For every $M\in e\sD_{G'}(G')$, the canonical arrow $\ig
M\rar{}\Ig M$ is an isomorphism.
 \sbr
 \item[$($i$\,')$] The canonical arrow $\varphi :\ig e\rar{}\Ig e$ is an isomorphism.
 \sbr
\item[$($ii$)$] The object $f=\ig e$ is a closed idempotent.
\end{enumerate}
\end{prop}

Clearly (i)$\Rightarrow$(i$'$). The implications
(i$'$)$\Rightarrow$(ii)$\Rightarrow$(i) will be proved in
\S\ref{ss:first'-implication}--\ref{ss:second'-implication} using
the adjointness between $\Ig :\sD_{G'}(G')\rar{}\sD_G(G)$ and
$\Rg:\sD_G(G)\rar{}\sD_{G'}(G')$, see
Corollary~\ref{c:induction-adjunctions}.

\begin{rem}
Only the implication (ii)$\Rightarrow$(i) will be used in the proofs
in Section \ref{s:proofs}.
\end{rem}

\subsection{Proof of the implication (i$'$)$\Rightarrow$(ii)} \label{ss:first'-implication}
Let $\e_G\in\sD_G(G)$, $\e_{G'}\in\sD_{G'}(G')$ be the unit objects.
Fix an idempotent arrow $\Rg (\e_G)=\e_{G'}\rar{}e$. By adjunction
(see Corollary~\ref{c:induction-adjunctions}), we get a morphism
$\theta : \e_G\rar{}\Ig e$. Let us prove that $\varphi^{-1}\theta :\e\rar{}f$ is an
idempotent arrow. It suffices to construct a morphism

\begin{equation}  \label{e:modulestructure}
\mu :f*\Ig e\rar{} f
\end{equation}
such that the compositions
\begin{equation}  \label{e:comp1}
f*\e_G\xrar{\ \ \id_f*\theta\ \ } f*\Ig e\rar{\mu} f
\end{equation}
and
\begin{equation}  \label{e:comp2}
f*f \xrar{\ \ \id_f*\varphi\ \ } f*\Ig e\rar{\mu} f
\end{equation}
are isomorphisms (indeed, looking at
\eqref{e:comp2} we see that $\mu$ is an isomorphism, and then looking at
\eqref{e:comp1} we see that $\id_f*\theta :f*\e_G\rar{}f*\Ig e$ is an isomorphism).

\mbr

We will construct \eqref{e:modulestructure} and prove the invertibility of
\eqref{e:comp1} and \eqref{e:comp2} without assuming that $\varphi$ is an isomorphism,
By Lemma~\ref{lemma3}, there is a functorial isomorphism
\begin{equation}   \label{e:5}
f*N\rar{\simeq}\ig(e*\Rg N), \quad N\in\sD_G(G).
\end{equation}
So to construct \eqref{e:modulestructure}, it suffices to define a morphism
$e*\Rg\Ig e\rar{} e$. We define it to be the composition
\[
e*\Rg\Ig e\rar{} e*e\iso e,
\]
where the first morphism comes from the adjunction $\Rg\Ig \rar{}\id$ and the second one from
the idempotent arrow $\e_{G'}\rar{}e$. Using \eqref{e:5}, it is easy to see that
\eqref{e:comp1} is an isomorphism. To show that \eqref{e:comp2} is an isomorphism,
use \eqref{e:5} and apply Lemma~\ref{lemma4}  to $M=e$.

\subsection{Proof of the implication (ii)$\Rightarrow$(i)}
\label{ss:second'-implication}
We keep the notation of Proposition \ref{p:when-ind-closed} and of
\S\ref{ss:idempotents-setup}. We now assume that both $e$ and $f$
are closed idempotents. We must prove property (i) in Proposition
\ref{p:when-ind-closed}.

\begin{lem}  \label{l:1}
If $M\in e\sD_{G'}(G')$ then $\Ig M\in f\sD_G(G)$.
\end{lem}

\begin{proof}
It suffices to show that for every $N\in\sD_G(G)$ the map
\begin{equation}   \label{e:bij1}
\Hom (f*N,\Ig M)\rar{} \Hom (N,\Ig M)
\end{equation}
induced by an idempotent arrow $\e_G\rar{}f$ is bijective. Then the
lemma will follow by applying the implication (ii)$\Rightarrow$(i)
of Proposition \ref{p:idempotent-monads}(b).

\mbr

By Corollary~\ref{c:induction-adjunctions}(i), \[\Hom (N,\Ig M)=\Hom
(\Rg N,M).\] Since $M\in e\sD_{G'}(G')$ an idempotent arrow
$\e_{G'}\rar{}e$ induces an isomorphism \[\Hom (e*\Rg N,M)\iso\Hom
(\Rg N,M)\] by
Proposition~\ref{p:interplay-weak-closed-idempotents}(a). Thus we
get a functorial isomorphism
\begin{equation} \label{e:bij2}
\Hom (N,\Ig M)\iso
\Hom (e*\Rg N,M), \quad M\in e\sD_{G'}(G'), N\in\sD_G(G).
\end{equation}
So to prove the bijectivity of \eqref{e:bij1}, it suffices to show
that the morphism \[e*\Rg N\rar{}e*\Rg (f*N)\] induced by an
idempotent arrow $\e_G\rar{}f$ is an isomorphism. Let
$\overline{e}\in\sD_{G'}(G)$ be the extension of $e\in\sD_{G'}(G')$
by zero, then it suffices to show that the morphism
$\overline{e}*N\rar{}\overline{e}*f*N$ is an isomorphism. This is
clear because by \eqref{e:ef=e}, $\overline{e}\in
f\sD_{G'}(G)=\sD_{G'}(G)f$ and therefore the morphism
$\overline{e}\rar{}\overline{e}*f$ is an isomorphism.
\end{proof}

Now let us prove that for any $M\in e\sD_{G'}(G')$ the arrow
$\can_M:\ig M\rar{}\Ig M$ is an isomorphism\footnote{The idea of the
argument we present below was borrowed from a proof of the result
that Fourier-Deligne transform commutes with Verdier duality, which
was explained to us by Dennis Gaitsgory and is reproduced in the
appendix on the Fourier-Deligne transform in \cite{intro}.}. By
Lemmas~\ref{l:1} and \ref{lemma2}(a), both $\Ig$ and $\ig$ can be
considered as functors from $e\sD_{G'}(G')$ to $f\sD_{G}(G)$. So it
suffices to show that for every $N\in f\sD_{G}(G)$ the map
\begin{equation}   \label{e:bij3}
\Hom (N,\ig M)\rar{}\Hom (N,\Ig M) 
\end{equation}
is bijective.

\mbr

Fix an idempotent arrow $\pi:\e_{G'}\rar{}e$.
By Lemma~\ref{lemma2}(d) and Remark \ref{r:induction-inverse-explicit},
$\ig :e\sD_{G'}(G')\rar{}f\sD_{G}(G)$ is an
equivalence whose quasi-inverse is canonically isomorphic to the functor
$N\longmapsto e*\Rg N$. Thus we get a canonical isomorphism
\begin{equation}   \label{e:bij4}
\Hom (N,\ig M)\iso\Hom (e*\Rg N,M).
\end{equation}

Let $h:\Hom (N,\ig M)\rar{}\Hom (e*\Rg N,M)$ be the composition of
\eqref{e:bij3} with \eqref{e:bij2}. To prove that the map
\eqref{e:bij3} is bijective, it suffices to show that $h$ is. But
$h$ equals \eqref{e:bij4}: both maps take $\be\in\Hom (N,\ig M)$ to
the composition
\begin{eqnarray*}
e*\Rg N && \xrar{\id_e*\Rg (\be )} e*\Rg\ig M\xrar{ \id_e*\Rg
(\can_M )} e*\Rg\Ig M \\ && \xrar{\ \ \id_e*\eta_M\ \ } e*M\xrar{\ \
(\pi*\id_M)^{-1}\ \ } M.
\end{eqnarray*}



\section{Properties of averaging functors}\label{s:averaging}

Throughout this section, $k$ denotes an algebraically closed field
of characteristic $p>0$, and $\ell$ denotes a prime different from
$p$. We also fix a perfect unipotent group $G$ over $k$ and a closed
subgroup $G'\subset G$, with the exception of \S\ref{ss:averaging-duality}, where $G$ and $G'$ are assumed to be ordinary unipotent algebraic groups over $k$.

\mbr

Our goal is to establish certain properties
of the functors $\avg$ and $\ig$ that were introduced in
\S\ref{ss:induction-functors} (they will be used in the proofs of
the main theorems of our work, given in \S\ref{s:proofs}) and to
prove Lemma \ref{lemma4}.

\subsection{Compatibility of $\cF'$ with
$\avg$}\label{ss:f-d-compatibility-averaging} Let us fix a perfect
connected unipotent group $H$ over $k$ equipped with a $G$-action by
\emph{group automorphisms}. The construction of the functor $\cF'$
given in \S\ref{ss:fourier-deligne-transform} easily generalizes to
the equivariant setting (note that the universal local system $\sE$
on $H\times H^*$ has a canonical $G$-equivariant structure), which
yields functors
\[
\cF'^{G'}:\sD_{G'}(H^*)\rar{}\sD_{G'}(H) \qquad\text{and}\qquad
\cF'^{G}: \sD_{G}(H^*)\rar{}\sD_G(H)
\]
Our goal is to prove that these functors are compatible with the
averaging functors $\avg:\sD_{G'}(H)\to\sD_G(H)$ and
$\avg:\sD_{G'}(H^*)\to\sD_G(H^*)$ (see Proposition
\ref{p:f-d-compatibility-averaging}).

\begin{defin}\label{d:avg-pushforward}
Let $X$ and $Y$ be perfect quasi-algebraic schemes over $k$ equipped
with a $G$-action, and let $f:X\rar{}Y$ be a $G$-morphism. We define
an isomorphism of functors\footnote{We are using a slight abuse of
notation. On the left hand side of \eqref{e:avg-pushforward}, $f_!$
is viewed as a functor from $\sD_{G'}(X)$ to $\sD_{G'}(Y)$, and
$\avg$ is computed on $X$. On the right hand side, $f_!$ is viewed
as a functor from $\sD_{G}(X)$ to $\sD_{G}(Y)$, and $\avg$ is
computed on $Y$.}
\begin{equation}\label{e:avg-pushforward}
f_!\circ\avg \rar{\simeq} \avg\circ f_! : \sD_{G'}(X)\rar{}\sD_G(Y)
\end{equation}
as follows. The Cartesian square
\[
\xymatrix{
 X \ar[d]_f \ar[rr]^{i_X\ \ \ \ \ \ \ } & & (G/G')\times X \ar[d]^{\id\times f} \\
 Y \ar[rr]^{i_Y\ \ \ \ \ \ \ } & & (G/G')\times Y
   }
\]
where $i_X(x)=(\overline{1},x)$ and $i_Y(y)=(\overline{1},y)$
determines an isomorphism of functors\footnote{We are using the
definition of $\avg$ given in \S\ref{ss:induction-functors}. In
particular, $G$ acts on $(G/G')\times X$ and on $(G/G')\times Y$
diagonally, via the translation action on $G/G'$ and the given
action on $X$ and $Y$. The functors $\Phi_X:\sD_G((G/G')\times
X)\rar{}\sD_{G'}((G/G')\times X)$ and $\Phi_Y:\sD_G((G/G')\times
Y)\rar{}\sD_{G'}((G/G')\times Y)$ are the forgetful ones.}
\[
i_Y^*\circ\Phi_Y\circ(\id\times f)_! \rar{\simeq} f_!\circ
i_X^*\circ\Phi_X : \sD_G((G/G')\times X)\rar{}\sD_{G'}(Y),
\]
which induces an isomorphism of functors
\[
(\id\times f)_!\circ (i_X^*\circ\Phi_X)^{-1} \rar{\simeq}
(i_Y^*\circ\Phi_Y)^{-1}\circ f_! : \sD_{G'}(X) \rar{}
\sD_G((G/G')\times Y).
\]
Composing the latter isomorphism with $\pr^Y_{2!}$, where
$\pr^Y_2:(G/G')\times Y\rar{}Y$ is the second projection, defines
\eqref{e:avg-pushforward}.
\end{defin}

\begin{defin}\label{d:avg-pullback}
Let $X$ and $Y$ be perfect quasi-algebraic schemes over $k$ equipped
with a $G$-action, and let $f:X\rar{}Y$ be a $G$-morphism. We define
an isomorphism of functors\footnote{We use an abuse of notation
similar to that employed in Definition \ref{d:avg-pushforward}.}
\begin{equation}\label{e:avg-pullback}
f^*\circ\avg \rar{\simeq} \avg\circ f^* : \sD_{G'}(Y)\rar{}\sD_G(X)
\end{equation}
as follows. With the notation of Definition \ref{d:avg-pushforward},
we have an isomorphism
\[
i_X^*\circ\Phi_X\circ(\id\times f)^* \rar{\simeq} f^*\circ
i_Y^*\circ\Phi_Y : \sD_G((G/G')\times Y)\rar{}\sD_{G'}(X),
\]
which induces an isomorphism
\[
(\id\times f)^*\circ (i_Y^*\circ\Phi_Y)^{-1} \rar{\simeq}
(i_X^*\circ\Phi_X)^{-1}\circ f^* : \sD_{G'}(Y) \rar{}
\sD_G((G/G')\times X).
\]
Composing the latter isomorphism with $\pr^X_{2!}$, where
$\pr^X_2:(G/G')\times X\rar{}X$ is the second projection, and
applying the proper base change theorem to the Cartesian square
\[
\xymatrix{
 (G/G')\times X \ar[d]_{\id\times f} \ar[rr]^{\ \ \ \ \ \ \ \pr^X_2} & & X \ar[d]^f \\
 (G/G')\times Y \ar[rr]^{\ \ \ \ \ \ \ \pr_2^Y} & & Y
   }
\]
defines \eqref{e:avg-pullback}.
\end{defin}

\begin{defin}\label{d:avg-tensor-product}
Let $X$ be a perfect quasi-algebraic scheme equipped with a
$G$-action, and let $L\in\sD_G(X)$. We define functorial
isomorphisms
\begin{equation}\label{e:avg-tensor-product}
L\tens\avg(M)\rar{\simeq}\avg(F(L)\tens M), \qquad M\in\sD_{G'}(X),
\end{equation}
where $F:\sD_G(X)\rar{}\sD_{G'}(X)$ is the forgetful functor, as
follows.

\mbr

With the notation of Remark \ref{r:factorization}, we have
$\pr_2\circ i=\id_X$, which yields functorial isomorphisms
\[(i^*\circ\Phi)((\pr_2^*L)\tens
N)\rar{\simeq}F(L)\tens((i^*\circ\Phi)(N))\] for all
$N\in\sD_G((G/G')\times X)$. Thus we obtain functorial isomorphisms
\[
(\pr_2^*L)\tens\bigl((i^*\circ\Phi)^{-1}(M)\bigr)\rar{\simeq}(i^*\circ\Phi)^{-1}(F(L)\tens
M).
\]
Applying $\pr_{2!}$ to both sides and using the projection formula
yields \eqref{e:avg-tensor-product}.
\end{defin}

\begin{prop}\label{p:f-d-compatibility-averaging} Let $H$ be as above. There is a functorial family of isomorphisms
\[\cF'^G\bigl(\avg(M)\bigr)\rar{\simeq}\avg\bigl(\cF'^{G'}(M)\bigr)\] in $\sD_G(H)$ for all
$M\in\sD_{G'}(H^*)$.
\end{prop}
\begin{proof}
Use Definitions \ref{d:avg-pushforward}, \ref{d:avg-pullback} and
\ref{d:avg-tensor-product}.
\end{proof}

\subsection{Weak semigroupal structure on
$\av_{G/G'}$}\label{ss:averaging-semigroupal} In this subsection we
let $H$ be any perfect quasi-algebraic group over $k$ equipped with
an action of $G$ by group automorphisms. Note that both
$\sD_{G'}(H)$ and $\sD_G(H)$ are monoidal categories with respect to
the functor of convolution with compact supports. The goal of this
subsection is to construct a weak semigroupal structure
(Def.~\ref{d:weak-semigroupal}(a)) on the functor
$\av_{G/G'}:\sD_{G'}(H)\rar{}\sD_G(H)$.

\mbr

One ingredient in the construction is Definition
\ref{d:avg-pushforward}. Another ingredient is

\begin{defin}\label{d:avg-external}
Let $X$ and $Y$ be perfect quasi-algebraic schemes over $k$ equipped
with a $G$-action, and let $G$ act on $X\times Y$ diagonally. We
construct a functorial collection of morphisms\footnote{Here,
$\boxtimes$ denotes the external tensor product, viewed either as a
functor from $\sD_{G'}(X)\times\sD_{G'}(Y)$ to $\sD_{G'}(X\times Y)$
or as a functor from $\sD_G(X)\times\sD_G(Y)$ to $\sD_G(X\times
Y)$.} (not necessarily isomorphisms)
\begin{equation}\label{e:avg-external}
(\avg M)\boxtimes(\avg N) \rar{} \avg(M\boxtimes N)
\end{equation}
for all $M\in\sD_{G'}(X)$ and $N\in\sD_{G'}(Y)$ as follows.
 \sbr
\begin{enumerate}[1)]
\item Let $i_X$, $i_Y$ be as in Definition \ref{d:avg-pushforward},
and let $i_{X\times Y}:X\times Y\rar{}(G/G')\times X\times Y$ by
given by $(x,y)\mapsto(\overline{1},x,y)$. As before, $G$ acts
diagonally on all product $k$-schemes appearing in the definition.
 \sbr
\item Define $\De:(G/G')\times X\times Y \rar{} (G/G')\times
X\times(G/G')\times Y$ by $(g,x,y)\mapsto(g,x,g,y)$.
 \sbr
\item Observe that we have a natural isomorphism\footnote{For any $G$-scheme $Z$, we have the forgetful functor
$\Phi_Z:\sD_G((G/G')\times Z)\rar{}\sD_{G'}((G/G')\times Z)$.}
\[
\De^* \bigl(
(i_X^*\circ\Phi_X)^{-1}(M)\boxtimes(i_Y^*\circ\Phi_Y)^{-1}(N) \bigr)
\rar{\simeq} (i_{X\times Y}^*\circ\Phi_{X\times Y})^{-1}(M\boxtimes
N),
\]
which gives rise to a morphism
\begin{eqnarray*}
(i_X^*\circ\Phi_X)^{-1}(M)\boxtimes(i_Y^*\circ\Phi_Y)^{-1}(N)
&\rar{}& \De_*(i_{X\times Y}^*\circ\Phi_{X\times Y})^{-1}(M\boxtimes
N) \\
&=& \De_!(i_{X\times Y}^*\circ\Phi_{X\times Y})^{-1}(M\boxtimes N).
\end{eqnarray*}
 \sbr
\item Write $\pr:(G/G')\times X \times(G/G')\times Y \rar{}X\times
Y$ for the projection. Applying $\pr_!$ to the last morphism defines
\eqref{e:avg-external}.
\end{enumerate}
\end{defin}

Now we can give the main definition of this subsection:

\begin{defin}[Weak semigroupal structure on
$\av_{G/G'}$]\label{d:semigroupal-averaging} Given
$M,N\in\sD_{G'}(H)$, we define a morphism
\[
\phi(M,N):(\avg M)*(\avg N)\rar{}\avg(M*N)
\]
as follows. Let $\mu:H\times H\rar{}H$ be the multiplication
morphism. Applying $\mu_!$ to the morphism constructed in Definition
\ref{d:avg-external}, we obtain a morphism
\[
(\avg M)*(\avg N)\rar{}\mu_!\bigl(\avg(M\boxtimes N)\bigr).
\]
We define $\phi(M,N)$ as the composition of the latter morphism and
the isomorphism
\[
\mu_!\bigl(\avg(M\boxtimes N)\bigr) \rar{\simeq}
\avg\bigl(\mu_!(M\boxtimes N)\bigr) \overset{\text{def}}{=}
\avg(M*N)
\]
constructed in Definition \ref{d:avg-pushforward}. It is
straightforward to check that $\phi$ is a weak semigroupal structure
on the functor $\avg:\sD_{G'}(H)\rar{}\sD_G(H)$.
\end{defin}

\subsection{The functor $\avg$ as a bimodule
functor}\label{ss:averaging-module} We remain in the setup of
\S\ref{ss:averaging-semigroupal}.
\begin{prop}
Let $X$ and $Y$ be perfect quasi-algebraic schemes over $k$ equipped
with a $G$-action, and let $G$ act on $X\times Y$ diagonally. There
exist isomorphisms
\begin{equation}\label{e:avg-module-external}
(\avg M)\boxtimes N \rar{\simeq} \avg(M\boxtimes N),
\end{equation}
functorial with respect to $M\in\sD_{G'}(X)$ and $N\in\sD_G(Y)$.
\end{prop}
\begin{proof}
Let us use the notation of Definition \ref{d:avg-external} and
observe that
\[
i_{X\times Y} = i_X\times\id_Y : X\times Y \into (G/G')\times
X\times Y.
\]
This implies that for $M\in\sD_{G'}(X)$, $N\in\sD_G(Y)$, there are
functorial isomorphisms
\[
\bigl((i_X^*\circ\Phi_X)^{-1}(M)\bigr)\boxtimes N \rar{\simeq}
(i_{X\times Y}^*\circ\Phi_{X\times Y})^{-1}(M\boxtimes N).
\]
Similarly, if $\pr_2^X:(G/G')\times X\rar{}X$ and $\pr_2^{X\times
Y}:(G/G')\times X\times Y\rar{}X\times Y$ are the natural
projections, we have $\pr_2^{X\times Y}=\pr_2^X\times\id_Y$. This
defines \eqref{e:avg-module-external}.
\end{proof}

\begin{cor}\label{c:averaging-module}
Let $H$ be a perfect quasi-algebraic group over $k$ equipped with an
action of $G$ by group automorphisms. There exist isomorphisms
\begin{equation}\label{e:avg-module}
(\avg M)*N \rar{\simeq} \avg(M*N),
\end{equation}
\begin{equation}\label{e:avgmodule}
N*(\avg M) \rar{\simeq} \avg(N*M),
\end{equation}
functorial with respect to $M\in\sD_{G'}(H)$ and $N\in\sD_G(H)$.
\end{cor}
\begin{proof}
As in Definition \ref{d:semigroupal-averaging}, let $\mu:H\times
H\rar{}H$ be the multiplication morphism and apply $\mu_!$ to the
isomorphism \ref{e:avg-module-external}. Composing the result with
the isomorphism \eqref{e:avg-pushforward} constructed for $f=\mu$
defines \eqref{e:avg-module}. The construction of \eqref{e:avgmodule}
is similar.
\end{proof}

One can check that isomorphisms
\eqref{e:avg-module}-\eqref{e:avgmodule} define on $\avg
:\sD_{G'}(H)\to\sD_G(H)$ the structure of a $\sD_G(H)$-bimodule
functor.

\subsection{Averaging functors and duality}\label{ss:averaging-duality}
In this subsection we fix an ordinary (as opposed to perfect) unipotent group $G$ over $k$ and a closed subgroup $G'\subset G$. Our goal is to establish Corollary~\ref{c:induction-duality}, which will be used
in \S\ref{ss:proof-p:duality-canonical} below.

\mbr

Let $X$ be a scheme of finite type over $k$ equipped with a $G$-action, and let
$\bD_X:\sD_G(X)\rar{}\sD_G(X)$ denote the Verdier duality functor. When viewed as a functor $\sD_{G'}(X)\rar{}\sD_{G'}(X)$, the Verdier duality functor will be denoted by $\bD_X'$.

\begin{lem}\label{l:averaging-duality}
There is a natural isomorphism of functors from $\sD_{G'}(X)$ to $\sD_G(X)$,
\begin{equation}\label{e:averaging-duality-nat-isom}
\av_{G/G'}[2d](d) \rar{\simeq} \bD_X\circ\Av_{G/G'}\circ\bD'_X ,
\end{equation}
where $d=\dim(G/G')$.
\end{lem}

\begin{proof}
Recall that $\Av_{G/G'}$ is right adjoint to the forgetful functor $F:\sD_G(X)\rar{}\sD_{G'}(X)$
(see Lemma \ref{l:averaging}). Since $\bD_X'\circ F\rar{\simeq}F\circ\bD_X$ the functor
$\bD_X\circ\Av_{G/G'}\circ\bD'_X$ is left adjoint to $F$. But $\av_{G/G'}$ is also left adjoint to $F$
(see Lemma \ref{l:averaging}).
\end{proof}

In Definition \ref{d:duality} we constructed a contravariant functor $\bD_G^-:\sD_G(G)\rar{}\sD_G(G)$.

\begin{cor}\label{c:induction-duality}
There is a natural isomorphism of functors $\sD_{G'}(G')\rar{}\sD_G(G)$,
\begin{equation}\label{e:induction-duality-nat-isom}
\ig[2d](d) \rar{\simeq} \bD^-_G\circ\Ig\circ\bD^-_{G'} ,
\end{equation}
where $d=\dim(G/G')$.
\end{cor}

\begin{proof}
Use Definition~\ref{d:induction} and Lemma~\ref{l:averaging-duality}.
%
%
%
\end{proof}

\subsection{Proof of Lemma \ref{lemma4}}\label{ss:proof-lemma4}
The last assertion of the lemma follows from the first one by taking
$M=e$. In turn, the first assertion results from the more general
\begin{prop}\label{p:lemma4}
Let $G$ be a perfect unipotent group over $k$, let $G'\subset G$ be
a closed subgroup, and let $N\in\sD_{G'}(G)$ and $L\in\sD(G)$ be
such that
\begin{equation}  \label{e:6-1}
L*\de_g*N=0 \mbox{\; for all } g\in G(k), g\not\in G'(k).
\end{equation}
Then the composition
\begin{equation}  \label{e:6-2}
F\bigl(\avg N\bigr) \rar{} F\bigl(\Avg N\bigr) \rar{} N
\end{equation}
becomes an isomorphism after convolution\footnote{Here convolution is interpreted as a functor $\sD(G)\times\sD_{G'}(G)\rar{}\sD(G)$.} with $L$ on the left; here
\[F:\sD_G(G)\rar{}\sD_{G'}(G)\] is the forgetful functor, the first
morphism in \eqref{e:6-2} is induced by the canonical morphism
$\avg\rar{}\Avg$, and the second one is the adjunction morphism.
\end{prop}

Indeed, if $e\in\sD_{G'}(G')$ is a weak idempotent that satisfies
the geometric Mackey condition with respect to $G$, then for every
$M\in e\sD_{G'}(G')$ and every $g\in G(k)$ such that $g\not\in
G'(k)$ we have $\overline{e}*\de_g*\overline{M}\cong
\overline{e}*\de_g*\overline{e}*\overline{M}=0$, so the proposition
can be applied to $L=\overline{e}$ and $N=\overline{M}$, which
yields the first assertion of Lemma \ref{lemma4}.

\mbr

The proposition is proved in \S\ref{sss:proof-p:lemma4} below. As a
first step, we will find a more explicit description of the
morphisms \eqref{e:6-2}
in the slightly more general setting of~\S\ref{ss:explicitdescription}.

\subsubsection{} \label{ss:explicitdescription}
We fix a perfect quasi-algebraic scheme $X$ over $k$ and an action
of $G$ on $X$. Then we have the functors
$\avg,\Avg:\sD_{G'}(X)\rar{}\sD_G(X)$ defined by
\eqref{e:averaging}--\eqref{e:!averaging} and a canonical morphism
$\avg\rar{}\Avg$. By Lemma \ref{l:averaging}, $\Avg$ is right
adjoint to the forgetful functor $F:\sD_G(X)\rar{}\sD_{G'}(X)$, so
we get morphisms
\begin{equation}\label{e:(1)}
F\circ\avg\rar{}F\circ\Avg\rar{}\Id_{\sD_{G'}(X)}.
\end{equation}
They will be described explicitly in Lemmas \ref{l:lem1} and
\ref{l:lem2} below.

\subsubsection{} First let us explicitly describe the functors
$F\circ\avg$, $F\circ\Avg$ and the morphism
$F\circ\avg\rar{}F\circ\Avg$ in terms of the $G'$-equivariant
embedding $i;X\rar{}(G/G')\times X$, $i(x)=(\overline{1},x)$ and the
projection $\pr_2:(G/G')\times X\rar{}X$.

\mbr

We have commutative diagrams
\begin{equation}\label{e:(2)}
\xymatrix{
  & & \sD_G((G/G')\times X) \ar[d]^\Phi \ar[dll]^\sim_{i^*\circ\Phi} \ar[rr]^{\ \ \ \ \ \ \ \pr_{2*}} & & \sD_G(X) \ar[d]^F \\
  \sD_{G'}(X) & & \sD_{G'}((G/G')\times X) \ar[ll]_{i^*} \ar[rr]^{\ \ \ \ \ \ \ \pr_{2*}} & & \sD_{G'}(X)
   }
\end{equation}
and
\begin{equation}\label{e:(3)}
\xymatrix{
  & & \sD_G((G/G')\times X) \ar[d]^\Phi \ar[dll]^\sim_{i^*\circ\Phi} \ar[rr]^{\ \ \ \ \ \ \ \pr_{2!}} & & \sD_G(X) \ar[d]^F \\
  \sD_{G'}(X) & & \sD_{G'}((G/G')\times X) \ar[ll]_{i^*} \ar[rr]^{\ \ \ \ \ \ \ \pr_{2!}} & & \sD_{G'}(X)
   }
\end{equation}
in which $F$ and $\Phi$ are the forgetful functors and
$i^*\circ\Phi:\sD_G((G/G')\times X)\rar{}\sD_{G'}(X)$ is an
equivalence.

\begin{lem}\label{l:lem1}
The functors $F\circ\avg$ and $F\circ\Avg$ canonically identify with
the compositions
\begin{equation}\label{e:(4)}
\sD_{G'}(X)\xrar{\ \ \Phi(i^*\Phi)^{-1}\ \ } \sD_{G'}((G/G')\times
X)\xrar{\ \ \pr_{2!}\ \ }\sD_{G'}(X)
\end{equation}
and
\begin{equation}\label{e:(5)}
\sD_{G'}(X)\xrar{\ \ \Phi(i^*\Phi)^{-1}\ \ } \sD_{G'}((G/G')\times
X)\xrar{\ \ \pr_{2*}\ \ }\sD_{G'}(X).
\end{equation}
The morphism $F\circ\avg\rar{}F\circ\Avg$ corresponds under these
identifications to the canonical morphism $\pr_{2!}\rar{}\pr_{2*}$.
\end{lem}

\begin{proof}
This immediately follows from the definitions
\eqref{e:averaging}--\eqref{e:!averaging} together with the
commutativity of the diagrams \eqref{e:(2)}--\eqref{e:(3)}.
\end{proof}

\subsubsection{} The diagram of $G'$-varieties
\begin{equation}\label{e:diagram}
X\overset{i}{\into} (G/G')\times X\rar{\pr_2}X, \qquad \pr_2\circ
i=\id_X,
\end{equation}
defines canonical morphisms
\begin{equation}\label{e:(6)}
\pr_{2!}\rar{}\pr_{2*}\rar{}i^*,
\end{equation}
where $\pr_{2!}$ and $\pr_{2*}$ are viewed as functors
$\sD_{G'}((G/G')\times X)\rar{}\sD_{G'}(X)$. Namely, the morphism
$\pr_{2*}\rar{}i^*$ in \eqref{e:(6)} is the composition
\begin{equation}\label{e:(7)}
\pr_{2*}\rar{}\pr_{2*}\circ i_*\circ i^* =(\pr_2\circ i)_*\circ
i^*=i^*
\end{equation}
and also the composition
\begin{equation}\label{e:(8)}
\pr_{2*}=(\pr_2\circ i)^*\circ\pr_{2*} =
i^*\circ\pr_2^*\circ\pr_{2*}\rar{}i^*.
\end{equation}

\begin{rem}  \label{r:thecomposition}
Formula \eqref{e:(7)} shows that the composition
$\pr_{2!}\rar{}\pr_{2*}\rar{}i^*$ equals the morphism
$\pr_{2!}\rar{}\pr_{2!}\circ i_*\circ i^*=\pr_{2!}\circ i_!\circ
i^*=i^*$.
\end{rem}

\begin{lem}\label{l:lem2}
Identify $F\circ\avg$ and $F\circ\Avg$ with the compositions
\eqref{e:(4)} and \eqref{e:(5)}. Then the diagram \eqref{e:(1)}
comes from the morphisms \eqref{e:(6)} and the equality
$i^*\Phi(i^*\Phi)^{-1}=\Id_{\sD_{G'}(X)}$.
\end{lem}

\begin{proof}
It suffices to show that the adjunction
$F\circ\Avg\rar{}\Id_{\sD_{G'}(X)}$ comes from the morphism
$\pr_{2*}\rar{}i^*$ defined by \eqref{e:(8)}. By definition,
\[
F\circ\Avg=i^*\Phi\circ\pr_2^*\circ\pr_{2^*}\circ(i^*\Phi)^{-1},
\]
and the adjunction $F\circ\Avg\rar{}\Id_{\sD_{G'}(X)}$ comes from
the adjunction $\pr_2^*\pr_{2*}\rar{}\Id$. It remains to consider
the commutative diagram
\[
\xymatrix{
  & \sD_G((G/G')\times X) \ar[d]^\Phi \ar[dl]^\sim_{i^*\circ\Phi}
  \ar[r]^{\ \ \ \ \ \ \ \pr_{2*}} & \sD_G(X) \ar[d]^F
  \ar[r]^{\pr_2^*\ \ \ \ \ \ \ }
  & \sD_G((G/G')\times X) \ar[d]^\Phi \\
  \sD_{G'}(X) & \sD_{G'}((G/G')\times X) \ar[l]_{i^*} \ar[r]^{\ \ \ \ \ \ \ \pr_{2*}}
   & \sD_{G'}(X) \ar[r]^{\pr_2^*\ \ \ \ \ \ \ } \ar[dr]^\sim_\Id & \sD_{G'}((G/G')\times
   X) \ar[d]^{i^*} \\
   & & & \sD_{G'}(X)
   }
\]
\end{proof}

\subsubsection{Proof of Proposition
\ref{p:lemma4}}\label{sss:proof-p:lemma4} In the proof we specialize
the earlier discussion to the case where $X=G$ and $G$ acts on
itself by conjugation. In particular, we will use the diagram
\eqref{e:diagram} in this setting.

\mbr

Define $\Nt\in\sD_{G'}((G/G')\times X)$ by
\begin{equation}  \label{e:Nt}
\Nt :=\Phi ((i^*\circ\Phi)^{-1} (N)).
\end{equation}
In view of Lemma \ref{l:lem2} and Remark~\ref{r:thecomposition}, we
can reformulate Proposition \ref{p:lemma4} as follows:

\mbr

\noindent {\sc Claim.} Under the assumptions of Proposition
\ref{p:lemma4}, the natural morphism
\[
\pr_{2!}\Nt\rar{}\pr_{2!}i_*i^*\Nt
\]
becomes an isomorphism after convolution with $L$ on the left.

\mbr

To prove the claim, let $U$ be the complement of $\{\overline{1}\}$
in $G/G'$, and let $j:U\times G\into(G/G')\times G$ denote the
inclusion map. The exact triangle
\[
j_!j^*\Nt=j_!j^!\Nt\rar{}\Nt\rar{}i_*i^*\Nt\rar{}j_!j^*\Nt[1]
\]
implies that it is enough to check that $L*\pr_{2!}j_!j^*\Nt=0$,
i.e.,
\begin{equation}\label{e:(star)}
L*\pi_! N'=0,
\end{equation}
where $\pi:=\pr_2\circ j:U\times G\rar{}G$ is the second projection
and  $N'=j^*\Nt$.

\mbr

\begin{lem}[Projection formula]\label{l:projection-formula}
We have
\begin{equation}\label{e:projection}
L*\pi_!N'\cong\pi_!(\pi^*L*N'),
\end{equation}
where on the right hand side we are using the convolution with
compact support\footnote{It is defined by a formula essentially
identical to \eqref{e:convol-S}, except that now $G$ plays the role
of $H$ and $U$ plays the role of $S$ (and the order of the factors
must be reversed) in \eqref{e:convol-S}.}
\[
\sD(U\times G)\times\sD(U\times G)\rar{}\sD(U\times G).
\]
\end{lem}

\begin{proof}
Let $\mu:G\times G\rar{}G$ denote the multiplication morphism.
Consider the commutative diagram
\[
\xymatrix{
  G\times U\times G \ar[d]_{\mu'} \ar[rr]^{\pi'} & & G\times G\ar[d]^\mu \\
  U\times G \ar[rr]^\pi & & G
   }
\]
where $\pi'(g_1,u,g_2)=(g_1,g_2)$ and $\mu'(g_1,u,g_2)=(u,g_1g_2)$.
By the K\"unneth formula,
\begin{eqnarray*}
L*\pi_!N' &\overset{\text{def}}{=}& \mu_! \left( L\boxtimes \pi_! N'
\right) \cong \mu_!\pi'_! \left( L\boxtimes N' \right) \\
&\cong& \pi_!\mu'_! \left( L\boxtimes N' \right)
\overset{\text{def}}{=} \pi_!(\pi^*L * N').
\end{eqnarray*}
\end{proof}

By Lemma~\ref{l:projection-formula}, to prove \eqref{e:(star)}, it
suffices to check that $\pi^*L*N'=0$. Equivalently, we must show
that for every $u\in U(k)$, the restriction of $\pi^*L*N'$ to
$\{u\}\times G\subset U\times G$ is equal to $0$.
This follows from \eqref{e:6-1}
together with
\begin{lem}
Let $i_u:G\into U\times G$ be given by $x\mapsto (u,x)$, and let
$g\in G(k)$ be any representative of $u$. Then
\[
i_u^*(\pi^*L*N')\cong L*\de_g*N*\de_{g^{-1}}.
\]
\end{lem}

\begin{proof} By the proper base change theorem,
\[
i_u^*(\pi^*L*N') \overset{\text{def}}{=} i_u^*\mu'_! \left(
L\boxtimes N' \right) \cong \mu_!\bigl(L\boxtimes i_u^*\Nt\bigr)
\cong L*(i_u^*\Nt).
\]
Define $\la_g:G/G'\rar{}G/G'$ and $c_g:G\rar{}G$ by $\la_g (x):=gx$,
$c_g (y):=gyg^{-1}$.
Then $i_u=(\la_g\times\id_G)\circ i$, whence
\begin{eqnarray*}
i_u^*\Nt &\cong& i^*(\la_g\times\id_G)^*\Nt \\
&\cong& i^*(\la_g\times\id_G)^*(\la_{g^{-1}}\times c_{g^{-1}})^*\Nt
\\
&\cong& i^*(\id\times c_{g^{-1}})^*\Nt\cong c_{g^{-1}}^*i^*\Nt\cong
c_{g^{-1}}^*N \overset{\text{def}}{=} \de_g*N*\de_{g^{-1}}
\end{eqnarray*}
(the second isomorphism uses the fact that
$\Nt\in\sD_{G'}((G/G')\times X)$ comes from a $G$-equivariant
complex $(i^*\circ\Phi)^{-1} (N)$, see \eqref{e:Nt}). The lemma
follows.
\end{proof}



\section{Proofs of the main results}\label{s:proofs}

Throughout this section, $k$ denotes an algebraically closed field
of characteristic $p>0$ and $\ell$ denotes a prime different from
$p$. Our goal is to prove the five main theorems
(\ref{t:properties-character-sheaves}, \ref{t:functional-dimension},
\ref{t:construction-L-packets}, \ref{t:idempotents} and
\ref{t:geometric-mackey}) and the propositions stated in
\S\ref{s:results}.

\mbr

The section is organized as follows. In
\S\ref{ss:weak-semigroupal} and \S\ref{ss:Heis-idemp-geom-Mackey} we
recall some results from \cite{characters}. In \S\ref{ss:stumbling_block}
we formulate a key compatibility lemma, which will be proved in Appendix~\ref{s:CFT}.
In \S\ref{ss:proof-p:uniqueness} we prove Proposition
\ref{p:uniqueness}. In \S\ref{ss:artin} we review the exactness
properties of pushforward and pullback functors $f_*,f_!,f^*$ and
induction functors $\ig,\Ig$ with respect to perverse
$t$-structures. In \S\ref{ss:key-proposition} we formulate a key
result from which the five main theorems are deduced in
\S\S\ref{ss:proof-t:idempotents}--\ref{ss:proof-t:properties-character-sheaves}
without difficulty. Propositions \ref{p:duality-canonical} and
\ref{p:duality} are proved in \S\ref{ss:proof-p:duality-canonical}
and \S\ref{ss:proof-p:duality}, respectively (the proof of
Proposition \ref{p:duality} uses Theorem
\ref{t:construction-L-packets}). Finally, the aforementioned key
result is proved in \S\ref{ss:proof-p:crucial}, using an auxiliary
proposition from \S\ref{ss:avg-closed-idemp} that relies on
Proposition \ref{p:f-d-compatibility-averaging} and an equivariant
version of Corollary \ref{c:FD-closed-idempotent}.

\mbr

We remark that the order in which the results of \S\ref{s:results} were formulated are different from the order in which they are proved here. On the other hand, the structure of the present section is linear: each argument we give relies only on the earlier proofs and/or the results of the preceding sections of the article.

\subsection{Weak semigroupal structure on the functor
$\ig$}\label{ss:weak-semigroupal}

\begin{lem}\label{lemma5}
Let $G$ be a perfect unipotent group over $k$, and let $G'\subset G$
be a closed subgroup.
 \sbr
\begin{enumerate}[$($a$)$]
\item The functor of induction with compact supports
$\ig:\sD_{G'}(G')\rar{}\sD_G(G)$ has a natural weak semigroupal
structure
\begin{equation}\label{e:semigroupal-structure}
\bigl(\ig M\bigr)*\bigl(\ig N\bigr) \rar{} \ig(M*N)
\end{equation}
$($the morphism is defined for all $M,N\in\sD_{G'}(G')${}$)$.
 \sbr
\item If $M,N\in\sD_{G'}(G')$ satisfy
$\overline{M}*\de_x*\overline{N}=0$ for all $x\in G(k)\setminus
G'(k)$, then \eqref{e:semigroupal-structure} is an isomorphism
$($where $\overline{M}\in\sD(G)$ is the extension of $M$ by zero to
$G$, and $\de_x$ denotes the delta-sheaf at $x${}$)$.
 \sbr
\item Suppose that $e\in\sD_{G'}(G')$ is a weak idempotent
satisfying the geometric Mackey condition with respect to $G$. If
$M,N\in e\sD_{G'}(G')$, then $\overline{M}*\de_x*\overline{N}=0$ for
all $x\in G(k)\setminus G'(k)$.
\end{enumerate}
\end{lem}

\begin{proof}
(a) If $\iota:G'\into G$ is the inclusion morphism, then, by
definition (cf.~\S\ref{ss:induction-functors}), we have
$\ig=\av_{G/G'}\circ\iota_*$, where $G$ acts on itself by
conjugation. The functor $\iota_*:\sD_{G'}(G')\rar{}\sD_{G'}(G)$ has
an obvious strong semigroupal structure. Combining it with the weak
semigroupal structure on $\avg$ constructed in Definition
\ref{d:semigroupal-averaging}, we obtain a weak semigroupal
structure on $\ig$.

\mbr

\noindent (b) It is not hard to check that the weak semigroupal
structure on $\ig$ constructed in part (a) coincides with that
defined in \cite[\S5.5.2]{characters}. Hence the desired assertion
follows from Proposition 5.11 of \emph{op.~cit.}

\mbr

\noindent (c) This follows immediately from the observation that
$\overline{M}\cong\overline{M}*\overline{e}$ and
$\overline{N}\cong\overline{e}*\overline{N}$ whenever $M,N\in
e\sD_{G'}(G')$, together with Definition
\ref{d:geom-mackey-condition}.
\end{proof}

\begin{rem}
In the situation of Lemma \ref{lemma5}(c), the object $\ig
e\in\sD_G(G)$ is a weak idempotent by Lemma \ref{lemma2}(a).
Moreover, Lemma~\ref{l:Hecke-monoidal},
Lemma~\ref{lemma2}(d), and Lemma~\ref{lemma5} imply
\end{rem}

\begin{cor}\label{c:ind-monoidal-equivalence}
Suppose that in the situation of Lemma \ref{lemma5}$($c$)$, the
idempotents $e$ and $f=\ig e$ are closed. Then the semigroupal
categories $e\sD_{G'}(G')$ and $f\sD_G(G)$ are monoidal, and the weakly
semigroupal functor $\ig$ restricts to a monoidal equivalence
\begin{equation}  \label{e:monoidal-equivalence}
e\sD_{G'}(G') \rar{\sim} f\sD_G(G).
\end{equation}
\end{cor}

\begin{rem}   \label{r:canonical_bij}
In the situation of Corollary~\ref{c:ind-monoidal-equivalence} there is a canonical bijection
$\varphi :\Ar_e\iso\Ar_f$,
where $\Ar_e$ is the set of idempotent arrows $\e\rar{}e$. Namely,
$\varphi :\Ar_e\iso\Ar_f$ is the composition
\begin{equation}  \label{e:bijections}
\Ar_e\iso\Isom (e*e,e)\iso\Isom (f*f,f)\iso\Ar_f,
\end{equation}
where $\Isom (e*e,e)$ stands for the set of isomorphisms $e*e\iso e$,
the middle bijection in \eqref{e:bijections} comes from the monoidal equivalence
\eqref{e:monoidal-equivalence}, and the other two bijections in  \eqref{e:bijections}
are provided by Corollary~\ref{c:closed-idemp-arrows}.
\end{rem}

\begin{rem}
The weak semigroupal functor $\ig$ from Lemma~\ref{lemma5}(a) gives an algebra structure on $\ig (\e_{G'})$.
Unless $G/G'$ is finite, this algebra is not unital.
\end{rem}

\subsection{Compatibility of $\ig$ with braidings and twists}   \label{ss:stumbling_block}

\begin{lem}\label{l:compatibility-ind-braiding-twists}
Let $G$ be a perfect unipotent group over $k$, and let $G'\subset G$ be a closed subgroup.
 \sbr
\begin{enumerate}[$($a$)$]
\item For all $M,N\in\sD_{G'}(G')$, the diagram
\[
\xymatrix{
  \ig(M) * \ig(N) \ar[d]_{\be_{\ig(M),\ig(N)}} \ar[rr] & & \ig(M*N) \ar[d]^{\ig(\be'_{M,N})} \\
  \ig(N) * \ig(M) \ar[rr] & & \ig(N*M)
   }
\]
commutes, where $\be$ $($respectively, $\be'${}$)$ is the braiding on $\sD_G(G)$ $($respectively, on $\sD_{G'}(G')${}$)$ constructed in Definition \ref{d:braiding-equivariant-derived}, and the horizontal arrows come from the weak semigroupal structure on $\ig$ constructed in Lemma \ref{lemma5}(a).
 \sbr
\item For all $M\in\sD_{G'}(G')$, we have $\te_{\ig M}=\ig(\te'_M)$, where $\te$ $($respectively, $\te'${}$)$ is the twist on $\sD_G(G)$ $($respectively, on $\sD_{G'}(G')${}$)$ constructed in Definition \ref{d:twist-equivariant-derived}.
\end{enumerate}
\end{lem}

The lemma will be proved in Appendix~\ref{s:CFT} (see Corollary~\ref{c:what_we_struggled_for} and
\S\ref{sss:conclusion}).

\subsection{Heisenberg idempotents and the geometric Mackey
condition}\label{ss:Heis-idemp-geom-Mackey}
\begin{lem}\label{lemma7}
Let $G$ be a perfect unipotent group over $k$, let $(H,\cL)$ be an
admissible pair for $G$, and let $G'$ denote its normalizer in $G$.
Let $e'_{\cL}$ denote the Heisenberg minimal idempotent in
$\sD_{G'}(G')$ constructed in
\S\ref{ss:Heisenberg-minimal-idempotents} $($cf.~Lemma
\ref{lemma6}$)$. Then $e'_{\cL}$ satisfies the geometric Mackey
condition with respect to $G$.
\end{lem}

\begin{proof}
This is shown in \cite[\S9.5]{characters}.
\end{proof}

\subsection{Proof of Proposition
\ref{p:uniqueness}}\label{ss:proof-p:uniqueness} To prove the
existence of $A$ satisfying properties (a)--(c) of the proposition,
it suffices to show that if normal subgroups $A_1,A_2\subset G$
satisfy (a) and (b), then so does $A_1A_2$. To this end, it is
enough to check that the homomorphism
\begin{equation}\label{(1)}
(A_1A_2)^*\rar{}A_1^*\times A_2^*
\end{equation}
has finite kernel. But
\[
\Ker\bigl((A_1A_2)^*\rar{}A_1^*\bigr) = (A_1A_2/A_1)^*=
\bigl(A_2/(A_1\cap A_2)\bigr)^*
\]
by property \S\ref{ss:serre-duality-properties}(3), the homomorphism
\[
\bigl(A_2/(A_1\cap A_2)\bigr)^* \rar{} \bigl(A_2/(A_1\cap
A_2)^\circ\bigr)^*
\]
has finite kernel by \S\ref{ss:serre-duality-properties}(4), and the
homomorphism $\bigl(A_2/(A_1\cap A_2)^\circ\bigr)^*\rar{}A_2^*$ is
injective by \S\ref{ss:serre-duality-properties}(3), so \eqref{(1)}
has finite kernel.

\mbr

To finish the argument, we will prove the following
\begin{lem}\label{l:uniqueness-lemma}
Let $A\subset G$ be a connected normal subgroup and $\cN$ a
$G^\circ$-invariant multiplicative local system on $A$. Put
$G_1=N_G(\cN)$, so that $G^\circ\subset G_1\subset G$, and let
$e_1=\av_{G/G_1}(\cN\tens\bK_A)\in\sD_G(A)\subset\sD_G(G)$. The
following properties are equivalent:
\begin{enumerate}[$($i$)$]
\item $A\subset H$ and $\cL\bigl\lvert_A$ is $G$-conjugate to $\cN$;
 \sbr
\item $f*e_1\cong f$, where $f$ is constructed in Theorem
\ref{t:construction-L-packets}$($a$)$.
\end{enumerate}
\end{lem}

\begin{rem}\label{r:uniqueness-remark}
The lemma implies that the $G$-orbit of
$\bigl(A,\cL\bigl\lvert_A\bigr)$ depends only on $f$, and hence
completes the proof of Proposition \ref{p:uniqueness}. Indeed,
property (ii) of the lemma manifestly depends only on $f$. On the
other hand, if $g_1,\dotsc,g_n$ are representatives of the cosets of
$G_1$ in $G$ and $\cN^{g_j}$ are the corresponding conjugates of
$\cN$, then $e_1\cong\oplus_{j=1}^n (\cN^{g_j}\tens\bK_A)$ in
$\sD(A)$, which shows that the data of $e_1$ is equivalent to the
data of the $G$-orbit of $\bigl(A,\cL\bigl\lvert_A\bigr)$.
\end{rem}

\begin{proof}[Proof of Lemma \ref{l:uniqueness-lemma}]
If (i) holds, then $e'_{\cL}*e_1\cong e'_{\cL}$, as
$\cL\bigl\lvert_A\cong\cN^{g_j}$ for a unique $j$ with the notation
of Remark \ref{r:uniqueness-remark}. By definition, $f=\ig
e'_{\cL}=\avg(i_*e'_{\cL})$, where $i:G'\into G$ is the embedding.
So (ii) follows from Corollary \ref{c:averaging-module}.

\mbr

Now assume (ii). By Lemma \ref{lemma7}, $e'_{\cL}\in\sD_{G'}(G')$
satisfies the geometric Mackey condition with respect to $G$. The
argument that was used to prove Proposition \ref{p:Lemma} can be
repeated verbatim in this case, and it shows that $A\subset G'$ and
$e'_{\cL}*e_1\cong e'_{\cL}$. Property (i) follows.
\end{proof}

\subsection{Exactness properties of $\ig$ and $\Ig$}\label{ss:artin}
Let us recall a result of M.~Artin:
\begin{thm}[\cite{bbd}, Thm.~4.1.1]   \label{t:artin}
If $f:X\rar{}Y$ is an affine morphism of schemes of finite type over
$k$, the functor $f_*:\sD(X)\rar{}\sD(Y)$ is left exact with respect
to the perverse $t$-structures, i.e., takes ${}^p\sD^{\leq 0}(X)$
into ${}^p\sD^{\leq 0}(Y)$.
\end{thm}
\begin{cor}[\emph{op.~cit.}, Cor.~4.1.2]\label{c:artin}
Under the same assumptions, the functor $f_!:\sD(X)\rar{}\sD(Y)$ is
right $t$-exact, i.e., takes ${}^p\sD^{\geq 0}(X)$ into
${}^p\sD^{\geq 0}(Y)$.
\end{cor}
We also recall
\begin{prop}\label{p:pullback-perverse}
If $f:X\rar{}Y$ is a smooth morphism of $k$-schemes everywhere of
relative dimension $d$, then $f^*[d]$ takes $\Perv(Y)$ into
$\Perv(X)$.
\end{prop}

Using the construction of induction functors presented in
\S\ref{ss:induction-functors} together with Theorem \ref{t:artin},
Corollary \ref{c:artin} and Proposition \ref{p:pullback-perverse},
one obtains

\begin{lem}\label{l:ind-Ind-exactness}
Let $G$ be a perfect unipotent group over $k$, and let $G'\subset G$
be a closed subgroup. Then
\[
\ig \bigl( {}^p\sD^{\geq 0}_{G'}(G') \bigr) \subset {}^p\sD^{\geq
0}_G(G) \bigl[ -\dim(G/G') \bigr]
\]
and
\[
\Ig \bigl( {}^p\sD^{\leq 0}_{G'}(G') \bigr) \subset {}^p\sD^{\leq
0}_G(G) \bigl[ -\dim(G/G') \bigr].
\]
\end{lem}

\subsection{The key proposition}\label{ss:key-proposition} The next
proposition will be proved in \S\ref{ss:proof-p:crucial}.
\begin{prop}\label{p:crucial}
Let $G$ be a perfect unipotent group over $k$. For every nonzero
$N\in\sD(G)$, there exists a closed idempotent $f\in\sD_G(G)$ such
that $f*N\neq 0$ and $f\cong\ind_{G'}^Ge'$, where $G'\subset G$ is
the normalizer of some admissible pair $(H,\cL )$ and $e'$ is the
Heisenberg minimal idempotent $e'_{\cL}\in\sD_{G'}(G')$
corresponding to $(H,\cL )$.
\end{prop}

\begin{rem}\label{r:crucial}
In the situation of Proposition \ref{p:crucial}, $f$ is
automatically minimal as a weak idempotent in $\sD_G(G)$. This
follows from Lemmas \ref{l:heisenberg-minimal-idempotent},
\ref{lemma7} and \ref{lemma1}.
\end{rem}

\subsection{Proof of Theorem
\ref{t:idempotents}}\label{ss:proof-t:idempotents} Let $N\in\sD(G)$
be nonzero, and let $f\in\sD_G(G)$ satisfy the conclusion of
Proposition \ref{p:crucial}. By Remark \ref{r:crucial}, $f$ is
minimal as a weak idempotent in $\sD_G(G)$. \emph{A fortiori}, we
see that:
\begin{itemize}
\item $f$ is minimal as a closed idempotent, which yields Theorem
\ref{t:idempotents}(c), and
 \sbr
\item $\sD_G(G)$ is a Jacobson monoidal category (Definition
\ref{d:jacobson-monoidal-categories}), so parts (a) and (b) of
Theorem \ref{t:idempotents} follow from Proposition
\ref{p:jacobson1}.
\end{itemize}

\subsection{Proof of Theorem
\ref{t:geometric-mackey}}\label{ss:proof-t:geometric-mackey} (a) By
assumption, $e$ is a minimal closed idempotent in $\sD_{G'}(G')$. By
Theorem \ref{t:idempotents}(a), $e$ is a minimal weak idempotent in
$\sD_{G'}(G')$. By Lemma \ref{lemma1}, $f=\ig e$ is a minimal weak
idempotent in $\sD_G(G)$. By Theorem \ref{t:idempotents}(b), $f$ is
a minimal closed idempotent in $\sD_G(G)$.

\mbr

\noindent (c) Since $f$ is a closed idempotent by part (a), the assertion follows
from Corollary \ref{c:ind-monoidal-equivalence} and Lemma \ref{l:compatibility-ind-braiding-twists}.

\mbr

\noindent (d) Use the fact that $f$ is closed and the implication
(ii)$\Rightarrow$(i) of Proposition \ref{p:when-ind-closed}.

\mbr

\noindent (e) Combine Lemma \ref{l:ind-Ind-exactness} with assertion
(d) of the theorem.

\mbr

\noindent (b) This follows from part (e).

\subsection{Proof of Theorem \ref{t:construction-L-packets}}\label{ss:proof-t:construction-L-packets}
By Lemma \ref{l:heisenberg-minimal-idempotent},
$e'_{\cL}\in\sD_{G'}(G')$ is a minimal closed idempotent. By Lemma
\ref{lemma7}, it satisfies the geometric Mackey condition with
respect to $G$. So Theorem \ref{t:construction-L-packets}(a) follows
from Theorem \ref{t:geometric-mackey}(a).

\mbr

The equality $n_{e'_{\cL}}=\dim H$ follows from the definition of
$e'_{\cL}$. Using this equality and Theorem
\ref{t:geometric-mackey}(b), we get Theorem
\ref{t:construction-L-packets}(b).

\mbr

Let $N\in\sD_G(G)$ be a minimal closed idempotent. Let $f$ be as in
Proposition \ref{p:crucial} (i.e., $f*N\neq 0$ and $f$ is obtained
from an admissible pair by induction with compact supports). To
prove Theorem \ref{t:construction-L-packets}(c) for $N$, it suffices
to show that $N\cong f$. This is clear because $f*N\neq 0$, and both
$N$ and $f$ are minimal closed idempotents in $\sD_G(G)$ (for $f$
this follows from assertion (a), proved above).

\subsection{Proof of Theorem \ref{t:Heisenberg-L-packets}}  \label{ss:proof-t:Heisenberg-L-packets}
Theorem \ref{t:properties-character-sheaves} for $(G',e'_{\cL})$ in place of $(G,e)$ was
proved by T.~Deshpande, see  \cite[Theorems~1.1-1.5]{tanmay}. To prove
Theorem \ref{t:functional-dimension} for $(G',e'_{\cL})$, we need some notation.

\mbr

Let $(H,\cL)$ be the admissible pair for $G'$ that gives rise to the minimal Heisenberg idempotent
$e'_{\cL}$. We can view $e'_{\cL}$ as a closed idempotent either in $\sD_{G'}(G')$ or in
$\sD_{G'^\circ}(G'^\circ)$. Accordingly, we have modular categories
\[
\sM:=\{M\in e'_{\cL}\sD_{G'}(G')\, \big\vert\, M[-\dim H] \mbox{ is perverse}\},
\]

\[
\sM_1:=\{M\in e'_{\cL}\sD_{G'^\circ}(G'^\circ)\, \big\vert\, M[-\dim H] \mbox{ is perverse}\}.
\]
Let $\tau^+(\sM )$ and $\tau^+(\sM_1 )$ be their Gauss sums. S.~Datta \cite[\S2.4]{swarnendu} proved
Theorem~\ref{t:functional-dimension} for $(G'^\circ ,e'_{\cL})$ in place of $(G,e)$. So to prove
Theorem~\ref{t:functional-dimension} for $(G' ,e'_{\cL})$ it suffices to show that $\tau^+(\sM )/\tau^+(\sM_1 )$
is a power of $p$. In fact, we will prove that
\begin{equation} \label{e:Gauss_relation}
\tau^+(\sM_1 )=\tau^+(\sM )/|\Ga |\, ,
\end{equation}
where $\Ga:=\pi_0(G')=G'/G'^\circ$ is a $p$-group by unipotence of $G'$.

\mbr

To this end, consider the full subcategory $\cE:=\{M\in \sM \, \big\vert\, \Supp M\subset H\}$.
Note that $\cE$ is braided equivalent to the category of finite-dimensional representations of $\Ga$
and the twist on $\sM$ induces the trivial twist on $\cE$. In this situation we can apply
\cite[Thm.~6.16]{dgno}, which interprets the r.h.s. of \eqref{e:Gauss_relation} as the Gauss sum of a
certain modular category. The latter identifies with $\sM_1$ (to see this, combine \cite[Lem.~1.4]{tanmay},
\cite[Thm.~ 4.44]{dgno}, and \cite[Prop.~4.56(i)]{dgno}).

\subsection{Proof of Theorems
\ref{t:properties-character-sheaves} and
\ref{t:functional-dimension}}
\label{ss:proof-t:properties-character-sheaves}
\label{ss:proof-t:functional-dimension} Let $e\in\sD_G(G)$ be a
minimal closed idempotent. By Theorem
\ref{t:construction-L-packets}(c), we have $e\cong\ig e'_{\cL}$ for
some admissible pair $(H,\cL)$ for $G$, where $G'$ is the normalizer
of $(H,\cL)$ in $G$ and $e'_{\cL}\in\sD_{G'}(G')$ is the Heisenberg
minimal idempotent defined by $\cL$ (see
\S\ref{ss:Heisenberg-minimal-idempotents}).

\mbr

By Theorem \ref{t:geometric-mackey}(c), $\ig$ restricts to a
monoidal triangulated equivalence
\begin{equation}\label{(*)}
e'_{\cL}\sD_{G'}(G') \rar{\sim} e\sD_G(G),
\end{equation}
and by Theorem \ref{t:geometric-mackey}(e), the equivalence
\eqref{(*)} restricts to an equivalence between
$\cM_{e'_{\cL}}^{perv}$ and $\cM_e^{perv}[-\dim(G/G')]$. By Lemma \ref{l:compatibility-ind-braiding-twists}, the monoidal equivalence \eqref{(*)} is compatible with the canonical braidings (Definition \ref{d:braiding-equivariant-derived}) and twists (Definition \ref{d:twist-equivariant-derived}) on both categories. In
addition, by Theorem \ref{t:geometric-mackey}(b),  $n_e\le n_{e'_{\cL}}$ and
the functional dimensions $d_{e'_{\cL}}$ and $d_e$ differ by an integer.
Theorem~\ref{t:Heisenberg-L-packets} shows that
Theorems~\ref{t:properties-character-sheaves} and \ref{t:functional-dimension}
hold for the idempotent $e'_{\cL}\in\sD_{G'}(G')$. Hence
Theorems~\ref{t:properties-character-sheaves} and \ref{t:functional-dimension} also
hold for $e\in\sD_G(G)$ with a possible exception of the inequality
\begin{equation}  \label{e:nonnegativity}
n_e\ge 0,
\end{equation}
which is a part of Theorem~\ref{t:properties-character-sheaves}(b).

\mbr

Now let us prove \eqref{e:nonnegativity}. Recall that $e[-n_e]$ is perverse by the definition of
$n_e$. So
\begin{equation}  \label{e:Ext-vanishing}
\Ext^i(\e ,e[-n_e])=0 \mbox{ for }i<0
\end{equation}
(this follows from the definition of a perverse sheaf and the fact
that $\e$ is a delta-sheaf). On the other hand, an idempotent arrow
$\e\rar{}e$ is a nonzero element of $\Hom (\e ,e)$, so
$\Ext^{n_e}(\e ,e[-n_e])=\Hom (\e ,e)\ne 0$. Comparing this with
\eqref{e:Ext-vanishing}, we get \eqref{e:nonnegativity}.

\subsection{Proof of Proposition
\ref{p:duality-canonical}}\label{ss:proof-p:duality-canonical}
We begin by proving the proposition in the case where $H=G$. In this case $n_f=n_{e_{\cL}'}=\dim H=\dim G$ by Theorem \ref{t:construction-L-packets}(b). Moreover, $f=\cL\tens\bK_G$, where $\cL$ is a multiplicative local system on $G$ and $\bK_G$ is the dualizing complex of $G$. Since $G$ is smooth of dimension $n_f$, 
there is a canonical identification $\bK_G\rar{\simeq}\ql [2n_f](n_f)$ and therefore
\begin{equation}   \label{e:dualizingcomplex}
f\rar{\simeq}\cL[2n_f](n_f).
\end{equation}
By definition, $\bD_G^-=\bD_G\circ\iota^*$, where $\iota:G\rar{}G$ is given by $g\mapsto g^{-1}$.
Since $\cL$ is multiplicative $\iota^*\cL=\cL^{-1}$, so
$\bD_G^-f=\bD_G(\cL^{-1}\tens\bK_G)=\cL\tens\bD_G(\bK_G)=\cL$. Combining this with
\eqref{e:dualizingcomplex} one gets a canonical isomorphism
$\bD_G^-f\rar{\simeq}f[-2n_f](-n_f)$, completing the proof of the proposition when $H=G$.

%
%

\mbr

Next we treat the general case. Writing $d=\dim(G/G')$, we recall that $n_f=n_{e'_{\cL}}-d$ by Theorem \ref{t:construction-L-packets}(b). The first part of the proof yields a natural isomorphism
\[
\bD_{G'}^- e'_{\cL} \rar{\simeq} e'_{\cL}[-2n_{e'_{\cL}}](-n_{e'_{\cL}}).
\]
Applying the functor $\bD_G^-\circ\Ig$, we obtain a natural isomorphism
\[
\bigl(\bD_G^-\circ\Ig\bigr) \bigl(e'_{\cL}[-2n_{e'_{\cL}}](-n_{e'_{\cL}})\bigr) \rar{\simeq} \bigl(\bD_G^-\circ\Ig\circ \bD_{G'}^-\bigr) (e'_{\cL}).
\]
Composing the latter with the inverse of the isomorphism \eqref{e:induction-duality-nat-isom} provided by Corollary \ref{c:induction-duality}, we obtain a natural isomorphism
\[
\bigl(\bD_G^-\circ\Ig\bigr) \bigl(e'_{\cL}[-2n_{e'_{\cL}}](-n_{e'_{\cL}})\bigr) \rar{\simeq} \ig(e'_{\cL})[2d](d),
\]
which is the same thing as an isomorphism
\[
\bigl(\bD_G^-\circ\Ig\bigr) (e'_{\cL}) \rar{\simeq} \ig(e'_{\cL})[-2n_f](-n_f).
\]
Finally, the natural morphism $f=\ig(e'_{\cL})\rar{}\Ig(e'_{\cL})$ is an isomorphism by Theorem \ref{t:geometric-mackey}(d) and Lemma \ref{lemma7}. This yields a natural isomorphism
\[
\bD^-_G f \rar{\simeq} f[-2n_f](-n_f),
\]
as desired.

\subsection{Proof of Proposition \ref{p:duality}}\label{ss:proof-p:duality}
By Theorem \ref{t:construction-L-packets}(c), every minimal closed
idempotent $e\in\sD_G(G)$ arises from some admissible pair for $G$. In
view of Remark \ref{r:duality}(i), we see that Proposition
\ref{p:duality}(a) follows from Proposition
\ref{p:duality-canonical}. Now we will prove Proposition \ref{p:duality}(b) using the
language of Grothendieck-Verdier categories (see Definitions~\ref{def:dualizing}, \ref{def:GV}, and
\ref{def:r-category} from Appendix~\ref{s:dualityformalism}).

\mbr

By Example \ref{example:dual-gen4}, $\sD_G(G)$ is an r-category, where the duality functor is the functor $\bD_G^-$ from Definition \ref{d:duality}. In particular, $\sD_G(G)$ is a Grothendieck-Verdier category.

\mbr

By Lemma \ref{l:GV-Hecke-subcategory}, $e\sD_G(G)$ is a Grothendieck-Verdier category with dualizing object $\bD_G^- e$, and the corresponding duality functor can be identified with $\bD_G^-$.
By part (a) of Proposition \ref{p:duality}, we have
$\bD_G^- e \cong e[-2n_e](-n_e)\tens L_e$ for a certain line $L_e$ over $\ql$. In particular, $\bD_G^- e$ is an invertible object of the monoidal category $e\sD_G(G)$, with inverse $e[2n_e](n_e)\tens L_e^{-1}$. Hence $e$ is also a dualizing object of $e\sD_G(G)$, and the duality functor that $e$ defines is given by $M\longmapsto (\bD_G^- M)[2n_e](n_e)\tens L_e^{-1}$. This implies Proposition \ref{p:duality}(b) since we already saw in Theorem
\ref{t:properties-character-sheaves}(c) that $e\sD_G(G)$ is a rigid
monoidal category with unit object $e$.

\bbr

\emph{The rest of this section is devoted to a proof of Proposition
\ref{p:crucial}. It will be given in
\S\S\ref{sss:crucial-reduction}--\ref{sss:crucial-equivariant} after
various preliminaries.}

\subsection{Closed idempotents via averaging}\label{ss:avg-closed-idemp} Let $H$ be a connected
perfect unipotent group over $k$, and let $G$ be a perfect unipotent
group acting on $H$ by group automorphisms. Fix a multiplicative
local system $\cL$ on $H$, and let $G'\subset G$ be the stabilizer
of the corresponding point of $H^*(k)$. Equip the categories
$\sD_{G'}(H)$ and $\sD_G(H)$ with the monoidal structure given by
convolution with compact supports.
%
%
Since $G'$ stabilizes the (isomorphism class of) $\cL$, both $\cL$ and $e_{\cL}=\cL\tens\bK_H$ can be viewed as objects of $\sD_{G'}(H)$.

\begin{lem}
$e_{\cL}$ is a closed idempotent in $\sD_{G'}(H)$.
\end{lem}

\begin{proof}
Let $s:H\rar{}\Spec(k)$ and $1:\Spec(k)\rar{}H$ denote the structure morphism and the identity of $H$, respectively. Then $\bK_H\cong s^!\ql$ and $\e\cong 1_!\ql$ (the unit object of $\sD_{G'}(H)$). The natural isomorphism $\ql\rar{\simeq}1^!\bK_H$ induces by adjunction a morphism $\e\rar{}\bK_H$. On the other hand, since $\cL^\vee$ is a multiplicative local system, its fiber at the identity element of $H(k)$ has a canonical trivialization, which yields an isomorphism $\e\tens\cL^\vee\rar{\simeq}\e$. The composition $\e\tens\cL^\vee\rar{}\bK_H$ corresponds to a morphism $\e\rar{}e_{\cL}$. By \cite[\S8.3]{characters} this morphism becomes an isomorphism in $\sD(H)$, and hence also in $\sD_{G'}(H)$, after convolving with $e_{\cL}$.
\end{proof}

Recall that the averaging functor $\avg:\sD_{G'}(H)\rar{}\sD_G(H)$ was defined in \S\ref{ss:induction-functors}. The next result is used in the proof of Lemma \ref{l:crucial}.

\begin{prop}\label{p:avg-closed-idemp}
The object $\avg(e_{\cL})\in\sD_G(H)$ is a closed idempotent.
\end{prop}
\begin{proof}
In the proof we will use the functors
\[
\cF'^{G'}:\sD_{G'}(H^*)\rar{}\sD_{G'}(H) \qquad\text{and}\qquad
\cF'^{G}: \sD_{G}(H^*)\rar{}\sD_G(H)
\]
(see \S\ref{ss:f-d-compatibility-averaging}) together with
Proposition \ref{p:f-d-compatibility-averaging} and Corollary
\ref{c:FD-closed-idempotent}.

\mbr

Let $\de_\cL\in\sD_{G'}(H^*)$ denote the delta-sheaf at the point
$[\cL]\in H^*(k)$, let $\cO_\cL$ be the $G$-orbit of $[\cL]$, and
let $i:\cO_{\cL}\into H^*$ denote the inclusion morphism. Since $G$
is unipotent and $H^*$ is affine, $i$ is closed. Moreover,
$\avg(\de_\cL)\cong i_!(\ql)_{\cO_\cL}$ and
$e_\cL\cong\cF'^{G'}(\de_\cL)$, whence by Proposition
\ref{p:f-d-compatibility-averaging}, we have
\[
\avg(e_\cL)\cong \avg(\cF'^{G'}(\de_{\cL})) \cong
\cF'^G(\avg(\de_\cL))\cong \cF'^G\bigl( i_!(\ql)_{\cO_\cL} \bigr).
\]
Now Corollary \ref{c:FD-closed-idempotent} yields an idempotent
arrow
\[
\cF'^G(\pi_{\cO_\cL})\circ\pi_0 : \e\rar{}\cF'^G\bigl(
i_!(\ql)_{\cO_\cL} \bigr)
\]
in $\sD(H)$, and we only need to check that this arrow is a morphism
in $\sD_G(H)$. But
$\pi_{\cO_\cL}:(\ql)_{H^*}\rar{}i_!(\ql)_{\cO_\cL}$ is clearly
$G$-equivariant, and $\pi_0$ is $G$-equivariant because its
construction is canonical. This completes the proof.
\end{proof}

\subsection{Proof of Proposition \ref{p:crucial}}\label{ss:proof-p:crucial}
The argument has three stages. First we will prove an auxiliary
Lemma \ref{l:crucial} (which does not involve the object $N$). Next
we will reduce the proposition to the case where $N$ is
$G$-equivariant, using Corollary \ref{c:averaging-module}. Finally,
we will prove the proposition in this special case with the aid of
Lemma \ref{l:crucial} and Proposition \ref{p:Lemma}.

\subsubsection{An auxiliary lemma} Let us recall from \S\ref{ss:reduction-formulation} that
$\sP_{norm}(G)$ denotes the set of pairs $(A,\cN)$, where $A$ is a
normal connected subgroup of $G$ and $\cN$ is a multiplicative local
system on $A$.

\begin{lem}\label{l:crucial}
Let $(A,\cN)\in\sP_{norm}(G)$. Denote by $G_1$ the normalizer of
$\cN$ in $G$, consider the closed idempotent
$\cN\tens\bK_A\in\sD_A(A)$ determined by $\cN$
$($cf.~\S\ref{ss:Heisenberg-minimal-idempotents}$)$, and let
$e_1\in\sD_{G_1}(G_1)$ be its extension by zero\footnote{$e_1$ has a
canonical $G_1$-equivariant structure because $\cN$ is
$G_1$-invariant.} to $G_1$.

\begin{enumerate}[$($a$)$]
\item $e_1$ is a closed idempotent in $\sD_{G_1}(G_1)$.
 \sbr
\item $e_1$ satisfies the geometric Mackey condition with
respect to $G$.
 \sbr
\item Let $e_{\cO}=\igo e_1$. Then $e_{\cO}$ is a closed
idempotent in $\sD_G(G)$.
 \sbr
 \item The semigroupal categories $e_1\sD_{G_1}(G_1)$ and $e_{\cO}\sD_G(G)$ are mo\-noidal.
 \sbr
\item The functor $\igo$ restricts to a monoidal equivalence
$e_1\sD_{G_1}(G_1)\rar{\sim}e_{\cO}\sD_G(G)$. \sbr
\item $e_1*e_{\cO}\cong e_1$.
\end{enumerate}
\end{lem}
\begin{proof}
(a) We have a canonical idempotent arrow $\e\rar{}\cN\tens\bK_A$ in
$\sD_A(A)$, and it is clear that it is in fact $G_1$-equivariant.
Thus $\cN\tens\bK_A$ is also a closed idempotent in $\sD_{G_1}(A)$,
which proves (a).

\mbr

\noindent (b) We must prove that
$\overline{e_1}*\de_x*\overline{e_1}=0$ for all $x\in G(k)$ such
that $x\not\in G_1(k)$. The last identity is equivalent to
$\overline{e_1}*\de_x*\overline{e_1}*\de_{x^{-1}}=0$. But
$\de_x*\overline{e_1}*\de_{x^{-1}}$ is the extension by zero of the
object $\cN^x\tens\bK_A\in\sD_A(A)$, where $\cN^x$ denotes the
pullback of $\cN$ by the automorphism $a\mapsto x^{-1}ax$ of $A$.
Since $x\not\in G_1(k)$, we have $\cN^x\not\cong\cN$, whence
$(\cN\tens\bK_A)*(\cN^x\tens\bK_A)=0$. This implies that
$\overline{e_1}*\de_x*\overline{e_1}*\de_{x^{-1}}=0$.

\mbr

\noindent (c) With the notation of Proposition
\ref{p:avg-closed-idemp}, $e_{\cO}$ is the extension by zero of the
object $\av_{G/G_1}(e_{\cN})\in\sD_G(A)$. Applying Proposition
\ref{p:avg-closed-idemp} to the conjugation action of $G$ on $A$
yields (c). 

\mbr

Statements (d) and (e) follow from (a)--(c) and
Corollary~\ref{c:ind-monoidal-equivalence}. Finally, statement (f)
follows from the last assertion of Lemma \ref{lemma4}.
\end{proof}

\subsubsection{Reduction of Proposition \ref{p:crucial} to the equivariant
case}\label{sss:crucial-reduction} Let us assume for the moment that
the proposition holds for any $0\neq N\in\sD_G(G)$. We will explain
how to prove it in full generality.

\mbr

Let $N\in\sD(G)$ be nonzero. It suffices to prove Proposition
\ref{p:crucial} assuming that the stalk of $N$ at $1\in G$ is
nonzero (otherwise we can replace $N$ with its right translation by
an appropriate $g\in G(k)$). Define $\Nt\in\sD_G(G)$ by $\Nt:=\av_G
(N)$, where $\av_G : \sD (G)\rar{}\sD_G(G)$ is the functor of
averaging with compact support (see \eqref{e:!averaging} and
Definition~\ref{d:averaging}). Then $\Nt\ne 0$ (because the stalk of
$\Nt$ at $1\in G$ is nonzero). So by assumption, $\Nt*f\ne 0$ for
some closed idempotent $f\in\sD_G(G)$ satisfying the last
requirement of Proposition \ref{p:crucial}. Since
$\Nt*f\cong\av_G(N*f)$ by Corollary \ref{c:averaging-module}, we see
that $N*f\ne 0$.

\subsubsection{Proof of Proposition \ref{p:crucial} in the
equivariant case}\label{sss:crucial-equivariant} We now fix a
nonzero $N\in\sD_G(G)$ and complete the proof of Proposition
\ref{p:crucial} in this special case.

\mbr

Let $(A,\cN)$ be maximal among all pairs in $\sP_{norm}(G)$ that are
compatible with $N$ in the sense of Definition~\ref{d:compatible}.
Let $G_1$, $e_1$, and $e_{\cO}$ have the same meaning as in
Lemma~\ref{l:crucial}. Compatibility of $(A,\cN)$ with $N$ means
that
\begin{equation} \label{e:e_1Nnonzero}
e_1*N\ne 0,
\end{equation}
where $e_1$ is viewed as an object of $\sD_{G_1}(G)$.

\mbr

We consider two cases. If $G_1=G$, then $(A,\cN)$ is an admissible
pair for $G$ by Proposition \ref{p:normal-admissible-pairs}. Thus
the idempotent $f:=e_1$ has the properties required in the
formulation of Proposition \ref{p:crucial}.

\mbr

Now assume that $G_1\subsetneq G$. By Lemma~\ref{l:crucial}(e),
$e_{\cO}*N=\igo N_1$ for some $N_1\in e_1\sD_{G_1}(G_1)$. Combining
\eqref{e:e_1Nnonzero} with Lemma~\ref{l:crucial}(f) we see that
$e_{\cO}*N\ne 0$ and therefore $N_1\ne 0$. Since $G_1\neq G$ we may
assume by induction that there exists a closed idempotent
$f_1\in\sD_{G_1}(G_1)$ that satisfies
\begin{equation} \label{e:f_1N_1nonzero}
f_1*N_1\neq 0
\end{equation}
and has the form $f_1\cong\ind_{G'}^{G_1}e'_{\cL}$ for some
admissible pair $(H,\cL)$ for $G_1$, where $G'$ is the normalizer of
$(H,\cL)$ in $G_1$ and $e'_{\cL}\in\sD_{G'}(G')$ is the Heisenberg
minimal idempotent corresponding to $(H,\cL)$.

\mbr

Since $N_1\in e_1\sD_{G_1}(G_1)$ it follows from
\eqref{e:f_1N_1nonzero} that $f_1*e_1\ne 0$. By Remark
\ref{r:crucial}, $f_1$ is minimal as a weak idempotent in
$\sD_{G_1}(G_1)$, so
\begin{equation} \label{e:f_1e_1}
f_1*e_1\cong f_1,
\end{equation}
i.e., $f_1\in e_1\sD_{G_1}(G_1)$.

 \mbr

Define $f\in e_{\cO}\sD_G(G)$  by $f:=\igo f_1$. Let us show that
$f$ has the properties required in the formulation of Proposition
\ref{p:crucial}. By Lemma~\ref{l:crucial}(e), $f\in e_{\cO}\sD_G(G)$
is a closed idempotent, so by Lemma~\ref{l:crucial}(c), $f$ is also
a closed idempotent in $\sD_G(G)$. To show that $f*N\neq 0$, note
that since $f\in e_{\cO}\sD_G(G)$ one has \[f*N\cong
f*e_{\cO}*N\cong (\igo f_1)*(\igo N_1),\] and $(\igo f_1)*(\igo
N_1)\ne 0$ by \eqref{e:f_1N_1nonzero} and Lemma~\ref{l:crucial}(e).

\mbr

It remains to check that  $f$ satisfies the last requirement of
Proposition \ref{p:crucial}. By \eqref{e:f_1e_1} and Lemma
\ref{lemma7}, we can apply Proposition \ref{p:Lemma}, replacing
$G,f,e$ with $G_1,f_1,e'_{\cL}$, respectively. We deduce that
$G'\supset A$ and $e'_{\cL}*e_1\cong e'_{\cL}$, which implies that
$(A,\cN)\leq(H,\cL)$ with respect to the partial order introduced in
Definition \ref{d:pairs-order}. Lemma \ref{l:admissibility} shows
that $(H,\cL)$ is admissible for $G$, and that $G'$ is equal to the
normalizer of $(H,\cL)$ in $G$.


\appendix

\section{Grothendieck-Verdier categories and r-categories}  \label{s:dualityformalism}


Let $G$ be an algebraic group. The monoidal categories $(\sD(G),*)$ and $(\sD_G(G),*)$ are
usually not rigid, but they have a weaker type of duality, which goes back to Grothendieck
and Verdier. In this appendix we give an axiomatic treatment of the Grothendieck-Verdier
formalism in monoidal categories. A more complete exposition of the subject can be found in
\cite{GV}.

\mbr

Throughout this appendix, with the exception of \S\ref{ss:ribbon_G}, we interpret $\sD_G(G)$ as the bounded derived category $D^b_c\bigl((\Ad G)\backslash G,\ql\bigr)$ of the stack quotient of $G$ with respect to its conjugation action on itself (cf.~\cite{Las-Ols06}).


\subsection{Grothendieck-Verdier categories and r-categories}
\label{ss:dual-general}

\subsubsection{Definitions and examples}

\begin{defin}\label{d:dualizing-functor}
Let $\sC$ be a category and $\Phi:\sC\times\sC\rar{}\cS{}ets$ a functor, which is contravariant in both arguments. We say that $\Phi$ is a \emph{dualizing functor} if
for every $Y\in\sC$ the functor $X\mapsto\Phi(X,Y)$ is representable by some object $DY\in\sC$ and the contravariant functor $D:\sC\rar{}\sC$ is an antiequivalence. $D$ is called the \emph{duality functor} with respect to $\Phi$.
\end{defin}

\begin{defin}     \label{def:dualizing}
An object $K$ in a  monoidal category $\cM$ is said to be {\em dualizing\,} if the functor $\Phi(X,Y)=\Hom(X\tens Y,K)$ is dualizing. The corresponding duality functor is called the
\emph{duality functor with respect to $K$}.
\end{defin}

\begin{rem}
One can show that if a dualizing object exists then it is unique up to tensoring by an
invertible object, see \cite[Proposition 1.3(i)]{GV} for more details.
\end{rem}

\begin{defin}      \label{def:GV}
A {\em Grothendieck-Verdier category\,} is a pair $(\cM ,K)$, where $\cM$ is a monoidal category
and $K\in\cM$ is a dualizing object.
\end{defin}

By abuse of language, we will usually say ``Grothendieck-Verdier category $\cM$" instead of
``Grothendieck-Verdier category $(\cM ,K)$".

\mbr

Below we give some examples of Grothendieck-Verdier categories. More examples of
such categories can be found in \cite{GV} and in the works by M.~Barr, who studied them under the name of {\em $\ast$-autonomous categories\,}  (e.g., see  \cite{Barr79,Barr95,Barr96,Barr99}).


\begin{example}   \label{example:GVoriginal}
Let $\cM=(\sD (X),\otimes )$, where $X$ is a scheme of finite type over a field $k$
and $\sD(X)$ is the bounded derived category of constructible $\ell$-adic
sheaves on $X$, $\ell\ne$ char $k$.\footnote{For a general field $k$, the definition of $\sD(X)$ is
given in  \cite{Jannsen,Ekedahl}. If $k$ is algebraically closed it is equivalent to the definition
from \cite[\S\S1.1.2--1.1.3]{deligne-weil-2}.}
Let $K_X\in\sD (X)$ be the dualizing complex. Then $(\cM ,K_X)$ is a Grothendieck-Verdier category. In this case $D$ is the usual Verdier duality functor $ \bD_X$.
\end{example}

\begin{defin}\label{def:r-category}
A monoidal category $\cM$ is said to be an {\em r-category\,} if
the unit object $\e\in\cM$ is dualizing.
\end{defin}

So any r-category can be considered as a Grothendieck-Verdier category with $K=\e$.
The letter `r' in the name ``r-category" is related to the words ``rigid'' and ``regular'', see
Examples~\ref{example:dual-gen1}-\ref{example:dual-gen2} below.


\begin{example} \label{example:dual-gen1}
Any rigid monoidal category is an  r-category. The next example shows that the converse is false.
\end{example}

\begin{example}   \label{example:dual-gen2}
Let $X$ be a smooth $k$-scheme (or if you wish, a regular scheme of finite type over $k$).
Suppose that $X$ has pure dimension $d$. Then the monoidal category
$(\sD (X),\otimes )$ is an r-category and $D: \sD (X)\to\sD (X)$ is the functor
$N\mapsto( \bD_X  N)[-2d](-d)$. If $d>0$ then $(\sD (X),\otimes )$ is not rigid because
$D(M_1\tens M_2)\not\simeq D(M_2)\tens D(M_1)$ for some $M_1,M_2\in\sD (X)$.
E.g., take $M_1=M_2=i_*\ql$, where $i:\Spec k\into X$ is a point;  then
$D(M_1\tens M_2)= D(i_*\ql)=i_*\ql [-2d](-d)$ while $D(M_2)\tens D(M_1)=i_*\ql [-4d](-2d)$.
\end{example}


\begin{example}     \label{example:dual-gen4}
Let $G$ be any algebraic group (not necessarily unipotent or even affine) over a field $k$.
By Lemma \ref{l:duality-on-D(G)} below, the monoidal categories $\sD (G)$ and $\sD_G(G)$
equipped with the functor of convolution with compact support (see Definition~\ref{d:convolution}) are   r-categories with $D$ being the functor $\bD_G ^-$ from Definition~\ref{d:duality}. One can show that these r-categories are rigid if and only if $G$ is proper, see \cite[Corollary 3.8]{GV}.
\end{example}

\begin{lem}\label{l:duality-on-D(G)}
Let $\cM$ denote either $(\sD(G),*)$ or $(\sD_G(G),*)$. There is a family of isomorphisms $\Hom(M*N,\e)\cong\Hom(M,\bD_G^-N)$, functorial in $M,N\in\cM$.
\end{lem}

\begin{proof}
By Example \ref{example:GVoriginal}, there are canonical isomorphisms
\[
\Hom(M,\bD_G^-N) \cong \Hom(M,\bD_G(\iota^*N)) \cong \Hom(M\tens\iota^*N,\bK_G)
\]
for all $M,N\in\cM$, where $\iota:G\rar{}G$ is given by $g\mapsto g^{-1}$. Hence we need to identify $\Hom(M\tens\iota^*N,\bK_G)$ with $\Hom(M*N,\e)$.

\mbr

Let $p:G\rar{}\Spec k$ denote the structure map, and let $1:\Spec k\rar{}G$ denote the unit of $G$. Adjunction yields functorial isomorphisms
\[
\Hom(M*N,\e)\cong\Hom(1^*(M*N),\ql)
\]
for all $M,N\in\cM$, and the proper base change theorem identifies $1^*(M*N)$ with $p_!(M\tens\iota^*N)$. Using adjunction again, we get isomorphisms
\[
\Hom(p_!(M\tens\iota^*N),\ql)\cong\Hom(M\tens\iota^*N,p^!\ql)=\Hom(M\tens\iota^*N,\bK_G)
\]
functorial in $M,N\in\cM$, completing the proof.
\end{proof}

\begin{example}     \label{example:dual-gen5}
Here  is a generalization of the previous example.
Suppose we have a groupoid in the category of schemes of finite type over a field $k$. Let $X$ denote its ``scheme of objects'' and $\Ga$ its ``scheme of morphisms".
Then $\sD(\Ga)$ has a natural structure of
Grothendieck-Verdier category, see \cite[Example 2.2]{GV} for details. If $X$ is a point
we get the Grothendieck-Verdier category $\sD(G)$ from Example~\ref{example:dual-gen4}.
On the other hand, one can take any $X$ and set $\Gamma=X\times X$.
\end{example}

One can get more examples of Grothendieck-Verdier categories by using
Lem\-ma~\ref{l:GV-Hecke-subcategory}(b) below.

\subsubsection{Some canonical isomorphisms} \label{sss:rules_of_operation}

\begin{rems}   \label{r:r}
\begin{enumerate}[(i)]
\item By definition, in any  Grothendieck-Verdier category $\cM$ one has an isomorphism
\begin{equation}  \label{e:duality1}
\Hom (X\otimes Y, K )\iso\Hom (X,DY)
\end{equation}
functorial in $X,Y\in\cM$. Since $D$ is an antiequivalence the right-hand side
of \eqref{e:duality1} identifies with $\Hom (Y,D^{-1}X)$. So one
also has an isomorphism
\begin{equation}  \label{e:duality2}
\Hom (X\otimes Y, K )\iso\Hom (Y,D^{-1}X)
\end{equation}
functorial in $X,Y\in\cM$. Thus a Grothendieck-Verdier category equipped with the opposite
tensor product is still a Grothendieck-Verdier category, {\em but $D$ gets replaced by $D^{-1}$.}
\item
By \eqref{e:duality2}, in any Grothendieck-Verdier category $\cM$ one has a functorial isomorphism
$\Hom (D^2Y\otimes X, K )\iso\Hom (X,DY)$.
Combining it with \eqref{e:duality1} one gets a functorial isomorphism
\begin{equation}  \label{e:duality3}
g:\Hom (X\otimes Y, K )\iso\Hom (D^2Y\otimes X, K ), \quad X,Y\in\cM .
\end{equation}
Equivalently, $g$ is characterized by the commutativity of the diagram
\begin{equation}   \label{e:duality4}
\xymatrix{
\Hom (X\otimes Y, K )  \ar[d] \ar[rr]^{g} & &
\Hom (D^2Y\otimes X, K )\ar[d]\\
 \Hom (X,DY) \ar[rr]^{D} & &
 \Hom (D^2Y,DX)}
\end{equation}
whose vertical arrows come from \eqref{e:duality1}.
\item
In any Grothendieck-Verdier category  there exist right and left
internal $\Hom$'s.
More precisely, if one sets
\begin{equation}   \label{e:internal_Hom2a}
\HOM(X,Z)=D^{-1}(DZ\tens X),
\end{equation}

\begin{equation}  \label{e:internal_Hom1a}
\HOM'(Y,Z)=D(Y\tens D^{-1}Z)
\end{equation}
then\eqref{e:duality1} and \eqref{e:duality2} yield
functorial isomorphisms
\begin{equation}   \label{e:internal_Hom2b}
\Hom (X\otimes Y,Z)\iso\Hom (Y, \HOM(X,Z)),
\end{equation}

\begin{equation}  \label{e:internal_Hom1b}
\Hom (X\otimes Y,Z)\iso\Hom (X, \HOM'(Y,Z)).
\end{equation}


\item From \eqref{e:duality1} and  \eqref{e:duality2} one gets canonical isomorphisms
\begin{equation}   \label{e:D1}
D\e\iso K, \quad\quad D^{-1}\e\iso K.
\end{equation}
and therefore canonical isomorphisms
\begin{equation}  \label{e:D^2(1)}
\e\iso D^2\e ,
\end{equation}
\begin{equation}  \label{e:D^2K}
K\iso D^2K
\end{equation}
(the 
latter is the composition $K\iso D\e\iso D^2D^{-1}\e\iso D^2K$).
\end{enumerate}
\end{rems}

\subsubsection{$D^2$ as a monoidal equivalence}  \label{sss:D^2as_mon}
By \eqref{e:duality3}, for each $X,Y_1,Y_2\in\cM$ one has a canonical isomorphism
\begin{equation}  \label{e:D^21}
\Hom (X\otimes Y_1\otimes Y_2, K )\iso\Hom (D^2(Y_1\otimes Y_2)\otimes X, K ).
\end{equation}
On the other hand, writing $X\otimes Y_1\otimes Y_2$ as $(X\otimes Y_1)\otimes Y_2$ and applying
\eqref{e:duality3} twice one gets an isomorphism
\begin{equation}   \label{e:D^22}
\Hom (X\otimes Y_1\otimes Y_2, K )\iso\Hom (D^2Y_1\otimes D^2Y_2\otimes X, K ).
\end{equation}
Combining \eqref{e:D^21} and \eqref{e:D^22} one gets a functorial isomorphism
\begin{equation} \label{e:D^23}
\Hom (D^2(Y_1\otimes Y_2)\otimes X, K )\iso\Hom (D^2Y_1\otimes D^2Y_2\otimes X, K ),
\quad X,Y_1,Y_2\in\cM .
\end{equation}
Using Yoneda's lemma and the isomorphism $\Hom (Z\tens X,K)\iso\Hom (Z,DX)$ we see that
the isomorphism~\eqref{e:D^23} comes from a unique functorial isomorphism
\begin{equation}   \label{e:D^24}
D^2(Y_1\otimes Y_2)\iso D^2Y_1\otimes D^2Y_2 , \quad Y_1,Y_2\in\cM .
\end{equation}

\begin{prop} \label{p:DD'}
The isomorphism \eqref{e:D^24} defines a monoidal structure on the functor $D^2:\cM\iso\cM$.
\end{prop}

A proof is given in \cite[\S11.1]{GV}. We do not use Proposition~\ref{p:DD'} in the main body of this article.

\subsection{Pivotal structures on Grothendieck-Verdier categories} \label{ss:pivotal}
\subsubsection{The notion of a pivotal structure}
\begin{defin}  \label{d:pivotal1}
A {\em pivotal structure\,} on a  Grothendieck-Verdier category $\cM$ is a functorial isomorphism
\begin{equation}   \label{e:psi1}
\psi_{X,Y}:\Hom (X\otimes Y,K )\iso\Hom (Y\otimes X,K ), \quad X,Y\in\cM
\end{equation}
such that
\begin{equation}   \label{e:pretty2}
\psi_{X\tens Y,Z}\circ \psi_{Y\tens Z,X}\circ\psi_{Z\tens X,Y}=\id , \quad X,Y\in\cM ;
\end{equation}
\begin{equation}   \label{e:pretty1}
\psi_{X,Y}\circ \psi_{Y,X}=\id , \quad X,Y\in\cM .
\end{equation}
\end{defin}

In particular, one has the notion of a pivotal structure on an r-category
(which can be considered as a Grothendieck-Verdier category with $K=\e$).

\begin{defin} \label{d:pivotal2}
A {\em pivotal Grothendieck-Verdier category\,}
is a Grothendieck-Verdier category with a pivotal structure. A {\em pivotal r-category\,}
is an  r-category with a pivotal structure.
\end{defin}

The name ``pivotal category" goes back to \cite[Definition 1.3]{FY}.

\begin{example}\label{ex:pivotal-D(X)-tensor-product}
A symmetric \GV category has a canonical pivotal structure: the isomorphisms
$\Hom(M\tens N,K)\rar{\simeq}\Hom(N\tens M,K)$ are induced by the symmetry isomorphisms
$M\tens N\rar{\simeq}N\tens M$. In particular, one thus gets a canonical pivotal structure on the
\GV category $(\sD(X),K_X)$ from Example~\ref{example:GVoriginal}.
\end{example}

\begin{lem}   \label{l:pivotal}
Let $\cM$ be a Grothendieck-Verdier category and $\psi$ an isomorphism \eqref{e:psi1}
satisfying \eqref{e:pretty2}. Then $\psi$ satisfies \eqref{e:pretty1} if and only if $\psi_{K,\e}=\id$.
\end{lem}

\begin{proof}
Setting $Z=\e$ in \eqref{e:pretty2} we see that \eqref{e:pretty1} holds  if and only if the isomorphism
$\psi_{X,\e}:\Hom (X,K )\to\Hom (X,K )$ equals the identity for all $X$. By Yoneda's lemma, this happens if and only if  $\psi_{K,\e}=\id$.
\end{proof}

\begin{cor}   \label{c:pivotal}
If $\cM$ is an r-category then \eqref{e:pretty2} implies \eqref{e:pretty1}.
\end{cor}

\begin{rem}   \label{r:pivotal1}
By \eqref{e:pretty2} and \eqref{e:pretty1}, a pivotal structure on a Grothendieck-Verdier category defines for any integers $n\ge m\ge 1$
a {\em canonical\,} isomorphism
\[
\Hom (X_1\tens\ldots\tens X_n,K)\iso\Hom (X_m\tens\ldots\tens X_n\tens X_1\tens\ldots \tens
X_{m-1},K),\;  X_i\in\cM.
\]
\end{rem}

\subsubsection{Pivotal structures and isomorphisms $\id\iso D^2$}
\begin{rem}        \label{r:pivotal2}
By \eqref{e:duality1}-\eqref{e:duality2} and Yoneda's lemma, a functorial isomorphism
\eqref{e:psi1} is the same as an isomorphism
$D^{-1}\iso D$ or equivalently an isomorphism $\id\iso D^2$.
\end{rem}

Remark~\ref{r:pivotal2} yields an injective map from the set of pivotal structures on a \GV category
$\cM$ to the set of isomorphisms $f:\id\iso D^2$.

\begin{prop}    \label{p:2pivotal}
An isomorphism $f:\id\iso D^2$ belongs to the image of this map if and only if it satisfies the
following conditions:
 \sbr
\begin{enumerate}[$($i$)$]
\item $f$ is monoidal;
 \sbr
\item $f_K:K\iso D^2K$ equals the isomorphism \eqref{e:D^2K}.
\end{enumerate}
\end{prop}

A proof is given in \cite[\S13]{GV}. We do not use
Proposition~\ref{p:2pivotal} in the main body of this article.


\begin{rems}        
\begin{enumerate}[(i)]
\item If $\cM$ is an r-category then condition (ii) from Proposition~\ref{p:2pivotal} clearly
follows from condition (i). For arbitrary \GV categories this is false, see \cite[Remark 5.8(i)]{GV}.
 \sbr
\item By the previous remark, in the case of r-categories a pivotal structure can
equivalently be defined to be a monoidal isomorphism $f:\id\iso D^2$. It is this
definition that was used in works on rigid monoidal categories
(e.g., see \cite[Definition 2.7]{ENO}).
 \sbr
\item Here is a way to combine the two conditions on $f$ from Proposition~\ref{p:2pivotal} into one. Let $\fA$ be the 2-groupoid of pairs consisting of a monoidal category and an object in it. A
Grothendieck-Verdier category $(\cM ,K)$ is an object in $\fA$. The monoidal structure on $D^2$
and the isomorphism $K\iso D^2(K)$ defined in Remark~\ref{r:r}(iv) allow us to
consider $D^2$ as a 1-automorphism of $(\cM ,K)\in\fA$. The two conditions on $f$ from
Proposition~\ref{p:2pivotal}  mean that $f:\id\iso D^2$ is a 2-isomorphism in $\fA$.
\end{enumerate}
\end{rems}



\subsubsection{The canonical pivotal structure on $\sD(G)$ and $\sD_G(G)$} \label{sss:pivotal_G}

\begin{example}   \label{example:pivotal}
We will write $\cM$ for one of the r-categories $\sD(G)$ and $\sD_G(G)$ (cf.~Example \ref{example:dual-gen4}). Let us give a description of $\Hom(M_1*\dotsb*M_n,\e)$, $M_i\in\cM$,
which makes the pivotal structure on $\cM$ obvious. First,
\begin{equation}  \label{e:ex_pivotal1}
\Hom(M_1*\dotsb*M_n,\e)=\Hom (1^* (M_1*\dotsb*M_n), \ql ),
\end{equation}
where $1:\Spec k\rar{}G$ is the unit of $G$ (of course, in the case $\cM=\sD_G(G)$ the right-hand side of \eqref{e:ex_pivotal1} should be understood as $\Hom$ in the category $\sD_G (\Spec k)$).
Now define $Z_n\subset G^n$ by the equation
$g_1\ldots g_n=1$ and let $\pi_1,\ldots ,\pi_n:Z_n\rar{}G$ be the projections. Then by proper base
change,
\begin{equation}  \label{e:ex_pivotal2}
1^* (M_1*\dotsb*M_n)=p_! (\pi_1^*M_1\tens\dotsb\tens\pi_n^*M_n), \quad\quad p:Z_n\rar{}\Spec k .
\end{equation}
Combining \eqref{e:ex_pivotal1} with \eqref{e:ex_pivotal2} and using the invariance of
$Z_n\subset G^n$ with respect to cyclic permutations of the $n$ coordinates we get
%
%
%
%
%
a canonical isomorphism
\[
\Hom(M_1*\dotsb*M_n,\e) \rar{\simeq} \Hom(M_2*\dotsb*M_n*M_1,\e)
\]
whose $n$-th power (in the obvious sense) equals the identity .
%
%
It is easy to see that the isomorphisms $\Hom(M_1*M_2,\e)\rar{\simeq}\Hom(M_2*M_1,\e)$ that we obtain in the case $n=2$ of this construction define a pivotal structure on the r-category $\cM$ (which is either $\sD(G)$ or $\sD_G(G)$).
\end{example}

\begin{rems}\label{r:pivotal-groupoids}
\begin{enumerate}[(i)]
\item The \GV category from Example \ref{example:dual-gen5} 
has a canonical pivotal structure (in the spirit of Example~\ref{example:pivotal}).
 \sbr
\item By Remark~\ref{r:pivotal2},  the pivotal structure on $\sD(G)$ (respectively, $\sD_G(G)$) from Example~\ref{example:pivotal}
yields an isomorphism $f:\Id\rar{\simeq} (\bD_G^-)^2$. Let us compute it.
\end{enumerate}
\end{rems}

\begin{lem}\label{l:obvious-isomorphism}
The isomorphism $f:\Id\rar{\simeq}(\bD_G^-)^2$ coming from the pivotal structure of Example \ref{example:pivotal} is equal to the composition
\[
\Id \rar{\simeq} \bD_G^2 \rar{\simeq} (\iota^*)^2\circ\bD_G^2 \rar{\simeq} \iota^*\circ\bD_G\circ\iota^*\circ\bD_G=(\bD_G^-)^2,
\]
where the first isomorphism is the standard one and the other two come from the natural identifications $(\iota^*)^2\cong\Id$ and $\bD_G\circ\iota^*\cong\iota^*\circ\bD_G$.
\end{lem}

This lemma will be deduced in \S\ref{sss:proof-l:obvious-isomorphism} from a more general Lemma \ref{l:obvious-isomorphism-abstract}.

\subsubsection{Quasi-pivotal structures}\label{sss:quasi-pivotal}
Let $\sC$ be a category and $\Phi:\sC\times\sC\rar{}\cS{}ets$ a dualizing functor
(see Definition~\ref{d:dualizing-functor}). Let $D$ be the duality functor with respect to $\Phi$, i.e.,
$\Hom (X,DY)=\Phi (X,Y)$ for $X,Y\in\sC$.
\begin{defin}
 A \emph{quasi-pivotal structure} on $\Phi$ is a functorial family of isomorphisms
\[
\psi_{X,Y}:\Phi(X,Y)\rar{\simeq}\Phi(Y,X), \qquad X,Y\in\sC.
\]
\end{defin}

\begin{rem}\label{r:quasi-pivotal}
If 
$\psi$ is a quasi-pivotal structure on $\Phi$, then for $X,Y\in\sC$ we obtain a functorial isomorphism
\[
\Hom(X,Y)\rar{D}\Hom(DY,DX)=
\Phi(DY,X)\xrar{\ \psi_{DY,X}\ }\Phi(X,DY)=
\Hom (X,D^2 Y)
\]
and hence a functorial isomorphism $Y\rar{\simeq}D^2 Y$. This defines a bijection between quasi-pivotal structures on $\Phi$ and isomorphisms of functors $\Id_{\sC}\rar{\simeq}D^2$.
\end{rem}

From now on we assume that we are given a triple $(\sC,\Phi,\psi)$, where $\sC$ is a category, $\Phi:\sC\times\sC\rar{}\cS{}ets$ is a dualizing functor and $\psi$ is a quasi-pivotal structure on $\Phi$. Suppose moreover that we are given an action of $\bZ/2\bZ$ on $(\sC,\Phi,\psi)$. 
We write $\tau:\sC\rar{\sim}\sC$ for the autoequivalence defined by $1\in\bZ/2\bZ$.

\begin{rems}\label{rems:quasi-pivotal}
\begin{enumerate}[(1)]
\item The functor $\Phi^\tau:(X,Y)\longmapsto\Phi(\tau X,Y)$ is also dualizing, and the corresponding duality functor is $\tau^{-1}\circ D\cong\tau\circ D$, where $D$ is the duality functor with respect to $\Phi$.
 \sbr
\item The functor $D$ is $(\bZ/2\bZ )$-equivariant, so there is a natural isomorphism
$$\tau\circ D\rar{\simeq}D\circ\tau\, .$$
 \sbr
\item The composition
    \[
    \psi^\tau_{X,Y} : \Phi(\tau X,Y) \rar{\simeq} \Phi(X,\tau Y) \rar{\simeq} \Phi(\tau Y,X), \qquad X,Y\in\sC,
    \]
    defines a quasi-pivotal structure on $\Phi^\tau$, where the first isomorphism comes from the $\bZ/2\bZ$-equivariant structure on $\Phi$ and the second one is $\psi_{X,\tau Y}\,$.
\end{enumerate}
\end{rems}

\begin{lem}\label{l:obvious-isomorphism-abstract}
The isomorphism $\Id_{\sC}\rar{\simeq}(\tau\circ D)^2$ coming from the quasi-pivotal structure described in Remark \ref{rems:quasi-pivotal}(3) via the construction of Remark \ref{r:quasi-pivotal} is equal to the composition
\[
\Id_{\sC}\rar{\simeq} D^2 \rar{\simeq} \tau^2\circ D^2 \rar{\simeq} (\tau\circ D)^2,
\]
where the first isomorphism corresponds to $\psi$ as in Remark \ref{r:quasi-pivotal}, the second one comes from the natural identification $\Id_{\sC}\cong\tau^2$, and the third one is induced by the isomorphism $\tau\circ D\rar{\simeq}D\circ\tau$ of Remark \ref{rems:quasi-pivotal}(2).
\end{lem}

The proof is completely straightforward, so we omit it.

\subsubsection{Proof of Lemma \ref{l:obvious-isomorphism}}\label{sss:proof-l:obvious-isomorphism} Let us specialize \S\ref{sss:quasi-pivotal} to the following setting. Take $\sC$ to be either $\sD(G)$ or $\sD_G(G)$, define $\Phi(M_1,M_2)=\Hom(M_1\tens M_2,K_G)$, where $K_G\in\sC$ is the dualizing complex, and let $\psi$ be induced by the standard symmetry isomorphism $M_1\tens M_2\rar{\simeq}M_2\tens M_1$. The action of $\bZ/2\bZ$ on the triple $(\sC,\Phi,\psi)$ comes from $\tau:=\iota^*$, where $\iota:G\rar{}G$ is the inversion map.

\mbr

We claim that the new duality functor $\Phi^\tau$ can be naturally identified with the functor $(M_1,M_2)\longmapsto\Hom(M_1*M_2,\e)$, so that $\psi^\tau$ becomes identified with the pivotal structure defined in Example \ref{example:pivotal}. Indeed, with the notation of Example \ref{example:pivotal}, we can identify $G$ with $Z_2\subset G\times G$ via the map $g\mapsto(g^{-1},g)$. Under this identification, $\pi_1$ becomes $\iota$ and $\pi_2$ becomes the identity map on $G$. Hence
\[
\Hom(M_1*M_2,\e)\rar{\simeq}\Hom(p_!(\iota^*M_1\tens M_2),\ql)\rar{\simeq}\Hom(\iota^*M_1\tens M_2,K_G).
\]
This implies that $\Phi^\tau(M_1,M_2)=\Hom(M_1*M_2,\e)$, and the fact that $\psi^\tau$ coincides with the pivotal structure described in Example \ref{example:pivotal} follows from the construction. Applying Lemma \ref{l:obvious-isomorphism-abstract} completes the proof.

\subsection{Braided Grothendieck-Verdier categories}

\subsubsection{The functors $D^2$ and $D^4$} The next result is proved in \cite[Lemma 6.8]{GV}.

\begin{lem}
Let $(\cM,K,\be)$ be a braided Grothendieck-Verdier category, and let
$\vp^\pm:D^{-1}\rar{\simeq}D$ be the isomorphisms induced by the compositions
\[
\Hom(Y,D^{-1}X) \rar{\simeq} \Hom(X\tens Y,K) \xrar{\ \ (\be^\pm_{Y,X})^*\ \ } \Hom(Y\tens X,K) \rar{\simeq} \Hom(Y,DX)
\]
for all $X,Y\in\cM$, where $\be_{X,Y}^+:=\be_{X,Y}$ and $\be_{X,Y}^-:=\be_{Y,X}^{-1}$. Then
\[
D(\vp^\pm_X) = (\vp^\mp_{DX})^{-1} \;\mbox{ for all } X\in\cM .
\]
\end{lem}

\begin{defin}\label{d:isom-from-Id-to-D^2-in-a-braided-GV-category}
If $(\cM,K,\be)$ is a braided Grothendieck-Verdier category, we define isomorphisms of functors $\vt^\pm:\Id_{\cM}\rar{\simeq}D^2$ as follows:
\[
\vt^\pm_X = \vp^\pm_{DX} = D(\vp^\mp_X)^{-1} : X\rar{\simeq}D^2 X \qquad \forall\,X\in\cM,
\]
where the second equality holds by the lemma above.
\end{defin}

\begin{rem}
In general the isomorphisms $\vt^\pm$ are \emph{not} monoidal.
\end{rem}

\begin{lem}
Let $(\cM,K,\be)$ be a braided Grothendieck-Verdier category. Then the monoidal functor $D^2:\cM\rar{\sim}\cM$ is braided.
\end{lem}

This result is \cite[Proposition 6.1(i)]{GV}. The lemma implies that the functor $D^4:\cM\rar{\sim}\cM$ is also braided. Note that we can consider $\vt^+\vt^-$ as an isomorphism of functors between $\Id_{\cM}$ and $D^4$.

\begin{lem}\label{l:braided-GV-isomorphism-from-Id-to-D^4}
In any braided \GV category $(\cM,K,\be)$ the isomorphism $\vt^+\vt^- : \Id_{\cM}\rar{\simeq}D^4$ is monoidal.
One has $\vt^-\vt^+=\vt^+\vt^-$.
\end{lem}

For the proof, see Definition 6.11 and Remark 6.12 in \cite{GV}.

\subsubsection{Pivotal structures and twists} Let us recall the following

\begin{defin}   \label{d:tw}
If $(\cM,\be)$ is a braided monoidal category, a \emph{twist} on $(\cM,\be)$ is an automorphism $\te:\Id_{\cM}\rar{\simeq}\Id_{\cM}$ of the identity functor satisfying
\[
\te_{X\tens Y} = \be_{Y,X}\circ\be_{X,Y}\circ(\te_X\tens\te_Y) \qquad \forall\,X,Y\in\cM.
\]
\end{defin}

Now suppose $(\cM,K,\be)$ is a braided \GV category equipped with a pivotal structure $\psi$. There exists a unique automorphism $\te$ of $\Id_{\cM}$ such that for all $X,Y\in\cM$ the isomorphism
\[
\Hom(Y,D^{-1}X)\rar{\simeq}\Hom(Y,D^{-1}X), \qquad g\longmapsto g \circ \te_Y \, ,
\]
is equal to the composition
\[
\Hom(Y,D^{-1}X) \rar{\simeq} \Hom(X\tens Y,K) \xrar{\ \ \psi_{X,Y}\ \ } \Hom(Y\tens X,K)
\]
\[
\xrar{\ \ (\be_{Y,X}^{-1})^*\ \ } \Hom(X\tens Y, K) \rar{\simeq} \Hom(Y,D^{-1}X).
\]

\begin{rem}\label{r:pivotal-structure-in-terms-of-twist}
It is clear that $\psi$ can be expressed in terms of $\te$ as follows:
\[
\psi_{X,Y} = \be_{Y,X}^*\circ(\id_X\tens\te_Y)^* : \Hom(X\tens Y,K) \rar{\simeq} \Hom(X\tens Y,K) \rar{\simeq} \Hom(Y\tens X,K)
\]
for all $X,Y\in\cM$.
\end{rem}

\begin{lem}\label{l:bijection-pivotal-structures-twists}
The map $\psi\mapsto\te$ constructed above is a bijection between the set of pivotal structures on $(\cM,K)$ and the set of twists $\te$ on $(\cM,\be)$ satisfying $\te_K=\id_K$.
\end{lem}

For the proof, see Proposition 7.1 and Remark 7.2 in \cite{GV}.

\begin{rems}\label{rr:twists-pivotal-involutions}
\begin{enumerate}[(1)]
\item In the situation of Lemma \ref{l:bijection-pivotal-structures-twists}, let $f:\Id_{\cM}\rar{\simeq}D^2$ be the monoidal isomorphism corresponding to $\psi$ as in Remark \ref{r:pivotal2} (see also Proposition \ref{p:2pivotal}). Unwinding the definitions, one sees that in terms of $f$, the twist corresponding to $\psi$ is given by $\te=(\vt^+)^{-1}\circ f$, where $\vt^+:\Id_{\cM}\rar{\simeq}D^2$ is the isomorphism constructed in Definition \ref{d:isom-from-Id-to-D^2-in-a-braided-GV-category}.
 \sbr
\item Let $(\cM,K,\be)$ be a braided \GV category, and let $\te$ be a twist on $(\cM,\be)$ such that $\te_K=\id_K$. If one defines $\te'_X=D^{-1}(\te_{DX})$ for all $X\in\cM$, then $\te'$ is also a twist on $(\cM,\be)$ satisfying $\te'_K=\id_K$ \cite[Prop.~7.3(iii)--(iv)]{GV}.
 \sbr
\item The map $\te\mapsto\te'$ is an involution of the set of twists of $(\cM,\be)$ that act as the identity on $K$ \cite[Prop.~7.3(i)--(ii)]{GV}.
 \sbr
\item The involution $\te\mapsto\te'$ can be described in different terms \cite[Prop.~7.3(v)]{GV}. Let $\psi,\psi'$ be the pivotal structures corresponding to $\te$ and $\te'$, and let $f,f':\Id_{\cM}\rar{\simeq}D^2$ be the monoidal isomorphisms corresponding to $\psi,\psi'$. Then $ff':\Id_{\cM}\rar{\simeq}D^4$ is equal to the monoidal isomorphism $\vt^+\vt^-:\Id_{\cM}\rar{\simeq}D^4$ of Lemma \ref{l:braided-GV-isomorphism-from-Id-to-D^4}.
\end{enumerate}
\end{rems}

\subsection{Ribbon \GV categories}  \label{ss:ribbon_GV}
\begin{defin}\label{d:ribbon}
A \emph{ribbon structure} on a braided \GV category $(\cM,K,\be)$ is a twist $\te$ on
$(\cM,\be)$ such that
\begin{equation}   \label{e:ribbon}
\te_X=D^{-1}(\te_{DX})\quad \mbox{ for all } X\in\cM.
\end{equation}
A \emph{ribbon \GV category} is a braided \GV category with a ribbon structure.
A \emph{ribbon r-category} is an r-category with a ribbon structure.
\end{defin}

\begin{rem}\label{remark-from-GV}
The identity \eqref{e:ribbon} holds if and only if for any $X,Y\in\cM$ and $B:X\tens Y\rar{} K$ one has
\begin{equation}   \label{e:without_D}
B\circ (\id_X\tens\te_Y)=B\circ (\te_X\tens\id_Y).
\end{equation}
Note that unlike \eqref{e:ribbon}, formula \eqref{e:without_D} makes sense in \emph{any}
braided category with a fixed object $K$ ($K$ does not have to be dualizing and $\cM$ does not
have to be Grothendieck-Verdier). We do not know whether condition \eqref{e:without_D} is really
interesting in this generality.
\end{rem}

\begin{lem}   \label{l:ribbon}
A twist $\te$ satisfies \eqref{e:ribbon} if and only if $\te_K=\id_K$ and $\te'=\te$, where
$\te'$ is defined in Remark \ref{rr:twists-pivotal-involutions}(2).
\end{lem}

\begin{proof}
We only have to show that the equality $\te_K=\id_K$
follows from \eqref{e:ribbon}. This is clear because $K=D\e$ and $\te_{\e}=\id_{\e}$.
\end{proof}

\begin{cor}   \label{c:ribbon}
The correspondence between twists and pivotal structures
$($see Lemma \ref{l:bijection-pivotal-structures-twists}$)$ induces a bijection between
ribbon structures on $(\cM,K,\be)$ and those pivotal structures on $(\cM,K)$ for which the corresponding monoidal isomorphism $f:\Id_{\cM}\rar{\simeq}D^2$
is invariant under the involution $f\mapsto f'$ from Remark \ref{rr:twists-pivotal-involutions}(4).
\end{cor}

\begin{proof}
This follows from Lemma~\ref{l:ribbon} and Remark \ref{rr:twists-pivotal-involutions}(4).
\end{proof}

\subsection{The canonical ribbon structure on $\sD_G(G)$}   \label{ss:ribbon_G}
The r-category $(\sD_G(G),*)$ has a natural ribbon structure. For an arbitrary algebraic
group $G$ it is described in Appendix~\ref{s:CFT} below.
In this subsection we define it for algebraic groups $G$ such that $G^\circ$ is unipotent. This assumption allows us to use the \emph{ad hoc} construction of $\sD_G(G)$ given in Definition \ref{d:equiv-derived}.

\begin{defin}[Braiding on $\sD_G(G)$]\label{d:braiding-equivariant-derived}
Let $M,N\in\sD_G(G)$, then the braiding $\be_{M,N}:M*N\iso
N*M$ is defined as follows. Consider the commutative diagram
\[
\xymatrix{
  G\times G \ar[d]_{\tau} \ar[r]^{\xi} &  G\times G \ar[d]^{\mu} \\
  G\times G \ar[r]^{\mu} & G,
   }
\]
where $\tau(g,h):=(h,g)$ and $\xi(g,h):=(g,g^{-1}hg)$. We have
$M*N=\mu_!(M\boxtimes N)$, and the above diagram shows that
$N*M=(\mu\tau )_!(M\boxtimes N)=\mu_!\xi_!(M\boxtimes N)$. We define
$\be_{M,N}:\mu_!(M\boxtimes N)\iso\mu_!\xi_!(M\boxtimes N)$ by
$\be_{M,N}:=\mu_!(f)$, where $f:M\boxtimes N\iso\xi_!(M\boxtimes N)$
comes from the $G$-equivariant structure on $N$.
\end{defin}

\begin{rem}
Checking the axioms
of a braiding for $\be$ is straightforward and is similar to the well
known case where $G$ is finite. In this case $\sD_G(G)$ is the
derived category of modules over the so-called quantum double of the
group algebra of $G$ (see, e.g.,  \S3.2 of \cite{BK}). These modules
form a braided category which is not symmetric unless $\abs{G}=1$.
\end{rem}

\begin{defin}[Twist on $\sD_G(G)$]\label{d:twist-equivariant-derived}
Let $c:G\times G\to G$ be the conjugation action morphism
$c(g,h)=ghg^{-1}$, let $p_2:G\times G\to G$ denote the second
projection, and $\De:G\to G\times G$ the diagonal. Then
$c\circ\De=\id_G=p_2\circ\De$. For each $M\in\sD_G(G)$, the
$G$-equivariant structure on $M$ yields an isomorphism $p_2^*M\iso c^*M$.
Pulling it back by $\De$ we get an isomorphism
$\te_M:M= \De^*p_2^*M\rar{\simeq}\De^*c^*M=M$.
\end{defin}

By construction, $\te$ is an automorphism of the identity functor on $\sD_G(G)$, and one can check that it is related to the braiding $\be$ of Definition \ref{d:braiding-equivariant-derived} as follows:
\[
\te_{M*N} = \be_{N,M}\circ\be_{M,N}\circ (\te_M*\te_N) \qquad
\forall\, M,N\in\sD_G(G).
\]
In fact, this follows from part (a) of

\begin{prop}\label{p:ribbon-equivariant-derived}
\begin{enumerate}[$($a$)$]
\item If the r-category $\sD_G(G)$ is equipped with the pivotal structure of Example \ref{example:pivotal} and the braiding of Definition \ref{d:braiding-equivariant-derived}, the corresponding twist $($cf.~Lemma \ref{l:bijection-pivotal-structures-twists}$)$ is the automorphism $\te$ constructed in Definition \ref{d:twist-equivariant-derived}.
 \sbr
\item The twist $\te$ is a ribbon structure on the braided r-category $(\sD_G(G),\be)$.
\end{enumerate}
\end{prop}

By Remark \ref{r:pivotal-structure-in-terms-of-twist}, the proposition follows at once from the next two lemmas.

\begin{lem}\label{l:ribbon1}
If $M,N\in\sD_G(G)$, then the isomorphism
\[
\psi_{M,N} : \Hom(M*N,\e) \rar{\simeq} \Hom(N*M,\e)
\]
from Example \ref{example:pivotal} is equal to the pullback map via the composition
\[
(\id_M*\te_N)\circ\be_{N,M} : N*M \rar{\simeq} M*N.
\]
\end{lem}

\begin{lem}\label{l:ribbon2}
For all $M\in\sD_G(G)$ we have $\te_{\bD_G^-M}=\bD_G^-(\te_M)$.
\end{lem}

For Lemma \ref{l:ribbon2} we refer to \cite[Proposition 7.2]{tanmay}.

\begin{proof}[Proof of Lemma \ref{l:ribbon1}]
We use the notation of Example \ref{example:pivotal}. In particular,
\[
Z_2 = \bigl\{ (g,h)\in G\times G \st gh=1 \bigr\},
\]
$\pi_1,\pi_2:Z_2\rar{}G$ are the two projections, and $p:Z_2\rar{}\Spec k$ is the structure morphism. The composition
\[
\Hom(M*N,1_*\ql) \rar{\simeq} \Hom(1^*(M*N),\ql) \rar{\simeq} \Hom\bigl( p_!(\pi_1^*(M)\tens\pi_2^*(N)),\ql\bigr)
\]
yields an identification\footnote{The Hom on the left hand side is computed in $\sD_G(G)$, while the Hom on the right hand side is computed in $\sD_G(\Spec k)$.}
\begin{equation}\label{e:iden1}
\Hom(M*N,\e) \rar{\simeq} \Hom\bigl( p_!(\pi_1^*(M)\tens\pi_2^*(N)),\ql\bigr).
\end{equation}
Similarly, we have an identification
\begin{equation}\label{e:iden2}
\Hom(N*M,\e) \rar{\simeq} \Hom\bigl( p_!(\pi_1^*(N)\tens\pi_2^*(M)),\ql\bigr).
\end{equation}

\mbr

Let $\tau:G\times G\rar{\simeq}G\times G$ and $\xi:G\times G\rar{\simeq}G\times G$ be as in Definition \ref{d:braiding-equivariant-derived}, that is, $\tau(g,h)=(h,g)$ and $\xi(g,h)=(g,g^{-1}hg)$. Note that both $\tau$ and $\xi$ preserve $Z_2\subset G\times G$; moreover, $\xi\bigl\lvert_{Z_2}=\id_{Z_2}$. We also have $p\circ\tau=p$ and $\pi_1\circ\tau=\pi_2$.

\mbr

The natural isomorphism $p_!\rar{\simeq}p_!\tau_!$ yields an isomorphism
\begin{equation}\label{e:compos-auxil-2}
p_! ( \pi_1^*(N) \tens \pi_2^*(M) )  \rar{\simeq}  p_! \tau_! ( \pi_1^*(N) \tens \pi_2^*(M) ) \cong p_!(\pi_1^*(M)\tens\pi_2^*(N)).
\end{equation}
The induced isomorphism
\[
\Hom\bigl( p_!(\pi_1^*(M)\tens\pi_2^*(N)),\ql\bigr) \rar{\simeq} \Hom\bigl( p_!(\pi_1^*(N)\tens\pi_2^*(M)),\ql\bigr)
\]
coincides with $\psi_{M,N}$ modulo the identifications \eqref{e:iden1} and \eqref{e:iden2}.

\mbr

On the other hand, consider the composition
\begin{equation}\label{e:compos-auxil}
N\boxtimes M \xrar{\ \ f\ \ }  \xi_! (N\boxtimes M) \xrar{\ \ \xi_!(\id_N\boxtimes\te_M)\ \ } \xi_!(N\boxtimes M),
\end{equation}
where $f:N\boxtimes M\rar{\simeq}\xi_!(N\boxtimes M)$ is the isomorphism coming from the $G$-equivariant structure on $M$, used in Definition \ref{d:braiding-equivariant-derived}. If we restrict \eqref{e:compos-auxil} to $Z_2$, we obtain the identity automorphism of $\pi_1^*(N)\tens\pi_2^*(M)$ (here we used the definition of $\te_M$; recall also that $\xi\bigl\lvert_{Z_2}=\id_{Z_2}$). This implies that the isomorphism $1^*(N*M)\rar{\simeq}1^*(M*N)$ induced by $(\te_M*\id_N)\circ\be_{N,M}:N*M\rar{\simeq}M*N$ is equal to the composition \eqref{e:compos-auxil-2}. Equivalently, $\psi_{M,N}$ is equal to the pullback map via the composition
\[
(\te_M*\id_N)\circ\be_{N,M} : N*M \rar{\simeq} M*N.
\]
Finally, observe that since $\te_{\e}=\id_{\e}$, Lemma \ref{l:ribbon2} implies that
\[
(\te_M*\id_N)^* = (\id_M*\te_N)^* : \Hom(M*N,\e)\rar{}\Hom(M*N,\e),
\]
which completes the proof.
\end{proof}

\subsection{Hecke subcategories of Grothendieck-Verdier categories}
The notion of a closed idempotent in a monoidal category was introduced in Definition \ref{defs:idempotents}(d).

\begin{lem}\label{l:GV-idempotents}
Let $(\cM,K)$ be a Grothendieck-Verdier category.
Let $e\in\cM$ be a closed idempotent. Then 
\begin{equation}   \label{e:GV-idempotents1}
D(e\cM)=\cM e \, ,
\end{equation}
\begin{equation}    \label{e:GV-idempotents2}
D(\cM e)= D^2 e\cdot\cM \, .
\end{equation}
\end{lem}

Note that $D^2 e$ is a closed idempotent: this follows from
\eqref{e:D^2(1)} and \eqref{e:D^24}. The notation $e\cM$ and $\cM e$ was introduced in \S\ref{ss:hecke-subcategories}.

\begin{proof}
Let $\pi:\e\rar{}e$ be an idempotent arrow. If $Y\in e\cM$ the
morphism $\pi\tens\id_Y:Y\rar{}e\tens Y$ is an isomorphism. By
\eqref{e:duality1}, for every $X\in\cM$, the morphism
$\id_X\tens\pi:X\rar{}X\tens e$ induces a bijection $\Hom(X\tens e,D
Y)\rar{\simeq}\Hom(X,DY)$. Now Proposition
\ref{p:idempotent-monads}(b) implies that $DY\in \cM e$. This proves that
\begin{equation}   \label{e:GV-idempotents3}
D(e\cM)\subset\cM e \, .
\end{equation}
Now apply \eqref{e:GV-idempotents3} to $\cM$ equipped with the opposite tensor product.
Then the dualization functor equals $D^{-1}$, so we get $D^{-1} (\cM e)\subset e\cM$, i.e.,
$\cM e\subset D(e\cM )$. Combining this with \eqref{e:GV-idempotents3} we get
\eqref{e:GV-idempotents1}.

\mbr

To prove \eqref{e:GV-idempotents2}, apply $D$ to
\eqref{e:GV-idempotents1} and note that $D^2(e\cM )=D^2 e\cdot\cM$
by \eqref{e:D^24}.
\end{proof}

\begin{lem}\label{l:GV-Hecke-subcategory}
Let $(\cM,K)$ be a Grothendieck-Verdier category, and
let $e\in\cM$ be a closed idempotent such that $D^2e\simeq e$.
\begin{enumerate}[$($a$)$]
\item We have $D(e\cM)=\cM e$ and $D(\cM e)= e\cM$.
 \sbr
\item $De\in e\cM e$ is a dualizing object of the monoidal category $e\cM e$, so
$(e\cM e, De)$ is a Grothendieck-Verdier category.
 \sbr
\item The corresponding duality functor $e\cM e\rar{\sim}e\cM e$ can be identified with the restriction of $D$ to $e\cM e$. This identification is canonical as soon as one chooses an idempotent arrow
$\pi:\e\rar{}e$.
\end{enumerate}
\end{lem}

\begin{proof}
Statement (a) follows from Lemma~\ref{l:GV-idempotents}. Let us prove (b) and (c).
By Lemma \ref{l:Hecke-monoidal}, $e\cM e$ is a monoidal category. By  part (a),
$D (e\cM e)=e\cM e$; in particular, $D (e)\in e\cM e$. Fix an idempotent arrow
$\pi:\e\rar{}e$. Given $X,Y\in e\cM e$, we have canonical isomorphisms
\[
\Hom(X,DY) \cong \Hom(X\tens Y,K) \cong \Hom(X\tens Y\tens e,K) \cong
\Hom(X\tens Y, De),
\]
where the middle one comes from $\pi:\e\rar{}e$
and the other two come from \eqref{e:duality1}. This implies both (b) and (c).
\end{proof}

One can ask which \GV categories can be realized as $e\cM e$, where $\cM$ is an
r-category, and $e\in\cM$ is a closed idempotent such that $D^2e\simeq e$. An answer
to this question is given in \cite[\S9]{GV}.

\subsubsection{Hecke subcategories of pivotal \GV categories}

\begin{lem}\label{l:pivotal-Hecke-subcategory}
Let $(\cM,K)$ be a pivotal Grothendieck-Verdier category, and
$e\in\cM$ a closed idempotent. Then $(e\cM e, De)$ is a Grothendieck-Verdier category.
Moreover, it has a unique pivotal structure $\widetilde{\psi}$ such that for all $X,Y\in e\cM e$ and every
idempotent arrow $\pi:\e\rar{}e$ the diagram
\begin{equation}    \label{e:induced_pivotal}
\xymatrix{
 \Hom (X\tens Y,De)  \ar[d]_{} \ar[rr]^{\ \ \ \widetilde{\psi}_{X,Y}\ \ \ } & &
 \Hom (Y\tens X,De)\ar[d]^{} \\
 \Hom (X\tens Y,K) \ar[rr]^{\ \ \ \psi_{X,Y}\ \ \  } & &
 \Hom (Y\tens X,K)}
\end{equation}
in which the vertical arrows come from $D\pi:De\rar{}D\e=K$, commutes.
\end{lem}

\begin{proof}
The first statement follows from Lemma~\ref{l:GV-Hecke-subcategory} because
in a pivotal category $D^2\cong\Id$ (see Remark~\ref{r:pivotal2}).
Now fix an idempotent arrow $\pi:\e\rar{}e$. Then for
every $Z\in\cM e$ the map $\Hom (Z,De)\rar{}\Hom (Z,K)$ induced by $D\pi:De\rar{}D\e=K$ is bijective
because it equals the composition $\Hom (Z,De)\iso \Hom (Z\tens e,K)\iso\Hom (Z,K)$, where the
first arrow comes from \eqref{e:duality1} and the second one from $\id_Z\tens\pi :Z\iso Z\tens e$.
Since the vertical arrows in \eqref{e:induced_pivotal} are bijections there is a unique pivotal
structure $\widetilde{\psi}$ on $e\cM e$ such that the diagram \eqref{e:induced_pivotal}
corresponding to our fixed idempotent arrow $\pi:\e\rar{}e$ commutes.
We have to show that then it commutes for any idempotent arrow $\pi' :\e\rar{}e$.
By Corollary~\ref{c:uniqueness-idempotent-arrows}, $\pi'=f\circ\pi$ for some $f\in\Aut e$, so it remains
to show that the map
\begin{equation}  \label{e:our_psi}
\widetilde{\psi}_{X,Y}:\Hom (X\tens Y,De)\rar{}\Hom (Y\tens X,De), \quad X,Y\in e\cM e
\end{equation}
commutes with $\End e$ acting on both sides of \eqref{e:our_psi} via 
$D:\End e\rar{}\End (De)$. It is easy to check that this action of $\End e$ equals the one that comes
from the map $\varphi :\End e\rar{}\End X$ ($\varphi$ is defined because $e$ is a unit object of the monoidal category $e\cM e$ and $X\in e\cM e$). So the commutation of $\widetilde{\psi}_{X,Y}$ with $\End e$ follows from functoriality of $\widetilde{\psi}_{X,Y}$ with respect to~$X$.
\end{proof}

\subsubsection{Hecke subcategories of braided \GV categories}

\begin{lem}
Let $(\cM,K,\be)$ be a braided \GV category, let $D:\cM\rar{\sim}\cM$ be the corresponding duality functor, and let $e\in\cM$ be a closed idempotent. The Hecke subcategory $e\cM e=e\cM=\cM e\subset\cM$ is stable under $D$, and is a braided \GV category with dualizing object $De$. The corresponding dualizing functor can be identified with the restriction of $D$ to $e\cM e$.
\end{lem}

\begin{proof}
In Definition \ref{d:isom-from-Id-to-D^2-in-a-braided-GV-category} we described an isomorphism of functors $\Id_{\cM}\rar{\simeq}D^2$. In particular, $D^2e\cong e$. All statements of the lemma now follow from Lemma \ref{l:GV-Hecke-subcategory}.
\end{proof}

\subsubsection{Hecke subcategories of ribbon \GV categories}

\begin{lem}
Let $(\cM,K,\be)$ be a braided \GV category, fix a closed idempotent $e\in\cM$, and let $\widetilde{\cM}=e\cM e=e\cM=\cM e\subset\cM$ be the Hecke subcategory defined by $e$.
 \sbr
\begin{enumerate}[$($a$)$]
\item Suppose $\psi$ is a pivotal structure on $\cM$ and $\widetilde{\psi}$ is the induced pivotal structure on $\widetilde{\cM}$ $($see Lemma \ref{l:pivotal-Hecke-subcategory}$)$. If $\te$ and $\widetilde{\te}$ are the twists on $\cM$ and $\widetilde{\cM}$ corresponding to $\psi$ and $\widetilde{\psi}$ as in Lemma \ref{l:bijection-pivotal-structures-twists}, then $\widetilde{\te}=\te\bigl\lvert_{\widetilde{\cM}}$.
 \sbr
\item In the situation of $($a$)$, if $\te$ is a ribbon structure on $\cM$, then $\widetilde{\te}$ is a ribbon structure on $\widetilde{\cM}$.
\end{enumerate}
\end{lem}

\begin{proof}
(a) Choose an idempotent arrow $\pi:\e\rar{}e$. The diagram
\[
\xymatrix{
  \Hom (X\tens Y,De)  \ar[d]_{} \ar[r]^{\ \ \ \widetilde{\psi}_{X,Y}\ \ \ } &
 \Hom (Y\tens X,De)\ar[d]^{} \ar[r]^{(\be_{Y,X}^{-1})^*} & \Hom(X\tens Y, De) \ar[d] \\
 \Hom (X\tens Y,K) \ar[r]^{\ \ \ \psi_{X,Y}\ \ \  } &
 \Hom (Y\tens X,K) \ar[r]^{(\be_{Y,X}^{-1})^*} & \Hom(X\tens Y, K)
   }
\]
in which the vertical arrows come from $D\pi:De\rar{}D\e=K$, commutes for all $X,Y\in\widetilde{\cM}$. Now the claim follows from the definitions of $\te$ and $\widetilde{\te}$ and the fact that $\pi$ identifies the duality functor for $(\widetilde{\cM},De)$ with the restriction of $D$ to $\widetilde{\cM}$ (see Lemma \ref{l:GV-Hecke-subcategory}(c)).

\mbr

(b) This follows from Definition \ref{d:ribbon} and the fact that the duality functor for $(\widetilde{\cM},De)$ can be identified with the restriction of $D$ to $\widetilde{\cM}$.
\end{proof}

\section{The structures on $\sD_G(G)$ (a topological field theory approach)} \label{s:CFT}
\begin{convention}
{\em Throughout this appendix, $\sD_G(G)$ is understood as the bounded derived category of constructible $\ell$-adic complexes \cite{Las-Ols06} on the stack $\Ad(G)\backslash G$ obtained by taking the quotient of $G$ by the conjugation action of $G$ on itself.}
\end{convention}

The convention above is necessary because we do not require $G$ to be unipotent. On the other hand, to be able to apply the results of this appendix to Lemma \ref{l:compatibility-ind-braiding-twists}, we need to know that in the unipotent case the naive definition of $\sD_G(G)$ is equivalent to the correct one. This is proved in Proposition \ref{p:equivalence-definitions-equivariant-category} in Appendix \ref{a:equivalence-equivariant-derived}.

\subsection{Overview}
To every algebraic stack $\sX$ satisfying a certain ``perfectness" condition
D.~Ben-Zvi, J.~Francis, and D.~Nadler  \cite[\S6]{BFN08} associate a 2-dimensional topological
field theory (TFT), denoted by $Z_{\sX}$. If $G$ is an algebraic group and $\sX$ is its classifying
stack $BG$ then $Z_{\sX} (S^1)$ (i.e., the value of $Z_{\sX}$ on the standard circle $S^1$) is the equivariant derived
category of \emph{quasicoherent} sheaves on $G$. This implies that  the latter category is equipped
with a braided structure and a twist. Note that using the language of 2-dimensional TFT to define a braided structure is natural because
the braid groups are most naturally defined in terms of $\bR^2$.


In this appendix we describe a similar 
construction for \emph{constructible} sheaves
instead of quasicoherent ones. In particular, for any algebraic group $G$ over any field
we define in \S\ref{ss:The_pre-sTFT} a canonical braided structure and a twist on $\sD_G(G)$.
Moreover, we define an action of the surface operad on $\sD_G(G)$.


The main difference\footnote{There is also another difference, see \S\ref{sss:the_other_diff}.} between the constructible case and the quasicoherent one is that the
constructible derived category $\sD (X_1\times X_2)$ is usually \emph{not generated} by objects of
the form $M_1\boxtimes M_2$, $M_i\in\sD (X_i)$. Because of this,
we get not a TFT but 
a \emph{pre-TFT} in the sense of \S\ref{sss:informal-preTFT}
(this is a ``lax" version of the notion of TFT).


In \S\ref{ss:lax_functoriality} we study how the pre-TFT corresponding to an algebraic stack $\sX$
depends on $\sX$. This allows us to prove Lemma~\ref{l:compatibility-ind-braiding-twists}
(on the compatibility of the functor $\ig :\sD_{G'}(G')\rar{}\sD_{G}(G)$ with braidings and twists).

\S\ref{ss:GV-stacks} is devoted to \GV duality in $\sD_G(G)$ and more generally, in $Z_{\sX} (S^1)$,
where $\sX$ is any algebraic stack of finite type over a field. We construct a dualizing object
$K\in Z_{\sX} (S^1)$ and show that the braiding and twist from \S\ref{ss:The_pre-sTFT} define on
$(Z_{\sX} (S^1),K)$ a structure of ribbon \GV category in the sense of \S\ref{ss:ribbon_GV}.


Unlike  \cite[\S 6]{BFN08}, we use $n$-categories only for $n\le 2$.
Some remarks on the $\infty$-categorical setting are given in \S\ref{ss:infty-cats}.

\begin{convention}
{\em The words ``2-category" and ``2-functor" are always understood in the ``weak" sense}
(as opposed to the ``strict" one).
\end{convention}

\subsection{The notion of pre-TFT}   \label{ss:pre-TFT}
\subsubsection{The 2-categories $\Cob$,  $\Cobin$, $\Cobout$}   \label{sss:Cob}
We follow \cite[\S 1.1 and \S 1.4]{Lur09}.
In this subsection ``manifold" means ``$C^{\infty}$-manifold".
If $M$ and $N$ are $(n-1)$-dimen\-sional closed oriented manifolds then a \emph{bordism} from $M$ to $N$ is an $n$-dimensional oriented manifold $B$ equipped with an oriented
diffeomorphism $\al:\overline{M} \coprod N\iso\partial B$ (here $\overline{M}$ is the
manifold $M$ with the opposite orientation). If $(B',\al')$ is another bordism
from $M$ to $N$ then by a diffeomorphism between $(B,\al )$ and $(B',\al')$ we mean an
oriented diffeomorphism $f:B\iso B'$ such that $f\circ\al=\al'$.

 \begin{defin}     \label{d:(2,1)}
A \emph{(2,1)-category} is a 2-category 
whose 2-morphisms are invertible.
 \end{defin}

 \begin{defin}        \label{d:Cob}
 $\Cob$ is the following (2,1)-category:
 \begin{itemize}
\item  its objects are closed, oriented 1-dimensional $C^{\infty}$-manifolds;
\mbr
\item for any $M,N \in \Cob$ the category of 1-morphisms $\Mor (M,N)$  is the group\-oid
 whose objects are bordisms from $M$ to $N$ and whose isomorphisms
 are isotopy classes of diffeomorphisms between bordisms.
 \mbr
 \item
 1-morphisms are composed by gluing bordisms.
 \end{itemize}
 \end{defin}

\begin{rem}
 In  \cite{Lur09} and \cite{BFN08} the above (2,1)-category is denoted by $\Cob (2)$ and 2Cob, respectively (here ``2" indicates the dimension of the bordisms).
 \end{rem}

 \begin{defin}       \label{d:+Cob}
Let $\Cobin$ (resp. $\Cobout$) denote the (2,1)-category  that one gets from $\Cob$ by considering only those bordisms $B$ from $M$ to $N$ for which the map $\pi_0 (M)\rar{}\pi_0(B)$ (resp. $\pi_0 (N)\rar{}\pi_0(B)$) is surjective.
 \end{defin}

%

 We have obvious 2-functors $\Cobin\rar{}\Cob$ and $\Cobout\rar{}\Cob$.
 The $(2,1)$-categories $\Cob$, $\Cobin$, and $\Cobout$ are symmetric monoidal
 with respect to disjoint union. (The precise meaning of this statement is explained in
 Remark~\ref{r:Lurie's_def} below.)

\subsubsection{Pre-definition of a pre-TFT}    \label{sss:informal-preTFT}
Let $\Cat$ denote the 2-category of categories.

\begin{predef}   \label{d:informal-preTFT}
A \emph{2-dimen\-sional pre-TFT}
(resp.  \emph{2-dimen\-sional incoming pre-TFT,
2-dimen\-sional outgoing pre-TFT}\,)
with values in $\Cat$ is the following collection of data:

\begin{enumerate}[(i)]
\item a 2-functor $Z: \Cob\rar{}\Cat$ (resp.  $Z: \Cobin\rar{}\Cat$, $Z: \Cobout\rar{}\Cat$);

\mbr
\item  for every $n\ge 0$ and every closed oriented 1-manifolds $X_1,\ldots ,X_n$, a functor
\begin{equation}   \label{e:pseudomonoidal}
\prod_i Z(X_i)\rar{} Z(\bigsqcup_i X_i);
\end{equation}

\mbr
\item certain compatibility data and conditions for the functors \eqref{e:pseudomonoidal}.
\end{enumerate}
\end{predef}

We skip the precise list of the compatibilities mentioned in (iii) (the reader can easily guess it).
Instead, in \S\ref{ss:Segal} we give a definition of pre-TFT in the format of  \cite{Lur07};
the idea is to combine data (i)-(iii) into a single 2-functor.


\begin{rem}       \label{r:Sullivan}
In Pre-definition~\ref{d:informal-preTFT} ``incoming" and ``outgoing" are abbreviations for the names
``positive incoming boundary" and ``positive outgoing boundary", which were
suggested (in the case of TFT's) by Ralph Cohen and used
by M.~Chas and D.~Sullivan in \cite{CS04,S04}. The synonym for  
``incoming" used by J.~Lurie in Definition 4.2.10 and Theorem 4.2.11 from  \cite{Lur09} is ``noncompact".
\end{rem}


\subsubsection{The structure on the category $Z(S^1)$}   \label{sss:surface_operad}
Let $Z$ be a 
pre-TFT. Then for any $X,X_1,\ldots X_n\in\Cobout$ and any
1-morphism $f:\bigsqcup\limits_i X_i\rar{}X$ in $\Cobout$ one gets a canonical functor
$\prod\limits_i Z(X_i)\rar{} Z(X)$ by composing the functor \eqref{e:pseudomonoidal} with $Z(f)$.
In particular, for every finite set $I$ any connected bordism from $S^1\times I$ to $S^1$
defines a functor $Z(S^1)^I\rar{}Z(S^1)$. It is clear how such functors are composed:
they define an action of the \emph{surface operad} on $Z(S^1)$ (this operad was
introduced in \cite{Til00}). As explained, e.g., in  \cite[\S 3.1]{Til98}, an action of
the genus 0 part of the surface operad on a category $\cC$ defines\footnote{In fact, it is known that
an action of the genus 0 surface operad on a category $\cC$ is \emph{the same} as a structure of braided monoidal category with a twist on $\cC$. This follows from \cite[Proposition 7.6]{SW03} and the fact that the genus 0 surface operad is equivalent to the framed disk operad.} a structure of
braided monoidal category with a twist on $\cC$. In particular, if $Z$ is a 
pre-TFT then \emph{the category $Z(S^1)$ is equipped with a canonical braided monoidal structure and twist.} The same is true if $Z$ is an outgoing pre-TFT. If $Z$ is an
 \emph{incoming} pre-TFT then $Z(S^1)$ is a braided \emph{semigroupal} category
(see \S\ref{ss:monoidal-notation}) equipped with a twist.


%

\subsubsection{Precise definition of a pre-TFT}   \label{ss:Segal}
We recommend to skip this subsection. It is merely an exegesis of certain parts of Lurie's
article \cite{Lur07} (this article will be incorporated into his book ``Higher algebra''). The idea is to combine data (i)-(iii) from Pre-definition~\ref{d:informal-preTFT}
into a single 2-functor.


\begin{defin}   \label{def:Segal1}
Let $I,J$ be sets. A \emph{partially defined map} $I\dasharrow J$ is a pair $(I_f,f)$, where $I_f\subset I$ and
$f:I_f\rar{}J$ is a usual map.
\end{defin}

%

For partially defined maps there is an obvious notion of composition.

\begin{defin}    \label{def:Segal2}
\emph{Segal's category}, denoted by $\cS$, is the category whose objects are finite sets and whose morphisms are partially defined maps.
\end{defin}


\begin{rems}     \label{r:Segal1}
\begin{enumerate}[(i)]
\item According to \cite[Definition 1.1.7]{Lur07}, Segal's category (denoted by $\Ga$)
is the category of finite sets equipped with a based point. One has an equivalence
$\Ga\rar{\simeq}\cS$ (removing the base point).

\mbr
\item  The category introduced in G.~Segal's original work \cite[Definition 1.1]{Seg74}
is dual to $\cS$.
\end{enumerate}
\end{rems}

Now define a $(2,1)$-category $\Cob^{\otimes}$ as follows.
Its objects are triples $(M,I,\pi )$, where $M\in\Cob$, $I$ is a finite set, and $\pi :M\rar{}I$ is a locally constant map.
Given such a triple and an element $i\in I$ we set $M_i:=\pi^{-1}(i)$. Define a 1-morphism
$(M,I,\pi )\rar{}(M',I',\pi')$ to be the following collection of data:
\begin{itemize}
\item a partially defined map $f:I\dasharrow I'$;

\mbr
\item  for each $j\in I'$, a 1-morphism in $\Cob$ from $\bigsqcup\limits_{i\in f^{-1}(j)} M_i$ to $M'_j$\,.
\end{itemize}
The 2-morphisms in $\Cob^{\otimes}$ come from $\Cob$. The composition of 1-morphisms
and 2-morphisms in $\Cob^{\otimes}$ is clear.

\begin{example}   \label{example:most_stupid}
Let $(M,I,\pi )\in\Cob^{\otimes}$. Then for each $i\in I$ one has in $\Cob^{\otimes}$ a
canonical 1-morphism
\begin{equation}    \label{e:most_stupid}
(M,I,\pi )\rar{} (M_i,\{i \}, \pi_i)\, ,
\end{equation}
where $\pi_i$ is the unique map $M_i\rar{}\{i \}$. To define \eqref{e:most_stupid},
use the identity 1-morphism $M_i\rar{}  M_i$ in $\Cob$ and the
partially defined map $f_i:I\dasharrow\{ i\}$ such that  $f(i):=i$ and if $i'\ne i$ then $f(i')$ is not defined.
\end{example}

\begin{defin}     \label{d:preTFT}
A \emph{2-dimensional pre-TFT} with values in $\Cat$ is a 2-functor $Z:\Cob^{\otimes}\rar{}\Cat$ with the following \emph{Segal property}: for every $(M,I,\pi )\in\Cob^{\otimes}$ the functor
$Z(M,I,\pi )\rar{}\prod\limits_{i\in I} Z(M_i,\{i \}, \pi_i)$ induced by the 1-morphisms \eqref{e:most_stupid}
is an equivalence.
\end{defin}

Replacing in Definition~\ref{d:preTFT} $\Cob^{\otimes}$ by similar (2,1)-categories
$\Cobin^{\otimes}$ and $\Cobout^{\otimes}$ one gets the precise notions of 2-dimensional
\emph{incoming pre-TFT} and \emph{outgoing pre-TFT}.

Let us explain the relation between Definition~\ref{d:preTFT} and the informal
Definition~\ref{d:informal-preTFT}. Considering in $\Cob^{\otimes}$ only objects
$(M,I,\pi )$ such that $I$ has a single element and only those 1-morphisms
$(M,I,\pi )\rar{}(M',I',\pi' )$ for which the partially defined map $f:I\dasharrow I'$ is defined everywhere
we get a $(2,1)$-category equivalent to $\Cob$.
If $Z:\Cob^{\otimes}\rar{}\Cat$ is a 2-dimensional pre-TFT in the sense of Definition~\ref{d:preTFT} then the restriction of $Z$ to $\Cob$ is a 2-dimensional pre-TFT
in the sense of Pre-definition~\ref{d:informal-preTFT}. Conversely, if
$Z:\Cob\rar{}\Cat$ is a 2-dimensional pre-TFT in the sense of Definition~\ref{d:informal-preTFT}
then one extends $Z$ to a 2-functor $\Cob^{\otimes}\rar{}\Cat$ by setting
$Z(M,I,\pi ):=\prod\limits_{i\in I} Z(M_i)$ and using \eqref{e:pseudomonoidal} to define $Z$ on 1-morphisms.

\begin{rem}   \label{r:cofibered}
As explained by Grothendieck in expos\'e VI of SGA 1, given a 2-functor from a category
$\cC$ to $\Cat$ it is convenient to pass to the corresponding category cofibered over $\cC$.
Similarly, in Definition~\ref{d:preTFT}  one could pass from the 2-functor $Z$ to the
corresponding 2-category cofibered in categories over $\Cob^{\otimes}$. This is what
J.~Lurie does systematically in \cite{Lur07}.
\end{rem}

\begin{rem}   \label{r:Lurie's_def}
The pair consisting of the $(2,1)$-category $\Cob^{\otimes}$ and the functor
$\Cob^{\otimes}\to\cS$ defined by $(M,I,\pi )\mapsto I$ is a symmetric monoidal $(2,1)$-category
in the sense of  \cite[Definition 1.2.11]{{Lur07}}.
\end{rem}

\subsection{The notion of pre-sTFT}   \label{ss:pre-sTFT}
In \cite[Definition 6.4]{BFN08} Ben-Zvi, Francis, and Nadler introduce a version of the $(2,1)$-category
$\Cob$ in which manifolds are replaced by topological spaces satisfying a finiteness condition. Similarly, we will consider a version of $\Cob$ in which manifolds are replaced by groupoids
satisfying a finiteness condition. This will lead us to the notion of pre-sTFT, where ``s" stands for ``strong" (and maybe for ``stupid", see Remark~\ref{r:strong_or_stupid} below).

\subsubsection{The 2-categories $\sCob$, $\sCobin$, $\sCobout$}   \label{sss:sCob}
\begin{defin}      \label{d:finite_type}
A groupoid $\Ga$ has {\em finite presentation\,} if it has finitely many
isomorphism classes of objects and the automorphism group of each object of $\Ga$
has  finite presentation. The (2,1)-category of groupoids of finite presentation is denoted by~$\fG$.
\end{defin}

\begin{defin}   \label{d:sbordism}
Let $\Ga_1,\Ga_2 \in\fG$. A \emph{bordism} from $\Ga_1$ to $\Ga_2$ is a diagram
\[
\Ga_1\rar{}\Ga\lar{}\Ga_2, \quad\quad \Ga\in\fG.
\]
\end{defin}

Bordisms from $\Ga_1$ to $\Ga_2$ form a 
$2$-groupoid. Namely, a 1-morphism from a bordism
$\Ga_1\rar{f_1}\Ga\lar{f_2}\Ga_2$ to a bordism $\Ga_1\rar{f'_1}\Ga'\lar{f'_2}\Ga_2$ is defined to be a triple
consisting of an equivalence $F:\Ga\iso {}\Ga'$ 
and isomorphisms $F\circ f_1\iso f_1'$, $F\circ f_2\iso f_2'$; such triples clearly form a
groupoid.\footnote{This groupoid is often a set. This happens if and only if every object $\ga'\in\Ga'$ such that
$\Aut\ga'$ has nontrivial center belongs to the essential image of $\Ga_1\bigsqcup\Ga_2$.}
Now truncate the 
$2$-groupoid of bordisms to a 
$1$-groupoid.\footnote{This truncation (which is not very barbarous by the previous footnote) allows us to avoid $n$-categories for $n>2$.} 

\begin{defin}  \label{d:sgroupoid}
This groupoid is called the \emph{groupoid of bordisms} from $\Ga_1$ to $\Ga_2$.
\end{defin}

 \begin{defin}        \label{d:sCob}
 $\sCob$ is the following (2,1)-category:
 \begin{itemize}
\item  its objects are groupoids of finite presentation;
\mbr
\item for any $\Ga_1,\Ga_2 \in \sCob$ the category of 1-morphisms $\Mor (\Ga_1,\Ga_2)$  is the groupoid of bordisms  $\Ga_1\rar{}\Ga\lar{}\Ga_2$;

\mbr
 \item the composition of bordisms $\Ga_1\rar{}\Ga_{12}\lar{}\Ga_2$ and
 $\Ga_2\rar{}\Ga_{23}\lar{}\Ga_3$ is the bordism $\Ga_1\rar{}\Ga_{13}\lar{}\Ga_3$, where
 $\Ga_{13}$ is the categorical pushout $\Ga_{12}\bigsqcup_{\Ga_2}\Ga_{23}\,$.
 \end{itemize}
 \end{defin}

 \begin{defin}       \label{d:+sCob}
Let $\sCobin$ (resp.~$\sCobout$) denote the (2,1)-category  that one gets from $\Cob$ by considering only those
bordisms $\Ga_1\rar{}\Ga\lar{}\Ga_2$ for which the map $\pi_0 (\Ga_1)\rar{}\pi_0(\Ga)$ (resp. $\pi_0 (\Ga_2)\rar{}\pi_0(\Ga)$) is surjective.
 \end{defin}

%

\subsubsection{Definition of a pre-sTFT}    
Let $\Cat$ denote the 2-category of categories.

\begin{predef}   \label{d:informal-presTFT}
A \emph{pre-sTFT} (resp.  \emph{incoming pre-sTFT, outgoing pre-sTFT}\,) with values in $\Cat$ is the following collection of data:

\begin{enumerate}[(i)]
\item a 2-functor $Z: \sCob\rar{}\Cat$ (resp.  $Z: \sCobin\rar{}\Cat$, $Z: \sCobout\rar{}\Cat$);

\mbr
\item  for every $n\ge 0$ and every groupoids of finite presentation $\Ga_1,\ldots ,\Ga_n$,  a functor
\begin{equation}   \label{e:spseudomonoidal}
\prod_i Z(\Ga_i)\rar{} Z(\bigsqcup_i \Ga_i);
\end{equation}

\mbr
\item certain compatibility data and conditions for the functors \eqref{e:spseudomonoidal}
similar to those in Definition~\ref{d:informal-preTFT}.
\end{enumerate}
\end{predef}

To formulate a complete definition of pre-sTFT,
define a (2,1)-category $\sCob^{\otimes}$ similarly to the  (2,1)-category $\Cob^{\otimes}$
from \S\ref{ss:Segal} and then, just as in Definition~\ref{d:preTFT}, define a pre-sTFT to be
a 2-functor $\sCob^{\otimes}\rar{}\Cat$ having the Segal property.

\subsubsection{From a pre-sTFT to a pre-TFT}   \label{sss:destupidization}
 Associating to a manifold its fundamental groupoid one gets 2-functors
 \begin{equation}
 \Pi: \Cob\to{}\sCob,\quad\Pi_{in}: \Cobin\to{}\sCobin, \quad   \Pi_{out}: \Cobout\to{}\sCobout \;.
  \end{equation}
If $Z: \sCob\rar{}\Cat$ is a pre-sTFT then $Z\circ\Pi :\Cob\rar{}\Cat$ is a pre-TFT.
Similarly, an incoming (resp. outgoing) pre-sTFT defines an incoming (resp. outgoing) boundary pre-TFT.

\subsubsection{The canonical braiding and twist on $Z(B\bZ)$}   \label{sss:Z(BbZ)}
For any group $H$, let $BH$ denote the corresponding groupoid (i.e., $BH$ has one object
with automorphism group $H$). The fundamental groupoid of the standard circle $S^1$ equals
$B\bZ$. Combining this with  \S\ref{sss:destupidization} and \S\ref{sss:surface_operad} we see that if
$Z$ a pre-sTFT (or an outgoing pre-sTFT) then \emph{the category $\cM:=Z (B\bZ )$ is equipped with a canonical braided monoidal structure and twist} and if $Z$ is an incoming pre-sTFT then $\cM$ is a braided semigroupal category equipped with a twist.
In this subsection we will 
describe the same braided semigroupal structure and twist  in concrete algebraic terms, without referring to
\S\ref{sss:surface_operad}. The reader may prefer to skip this description and go directly to
\S\ref{ss:The_pre-sTFT}.

\mbr

Let $F_n$ be the group freely generated by $x_1,\ldots ,x_n$. For each $u\in F_n$ let
$\phi_u:B\bZ\rar{} BF_n$ be the functor induced by the homomorphism $\bZ\rar{} F_n$ that takes
$1\in\bZ$ to $u$.
For each $n>0$, we have a bordism
\begin{equation}    \label{e:n-th_bordism}
\Ga_n\rar{f}BF_n\lar{g}B\bZ , \quad\quad
\Ga_n:=\underbrace{B\bZ\bigsqcup\ldots \bigsqcup B\bZ}_n
\end{equation}
where the restriction of $f$ to the $i$-th copy of $B\bZ$
equals $\phi_{x_i}$ and $g:=\phi_{x_1\ldots x_n}$.
Let $\Phi_n:\cM^n\rar{} \cM$ be the composition
\begin{equation}   \label{e:C6}
\cM^n=Z(B\bZ)^n\rar{} Z(\Ga_n )\rar{}Z(B\bZ )=\cM,
\end{equation}
where the first arrow comes from the fact that $Z$ is a pre-sTFT and the second one comes
from the bordism \eqref{e:n-th_bordism}. Define the tensor product on $\cM$ to be
$\Phi_2:\cM\times \cM \rar{} \cM$. The associativity constraint for the tensor product is obvious, and
the $n$-fold tensor product on
$\cM$ identifies with $\Phi_n$. If $Z$ is an incoming pre-sTFT then the bordism
\eqref{e:n-th_bordism} and the functor $\Phi_n$ are defined only for $n>0$.
If $Z$ is a pre-sTFT (or an outgoing pre-sTFT) then one also has the functor
$\Phi_0$; this is the unit object in $\cM$.

\mbr

Our next goal is to define the braiding and the twist. We will use two obvious
\begin{rems}\label{rems:computation_in_sCob}
Let $\cG$ be an arbitrary group.
\begin{enumerate}[(1)]
\item Giving a functor $\Ga_1=B\bZ\rar{}B\cG$ is the same as giving an element $g\in\cG$
(namely, $g$ is the image of $1\in\bZ$).
If functors $\Psi ,\Psi':B\bZ\rar{}B\cG$ correspond to $g,g'\in \cG$ then an isomorphism $\Psi\iso\Psi'$ is an
element $\ga\in\cG$ such that $\ga g\ga^{-1}=g'$.
 \sbr
\item Similarly, giving a functor $\Ga_2\rar{}B\cG$ is the same as giving a pair $(g_1,g_2)\in\cG^2$. For two pairs $(g_1,g_2),(g'_1,g'_2)\in\cG^2$, an isomorphism between the corresponding functors $\Ga_2\rar{}B\cG$ is the same thing as a pair $(\ga_1,\ga_2)\in\cG^2$ such that $\ga_1 g_1\ga_1^{-1}=g'_1$ and $\ga_2 g_2\ga_2^{-1}=g'_2$.
\end{enumerate}
\end{rems}

\begin{defin}\label{d:twist-general}
Consider the following autoequivalence of the bordism \eqref{e:n-th_bordism} with $n=1$:
\[
\xymatrix{
  \Ga_1 \ar[rr]^{f} \ar[dd]_{\Id} & & BF_1 \ar[dd]^{\Id} &  & B\bZ\ar[ll]_g \ar[dd]^{\Id} \\
   & & & \\
  \Ga_1 \ar[rr]^{f} & & BF_1 &  &B\bZ\ar[ll]_g
   }
\]
where the isomorphism $\Id\circ f\rar{\simeq}f\circ\Id$ is given by the element $x_1\in F_1$ (cf.~Remark \ref{rems:computation_in_sCob}(1)) and the isomorphism $\Id\circ g\rar{\simeq}g\circ\Id$ is the identity map. This autoequivalence defines an automorphism $\te$ of the functor $\Phi_1=\Id_{\cM}:\cM\rar{}\cM$, which is the \emph{twist} on $\cM$.
\end{defin}

\begin{defin}\label{d:braiding-general}
Consider the following autoequivalence of the diagram \eqref{e:n-th_bordism} with $n=2$:
\[
\xymatrix{
  \Ga_2 \ar[rr]^{f} \ar[dd]_{\tau} & & BF_2 \ar[dd]^{\nu} &  & B\bZ\ar[ll]_g \ar[dd]^{\Id} \\
   & & & \\
  \Ga_2 \ar[rr]^{f} & & BF_2 &  &B\bZ\ar[ll]_g
   }
\]
where $\tau$ interchanges the two copies of $B\bZ$, $\nu$ is induced by the homomorphism $F_2\rar{}F_2$ such that $x_1\mapsto x_2$ and $x_2\mapsto x_2^{-1}x_1 x_2$, the isomorphism $\nu\circ f\rar{\simeq}f\circ\tau$ is given by the pair $(1,x_2)\in F_2^2$ (cf.~Remark \ref{rems:computation_in_sCob}(2)), and the isomorphism $\nu\circ g\rar{\simeq}g\circ\Id$ is the identity. This autoequivalence defines a functorial isomorphism
\[
\be_{M_1,M_2}:M_1\tens M_2\rar{\simeq}M_2\tens M_1, \qquad M_1,M_2\in\cM,
\]
which is the \emph{braiding} on $\cM$.
\end{defin}

The fact that $( \cM ,\be ,\te )$ is indeed a braided category with a twist follows from
\S\ref{sss:surface_operad} and \S\ref{sss:destupidization}.

\subsection{The pre-sTFT associated to an algebraic stack}   \label{ss:The_pre-sTFT}
We fix a field $k$, and we will say ``stack" or ``scheme" instead of ``stack over $k$"
or ``scheme over $k$".

To any algebraic stack $\sX$ of finite type we will associate a pre-sTFT
$Z^-_{\sX}$ and an outgoing\footnote{If $\sX$ satisfies a certain condition
(which holds, e.g., for classifying stacks of unipotent groups) then the word
``outgoing" is unnecessary here, see Corollary~\ref{c:safe2} and Definition~\ref{d:safe}.} pre-sTFT  $Z_{\sX}$. By \S\ref{sss:Z(BbZ)}, each of the
categories $Z^-_{\sX}(B\bZ )$ and $Z_{\sX}(B\bZ )$ is monoidal and equipped with a braiding and a twist. If $\sX$ is the classifying stack of an algebraic group $G$  then $Z_{\sX}(B\bZ )=\sD_G(G)$
and $Z^-_{\sX}(B\bZ )$ is the bounded above derived category $\sD^-_G(G)$. Moreover,
the monoidal structure on $Z^-_{\sX}(B\bZ )=\sD^-_G(G)$ is defined by convolution with compact support (see Example~\ref{example:convolution} below).
So we get a braiding and a twist on the category
$\sD^-_G(G)$ (or $\sD_G(G)$\,) equipped with this monoidal structure.
In the case where $G^{\circ}$ is unipotent a braiding and a twist on $\sD_G(G)$ were
already defined in \S\ref{ss:ribbon_G}; it is straightforward to check that the two braidings and twists
are the same.


\subsubsection{The stack $\sX^{\Ga}$} \label{sss:X^Ga}

Let $\Ga$ be a groupoid and $\sX$ a stack. Define the stack $\sX^{\Ga}$ as follows:
for any scheme $S$, an $S$-point of $\sX^{\Ga}$ is a functor $\Ga\rar{}\sX (S)$, where
$\sX (S)$ is the groupoid of $S$-points of $\sX$.

Groupoids form a 2-category. So do stacks (see \cite{LM}). The 2-functor $(\sX,\Ga )\mapsto\sX^{\Ga}$ is covariant in $\sX$ and contravariant in $\Ga$.

\begin{rem}  \label{r:exactness}
Given a diagram of groupoids $\Ga_1\lar{}\Ga_{2}\rar{}\Ga_3$ one can form the
categorical pushout $\Ga =\Ga_1\bigsqcup_{\Ga_2}\Ga_3\,$.
 The above definition of $\sX^{\Ga}$ immediately implies that
 $\sX^{\Ga}$ is the fiber product of $\sX^{\Ga_{1}}$ and $\sX^{\Ga_{3}}$ over $\sX^{\Ga_2}$.
\end{rem}

\begin{example}   \label{example:[G]}
 Let $G$ be an algebraic group and $\sX=BG$. Let $A$ be an abstract group and $\Gamma=BA$.
 Then $\sX^{\Ga}$ is the quotient stack $\Hom (A,G)/G$, where $G$ acts on the
 scheme $\Hom (A,G)$ by conjugation. In particular, if $\Gamma=B\bZ$ then
 $\sX^{\Ga}$ is the quotient stack of $G$ by the adjoint action of $G$.
\end{example}

\begin{rem}
If $\Ga$ is the fundamental groupoid of a topological space $T$ and $\sX$ is the classifying
stack of an algebraic group $G$ then $\sX^{\Ga}$ is often called \emph{the stack of $G$-local systems} on $T$.
\end{rem}

Recall that according to \cite{LM}, a morphism (i.e., a 1-morphism) of  stacks
$f:\sX\rar{}\sY$ is said to be \emph{representable} if for any scheme $S$ equipped with
a morphism to $\sY$ the stack $\sX\times_{\sY}S$ is an algebraic space.

\begin{lem}  \label{l:X^Ga}
Let $\sX$ be an algebraic stack of finite type.
\begin{enumerate}[$($i$)$]
\item If $\Ga$ is a groupoid of finite presentation (see Definition~\ref{d:finite_type}) then
$\sX^{\Ga}$ is an algebraic stack of finite type.
 \sbr
\item If a functor $\Ga'\rar{}\Ga$ between groupoids of finite presentation is essentially surjective then the corresponding morphism $\sX^{\Ga}\to\sX^{\Ga'}$ is representable.
\end{enumerate}
\end{lem}

\begin{rem}
The above lemma and Lemmas~\ref{l:safe2},
\ref{l:properness} below remain valid if the automorphism groups
of the objects of $\Ga$ are assumed to be of finite type but not necessarily of finite presentation.
We do not need this fact.
\end{rem}

\begin{proof}[Proof of Lemma~\ref{l:X^Ga}]
Statement (i) can be deduced from (ii) as follows. First choose an essentially surjective
functor $I\rar{}\Ga$, where $I$ is a finite set (viewed as a discrete groupoid). Then use (ii) and the fact that $\sX^I$ is an algebraic stack of finite type.

To prove (ii), it suffices to consider the following two cases.
\begin{enumerate}[$($a$)$]
\item $\Ga$ is obtained from $\Ga'$ by freely adding an isomorphism $\ga_1\rar{\simeq}\ga_2$,
$\ga_1,\ga_2\in \Ga'$. In other words, $\Ga$ is the categorical pushout
$\Ga'\bigsqcup_{\{ 1, 2\}}{\{ 1\}}$, where the sets $\{ 1, 2\}$ and $\{ 1\}$ are considered as
groupoids and the functor $\{ 1, 2\}\rar{}\Ga'$ takes $i$ to $\ga_i$.
 \sbr
\item $\Ga$ is obtained from $\Ga'$ by killing some $f\in\Aut\ga$, $\ga\in\Ga$.  In other words, $\Ga$ is the categorical pushout $\Ga'\bigsqcup_{B\bZ}{\{ 1\}}$, where the functor $B\bZ\rar{}\Ga'$
takes the single object of $B\bZ$ to $\ga$ and the element $1\in\bZ$ to $f$.
(As before, ${\{ 1\}}$ is the groupoid with one object and one morphism.)
\end{enumerate}

In case (a) it suffices to use Remark~\ref{r:exactness} and the fact that the diagonal morphism
$\sX\rar{}\sX\times\sX$ is representable. In case (b) Remark~\ref{r:exactness} shows that it suffices
to prove the representability of the morphism $\al :\sX\rar{}\sX^{B\bZ}$ corresponding to the functor
$B\bZ\rar{}{\{ 1\}}$. But we have already considered case (a), so we know that
the morphism $\be :\sX^{B\bZ}\rar{}\sX$ corresponding to the functor
${\{ 1\}}\rar{}B\bZ$ is representable. Since $\be\al =\id_{\sX}$ it follows that $\al$ is representable.
\end{proof}

\subsubsection{The $\ell$-adic derived category of a stack}   \label{ss:review_of_LO}
In this subsection we follow Y.~Laszlo and M.~Olsson \cite{Las-Ols06}.
\begin{convention}
From now on all algebraic stacks are assumed to be of finite type over $k$.
By a morphism of stacks we mean a 1-morphism.
\end{convention}

For every algebraic stack $\sX$ Laszlo and Olsson \cite{Las-Ols06} define
the bounded derived category $D^b_c(\sX,\ql)$ and the unbounded derived categories
$D^-_c(\sX,\ql)$, $D^+_c(\sX,\ql)$. We will use the notation
\[
\sD (\sX ):=D^b_c(\sX,\ql),\quad \sD^- (\sX ):=D^-_c(\sX,\ql),\quad \sD^+ (\sX ):=D^+_c(\sX,\ql).
\]
Given a morphism $f:\sX\to\sY$ they define the functors $f^*,f^!:\sD (\sY)\to \sD (\sX)$
and similar functors for $\sD^-$ and $\sD^+$. They also define
$f_!:\sD^-(\sX)\to \sD^-(\sY)$ and $f_*:\sD^+(\sX)\to \sD^+(\sY)$.

The assignment $\sX\mapsto \sD^- (\sX )$,
$f\mapsto f_!$ is a 2-functor from the 2-category of algebraic stacks to that of triangulated
categories. With obvious changes, this is also true for $f^*$, $f^!$, and $f_*$.
One also has base change isomorphisms, just as for schemes.

In general, $f_!$ and $f_*$ do not map $\sD (\sX)$ to $\sD (\sY)$
(e.g., take $\sX$ to be the classifying stack of $\bG_m$ and $\sY=\Spec k$).
However, this phenomenon does not occur for the following class of morphisms.

\begin{defin}  \label{d:safe}
An algebraic stack $\sX$ is {\em safe\,} if for every geometric point $x$ of $\sX$ the algebraic group
$(G_x)^{\circ}_{\red}$ is unipotent (here $G_x$ is the automorphism group of $x$
and $(G_x)^{\circ}_{\red}$ is the neutral component of the reduced scheme $(G_x)_{\red}$).
A morphism of  algebraic stacks  is {\em safe\,} if all its fibers are.
\end{defin}

\begin{rems}   \label{r:safe}
\begin{enumerate}[$($i$)$]
\item Representable morphisms are safe.
 \sbr
\item Morphisms from a safe stack to any algebraic stack are safe.
\end{enumerate}
\end{rems}

\begin{lem}  \label{l:safe1}
If a morphism $f:\sX\to\sY$ of  algebraic stacks is safe then
$f_!:\sD^-(\sX)\to \sD^-(\sY)$ and $f_*:\sD^+(\sX)\to \sD^+(\sY)$
map $\sD (\sX)$ to $\sD (\sY)$. 
\end{lem}

\begin{proof}
By base change and  \cite[Theorem 11.5]{LM}, it suffices to consider the case where
$\sY=\Spec k$ and $\sX$ is the classifying stack of a group scheme $G$ such that
$G^{\circ}_{\red}$ is unipotent.
\end{proof}

\subsubsection{The theory $Z^-_{\sX}$}   \label{sss:unboundedZ}
For any algebraic stack $\sX$ of finite type we will define a pre-sTFT \,$Z^-_{\sX}:\sCob\rar{}\Cat$.

By definition (see \S\ref{sss:sCob}), an object of $\sCob$ is a groupoid $\Ga$ of finite presentation.
Set $Z^-_{\sX}(\Ga ):= \sD^-(\sX^{\Ga})$.

Now let us define $Z^-_{\sX}$ on 1-morphisms. Recall that a 1-morphism in $\sCob$ is a bordism of groupoids, i.e., a diagram of groupoids of finite presentation
\begin{equation}   \label{e:bordism_groupoids}
\Ga_1\rar{}\Ga_{12}\lar{}\Ga_2 .
\end{equation}
This diagram defines a correspondence
\begin{equation}   \label{e:correspondence}
\xymatrix{
& \ar[dl]_f \ar[dr]^g \sX^{\Ga_{12}} & \\
\sX^{\Ga_{1}}& & \sX^{\Ga_{2}}\\
}
\end{equation}
and therefore a functor
\begin{equation}    \label{e:problematic}
g_!f^*:\sD^-(\sX^{\Ga_1})\rar{} \sD^-(\sX^{\Ga_2}).
\end{equation}



Recall that the composition of bordisms $\Ga_1\rar{}\Ga_{12}\lar{}\Ga_2$ and
 $\Ga_2\rar{}\Ga_{23}\lar{}\Ga_3$ is the bordism $\Ga_1\rar{}\Ga_{13}\lar{}\Ga_3$, where
 $\Ga_{13}$ is the categorical pushout $\Ga_{12}\bigsqcup_{\Ga_2}\Ga_{23}\,$.
Thus we have a commutative diagram of stacks
$$\xymatrix{
&&\sX^{\Ga_{13}} \ar[dl]_{\bar f}\ar[dr]^{\bar g}&&\\
& \ar[dl]_f \ar[dr]^g \sX^{\Ga_{12}}&&\sX^{\Ga_{23}} \ar[dl]_{f'} \ar[dr]^{g'}\\
\sX^{\Ga_1}& & \sX^{\Ga_2} & &\sX^{\Ga_3}
}
$$
in which the square is Cartesian by Remark~\ref{r:exactness}.
So the base change isomorphism $(f')^*g_!\rar{\simeq}\bar g_!\bar f^*$
provides a canonical isomorphism between the composition
\[
\sD^-(\sX^{\Ga_1})\rar{g_!f^*} \sD^-(\sX^{\Ga_2})\rar{g'_!(f')^*} \sD^-(\sX^{\Ga_3})
\]
and the functor $g'_! \bar g_!\bar f^*f^*:\sD^-(\sX^{\Ga_1})\rar{}\sD^-(\sX^{\Ga_3})$.

Finally, if $\Ga$ is a disjoint union of $\Ga_1,\ldots ,\Ga_n$ then $\sX^{\Ga}=\prod\limits_i \sX^{\Ga_i}$,
so we get a canonical functor
\[
\prod_i Z^-_\sX(\Ga_i)=\prod_i \sD^-(\sX^{\Ga_i})\rar{\boxtimes} \sD^-(\prod_i \sX^{\Ga_i})= \sD^-(\sX^{\Ga})=Z^-_\sX(\bigsqcup_i \Ga_i).
\]

\subsubsection{The theory $Z_{\sX}$}       \label{sss:boundedZ}
For $\Ga\in\sCob$ set $Z_{\sX}({\Ga}):=\sD (\sX^{\Ga})$; this is a full subcategory of $Z^-_{\sX}({\Ga})$.
To define $Z_{\sX}$ as a 2-functor, we have to ensure that the functor \eqref{e:problematic} preserves
the class of bounded complexes. By Lemma~\ref{l:safe1}, this is true if the morphism $g$ in diagram
\eqref{e:correspondence} is safe in the sense of Definition~\ref{d:safe}.

\begin{lem}  \label{l:safe2}
Let $\sX$ be an algebraic stack of finite type and $\alpha :\Ga'\rar{}\Ga$ a functor between groupoids of finite presentation. Suppose that either $\sX$ is safe or $\alpha$ is essentially surjective.
Then the morphism $\sX^{\Ga}\to\sX^{\Ga'}$ induced by $\alpha$ is safe.
\end{lem}

\begin{proof}
Just as in the proof of Lemma~\ref{l:X^Ga}(i), one shows that if  $\sX$ is safe then so is $\sX^{\Ga}$. By
Remark~\ref{r:safe}(ii), this implies that the morphism $\sX^{\Ga}\to\sX^{\Ga'}$ is safe.

If $\alpha$ is essentially surjective then use Lemma~\ref{l:X^Ga}(ii) and Remark~\ref{r:safe}(i).
\end{proof}

\begin{cor}    \label{c:safe2}
If $\sX$ is safe then $Z_{\sX}$ is a well-defined pre-sTFT. If $\sX$ is any algebraic stack of finite type then $Z_{\sX}:\sCobout\rar{}\Cat$ is an outgoing pre-sTFT. \qed
\end{cor}

\subsubsection{Examples of functors \eqref{e:problematic}}
\begin{example}   \label{example:convolution}
Let $\sX =BG$, where $G$ is an algebraic group. Then the functor \eqref{e:problematic} corresponding to the diagram \eqref{e:n-th_bordism} equals
$\mu_!:\sD^-_{G^n}(G^n)\rar{}\sD^-_G(G)$, where $\mu :G^n\rar{}G$ is the map
$(g_1,\ldots , g_n)\mapsto g_1\ldots g_n$. So the composition
\[
(\sD_G^- (G))^n\rar{\boxtimes}\sD^-_{G^n}(G^n)\rar{\eqref{e:problematic}}\sD_G^- (G)
\]
is the convolution with compact support $(M_1,\ldots ,M_n)\mapsto M_1*\ldots * M_n$.
\end{example}

The previous example was based on diagram \eqref{e:n-th_bordism}. One can consider
 \eqref{e:n-th_bordism} as the diagram of fundamental groupoids corresponding to the  bordism between $nS^1$ and $S^1$ given by the sphere with $n+1$ holes (here $nS^1$ stands for the disjoint union of $n$ copies of the standard circle $S^1$). In the next example
 we consider a more general situation of a bordism between $mS^1$ and $nS^1$ given by a connected compact oriented surface 
of genus $g$ with $m+n$ holes.



\begin{example}   \label{example:connected_oriented}
As before, let $\sX =BG$, where $G$ is an algebraic group.
Let $\pi$ be the group with generators $A_1,B_1,\ldots A_g,B_g, x_1,\ldots ,x_m,y_1,\ldots ,y_n$
and the defining relation
\begin{equation}
x_1\ldots x_m A_1B_1A_1^{-1}B_1^{-1}\ldots A_gB_gA_g^{-1}B_g^{-1}=y_1\ldots y_n \, .
\end{equation}
For each $i\in\{ 1,\ldots , m\}$ consider the homomorphism $\bZ\rar{}\pi$ such that $1\mapsto x_i\,$.
For each $j\in\{ 1,\ldots , n\}$ consider the homomorphism $\bZ\rar{}\pi$ such that $1\mapsto y_j\,$.
These homomorphisms define a diagram of groupoids
\[
\underbrace{B\bZ\bigsqcup\ldots \bigsqcup B\bZ}_m\rar{}B\pi\lar{}
\underbrace{B\bZ\bigsqcup\ldots \bigsqcup B\bZ}_n\ .
\]
This is a bordism in $\sCob$, which clearly comes from a bordism in $\Cob$ (the definitions of $\Cob$ and $\sCob$ were given in \S\ref{sss:Cob} and \S\ref{sss:sCob}).
The functor \eqref{e:problematic} corresponding to this bordism is the composition
\begin{equation}   \label{e:any_genus}
\sD^-_{G^m}(G^m)\rar{f^*} \sD^-_G(M)\rar{h_!} \sD^-_G(G^n)\rar{\av} \sD^-_{G^n}(G^n) \, .
\end{equation}
Here $M$ is the variety of homomorphisms $\pi\rar{}G$ (on which $G$ acts by conjugation),
the map $f:M\rar{}G^m$ (resp. $h:M\rar{}G^n$) corresponds to $x_1,\ldots ,x_m\in\pi$
(resp.  to $y_1,\ldots ,y_n\in\pi$), and $\av =\av_{G^n/G}$ is the functor of averaging with compact
support (see Definition~\ref{d:averaging}).
\end{example}

\begin{rems}  \label{r:surface_operad}
\begin{enumerate}[(i)]
\item In the case $G=PGL(N)$, $m=n=0$ the composition \eqref{e:any_genus}  was studied in \cite{HRV}.

\mbr

\item
Composing $\boxtimes :(\sD_G^- (G))^m\rar{}\sD^-_{G^m}(G^m)$ with a functor of the form  \eqref{e:any_genus} for $n=1$, one gets a functor  $(\sD_G^- (G))^m\rar{}\sD_G^- (G)$. All such functors
form an action of the surface operad on  $\sD_G^- (G)$. We already mentioned it in \S\ref{sss:surface_operad}.
\end{enumerate}
\end{rems}

In the next examples we consider some 1-morphisms in $\sCob$ that do not come from $\Cob$.


\begin{example}   \label{example:usual_product}
Consider the following diagram \eqref{e:bordism_groupoids}:  $\Ga_2=\Ga_{12}=B\bZ$, the functor
$\Ga_2\rar{}\Ga_{12}$ is the identity, $\Ga_1$ is the disjoint union of $n$ copies of $B\bZ$, and the
restriction of the functor $\Ga_1\rar{}\Ga_{12}$ to each copy of $B\bZ$ is the identity. Let $\sX$ be any
algebraic stack of finite type over $k$. Set $\sY:=\sX^{B\bZ}$.
Then \eqref{e:problematic} is the functor  $\De^*:\sD^- (\sY^n)\rar{}\sD^- (\sY)$, where
$\De :\sY\rar{}\sY^n$ is the diagonal morphism. So the composition
\[
\sD^- (\sY)^n\rar{\boxtimes}\sD^- (\sY^n)\rar{\eqref{e:problematic}}\sD^- (\sY)
\]
is the usual tensor product $(M_1,\ldots ,M_n)\mapsto M_1\otimes\ldots\otimes M_n$. Note that if $\sX$ is the classifying stack of an algebraic group $G$ then $\sD^- (\sY)=\sD^-_G(G)$, $\sD^- (\sY^n)=\sD^-_{G^n}(G^n)$.
\end{example}

\begin{example}   \label{example:Adams}
Consider the following diagram \eqref{e:bordism_groupoids}:  $\Ga_1=\Ga_2=\Ga_{12}=B\bZ$, the
functor $\Ga_2\rar{}\Ga_{12}$ is the identity, and the functor $\Ga_1\rar{}\Ga_{12}$ comes from
the homomorphism $\bZ\rar{}\bZ$ given by $n\mapsto mn$. Let $\sX =BG$, where $G$ is an algebraic group. Then \eqref{e:problematic} is the functor $\psi_m^*: \sD_G(G)\to\sD_G(G)$,
where $\psi_m:G\to G$ is the map $g\mapsto g^m$. Note that if $m=-1$ then $\psi_m^*$ comes from
a \emph{non-oriented bordism} between 1-manifolds.
\end{example}

\begin{rem}   \label{r:strong_or_stupid}
Examples~\ref{example:convolution} and \ref{example:usual_product} show that the pre-sTFT
$Z^-_{BG}$ encodes both the convolution on $Z^-_{BG} (B\bZ )=\sD^-_G(G)$ and the usual
tensor product. The pre-TFT corresponding to $Z^-_{BG}$ encodes the convolution but not the tensor product.
Probably this means that from the representation theorist's point of view, the pre-TFT is more adequate
than the pre-sTFT.
\end{rem}

\subsection{A ``lax" functoriality of $Z^-_{\sX}$ in $\sX$}   \label{ss:lax_functoriality}
We will show that a separated morphism $f:\sX\rar{}\sY$ between algebraic stacks of finite type
induces a ``lax 1-morphism" $f_!:Z^-_{\sX ,\In}\rar{}Z^-_{\sY ,\In}$ in the sense of
\S\ref{sss:thelaxnotion}, where $Z^-_{\sX ,\In}$ is the incoming pre-sTFT that one gets by restricting $Z^-_{\sX}$ to $\Cobin$. This implies that
$f_!^{B\bZ}:\sD^-(\sX^{B\bZ})\rar{}\sD^-(\sY^{B\bZ})$ is a weakly semigroupal\footnote{``Weakly" is related to ``lax", and ``semigroupal" (as opposed to ``monoidal") is related to ``incoming''.} functor compatible with the braidings and twists. In particular, this holds for
$\ig : \sD^-_{G'}(G')\rar{}\sD^-_G(G)$, where $G$ is an algebraic group and $G'\subset G$ is a closed subgroup.

\subsubsection{Lax 1-morphisms between pre-sTFT's}    \label{sss:thelaxnotion}

\begin{predef}   \label{d:informal-lax}
Let $Z,Z':\sCob\rar{}\Cat$ be pre-sTFT's. A \emph{lax 1-morphism} $F:Z\rar{}Z'$ is the following collection of data:

\begin{enumerate}[(i)]
\item for each $\Ga\in\sCob$, a functor $F^{\Ga}:Z(\Ga )\rar{}Z'(\Ga )$;

\mbr
\item  for each 1-morphism  $\al :\Ga_1\rar{}\Ga_2$ in $\sCob$, a morphism
\begin{equation}    \label{e:to_be_constructed}
\xi_{\al}:Z'(\al )\circ F^{\Ga_1}\rar{}F^{\Ga_2}\circ Z(\al )
\end{equation}
(note that both $Z'(\al )\circ F^{\Ga_1}$ and $F^{\Ga_2}\circ Z(\al )$ are
functors $Z(\Ga_1 )\rar{}Z'(\Ga_2 )$);

\mbr
\item for any $n\ge 0$ and any $\Ga_1,\ldots ,\Ga_n\in\sCob$,  a morphism
from the composition
$\prod\limits_i Z(\Ga_i)\rar{} \prod\limits_i Z'(\Ga_i)\rar{} Z'(\bigsqcup\limits_i \Ga_i)$
to the composition
$\prod\limits_i Z(\Ga_i)\rar{} Z(\bigsqcup\limits_i \Ga_i)\rar{} Z'(\bigsqcup\limits_i \Ga_i)$.
\end{enumerate}
These data should satisfy certain compatibility conditions. In particular, data (i)-(ii) should define
a \emph{lax natural transformation\footnote{According to \cite[\S 2]{Kelly74}, this means the following.
First, $\xi_{\al}$ should be functorial in $\al$ (this condition makes sense because all 1-morphisms
$\al :\Ga_1\rar{}\Ga_2$ in $\sCobin$ form a category and $Z(\al )$, $Z'(\al )$ depend functorially on $\al$.). Second, the assignment
$\al\mapsto\xi_{\al}$ should be compatible with the composition of $\al$'s and if $\Ga_1=\Ga_2$,
$\al=\Id$ then one should have $\xi_{\al}=\Id$.}} between 2-functors $Z$ and $Z'$.
\end{predef}

\sbr

A complete definition of lax 1-morphism can be concisely formulated in terms of \S\ref{ss:Segal}:
namely, a pre-sTFT is a 2-functor $\sCob^{\otimes}\rar{}\Cat$ with the Segal property, and a lax
1-morphism is a lax natural transformation between such functors.

Similarly, one defines the notion of lax 1-morphism between incoming pre-sTFT's
(or, say, outgoing pre-TFT's).

\begin{rem}   \label{r: weakly_semigroupal}
If $F:Z\rar{}Z'$ is a lax 1-morphism between incoming pre-sTFT's then
$F^{B\bZ}:Z(B\bZ )\rar{}Z'(B\bZ )$ has a natural structure of weakly semigroupal functor in the sense of Definition~\ref{d:weak-semigroupal} (``weakly" corresponds to ``lax", and ``semigroupal" corresponds to ``incoming''). This weakly semigroupal functor is compatible with the braidings and twists.
\end{rem}

\subsubsection{$f_!$ as a lax 1-morphism}    \label{sss:f!lax}
Let $f:\sX\rar{}\sY$ be a separated morphism between algebraic stacks of finite type.
Let $Z^-_{\sX ,\In}$ denote the restriction of $Z^-_{\sX}$ to $\sCobin\,$; this is an incoming pre-sTFT. We will define a lax 1-morphism $f_!: Z^-_{\sX ,\In}\rar{}Z^-_{\sY ,\In}$.
For any groupoid $\Ga\in\sCobin$ one has a morphism $f:\sX^{\Ga}\rar{}\sY^{\Ga}$ and therefore a functor $f^{\Ga}_!:Z_{\sX}^-(\Ga )\rar{}Z_{\sY}^-(\Ga )$ (recall that $Z_{\sX}^-(\Ga ):=\sD^-(\sX^{\Ga})$).
Thus one has datum (i) from Pre-definition~\ref{d:informal-lax}. Datum (iii) is the K\"unneth morphism
\[
f_!^{\Ga_1}M_1\boxtimes\ldots\boxtimes f_!^{\Ga_n}M_n\rar{\simeq}
f_!^{\Ga}(M_1\boxtimes\ldots\boxtimes M_n), \quad\quad M_i\in\sD^-(\sX^{\Ga_i}),
\]
where $\Ga:=\bigsqcup\limits_i\Ga_i$ (and therefore $f^{\Ga}$ is a morphism
$\prod\limits_i\sX^{\Ga_i}\rar{}\prod\limits_i\sY^{\Ga_i}$).

To define datum (ii), we will use that $f:\sX\rar{}\sY$ is separated.
By  \cite[Definition~7.6]{LM}, this means that the diagonal morphism
$\sX\rar{}\sX\times_{\sY}\sX$ is proper. The next lemma is proved just as Lemma~\ref{l:X^Ga}(ii).

\begin{lem}  \label{l:properness}
Let $f:\sX\rar{}\sY$ be a morphism of algebraic stacks of finite type and $\Phi:\Ga'\rar{}\Ga$  a
functor between groupoids of finite presentation. If $f$ is separated and
$\Phi$ is essentially surjective then the morphism
$\nu:\sX^{\Ga}\to\sX^{\Ga'} \times_{\sY^{\Ga'}}\sY^{\Ga}$ corresponding to $f$ and $\Phi$ is proper.
If, in addition, $f$ is representable then $\nu$ is a closed embedding.~\hfill\qedsymbol
\end{lem}

Now let us construct the morphism \eqref{e:to_be_constructed} corresponding to a 1-morphism $\al$
in $\sCobin$. Such $\al$ is, in fact, a  diagram of groupoids of finite presentation
$\Ga_1\twoheadrightarrow\Ga\leftarrow\Ga_2$. We have to construct a canonical morphism
\begin{equation}   \label{e:tobe_constructed}
\xi_{\al}:Z^-_{\sY}(\al )\circ f^{\Ga_1}_!\rar{}f^{\Ga_2}_!\circ Z^-_{\sX}(\al ),
\end{equation}
i.e., a morphism from the composition
\begin{equation} \label{e:compos1}
\sD^-(\sX^{\Ga_1}) \rar{f^{\Ga_1}_!}\sD^-(\sY^{\Ga_1})\rar{Z^-_{\sY}(\al )}\sD^-(\sY^{\Ga_2})
 \end{equation}
 to the composition
 \begin{equation} \label{e:compos2}
 \sD^-(\sX^{\Ga_1}) \rar{Z^-_{\sX}(\al )}\sD^-(\sX^{\Ga_2})\rar{f^{\Ga_2}_!}\sD^-(\sY^{\Ga_2}) .
 \end{equation}
Consider the diagram
\begin{equation}
\xymatrix{
     &&\sX^{\Ga} \ar[d]^{\nu} \ar[r]^{\tilde v} &\sX^{\Ga_2}\ar[dd]^{f^{\Ga_2}} \\
\sX^{\Ga_1}\ar[d]_{f^{\Ga_1}}&&\sX^{\Ga_1}\times_{\sY^{\Ga_1}}\sY^{\Ga}\ar[d]^p\ar[ll]_{\tilde u}& \\
\sY^{\Ga_1}&&\sY^{\Ga}\ar[ll]_u\ar[r]^v&\sY^{\Ga_2}
}
\end{equation}
in which $\nu$ is proper by Lemma~\ref{l:properness}. The compositions~\eqref{e:compos1} and \eqref{e:compos2} equal, respectively, $v_!u^*f^{\Ga_1}_!$ and
$f^{\Ga_2}_!\tilde v_!\nu^*\tilde u^*$. The required morphism from $v_!u^*f^{\Ga_1}_!=v_!p_!\tilde u^*$ to $f^{\Ga_2}_!\tilde v_!\nu^*\tilde u^*=v_!p_!\nu_!\nu^*\tilde u^*$
comes from the adjunction $\Id\to \nu_*\nu^*=\nu_!\nu^*$.

\subsubsection{The functor $\ig$}   \label{sss:ig}
\begin{example}     \label{example:what_we_really_need}
Let $G$ be an algebraic group and $G'\subset G$ a closed subgroup.
Let $\sX:=BG'$ and $\sY:=BG$ be the classifying stacks and $f:\sX\rar{}\sY$ the
natural morphism. Then $Z^-_{\sX}(B\bZ)=\sD^-(\sX^{B\bZ})=\sD^-_{G'}(G')$, so $f^{B\bZ}_!$ is a functor $\sD^-_{G'}(G')\rar{}\sD^-_G(G)$. Note that $f$ is representable: indeed, after base change
$\Spec k\rar{}BG$ it becomes the morphism $G/G'\rar{}\Spec k$.
So by Lemma~\ref{l:safe1}, $f^{B\bZ}_!$ maps $\sD_{G'}(G')$ to $\sD_{G}(G)$.
\end{example}

\begin{defin}     \label{d:induction_any_group}
Each of the functors $\sD^-_{G'}(G')\rar{}\sD^-_G(G)$ and $\sD_{G'}(G')\rar{}\sD_G(G)$
 from Example~\ref{example:what_we_really_need} is called \emph{induction with compact support} and denoted by $\ig\,$.
\end{defin}

 The morphism $f:BG'\rar{}BG$ from Example~\ref{example:what_we_really_need}
is separated, so combining the construction from \S\ref{sss:f!lax} with
Remark~\ref{r: weakly_semigroupal} one gets the following

\begin{cor}   \label{c:what_we_struggled_for}
Each of the functors
\[
\ig :\sD^-_{G'}(G')\rar{}\sD^-_G(G),\quad\quad \ig :\sD_{G'}(G')\rar{}\sD_G(G)
\]
has a canonical structure of weakly semigroupal functor compatible with the braidings and twists.
\end{cor}

\subsubsection{Conclusion}    \label{sss:conclusion}
If $G$ is unipotent a functor $\ig :\sD_{G'}(G')\rar{}\sD_G(G)$ and a weak semigroupal
structure on it were defined already in \S\ref{ss:induction-functors} and  \S\ref{ss:weak-semigroupal}.
It is easy to see that this weakly semigroupal  functor is equal to the functor
$\ig :\sD_{G'}(G')\rar{}\sD_G(G)$ from Corollary~\S\ref{c:what_we_struggled_for}.
Moreover, the construction of $\ig$ given in \S\ref{sss:f!lax}-\ref{sss:ig} is essentially identical to the one
given in \S\ref{ss:induction-functors},
\S\ref{ss:f-d-compatibility-averaging}-\ref{ss:averaging-semigroupal},
and \S\ref{ss:weak-semigroupal}, the only difference being the language
used.\footnote{Note that the key construction of the morphism \eqref{e:avg-external} is based on the equality $\De_*=\De_!$ used in step 3 of the construction, i.e., on the separatedness of $G/G'$
(which is equivalent to the separatedness of the morphism $BG'\rar{}BG$).} A serious advantage of
 the language used in this section is that it makes the compatibility of $\ig$ with the braidings and twists obvious: this compatibility immediately follows from the fact that the morphism \eqref{e:tobe_constructed} is functorial in $\al$.

\subsection{Grothendieck-Verdier duality and ribbon structure on $\sD (\sX^{B\bZ})$} \label{ss:GV-stacks}


Let $\sX$ be an algebraic stack  of finite type over $k$ and $\cM:=Z_{\sX}(B\bZ )=\sD (\sX^{B\bZ})$; e.g., if
$\sX$ is the classifying stack of an algebraic group $G$  then $\cM =\sD_G (G)$.
In \S\ref{sss:Z(BbZ)} and \S\ref{ss:The_pre-sTFT} we defined a canonical braided monoidal structure and twist
on $\cM$ using the structure of \emph{outgoing} pre-sTFT on $Z_{\sX}$.
Our next goal is to define a canonical structure of ribbon Grothendieck-Verdier on $\cM$.
The construction giben in \S\ref{sss:counit functor}-\ref{sss:ribbon-stacks} below uses $Z^-_{\sX}$ as well as $Z_{\sX}\subset Z^-_{\sX}$. The advantage of
$Z^-_{\sX}$ is that it is a \emph{``full"} pre-sTFT (not merely an outgoing one).

\begin{rem}
If the stack $\sX$ is safe in the sense of Definition~\ref{d:safe}
(e.g., if $\sX =BG$, where $G^{\circ}$ is unipotent) then considering $Z^-_{\sX}$ is not necessary
because by Corollary~\ref{c:safe2}, already $Z_{\sX}$ is a pre-sTFT.
\end{rem}

\subsubsection{The counit functor}   \label{sss:counit functor}
Set $\cM_0: =Z^-_{\sX} (\varnothing )=\sD^-(\Spec k)$. Let $\varepsilon :\cM\to\cM_0$ be the functor corresponding to the bordism\footnote{This bordism corresponds to the bordism
$S^1\subset\mbox{\{disk\} }\supset \varnothing$ in $\Cob$.}
 \begin{equation}           \label{e:zakleivanie}
B\bZ\to B(0)\leftarrow\varnothing
\end{equation}
in $\sCob$.
Explicitly,
 \begin{equation}           \label{e:counit}
\varepsilon (M)=R\Gamma_c (\sX ,1^*M ), \quad\quad M\in \cM=\sD (\sX^{B\bZ}),
\end{equation}
where $1:\sX\to\sX^{B\bZ}$ comes from the homorphism $\bZ\to 0$. (Note that if $\sX =BG$ then
$1:\sX\to\sX^{B\bZ}$ is obtained from $1:\Spec k\to G$ by passing to the quotient with respect to the
action of $G$ by conjugation.) Informally, we think of the functor $\varepsilon :\cM\to\cM_0$ as a
``counit" or ``augmentation".

\subsubsection{The dualizing object in $\cM$}   \label{sss:Dualizing-stacks}
Set $K_{\cM_0}:= \ql\in\cM_0$. There is a unique object $K_{\cM}\in\cM$ such that
\begin{equation}   \label{e:functor_representing_K}
\Hom (M,K_{\cM} ) =\Hom (\varepsilon (M),K_{\cM_0}), \quad\quad M\in \cM .
\end{equation}
Explicitly,
\begin{equation}   \label{e:K_itself}
K_{\cM}=1_*K_{\sX},
\end{equation}
where $K_{\sX}\in\sD (\sX)$ is the dualizing object.

\begin{example}      
If $\sX =BG$ then $K_{\cM}\simeq\e_{\cM}[-2d]$, where $d:=\dim G=-\dim BG$.
\end{example}

The next lemma is similar to Lemma \ref{l:duality-on-D(G)}.

\begin{lem}  \label{l:GV-stacks}
$K_{\cM}$ is a dualizing object in $\cM$. The corresponding dualizing functor is
$\bD^-:=\bD\circ\iota^*=\iota^*\circ\bD$, where $\bD:\sD (\sX^{B\bZ})\iso \sD (\sX^{B\bZ})$ is the Verdier duality functor and $\iota\in\Aut (\sX^{B\bZ})$ corresponds to $-1\in\Aut (\bZ )$. 
\end{lem}

\begin{proof}
We have to construct a functorial isomorphism
\begin{equation}    \label{e:duality_on_M}
\Hom (M_1*M_2,K_{\cM} )\iso \Hom (M_1\otimes \iota^* M_2,K_{\sX^{B\bZ}}),
\; M_i\in \cM =\sD (\sX^{B\bZ}),
\end{equation}
where $K_{\sX^{B\bZ}}\in\sD(\sX^{B\bZ})$ is the dualizing object. It follows from \eqref{e:functor_representing_K} that for $M_1,M_2\in \cM =\sD (\sX^{B\bZ})$ one has
\begin{equation}
\Hom (M_1*M_2,K_{\cM} ) =\Hom (N,K_{\cM_0}), \quad
N:=R\Gamma_c (\sX^{B\bZ} ,M_1\otimes \iota^* M_2).
\end{equation}
By usual Verdier duality, $\Hom (N,K_{\cM_0})=\Hom (M_1\otimes \iota^* M_2,K_{\sX^{B\bZ}})$.
\end{proof}

\begin{rem}  
If $\sX =BG$ then the functors $\bD^-$ and $\bD$ from Lemma~\ref{l:GV-stacks} differ by a shift from the functors $\bD_G^-:\sD_G (G)\iso\sD_G (G)$ and $\bD_G:\sD_G (G)\iso\sD_G (G)$ used in the main part of the article. This difference is not essential for our purposes. In particular, Proposition~\ref{p:ribbon-stacks} below implies that the braiding and twist from \S\ref{sss:Z(BbZ)} make $\cM =\bD_G (G)$ into a ribbon category even if one uses $\e_G$ rather than $ K_{\cM}$ as a dualizing object in $\cM$.
\end{rem}

\subsubsection{The pivotal structure on $\cM$}   \label{sss:pivotal-stacks}
As before, let $\iota\in\Aut (\sX^{B\bZ})$ denote the automorphism corresponding to $-1\in\Aut (\bZ )$.
For $M_1,M_2\in\cM$ let
\begin{equation}  \label{e:def_of_psi}
\psi_{M_1,M_2}:\Hom (M_1*M_2,K_{\cM})\iso\Hom (M_2*M_1,K_{\cM})
\end{equation}
be the isomorphism corresponding via \eqref{e:duality_on_M} to the isomorphism
$$\iota^*:\Hom (M_1\otimes \iota^* M_2,K_{\sX^{B\bZ}})\iso \Hom (\iota^* M_1\otimes M_2,K_{\sX^{B\bZ}}).$$

\begin{prop}\label{p:pivotal-stacks}
\begin{enumerate}[$($a$)$]
\item The isomorphism \eqref{e:def_of_psi} is a pivotal structure on $(\cM ,K_{\cM})$ (see Definition~\ref{d:pivotal1}).
 \sbr
\item The isomorphism $(\bD^-)^2\iso\Id$ corresponding to this pivotal structure by Remark~\ref{r:pivotal2} and Lemma~\ref{l:GV-stacks} is equal to the obvious isomorphism $$(\bD^-)^2=(\bD\circ\iota^*)^2\iso \bD^2\circ(\iota^*)^2\iso\Id.$$
\end{enumerate}
\end{prop}

The proof will be given in \S\S\ref{sss:proof-p:pivotal-stacks-a}--\ref{sss:proof-p:pivotal-stacks-b}. One can also deduce Pro\-po\-si\-tion~\ref{p:pivotal-stacks}(a) from Proposition \ref{p:ribbon-stacks} and Lemma \ref{l:bijection-pivotal-structures-twists}.

\subsubsection{The ribbon structure on $\cM$}   \label{sss:ribbon-stacks}
By \S\ref{sss:Z(BbZ)}, $\cM$ is a braided category equipped with a twist $\te$. In \S\ref{ss:ribbon_GV} we defined the notion of ribbon structure.
\begin{prop} \label{p:ribbon-stacks}
\begin{enumerate}[$($a$)$]
\item $\te$ corresponds (in the sense of Lemma~\ref{l:bijection-pivotal-structures-twists}) to the pivotal structure from \S\ref{sss:pivotal-stacks}.
 \sbr
\item $\te$ defines a ribbon structure on $(\cM ,K_{\cM})$.
\end{enumerate}
\end{prop}

The proof will be given in \S\S\ref{sss:proof-p:ribbon-stacks-a}--\ref{sss:proof-p:ribbon-stacks-b}.

\subsubsection{A formula for $\Hom (M_1*\ldots *M_n,K_{\cM} )$, $M_i\in\cM$}   \label{sss:A_formula}
Consider the following
bordism\footnote{It corresponds to the following bordism in $\Cob$:
$\underbrace{S^1\sqcup\dotsb\sqcup S^1}_n\subset\{S^2\text{ with } n \text{ holes}\} \supset \varnothing$ .} in
$\sCob$:
\begin{equation}   \label{e:thebordism}
\Ga_n\rar{f'} BF'_n\lar{}\varnothing \, ,
 \quad\quad \Ga_n:=\underbrace{B\bZ\bigsqcup\ldots \bigsqcup B\bZ}_n \, ,
\end{equation}
where $F'_n$ is the group generated by $x_1,\ldots ,x_n$ with the defining relation $x_1\cdot\ldots\cdot x_n=1$
and the restriction of $f'$ to the $i$-th copy of $B\bZ$ takes $1\in B\bZ$ to $x_i$.

\mbr

Since $Z^-_{\sX}$ is a pre-sTFT the bordism \eqref{e:thebordism} defines a functor
\[
\Phi'_n:\cM^n=(Z_{\sX}(B\bZ))^n\hookrightarrow(Z^-_{\sX}(B\bZ))^n\to Z^-_{\sX}(\varnothing )=\cM_0 \, .
\]
\begin{lem}   \label{l:Hom(XYZ,K)}
One has a functorial isomorphism
\[
\Hom (M_1*\ldots *M_n,K_{\cM} )\iso\Hom\bigl(\Phi'_n(M_1,\ldots ,M_n),K_{\cM_0}\bigr), \qquad M_i\in\cM \, .
\]
\end{lem}

\begin{proof}
By \eqref{e:functor_representing_K}, it suffices to check that
$\varepsilon (M_1*\ldots *M_n)=\Phi'_n(M_1,\ldots ,M_n)$.
This is clear since composing the bordisms  \eqref{e:n-th_bordism} and \eqref{e:zakleivanie}
one gets 
\eqref{e:thebordism}.
\end{proof}


\subsubsection{Proof of Proposition \ref{p:pivotal-stacks}(a)}\label{sss:proof-p:pivotal-stacks-a}

\begin{rems}   \label{r:cyclic_action}
\begin{enumerate}[(i)]
\item The diagram \eqref{e:thebordism} is acted upon by the cyclic subgroup $C_n$ of the symmetric group $S_n$ generated by the cycle $(2,3,\ldots ,n,1)$. Namely, $C_n$ acts on $\Ga_n$ (respectively, on $F'_n$) by permuting the $n$ copies of $B\bZ$ (respectively, the generators $x_1,\dotsc,x_n$ of $F'_n$) and the functor $f':\Ga_n\rar{}BF'_n$ from \eqref{e:thebordism} is $C_n$-equivariant (in the strict sense).
 \mbr
\item  The previous remark yields a functorial isomorphism
\begin{equation}\label{e:my_pivotal-stacks1}
\Phi'_n(M_n,M_1,\dotsc,M_{n-1}) \rar{\simeq} \Phi'_n(M_1,M_2,\dotsc,M_n), \quad M_i\in\cM
\end{equation}
whose $n$-th power (in the obvious sense) equals the identity.
\end{enumerate}
\end{rems}

To prove Proposition \ref{p:pivotal-stacks}(a) we have to show that the isomorphism $\psi$ defined in
\S\ref{sss:pivotal-stacks} has properties
\eqref{e:pretty2}-\eqref{e:pretty1}. Using Remark~\ref{r:cyclic_action}(ii) for $n=2,3$ and
 Lemma~\ref{l:Hom(XYZ,K)}, we obtain  functorial isomorphisms
\begin{equation}  \label{e:my_pivotal-stacks2}
\Hom(M_1*M_2,K_{\cM}) \rar{\simeq} \Hom(M_2*M_1,K_{\cM}), \qquad M_1,M_2\in\cM,
\end{equation}
\begin{equation}  \label{e:my_pivotal-stacks3}
\Hom(M_1*M_2*M_3,K_{\cM}) \rar{\simeq} \Hom(M_3*M_1*M_2,K_{\cM}), \quad M_1,M_2,M_3\in\cM
\end{equation}
such that the square of \eqref{e:my_pivotal-stacks2} and the cube of \eqref{e:my_pivotal-stacks3} are equal
to the identity. It is easy to see that \eqref{e:my_pivotal-stacks2} equals the isomorphism $\psi_{M_1,M_2}$ defined by \eqref{e:def_of_psi}
and  \eqref{e:my_pivotal-stacks3} equals $\psi_{M_1*M_2,M_3\,}$. Properties \eqref{e:pretty2}-\eqref{e:pretty1} follow.

\subsubsection{Proof of Proposition \ref{p:pivotal-stacks}(b)}\label{sss:proof-p:pivotal-stacks-b}
One proves the assertion using Lemma \ref{l:obvious-isomorphism-abstract} in exactly the same way as explained in \S\ref{sss:proof-l:obvious-isomorphism}. We skip the details.


%
%
%
%


\subsubsection{Proof of Proposition \ref{p:ribbon-stacks}(a)}\label{sss:proof-p:ribbon-stacks-a}
We have to show that for each $M_1,M_2\in\cM$ the isomorphism
\begin{equation}  \label{e:ribbon-stacks-1a}
\psi_{M_1,M_2}:\Hom(M_1*M_2,K_{\cM}) \rar{\simeq} \Hom(M_2*M_1,K_{\cM})
\end{equation}
is equal to
\begin{equation}  \label{e:ribbon-stacks-2a}
\be^*_{M_2,M_1}\circ(\id_{M_1}*\te_{M_2})^*:\Hom(M_1*M_2,K_{\cM}) \rar{\simeq} \Hom(M_2*M_1,K_{\cM}).
\end{equation}
To this end, we will describe the isomorphisms \eqref{e:ribbon-stacks-1a}-\eqref{e:ribbon-stacks-2a} in terms of $\sCob$.

\mbr

By Lemma~\ref{l:Hom(XYZ,K)}, we have
$\Hom (M_1*M_2,K_{\cM} )=\Hom\bigl(\Phi'_2(M_1,M_2),K_{\cM_0}\bigr)$, where
$\Phi'_2(M_1,M_2)=\varepsilon (M_1*M_2)$ comes from the bordism
\begin{equation} \label{e:ribbon-stacks-3a}
\Ga_2\rar{f'} BF'_2\lar{}\varnothing \, ,
 \quad\quad \Ga_2:=B\bZ\bigsqcup B\bZ \, ,
 \end{equation}
which is a special case of \eqref{e:thebordism}. The isomorphism \eqref{e:ribbon-stacks-1a} comes from
the auto\-equivalence
\begin{equation} \label{e:ribbon-stacks-4a}
\xymatrix{
  \Ga_2 \ar[rr]^{f'} \ar[dd]_{\tau} & & BF'_2 \ar[dd]^{\xi'} & \\
   & & & \\ 
  \Ga_2 \ar[rr]^{f'} & & BF'_2 &
   }
\end{equation}
of diagram \eqref{e:ribbon-stacks-3a} described in Remark~\ref{r:cyclic_action}(i); namely, $\tau$ interchanges the two copies of $B\bZ$ and $\xi'$ comes from the automorphism $F'_2$ interchanging the generators
$x_1,x_2\in F'_2$.

\mbr

On the other hand, the isomorphism \eqref{e:ribbon-stacks-2a} comes from the composition
\begin{equation} \label{e:ribbon-stacks-5a}
M_2*M_1 \xrar{\ \ \be_{M_2,M_1}\ \ } M_1*M_2 \xrar{\ \ \id_{M_1}*\te_{M_2}\ \ } M_1*M_2\, .
\end{equation}
Recall that the functor $(M_1,M_2)\mapsto M_1*M_2$ comes from the bordism
\begin{equation} \label{e:ribbon-stacks-6a}
\Ga_2\rar{f}BF_2\lar{g}B\bZ \, ,
\end{equation}
which is a special case of \eqref{e:n-th_bordism}.

\begin{lem}   \label{l:computation_in_sCob}
The composition \eqref{e:ribbon-stacks-5a} comes from the following autoequivalence
\begin{equation}  \label{e:GV-stacks-mydiagram}
\xymatrix{
  \Ga_2 \ar[rr]^{f} \ar[dd]_{\tau} & & BF_2 \ar[dd]^{\xi} &  &B\bZ\ar[ll]_g \ar[dd]^{\Id} \\
   & & & \\
  \Ga_2 \ar[rr]^{f} & & BF_2 &  &B\bZ\ar[ll]_g
   }
\end{equation}
of diagram \eqref{e:ribbon-stacks-6a}: the left vertical arrow is the same as in \eqref{e:ribbon-stacks-4a},
$\xi$ comes from the automorphism of $F_2$ interchanging the generators $x_1,x_2\in F_2$,
the isomorphism $\xi\circ f\iso f\circ\tau$ equals the identity, and the isomorphism $\xi\circ g\iso g$
is given by the element $x_2^{-1}\in F_2$ (see Remark \ref{rems:computation_in_sCob}(1)).
\end{lem}

The desired equality between \eqref{e:ribbon-stacks-1a} and \eqref{e:ribbon-stacks-2a} follows from
Lemma~\ref{l:computation_in_sCob} because the autoequivalence of the diagram \eqref{e:ribbon-stacks-3a} induced by \eqref{e:GV-stacks-mydiagram} equals \eqref{e:ribbon-stacks-4a}. Thus it remains to prove
Lemma~\ref{l:computation_in_sCob}.

\begin{proof}[Proof of Lemma \ref{l:computation_in_sCob}]
Combining Definitions \ref{d:twist-general} and \ref{d:braiding-general}, it is easy to check that the composition \eqref{e:ribbon-stacks-5a} comes from the autoequivalence
\begin{equation}  \label{e:GV-stacks-newdiagram}
\xymatrix{
  \Ga_2 \ar[rr]^{f} \ar[dd]_{\tau} & & BF_2 \ar[dd]^{\nu} & & B\bZ\ar[ll]_g \ar[dd]^{\Id} \\
   & & & \\
  \Ga_2 \ar[rr]^{f} & & BF_2 & & B\bZ\ar[ll]_g
   }
\end{equation}
of diagram \eqref{e:ribbon-stacks-6a} in which the notation is the same as in Definition \ref{d:braiding-general}, the isomorphism $\nu\circ f\rar{\simeq}f\circ\tau$ is given by the pair $(x_2,x_2)\in F_2^2$ (cf.~Remark \ref{rems:computation_in_sCob}(2)) and the isomorphism $\nu\circ g\rar{\simeq}g\circ\Id$ is the identity map.
Both \eqref{e:GV-stacks-mydiagram} and \eqref{e:GV-stacks-newdiagram} define 1-isomorphisms between
the bordisms $\Ga_2\rar{f}BF_2\lar{g}B\bZ$ and $\Ga_2\rar{f\circ\tau}BF_2\lar{g}B\bZ$. To prove the lemma,
it suffices to show that these 1-isomorphisms are 2-isomorphic. This means constructing an isomorphism
$\nu\iso\xi$ such that the corresponding isomorphisms $\nu\circ f\iso\xi\circ f$ and $\nu\circ g\iso\xi\circ g$
are equal to the compositions
\[
\nu\circ f\iso f\circ\tau\iso\xi\circ f, \quad \nu\circ g\iso g\iso\xi\circ g
\]
(in each of the compositions the first arrow comes from \eqref{e:GV-stacks-mydiagram} and the second one from \eqref{e:GV-stacks-newdiagram}\,). The isomorphism $\nu\rar{\simeq}\xi$
corresponding to the element $x_2\in F_2$ has the desired properties.
\end{proof}

\subsubsection{Proof of Proposition \ref{p:ribbon-stacks}(b)}\label{sss:proof-p:ribbon-stacks-b}
By Remark \ref{remark-from-GV}, it suffices to check that for all $M_1,M_2\in\cM$ the automorphism
$\te_{M_1}*\te_{M_2}^{-1}\in\Aut (M_1*M_2)$ induces the identity map from $\Hom(M_1*M_2,K_{\cM})$ to itself. By Lemma \ref{l:Hom(XYZ,K)}, it is enough to show that the automorphism
\begin{equation} \label{e:1ribbon-stacks-b}
\Phi'_2(\te_{M_1},\te_{M_2}^{-1})\in\Aut \Phi'_2(M_1,M_2)
\end{equation}
is trivial. By the definition of $\Phi'_2$ (see \S\ref{sss:A_formula}), the automorphism \eqref{e:1ribbon-stacks-b}
comes from a certain automorphism of the bordism
\begin{equation} \label{e:2ribbon-stacks-b}
\Ga_2\rar{f'} BF'_2\lar{}\varnothing \, ,
\end{equation}
which is a special case of the bordism \eqref{e:thebordism}. A general automorphism of the bordism
\eqref{e:2ribbon-stacks-b} is defined by a pair $(\al ,a)$ consisting of an equivalence $\al:BF'_2\iso BF'_2$ and an isomorphism of functors $a:\al\circ f'\iso  f'\,$; a pair $(\al ,a)$ corresponds to the identity automorphism of \eqref{e:thebordism} if $a$ comes from an isomorphism $\al\iso\Id_{BF'_2}\,$.
In view of Definition \ref{d:twist-general}, the automorphism in question corresponds to $\al =\Id_{BF'_2}\,$, $a_{\ga_1}=x_1\,$, $a_{\ga_2}=x_2^{-1}\,$,
where $\ga_1, \ga_2$ are the two objects of $\Ga_2$ and $x_1, x_2$ are the generators of $F'_2$.
Since $x_2^{-1}=x_1$ this automorphism is trivial.


\subsection{Some remarks on the $\infty$-categorical setting}
\label{ss:infty-cats}
\subsubsection{}
It is becoming customary to define $\sD (\sY )$ and  $\sD^- (\sY )$ as (stable) $\infty$-cate\-go\-ries rather
than merely as categories. In this setting the
construction of the theories $Z^-_{\sX}$ and $Z_{\sX}$
given in \S\ref{sss:unboundedZ}-\ref{sss:boundedZ} 
still goes through if the definition of pre-sTFT is modified accordingly.
Namely, the 
2-groupoid of bordisms defined in \S\ref{sss:sCob} should not be truncated to a 
1-groupoid; then 
$\sCob$ becomes a $(3,1)$-category rather than a $(2,1)$-category.

\subsubsection{}   \label{sss:the_other_diff}
In \cite{BFN08} Ben-Zvi, Francis, and Nadler consider the quasicoherent derived category of $\sX^Y$,
where $\sX$ is a \emph{derived} stack (rather than a ``classical" one) and $Y$ is \emph{any} topological
space (rather than a classifying space of a groupoid). This degree of generality would be useless to treat
the categories (or $\infty$-categories) $\sD (\sX^Y)$ and  $\sD^- (\sX^Y)$. Reason: unlike the quasicoherent case, for any derived stack $\sZ$ one has $\sD (\sZ )=\sD (\sZ^{cl})$ and  $\sD^- (\sZ )=\sD^- (\sZ^{cl})$, where
$\sZ^{cl}$ stands for the  classical stack underlying $\sZ$. On the other hand, if $\sX$ is a derived stack and
$Y$ is a topological space then $(\sX^Y )^{cl}$ depends only on $\sX^{cl}$ and the fundamental groupoid
$\Pi (Y)$. To see this, note that for any classical scheme $S$ one has
\[
\sX^Y (S):=\Mor (Y,\sX (S))=\Mor (Y,\sX^{cl} (S))=\Mor (\Pi (Y),\sX^{cl} (S));
\]
the latter equality holds because $\sX^{cl} (S))$ is a usual groupoid rather than an $\infty$-groupoid.

\section{Equivalence of two definitions of $\sD_G(X)$}\label{a:equivalence-equivariant-derived}

In this appendix $k$ denotes an algebraically closed field of arbitrary characteristic and $G$ is an algebraic group over $k$ acting on a scheme $X$ of finite type over $k$. We form the quotient stack $\sY:=G\backslash X$, write $\sD(\sY)=D^b_c(\sY,\ql)$ and let $\sD_G^{naive}(X)$ denote the category constructed in Definition \ref{d:equiv-derived} (where it was denoted $\sD_G(X)$). If $\fq:X\rar{}\sY$ is the quotient morphism, we obtain a Cartesian diagram
\begin{equation}\label{e:cart-diag}
\xymatrix{
  G\times X \ar[rr]^{\al} \ar[d]_{\pi} & & X \ar[d]^{\fq} \\
  X \ar[rr]^{\fq} & & \sY
   }
\end{equation}
where $\al$ is the action map and $\pi$ is the second projection. Hence given $N\in\sD(\sY)$, the pullback $\fq^*(N)\in\sD(X)$ acquires an isomorphism $\phi:\al^*\fq^*(N)\rar{\simeq}\pi^*\fq^*(N)$, which is easily seen to satisfy condition \eqref{e:compatibility} of Definition \ref{d:equiv-derived}. Therefore $\fq^*$ can be viewed as a functor $\sD(\sY)\rar{}\sD_G^{naive}(X)$.

\begin{prop}\label{p:equivalence-definitions-equivariant-category}
If $G^\circ$ is unipotent, the functor $\fq^*:\sD(\sY)\rar{}\sD_G^{naive}(X)$ is an equivalence.
\end{prop}

To prove the proposition we will construct a functor $\fq_*^G:\sD_G^{naive}(X)\rar{}\sD(\sY)$ and show that when $G^\circ$ is unipotent, the functors $\fq_*^G$ and $\fq^*$ are quasi-inverse to each other. Given $g\in G(k)$, by a slight abuse of notation we will also denote by $g:X\rar{\simeq}X$ the automorphism $x\mapsto\al(g,x)$. We have $\fq\circ g=\fq$. If $(M,\phi)\in\sD_G^{naive}(X)$, then for each $g\in G(k)$, the isomorphism $\phi$ induces an isomorphism $M\rar{\simeq}g_*(M)$ and hence an automorphism of $\fq_*(M)$. In this way we obtain an action of $G(k)$, viewed as an abstract group, on the object $\fq_*(M)$. Note that since $\fq$ is representable, the functor $\fq_*$ preserves boundedness, so $\fq_*(M)\in\sD(\sY)$.

\begin{lem}\label{lemm:1}
If $g\in G^\circ(k)$, then $g$ acts trivially on $\fq_*(M)$.
\end{lem}

The lemma is proved by a standard continuity argument, which we include for completeness and the lack of a suitable reference.

\begin{proof}
The statement becomes obvious if one rephrases the definition of the $G(k)$-action on $\fq_*(M)$ as follows. Set $N:=\fq_*(M)\in\sD(\sY)$. The $G$-equivariant structure on $M$ induces a $G$-equivariant structure on $N$ with respect to the trivial $G$-action on $\sY$, that is, an isomorphism $\phi_{\sY}:\pi_{\sY}^*(N)\rar{\simeq}\pi_{\sY}^*(N)$, where $\pi_{\sY}:G\times\sY\rar{}\sY$ is the projection. Rewrite $\phi_{\sY}$ as a morphism
\begin{equation}\label{e:*}
N \rar{} \pi_{\sY*}\pi_{\sY}^*(N)=N\tens R\Ga(G,\ql)
\end{equation}
Then the automorphism of $N$ corresponding to $g\in G(k)$ is the composition of \eqref{e:*} with the morphism $\id_N\tens\ev_g:N\tens R\Ga(G,\ql)\rar{}N$, where $\ev_g:R\Ga(G,\ql)\rar{}R\Ga(\Spec k,\ql)=\ql$ is induced by $g:\Spec k\rar{}G$. Clearly $\ev_g$ depends only on the image of $g$ in $G(k)/G^\circ(k)$.
\end{proof}

By Lemma \ref{lemm:1}, for every $M=(M,\phi)\in\sD_G^{naive}(X)$, we obtain an action of the finite group $\pi_0(G)\cong G(k)/G^\circ(k)$ on $\fq_*(M)$. Since $\sD(\sY)$ is a $\ql$-linear Karoubi-complete category, the endomorphism $P_M:=\frac{1}{\abs{\pi_0(G)}}\sum_{g\in\pi_0(G)}g$ of $\fq_*(M)$ has an image, which is also the kernel of $\id-P_M$. We denote it by $\fq_*^G(M)$, and it is clear that this construction defines a functor $\fq_*^G:\sD_G^{naive}(X)\rar{}\sD(\sY)$.

\mbr

The definition of $\fq_*^G$ shows that for each $M\in\sD_G^{naive}(X)$ we have natural morphisms $\fq_*^G(M)\to\fq_*(M)\to\fq_*^G(M)$. In particular, if $M\in\sD_G^{naive}(X)$, then the adjunction morphism $\fq^*\fq_*(M)\rar{}M$ induces a morphism
\begin{equation}\label{e:adj1}
\fq^*\fq_*^G(M)\rar{}M,
\end{equation}
and if $N\in\sD(\sY)$, then the adjunction morphism $N\rar{}\fq_*\fq^*(N)$ induces a morphism
\begin{equation}\label{e:adj2}
N\rar{}\fq_*^G\fq^*(N)
\end{equation}
Proposition \ref{p:equivalence-definitions-equivariant-category} follows from the next

\begin{lem}\label{lemm:2-3}
If $G^\circ$ is unipotent, then \eqref{e:adj1} and \eqref{e:adj2} are isomorphisms.
\end{lem}

\begin{rem}
It is not hard to show in general that $\fq_*^G:\sD_G^{naive}(X)\rar{}\sD(\sY)$ is right adjoint to $\fq^*:\sD(\sY)\rar{}\sD_G^{naive}(X)$, but we do not need this fact.
\end{rem}

\begin{proof}[Proof of Lemma \ref{lemm:2-3}]
Smooth base change \cite[\S12]{Las-Ols06} with respect to the morphism $\fq:X\rar{}\sY$ (cf.~diagram \eqref{e:cart-diag}) reduces the proof of the lemma to the special case where $X=G\times Y$ for some scheme $Y$ of finite type over $k$ and the $G$-action on $X$ is given by the left multiplication action of $G$ on itself. In this case the quotient morphism $\fq:X\to\sY$ can be identified with the second projection $\pr_2:G\times Y\to Y$. It is straightforward to check that the functor $\pr_2^*:\sD(Y)\rar{}\sD^{naive}_G(G\times Y)$ is an equivalence, with quasi-inverse $\iota^*:\sD^{naive}_G(G\times Y)\rar{}\sD(Y)$, where $\iota:Y\to G\times Y$ is given by $y\mapsto (1,y)$. On the other hand, since $G^\circ$ is unipotent, it is isomorphic to an affine space as a variety over $k$, so $H^j(G,\ql)=0$ for $j\geq 1$ and $H^0(G,\ql)$ can be identified with the space of functions $\pi_0(G)\rar{}\ql$ on which $\pi_0(G)$ acts by translations. These observations and the K\"unneth formula imply that the maps \eqref{e:adj1} and \eqref{e:adj2} are isomorphisms.
\end{proof}


\begin{thebibliography}{55}

\bibitem[Ad73]{adams} J.F.~Adams, {\em Idempotent functors in homotopy theory}, in: ``Manifolds --- Tokyo 1973 (Proc. Internat. Conf., Tokyo, 1973),'' pp.~247--253. Univ. Tokyo Press, Tokyo, 1975.

\bibitem[Ba79]{Barr79} M.~Barr,  {\em $\ast$-autonomous categories} (with an appendix by Po Hsiang Chu), Lecture Notes in Mathematics \textbf{752}, Springer-Verlag, Berlin, 1979.

\bibitem[Ba95]{Barr95} M.~Barr,  {\em Nonsymmetric $\ast$-autonomous categories}, Theoret. Comput. Sci. \textbf{139} (1995), no. 1-2, 115--130.

\bibitem[Ba96]{Barr96} M.~Barr,  {\em $\ast$-autonomous categories, revisited}, J. Pure Appl. Algebra \textbf{111} (1996), no. 1-3, 1--20.

\bibitem[Ba99]{Barr99} M.~Barr,  {\em $\ast$-autonomous categories: once more around the track}, Theory Appl. Categ. \textbf{6} (1999), 5--24 (electronic).

\bibitem[Be80]{begueri} L.~Begueri, {\em Dualit\'e sur un corps local \`a corps r\'esiduel alg\'ebriquement clos}, M\'em. Soc. Math. France (N.S.) 1980/81, no.~4.

\bibitem[BBD82]{bbd} A.A.~Beilinson, J.~Bernstein and P.~Deligne, \emph{Faisceaux Pervers}, in: ``Analyse et topologie sur les espaces singuliers (I)'', Ast\'erisque \textbf{100}, 1982.

\bibitem[BeDr04]{chiral} A.~Beilinson and V.~Drinfeld, ``Chiral Algebras'', Amer. Math. Soc. Colloq. Publ. \textbf{51}, American Mathematical Society, Providence, RI, 2004.

\bibitem[BK01]{BK} B.~Bakalov and A.Kirillov, Jr, ``Lectures on tensor categories and modular functors", University Lecture Series, \textbf{21}. American Mathematical Society, Providence, RI, 2001.

\bibitem[BFN08]{BFN08} D.~Ben-Zvi, J.~Francis, and D.~Nadler, {\em Integral transforms and Drinfeld centers in derived algebraic geometry}, J. Amer. Math. Soc. \textbf{23} (2010), no.~4, 909--966.

\bibitem[BL94]{ber-lunts} J.~Bernstein and V.~Lunts, ``Equivariant Sheaves and Functors'', Lecture Notes in Math. \textbf{1578}, Springer-Verlag, Berlin, 1994.


\bibitem[Bo10]{characters} M.~Boyarchenko, {\em Characters of unipotent groups over finite fields}, Selecta Math. \textbf{16} (2010), no.~4, 857--933.

\bibitem[Bo11]{characters-char-sheaves} M.~Boyarchenko, {\em Character sheaves and characters of unipotent groups over finite fields}, {\tt arXiv:1006.2476}, to appear in the American Journal of Mathematics.

\bibitem[BD06]{intro} M.~Boyarchenko and V.~Drinfeld, {\em A motivated introduction to character sheaves and the orbit method for unipotent groups in positive characteristic}, Preprint, September 2006, arXiv: {\tt math.RT/0609769}

\bibitem[BD08]{talk} M.~Boyarchenko and V.~Drinfeld, {\em Character sheaves on unipotent groups in characteristic $p>0$}, a talk at the conference ``Current Developments and Directions in the Langlands Program'', Northwestern University, May 14, 2008. The slides are available online (in PDF) at the URL \url{http://www.math.uchicago.edu/~mitya/LanglandsConf.pdf} and can also be found at {\tt arXiv:1301.0025}

\bibitem[BD11]{GV} M.~Boyarchenko and V.~Drinfeld, {\em A duality formalism in the spirit of Grothendieck and Verdier}, {\tt arXiv:1108.6020}, to appear in Quantum Topology.


\bibitem[Br81]{breen-exts} L.~Breen, {\em Extensions du groupe additif sur le site parfait}, in:  ``Algebraic surfaces (Orsay, 1976--78)'', pp.~238--262, Lecture Notes in Math. \textbf{868}, Springer, Berlin-New York, 1981.


\bibitem[CS04]{CS04} M.~Chas and D.~Sullivan, {\em Closed string operators in topology leading to Lie bialgebras and higher string algebra}, in: `` The legacy of Niels Henrik Abel", 771--784, Springer, Berlin, 2004.

\bibitem[Da10]{swarnendu} S.~Datta, {\em Metric groups attached to biextensions}, Transformation Groups \textbf{15} (2010), no.~1, 72--91.

\bibitem[DFH74]{dfh1} A.~Deleanu, A.~Frei and P.~Hilton, {\em Generalized Adams completion}, Cahiers Topologie G\'eom. Diff\'erentielle \textbf{15} (1974), 61--82.

\bibitem[DFH75]{dfh2} A.~Deleanu, A.~Frei and P.~Hilton, {\em Idempotent triples and completion}, Math. Z. \textbf{143} (1975), 91--104.

\bibitem[De76]{deligne} P.~Deligne, Letter to D.~Kazhdan, November 29, 1976 (unpublished).

\bibitem[De80]{deligne-weil-2} P.~Deligne, {\em La conjecture de Weil II}, Publ. Math. IHES \textbf{52} (1980), 137--252.

\bibitem[De90]{deligne-cat-tan} P.~Deligne, {\em Cat\'egories tannakiennes}, in: ``The Grothendieck Festschrift'', Vol. II,  111--195, Progr. Math. \textbf{87}, Birkh\"auser Boston, Boston, MA, 1990.

\bibitem[SGA$4\frac{1}{2}$]{sga4.5} P.~Deligne, with J.-F.~Boutot, L.~Illusie and J.-L.~Verdier, ``SGA $4\frac{1}{2}$: Cohomologie \'Etale'', Lecture Notes in Math. \textbf{569}, Springer, Heidelberg, 1977.

\bibitem[De10]{tanmay} T.~Deshpande,  {\em Heisenberg idempotents on unipotent groups}, Math. Res. Lett. \textbf{17} (2010), no.~3, 415--434.

\bibitem[DGNO10]{dgno} V.~Drinfeld, S.~Gelaki, D.~Nikshych and V.~Ostrik, {\em On Braided Fusion Categories I}, Selecta Math. (N.S.) \textbf{16} (2010), no.~1, 1--119.

\bibitem[Ek90]{Ekedahl} T.~Ekedahl, {\em On the adic formalism}, in: ``The Grothendieck Festschrift'', vol.~II, Progr. Math., \textbf{87}, pp. 197--218, Birkh\"auser, Boston, MA, 1990.


\bibitem[ENO05]{ENO} P.~Etingof, D.~Nikshych and V.~Ostrik, {\em On fusion categories}, Annals of Math. \textbf{162} (2005), 581--642.

\bibitem[FY89]{FY}  P.~J.~Freyd and D.~N.~Yetter, {\em Braided compact closed categories with applications to low-dimensional topology}, Adv. Math. \textbf{77} (1989), no. 2, 156--182.

\bibitem[Gr65]{greenberg} M.J.~Greenberg, {\em Perfect closures of rings and schemes}, Proc. AMS \textbf{16} (1965), 313--317.

\bibitem[EGA\,IV-3]{ega4-3} A.~Grothendieck, ``\'El\'ements de g\'eom\'etrie alg\'ebrique. IV. \'Etude locale des sch\'emas et des morphismes de sch\'emas. III.'' Inst. Hautes \'Etudes Sci. Publ. Math. \textbf{28}, 1966.

\bibitem[HRV08]{HRV} T.~Hausel and F.~Rodriguez-Villegas, {\em Mixed Hodge polynomials of character varieties (With an appendix by N.~M.~Katz)}, Invent. Math. \textbf{174} (2008), no. 3, 555--624.









\bibitem[Jan88]{Jannsen} U.~Jannsen, {\em Continuous \'etale cohomology}, Math. Ann., \textbf{280} (1988), 207--245.

\bibitem[JS93]{joyal-street} A.~Joyal and R.~Street, {\em Braided tensor categories},  Adv. Math. \textbf{102} (1993), no.~1, 20--78.

\bibitem[Ka09]{masoud} M.~Kamgarpour, {\em Stacky abelianization of algebraic groups}, Transform. Groups \textbf{14} (2009), no.~4, 825--846.

\bibitem[KS06]{categories-sheaves} M.~Kashiwara and P.~Schapira, ``Categories and Sheaves'', Grundlehren Math. Wiss. \textbf{332}, Springer-Verlag, Berlin, 2006.




\bibitem[KL85]{katz-laumon} N.M.~Katz and G.~Laumon, {\em Transformation de Fourier et majoration de sommes exponentielles}, Publ. Math. IHES \textbf{62} (1985), 361--418.

 \bibitem[Ke74]{Kelly74} G.~M~Kelly, {\em On clubs and doctrines}, Category Seminar (Proc. Sem., Sydney, 1972/1973), pp. 181--256. Lecture Notes in Math., Vol. 420, Springer, Berlin, 1974.

\bibitem[LO06]{Las-Ols06} Y.~Laszlo and M.~Olsson, {\em The six operations for sheaves on Artin stacks II: Adic coefficients}, Publ. Math. IHES. \textbf{107} (2008), 169--210.

\bibitem[LM00]{LM} G.~Laumon and L.~Moret-Bailly, ``Champs alg\'ebriques", Springer-Verlag, Berlin, 2000.

\bibitem[Lur07]{Lur07} J.~Lurie, {\em Derived Algebraic Geometry III: Commutative Algebra}, e-print, 2007, arXiv: {\tt math.CT/0703204}

\bibitem[Lur09]{Lur09} J.~Lurie, {\em On the Classification of Topological Field Theories}, e-print, 2009, arXiv: {\tt math.CT/0905.0465}

\bibitem[Lu03]{lusztig} G.~Lusztig, {\em Character sheaves and generalizations}, in: ``The unity of mathematics'' (editors: P.~Etingof, V.~Retakh, I.~M.~Singer), 443--455, Progress in Math. \textbf{244}, Birkh\"auser Boston, Boston, MA, 2006, arXiv: {\tt math.RT/0309134}


\bibitem[McL98]{m} S.~MacLane, ``Categories for the working mathematician'', 2nd ed. Graduate Texts in Mathematics \textbf{5}. Springer-Verlag, New York, 1998.




\bibitem[Sa96]{saibi} M.~Saibi, {\em Transformation de Fourier-Deligne sur les groupes unipotents}, Ann. Inst. Fourier (Grenoble) \textbf{46} (1996), no.~5, 1205--1242.

\bibitem[SW03]{SW03} P.~Salvatore and N.~Wahl, {\em Framed discs operads and Batalin-Vilkovisky algebras}, Q. J. Math.  \textbf{54} (2003), no. 2, 213--231.

\bibitem[Seg74]{Seg74} G.~Segal, {\em Categories and cohomology theories}, Topology \textbf{13} (1974), 293--312.

\bibitem[Se60]{serre} J.-P.~Serre, {\em Groupes proalg\'ebriques}, Publ. Math. IHES \textbf{7} (1960).





\bibitem[S04]{S04} D.~Sullivan, {\em Open and closed string field theory interpreted in classical algebraic topology}, in: ``Topology, geometry and quantum field theory", 344--357, London Math. Soc. Lecture Note Ser. \textbf{308}, Cambridge Univ. Press, Cambridge, 2004.

\bibitem[Til98]{Til98} U.~Tillmann, {\em $\sS$-structures for $k$-linear categories and the definition of a modular functor}, J. London Math. Soc. (2) \textbf{58} (1998), no. 1, 208--228.

\bibitem[Til00]{Til00} U.~Tillmann, {\em Higher genus surface operad detects infinite loop spaces}, Math. Ann. \textbf{317}  (2000), no.~3, 613--628.

\end{thebibliography}
\end{document}